\documentclass[letterpaper,11pt,oneside,reqno]{article}

\usepackage{amsmath,amssymb,amsthm,amsfonts}
\usepackage[pdftex,backref=page,colorlinks=true,linkcolor=blue,citecolor=red]{hyperref}
\usepackage[alphabetic,nobysame]{amsrefs}
\usepackage{graphicx,color}
\usepackage{upgreek}
\usepackage[mathscr]{euscript}

\allowdisplaybreaks
\numberwithin{equation}{section}

\usepackage{tikz}
\usetikzlibrary{shapes,arrows,positioning,decorations.markings}

\usepackage{array}
\usepackage{adjustbox}
\usepackage{cleveref}
\usepackage{enumerate}

\usepackage[DIV=12]{typearea}

\newtheorem{proposition}{Proposition}[section]
\newtheorem{lemma}[proposition]{Lemma}
\newtheorem{corollary}[proposition]{Corollary}
\newtheorem{theorem}[proposition]{Theorem}
\theoremstyle{definition}
\newtheorem{definition}[proposition]{Definition}
\newtheorem{remark}[proposition]{Remark}

\begin{document}

\title{Free Fermion Six Vertex Model:\\Symmetric Functions and Random Domino Tilings}
\author{Amol Aggarwal, Alexei Borodin, Leonid Petrov, Michael Wheeler}

\date{}

\maketitle

\begin{abstract}
	Our work deals with symmetric rational functions and probabilistic models based on the fully inhomogeneous six vertex (ice type) model satisfying the free fermion condition. Two families of symmetric rational functions $F_\lambda,G_\lambda$ are defined as certain partition functions of the six vertex model, with variables corresponding to row rapidities, and the labeling signatures $\lambda=(\lambda_1\ge \ldots\ge \lambda_N)\in \mathbb{Z}^N$ encoding boundary conditions. These symmetric functions generalize Schur symmetric polynomials, as well as some of their variations, such as factorial and supersymmetric Schur polynomials. Cauchy type summation identities for $F_\lambda,G_\lambda$ and their skew counterparts follow from the Yang--Baxter equation. Using algebraic Bethe Ansatz, we obtain a double alternant type formula for $F_\lambda$ and a Sergeev--Pragacz type formula for $G_\lambda$.

	In the spirit of the theory of Schur processes, we define probability measures on sequences of signatures with probability weights proportional to products of our symmetric functions. We show that these measures can be viewed as determinantal point processes, and we express their correlation kernels in a double contour integral form. We present two proofs: The first is a direct computation of Eynard--Mehta type, and the second uses non-standard, inhomogeneous versions of fermionic operators in a Fock space coming from the algebraic Bethe Ansatz for the six vertex model.
	
	We also interpret our determinantal processes as random domino tilings of a half-strip with inhomogeneous domino weights. In the bulk, we show that the lattice asymptotic behavior of such domino tilings is described by a new determinantal point process on $\mathbb{Z}^{2}$, which can be viewed as an doubly-inhomogeneous generalization of the extended discrete sine process.
\end{abstract}

\vspace{.6cm}

\setcounter{tocdepth}{1}
\tableofcontents
\setcounter{tocdepth}{3}

\section{Introduction}
\label{sec:introduction}

\subsection{Preface}
\label{sub:background}

\emph{Determinantal random point processes} (or \emph{fields}) 
originated in random matrix theory in the 1960s and were first singled out as a class by 
Macchi in 1975 \cite{macchi1975coincidence}
under the name \emph{fermion point processes}.
The book of Anderson--Guionnet--Zeitouni 
\cite[Section 4.6]{AndersonGuionnetZeitouniBook}
provides a brief historical summary of the random matrix origins.
The term \emph{determinantal} was adopted around the year 2000, 
see Borodin--Olshanski \cite{borodin2000distributions} and Soshnikov 
\cite{Soshnikov2000}.
By now there are quite a few surveys discussing various aspects of 
determinantal processes by 
Soshnikov \cite{Soshnikov2000},
Lyons \cite{lyons2003determinantal},
Johansson \cite{Johansson2005lectures},
K\"onig \cite{Konig2005},
Hough--Krishnapur--Peres--Virag \cite{peres2006determinantal},
Borodin \cite{Borodin2009},
Kulesza--Taskar \cite{kulesza2012determinantal},
and Decreusefond--Flint--Privault--Torrisi \cite{decreusefond2016determinantal}.

One of the most important determinantal processes
is the \emph{sine process} that goes back to Mehta--Gaudin \cite{mehta1960density} and Dyson
\cite{dyson1962brownian}. It describes universal bulk\footnote{The term ``bulk'' refers to the parts of the system
	where the space can be rescaled to form growing regions with unit
	particle density.} 
asymptotics of large determinantal systems in one space dimension, see, e.g.,
Yau \cite{yau2013wigner} for a historical overview. 
For one-dimensional discrete point processes (i.e., random subsets of
$\mathbb{Z}$), the corresponding universal object is the \emph{discrete sine process} 
introduced by Borodin--Okounkov--Olshanski \cite{Borodin2000b}.

Besides being universal bulk limits in one dimension, both
the continuous and the discrete sine processes
admit natural extensions to two dimensions arising, respectively,
from the Dyson Brownian motion, see Dyson \cite{dyson1962brownian},
Nagao--Forrester \cite{nagao1998multilevel},
and random plane partitions, see Okounkov--Reshetikhin 
\cite{okounkov2003correlation},
or more general dimer models, cf. Kenyon--Okounkov--Sheffield \cite{KOS2006} and
Johansson \cite{Johansson2005arctic}.
In the discrete case (which is the focus of the present paper),
the \emph{two-dimensional} (also called \emph{extended}) \emph{sine process} 
admits
a natural description
as a unique (cf. Sheffield \cite{Sheffield2008}) translation 
invariant ergodic Gibbs measure
on point configurations in $\mathbb{Z}^2$
of a given \emph{slope}. The slope consists of two real or a 
single complex parameter that encodes the particles' densities 
along the two coordinate directions.
In this case the Gibbs property means that the probability law
of the random configuration is invariant under 
uniform resampling in any finite window, conditioned
on the configuration on the boundary of this window.
The fact that such a rich family of Gibbs measures in 
$\mathbb{Z}^2$ enjoys a completely 
explicit determinantal description of their correlations is 
remarkable and very rare.

\medskip

The main probabilistic outcome of the present work is
the introduction of a wide class of new \emph{inhomogeneous
deformations} of the extended discrete sine process.
These deformations are determinantal point processes on $\mathbb{Z}^2$ with 
very explicit correlation kernels that depend, in addition to the complex slope 
parameter, on four bi-infinite sequences of real parameters
associated with the horizontal and the vertical coordinate directions (two sequences
per each direction). They seem to be out of reach of existing approaches to 
deformations of the extended sine processes such as various versions of the 
Schur processes, fermionic Fock space formalism with the Boson--Fermion 
correspondence, random matrix type ensembles, or periodic dimer models.

Free parameters varying by rows and columns is a salient feature of integrable 
lattice models, and those are indeed behind our construction. 
More precisely, we start with the
\emph{free fermion six vertex model},
show that it is described by determinantal (fermion) point processes,
and in a bulk limit obtain the inhomogeneous deformations of the 
extended discrete sine process.

\medskip

The six vertex model, first introduced as a two-dimensional model for residual entropy of
water ice
by Pauling in 
1935 \cite{pauling1935structure}, is a classical model in statistical
mechanics that gave birth to the domain of integrable
(exactly solvable) lattice models; see the book of Baxter
\cite{baxter2007exactly} for an introduction, and also 
Reshetikhin \cite{reshetikhin2010lectures} for a more recent survey of the six 
vertex model.
Integrable lattice models is a vast domain, and the present work belongs to a 
subdomain dealing with symmetric functions and associated stochastic systems.

The theory of symmetric functions, a classical introduction to which is 
Macdonald's book \cite{Macdonald1995}, studies remarkable families of symmetric 
and associated nonsymmetric polynomials with origins in diverse areas
of group theory, combinatorics, representation theory, 
noncommutative harmonic analysis, probability, and mathematical physics.
There are many works highlighting connections between symmetric
functions
and integrable vertex models; for some of the earlier papers see 
Kirillov--Reshetikhin \cite{kirillov1988bethe},
Fomin--Kirillov \cite{fomin1994grothendieck,fomin1996yang},
Lascoux--Leclerc--Thibon \cite{lascoux1997flag},
Gleizer--Postnikov \cite{gleizer2000littlewood},
Tsilevich \cite{tsilevich2006quantum},
Lascoux \cite{lascoux20076},
Zinn--Justin \cite{ZinnJustin20096Vertex},
Brubaker--Bump--Friedberg \cite{brubaker2011schur},
Bump--McNamara--Nakasuji \cite{bump2011factorial}, and
Korff \cite{korff2013cylindric}.

We focus on the (asymmetric) six vertex model
with vertex weights $a_1,a_2,b_1,b_2,c_1,c_2$
satisfying the \emph{free fermion condition}
$a_1a_2+b_1b_2=c_1c_2$.
This condition corresponds to the vanishing of the quantity~$\Delta$ 
associated to the model. See the references in Baxter
\cite[Ch. 8.10.III]{baxter2007exactly},
and also
Felderhof
\cite{felderhof1973direct},
\cite{felderhof1973diagonalization2},
\cite{felderhof1973diagonalization3}
for earlier works on 
the free fermion six vertex model.

We consider the six vertex model in which
the free fermion condition holds at each lattice site, 
but otherwise the weights are fully inhomogeneous and
are determined by the parameters
$(x_i,r_i)$ and
$(y_j,s_j)$ which are constant along the lattice rows and
columns, respectively.
The $x$'s and $y$'s are known as the \emph{rapidities}, while the $r$'s and $s$'s 
are the \emph{spin parameters}. 
This particular parametrization ensures that the vertex weights satisfy
a version of the \emph{Yang--Baxter equation}, which is a key algebraic 
property powering our results.
We mainly employ the
Yang--Baxter equation in the form of 
quadratic relations for the row operators $A,B,C,D$
that are standard in the Algebraic Bethe Ansatz, cf. 
Faddeev \cite{Faddeev_Lectures}, Korepin--Bogoliubov--Izergin \cite[Part~VII]{QISM_book}.

The structure of integrable lattice models
(in particular, in our six vertex model)
is very special as it is
powered by connections to quantum groups.
Quantum groups are deformations of universal enveloping algebras
of classical Lie groups (and their generalizations), 
which possess certain additional structure,
in particular, $R$-matrices satisfying the Yang--Baxter equation.
In this language, the free fermion six vertex weights
correspond to the $R$-matrix of the rank 1
quantum affine superalgebra $U_q(\widehat{\mathfrak{sl}}(1|1))$.
A recent paper \cite{agg-bor-wh2020-sl1n} by a subset of the authors
presented a detailed study of symmetric functions related to the higher
rank quantum affine superalgebras $U_q(\widehat{\mathfrak{sl}}(1|n))$
with $n>1$. In that case fermions are no longer `free', and most of the 
theory differs substantially. In particular, the results and proofs in the 
present work are largely independent from those of \cite{agg-bor-wh2020-sl1n}, 
although we do point to connections in a few places where those exist. 

\medskip

We define two families of 
functions $F_\lambda,G_\lambda$
indexed by integer tuples $\lambda=(\lambda_1\ge \ldots \ge \lambda_N\ge0 )$
as
certain partition functions of the 
free fermion six vertex model with boundary conditions
depending on $\lambda$. 
The functions $F_\lambda,G_\lambda$
are rational in
(a finite, $\lambda$-dependent subset of) the parameters
$x_i,r_i,y_j,s_j$.
Up
to a simple product factor in $F_\lambda$, 
both the functions are
symmetric with respect to simultaneous permutations 
of the row variables $(x_i,r_i)$,
which is a consequence of the Yang--Baxter equation.
When the horizontal parameters $y_j,s_j$ do not depend on $j$,
the functions $F_\lambda$ and $G_\lambda$ reduce,
respectively, to the ordinary symmetric Schur polynomials
and the supersymmetric Schur polynomials.
In another specialization of the parameters,
the functions $F_\lambda$ become the factorial Schur polynomials
(cf. Molev
\cite{molev2009comultiplication}, \cite{zinn2009littlewood}).

We establish the following results:
\begin{itemize}
	\item Cauchy type summation identities leading to a product form expression
		for $\sum_\lambda F_\lambda G_\lambda$, and their skew analogues.
	\item Torus biorthogonality of the functions $F_\lambda$ 
		and certain dual functions $F_\lambda^*$,
		with integration over the row rapidities $x_j$.
	\item A double alternant type formula for $F_\lambda$,
		and a Jacobi--Trudy type determinantal formula for~$G_\lambda$.
	\item Another explicit formula for $G_\lambda$
		involving a summation over 
		pairs of permutations which resembles (but does not imply)
		the Sergeev--Pragacz formula 
		for the supersymmetric
		Schur polynomials
		(cf. Hamel--Goulden
		\cite[(5)]{hamel1995lattice}).
\end{itemize}

By analogy with the Schur processes of \cite{okounkov2003correlation},
we define probability measures (called \emph{FG~measures})
on two-dimensional integer arrays
encoded by sequences $\lambda^{(1)},\ldots,\lambda^{(T)}$.
Under an FG measure, the 
probability weights are expressed through the 
functions $F_\lambda,G_\lambda$ and their skew analogues.
Thanks to the Cauchy type summation identities,
$\lambda^{(j)}$-marginals have weights proportional to $F_{\lambda^{(j)}}G_{\lambda^{(j)}}$
for certain specializations of $F$ and $G$ that vary with $j=1,\dots,T$.
We interpret the FG measures as
certain ensembles of random domino tilings of a half-strip,
in which the domino weights are inhomogeneous and depend 
on the parameters $(x_i,r_i)$ and $(y_j,s_j)$
varying in the two coordinate directions.

We show that the FG measures (and the corresponding random domino tilings) 
are \emph{determinantal}. 
Namely,
the random point configuration
\begin{equation*}
	\mathcal{S}^{(T)}=\{(t,\lambda^{(t)}_i+N+1-i)\colon 1\le t\le T,\,1\le i\le N\}\subset
	\{1,\ldots,T \}\times \mathbb{Z}_{\ge1}
\end{equation*}
has all correlation functions $\mathbb{P}[A\subseteq \mathcal{S}^{(T)}]$
(where $A$ is finite)
expressed as symmetric $|A|\times |A|$ determinants of a certain function
$K(t,a;t',a')$ called the \emph{correlation kernel}.
We write $K$ as a double contour integral which resembles (yet does not coincide with) 
some determinantal correlation kernels of multilevel $\beta=2$
random matrix ensembles.

Our kernel $K$ generalizes that of the Schur process 
first obtained in \cite{okounkov2003correlation}
via a vertex operator formalism in the fermionic Fock space.
We obtain our double contour integral 
formula for $K$ by employing an `inhomogeneous version' of the Fock space.
In particular, we establish an inhomogeneous analogue 
of the
Boson--Fermion correspondence (cf.~Kac \cite[Theorem 14.10]{Kac1990InfiniteDim}
for the homogeneous statement), which may be of independent interest.
The fermionic operators in our Fock space
arise as combinations of (doubly) infinite volume limits of the Algebraic Bethe Ansatz row operators
$A,B,C,D$ evaluated at certain special parameter values.
We realize the commutation relations for the inhomogeneous fermionic
operators, as well as the inhomogeneous Boson--Fermion
correspondence, as consequences of the
Yang--Baxter equation.

The double contour integral form of the correlation kernel $K$ of the FG measures
is well-suited for the asymptotic analysis in the bulk of the system by the method of 
steepest descent. Such analysis leads us to the generalization of
the extended discrete sine kernel that was mentioned above. 

\medskip

Having outlined our main results,
let us now proceed to describing them in greater detail.

\subsection{Symmetric rational functions}
\label{sub:intro_functions}

Let 
$\lambda=(\lambda_1\ge \ldots \ge\lambda_N\ge0)$, $\lambda_i\in \mathbb{Z}$,
be a nonincreasing integer sequence,
which we call a \emph{signature} with $N$ parts.
Central objects considered in the present work 
are families of rational functions
$F_\lambda(\mathbf{x};\mathbf{y};\mathbf{r};\mathbf{s})$
and
$G_\lambda(\mathbf{x};\mathbf{y};\mathbf{r};\mathbf{s})$
indexed by signatures~$\lambda$.
These functions depend on four sequences
of (generally speaking, complex) parameters
\begin{equation}
	\label{eq:xyrs_parameters}
	\mathbf{x}=(x_1,\ldots,x_k),
	\qquad 
	\mathbf{r}=(r_1,\ldots,r_k),
	\qquad
	\mathbf{y}=(y_1,y_2,\ldots ),
	\qquad 
	\mathbf{s}=(s_1,s_2,\ldots ).
\end{equation}
The functions $F_\lambda,G_\lambda$ are defined
as partition functions
of the \emph{free fermion six vertex model}.
By a partition function we mean the sum of weights
of all configurations of the six vertex model with 
given boundary conditions depending on $\lambda$,
where the weight of each particular configuration
is equal to the product of local single-vertex weights
$w_{\mathrm{6V}}\bigl(\hspace*{-2pt}
	\begin{tikzpicture}[baseline=-2.7,scale=.7]
    	\draw[fill] (0,0) circle [radius=0.025];
        \draw[line width = 1mm, red!30] (0.05, 0) -- (0.5, 0) 
				node[black, right,xshift=-3]{\scriptsize{$j_2$}};
        \draw[line width = 1mm, red!30] (0, 0.05) -- (0, 0.5) 
				node[black,left,xshift=2]{\scriptsize{$i_2$}};
        \draw[line width = 1mm, red!30] (-0.05, 0) -- (-0.5, 0) 
				node[black, left,xshift=3]{\scriptsize{$j_1$}};
        \draw[line width = 1mm, red!30] (0, -0.05) -- (0, -0.5) 
				node[black, left,xshift=2]{\scriptsize{$i_1$}};
			\end{tikzpicture}\hspace*{-2pt}
	\bigr)$, where $i_1,j_1,i_2,j_2\in \left\{ 0,1 \right\}$,
depending on the parameters $x,y,r,s$.
The parameters $x,y,r,s$, in their turn, depend on the lattice coordinates
of the vertex.
For the definition of $G_\lambda$ we take the vertex
weights
$w_{\mathrm{6V}}=W$ given by
\begin{equation}
	\label{eq:intro_6v_weights}
	\begin{split}	
	W
	\bigl( 
		\begin{tikzpicture}[baseline=-3,scale=.7,very thick]
        \draw[fill] (0,0) circle [radius=0.025];
        \draw [dotted] (0.5,0) -- (0.05,0);
        \draw [dotted] (-0.5,0) -- (-0.05,0);
        \draw [dotted] (0,0.05) -- (0, 0.5);
        \draw [dotted] (0,-0.05) -- (0,-0.5);
        \draw[dotted](-0,-0.05) -- (-0,-0.5);
        \draw[dotted](0,0.05) -- (0,0.5);
  \end{tikzpicture}
	\bigr)
	=a_1=1,\qquad 
	W
	\bigl( 
		\begin{tikzpicture}[baseline=-3,scale=.7,very thick]
			\draw[fill] (0,0) circle [radius=0.025];
        \draw [red] (0.5,0) -- (0.05,0);
        \draw [red] (-0.5,0) -- (-0.05,0);
        \draw [red] (0,0.05) -- (0, 0.5);
        \draw [red] (0,-0.05) -- (0,-0.5);
        \draw[red](-0,-0.05) -- (-0,-0.5);
        \draw[red](0,0.05) -- (0,0.5);
  \end{tikzpicture}
	\bigr)
	&=a_2=\frac{r^{-2}x-y}{s^{-2}y-x}
	,\qquad 
	W
	\bigl( 
		\begin{tikzpicture}[baseline=-3,scale=.7,very thick]
			\draw[fill] (0,0) circle [radius=0.025];
        \draw [dotted] (0.5,0) -- (0.05,0);
        \draw [dotted] (-0.5,0) -- (-0.05,0);
        \draw [dotted] (0,0.05) -- (0, 0.5);
        \draw [red] (0,-0.05) -- (0,-0.5);
        \draw[red](0,0.05) -- (0,0.5);
  \end{tikzpicture}
\bigr)
=b_1=\frac{s^{-2}y-r^{-2}x}{s^{-2}y-x},
	\\
	W
	\bigl( 
		\begin{tikzpicture}[baseline=-3,scale=.7,very thick]
 \draw[fill] (0,0) circle [radius=0.025];
        \draw [red] (0.5,0) -- (0.05,0);
        \draw [red] (-0.5,0) -- (-0.05,0);
        \draw [dotted] (0,0.05) -- (0, 0.5);
        \draw [dotted] (0,-0.05) -- (0,-0.5);
        \draw[dotted](-0,-0.05) -- (-0,-0.5);
        \draw[dotted](0,0.05) -- (0,0.5); 
  \end{tikzpicture}
	\bigr)=b_2=
	\frac{y-x}{s^{-2}y-x}
	,\qquad 
	W
	\bigl( 
		\begin{tikzpicture}[baseline=-3,scale=.7,very thick]
			\draw[fill] (0,0) circle [radius=0.025];
        \draw [red] (0.5,0) -- (0.05,0);
        \draw [dotted] (-0.5,0) -- (-0.05,0);
        \draw [red] (0,-0.05) -- (0,-0.5);
        \draw[dotted](0,0.05) -- (0,0.5);
  \end{tikzpicture}
	\bigr)&=c_1=
	\frac{x(r^{-2}-1)}{s^{-2}y-x}
	,\qquad 
	W
	\bigl( 
		\begin{tikzpicture}[baseline=-3,scale=.7,very thick]
\draw[fill] (0,0) circle [radius=0.025];
        \draw [dotted] (0.5,0) -- (0.05,0);
        \draw [red] (-0.5,0) -- (-0.05,0);
        \draw [dotted] (0,0.05) -- (0, 0.5);
        \draw [dotted] (0,-0.05) -- (0,-0.5);
        \draw[dotted](-0,-0.05) -- (-0,-0.5);
        \draw[red](0,0.05) -- (0,0.5);
  \end{tikzpicture}
	\bigr)
	=c_2=
	\frac{y(s^{-2}-1)}{s^{-2}y-x}
	\end{split}
\end{equation}
(notation $a_1,a_2,b_1,b_2,c_1,c_2$
is the classical convention in the six 
vertex model weights, see, e.g., 
\cite[Ch. 8]{baxter2007exactly}, 
\cite{reshetikhin2010lectures}).
The functions $F_\lambda$ involve the renormalized weights
\begin{equation*}
	\widehat{W}(i_1,j_1;i_2,j_2) := \frac{W(i_1,j_1;i_2,j_2)}{W(0,1;0,1)}.
\end{equation*}
This normalization is chosen so that $\widehat{W}(0,1;0,1)=1$.
One readily sees that 
each of the families of 
vertex weights $W$ and $\widehat{W}$ satisfies the
\emph{free fermion condition}
$a_1a_2+b_1b_2-c_1c_2=0$.
The free fermion condition is crucial throughout our work. 

\medskip

Having the vertex weights, 
we form partition functions
in the half-infinite strip
$\mathbb{Z}_{\ge1}\times\left\{ 1,\ldots,k  \right\}$
(where $k$ is the number of variables in $\mathbf{x},\mathbf{r}$
\eqref{eq:xyrs_parameters})
as in \Cref{fig:intro_F_G_defn},
and call them
$G_\lambda(\mathbf{x};\mathbf{y};\mathbf{r};\mathbf{s})$
and
$F_\nu(\mathbf{x};\mathbf{y};\mathbf{r};\mathbf{s})$.
\begin{figure}[ht]
	\centering
	\includegraphics[width=\textwidth]{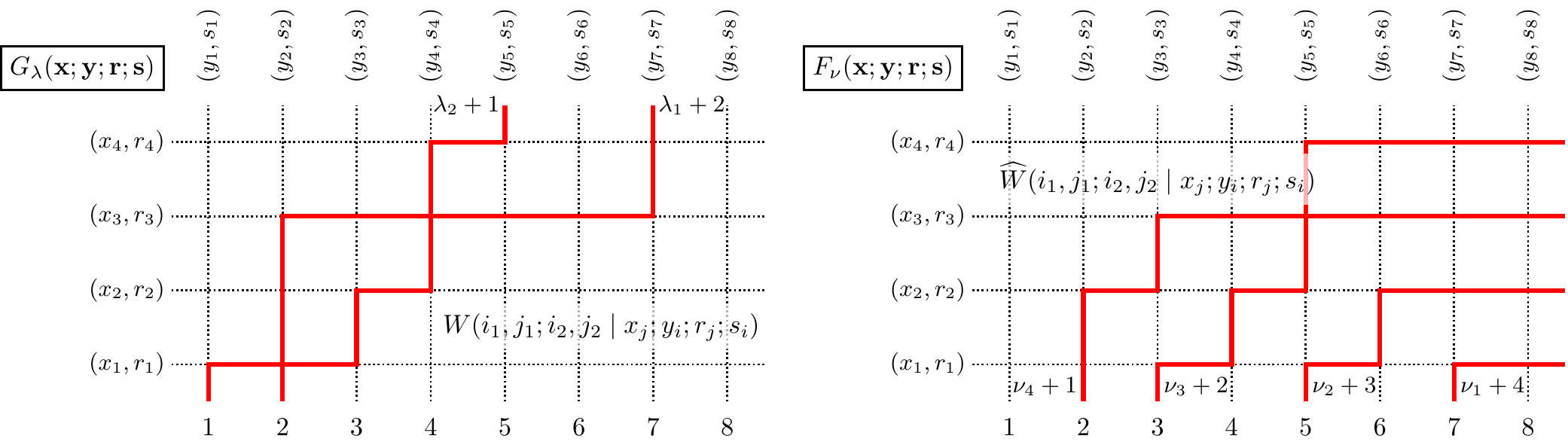}
	\caption{Left: An example of a 
		six vertex model configuration contributing to $G_\lambda$.
		The number $N$ of parts in $\lambda$ 
		and the number $k$ of the variables
		$\mathbf{x},\mathbf{r}$ \eqref{eq:xyrs_parameters}
		may differ (here $k=4,N=2$). The boundary conditions on the 
		left and right are empty, 
		are $\{N,N-1,\ldots,1 \}$
		at the bottom, 
		and are $\{\lambda_1+N, \lambda_2+N-1, \ldots, \lambda_N+1 \}$
		at the top.
		Right: An example of a configuration
		contributing to $F_\nu$. In contrast with $G_\lambda$, 
		the number of parts in $\nu$ must be equal to $k$ (here $k=4$).
		The boundary conditions are empty on the left 
		and at the top, fully packed on the right, 
		and are $\{ \nu_1+k, \nu_2+k-1, \ldots, \nu_k+1 \}$
		at the bottom.}
	\label{fig:intro_F_G_defn}
\end{figure}

\begin{remark}
	\label{rmk:intro_skew_functions}
	We also define skew
	functions
	$G_{\lambda/\mu}(\mathbf{x};\mathbf{y};\mathbf{r};\mathbf{s})$
	and 
	$F_{\nu/\varkappa}(\mathbf{x};\mathbf{y};\mathbf{r};\mathbf{s})$
	as partition functions.
	Namely, 
	for $G_{\lambda/\mu}$, the signature $\mu$ encodes the bottom 
	boundary in \Cref{fig:intro_F_G_defn}, 
	left, so that we have
	$G_{\lambda}=G_{\lambda/(0,0,\ldots,0 )}$.
	For $F_{\nu/\varkappa}$, the signature
	$\varkappa$ encodes the top boundary
	in \Cref{fig:intro_F_G_defn}, right, 
	so that $F_\nu=F_{\nu/\varnothing}$.
	See \Cref{sub:def_F_G_as_partition_functions}
	in the text for detailed definitions of all these functions.
	For brevity, in the Introduction
	we mostly stick to the non-skew functions.
\end{remark}

We have normalized the weights $W$ and $\widehat{W}$
so that the vertices occurring infinitely
many times
in \Cref{fig:intro_F_G_defn} 
have weight $1$. 
Therefore, the weights of individual six vertex model configurations
are well-defined.
Moreover, both partition functions 
$G_\lambda$ and $F_\nu$
involve only finitely many such
configurations, so there are no convergence issues. We see that 
$F_\nu(\mathbf{x};\mathbf{y};\mathbf{r};\mathbf{s})$
and
$G_\lambda(\mathbf{x};\mathbf{y};\mathbf{r};\mathbf{s})$
are rational functions in (a finite subset of) the parameters
\eqref{eq:xyrs_parameters}.

In \Cref{sec:particular_cases} we
consider particular cases
of the parameters $\mathbf{x},\mathbf{y},\mathbf{r},\mathbf{s}$
under which the functions $F_\lambda,G_\lambda$
become either the ordinary Schur
symmetric polynomials \cite[I.3]{Macdonald1995},
or their factorial 
or supersymmetric variations
\cite{BereleRegev},
\cite{macdonald1992schur_Theme},
\cite{molev2009comultiplication}.
Here let us formulate the supersymmetric setting.
\begin{proposition}[\Cref{prop:F_G_homogeneous_through_Schur} in the text]
	\label{prop:intro_supersymm}
	Take the horizontally homogeneous specialization
	$y_j=y$ and $s_j=s$ for all $j\ge1$.
	Then for any signature 
	$\lambda=(\lambda_1\ge \ldots\ge \lambda_N\ge0 )$ 
	we have
	\begin{equation*}
		\frac{F_\lambda(x_1,\ldots,x_N;\mathbf{y};\mathbf{r};\mathbf{s} )}
		{F_{(0,0,\ldots,0 )}(x_1,\ldots,x_N;\mathbf{y};\mathbf{r};\mathbf{s} )}
		=
		s_\lambda\left( \frac{1-s^2x_1}{s^2(1-x_1)},\ldots, \frac{1-s^2x_N}{s^2(1-x_N)} \right),
	\end{equation*}
	where $s_\lambda$ is 
	the ordinary Schur symmetric polynomial
	\cite[I.3]{Macdonald1995}.
	Moreover,
	\begin{equation*}
		\begin{split}
			&
			G_\lambda(x_1,\ldots,x_M;\mathbf{y};\mathbf{r};\mathbf{s} )
			\\&\hspace{30pt}
			=
			s_\lambda
			\left( 
				\biggl\{ \frac{s^2(1-x_j)}{1-s^2 x_j} \biggr\}_{j=1}^M
				\,\middle/\,
				\biggl\{ \frac{s^2(x_jr_j^{-2}-1)}{1-s^2r_j^{-2} x_j} \biggr\}_{j=1}^M
			\right)\,
			\prod_{i=1}^{M}\left( \frac{1-s^2r_i^{-2} x_i }{1-s^2x_i} \right)^N,
		\end{split}
	\end{equation*}
	where $s_\lambda(\cdots/\cdots)$ denotes the supersymmetric Schur 
	function
	\cite{BereleRegev}, \cite[(6.19)]{macdonald1992schur_Theme}.
\end{proposition}

Thus, one may view our
functions $F_\lambda,G_\lambda$
as generalizations 
of various Schur-like symmetric functions,
based on the inhomogeneous parameters $y_j,s_j$.
In fact, many of the properties of $F_\lambda,G_\lambda$
discussed 
in the rest of this subsection
resemble the ones of the ordinary Schur polynomials.

\medskip

The concrete parametrization
of the
vertex weights $W,\widehat{W}$ by $x,y,r,s$
is chosen so that the weights satisfy 
the Yang--Baxter equation with the 
cross vertex weights independent of 
$(y_i,s_i)$.
These cross vertex weights
are given in \Cref{fig:cross_vertex_weights},
and 
we refer to \Cref{sub:YBE} in the text for a detailed formulation 
of the Yang--Baxter equation.
In particular, the Yang--Baxter equation implies 
(see \Cref{prop:F_G_symmetry} in the text)
that the functions 
\begin{equation*}
	G_\lambda(x_1,\ldots,x_k;\mathbf{y};r_1,\ldots,r_k;\mathbf{s} )
	\quad 
	\textnormal{and}
	\quad 
	F_\nu(x_1,\ldots,x_k;\mathbf{y};r_1,\ldots,r_k;\mathbf{s} )
	\prod_{1\le i<j\le k}
	(x_i-r_j^{-2}x_j)
\end{equation*}
are symmetric 
under simultaneous permutations 
of the pairs of variables
$(x_i,r_j)$.

Another application of the Yang--Baxter equation 
(together with an explicit 
formula for $F_\lambda$ from \Cref{thm:intro_F_formula}
below)
brings the following \emph{Cauchy type summation identity}:
\begin{theorem}[\Cref{thm:F_G_Cauchy_big} in the text]
	\label{thm:intro_Cauchy}
	Fix integers $N,k\ge1$ and sets of complex variables
	$\mathbf{x}=(x_1,\ldots,x_N )$,
	$\mathbf{r}=(r_1,\ldots,r_N )$, 
	$\mathbf{w}=(w_1,\ldots,w_k )$, 
	$\boldsymbol\uptheta=(\theta_1,\ldots,\theta_k)$,
	and
	$\mathbf{y}=(y_1,y_2,\ldots, )$, 
	$\mathbf{s}=(s_1,s_2,\ldots )$,
	satisfying
	\begin{equation}
		\label{eq:intro_Cauchy_condition}
		\sup_{p\ge 1}\,\left|
		\frac{s_p^{-2} y_p-x_i}{y_p-x_i}
		\frac{y_p-w_j}{s_p^{-2} y_p-w_j}
		\right|<1
		\qquad \textnormal{for all $1\le i\le N$, $1\le j\le k$}.
	\end{equation}
	Then we have
	\begin{equation}
		\label{eq:intro_Cauhy_identity}
		\begin{split}
			&\sum_{\lambda=(\lambda_1\ge \ldots\ge \lambda_N\ge0 )}
			G_{\lambda}(\mathbf{w};\mathbf{y};\boldsymbol\uptheta;\mathbf{s})\,
			F_{\lambda}(\mathbf{x};\mathbf{y};\mathbf{r};\mathbf{s})
			\\&\hspace{80pt}=
			\frac{\prod_{1\le i\le j\le N}(r_i^{-2}x_i-x_j)
			\prod_{1\le i<j\le N}(s_i^{-2}y_i-y_j)}
			{\prod_{i,j=1}^N(y_i-x_j)}
			\prod_{i=1}^{N}\prod_{j=1}^{k}
			\frac{x_i-\theta_j^{-2}w_j}{x_i-w_j}.
		\end{split}
	\end{equation}
\end{theorem}

\begin{remark}
	An example of a fully inhomogeneous situation when condition
	\eqref{eq:intro_Cauchy_condition} 
	holds is $y_p=1-2^{-p}$, $s_p=1+2^p$, $p\ge1$, and 
	$\frac{1}{2}>w_j>x_i>\frac{1}{3}$ for all $i,j$.
\end{remark}

Let us now discuss
explicit formulas for the functions $F_\lambda$ and $G_\lambda$.
The first function $F_\lambda$ possesses an 
inhomogeneous analogue of the 
\emph{double alternant formula} for the Schur symmetric
polynomials
\cite[I.(3.1)]{Macdonald1995}.
Define inhomogeneous analogues of the power functions by
\begin{equation*}
	\varphi_k(x\mid \mathbf{y};\mathbf{s}):=
	\frac{1}{y_{k+1}-x}
	\prod_{j=1}^{k}
	\frac{x-s_j^{-2}y_j}{x-y_j},\qquad  k\ge0.
\end{equation*}

\begin{theorem}[\Cref{thm:F_formula} in the text]
	\label{thm:intro_F_formula}
	Let $\lambda$ be a signature with $N$ parts. 
	Then
	\begin{equation}
		\label{eq:intro_F_formula}
		F_\lambda(\mathbf{x};\mathbf{y};\mathbf{r};\mathbf{s})=
		\frac
		{\prod_{1\le i\le j\le N}(r_i^{-2}x_i-x_j)}
		{\prod_{1\le i<j\le N}(x_i-x_j)}\,
		\det\left[ \varphi_{\lambda_j+N-j}(x_i\mid \mathbf{y};\mathbf{s}) \right]_{i,j=1}^{N}.
	\end{equation}
\end{theorem}
We also obtain an explicit
formula for 
$G_\lambda(\mathbf{x};\mathbf{y};\mathbf{r};\mathbf{s})$,
see \Cref{thm:G_formula}
in the text. It involves summation over 
pairs of permutations which resembles (however, does not imply,
cf. 
\Cref{rmk:SP_formula_comparison})
the Sergeev--Pragacz formula 
\cite[(5)]{hamel1995lattice}
for the supersymmetric
Schur polynomials.
Our formula in \Cref{thm:G_formula} is 
quite long so we do not reproduce it here.

We prove explicit formulas for $F_\lambda$ and $G_\lambda$
in \Cref{appA:F_G_formula_proofs}
via computations with row operators
(for these operators, see \Cref{sub:intro_fermionic_operators} below
and \Cref{sub:row_operators} in the text).
These computations follow
\cite[Section 4.5]{BorodinPetrov2016inhom} 
(but are much more involved in the case of $G_\lambda$)
and are based on Algebraic Bethe Ansatz 
for quantum integrable systems, see, e.g.,
\cite[Part VII]{QISM_book}.
This approach can ultimately be traced
to our central tool,
the Yang--Baxter equation,
whose repeated application yields
quadratic relations for row operators.

\begin{remark}
	The inhomogeneous
	free fermion six vertex weights
	like \eqref{eq:intro_6v_weights}
	appeared
	(with a different parametrization)
	in 
	\cite{motegi2017izergin}.
	Moreover, in that paper
	a determinantal formula
	like \eqref{eq:intro_F_formula}
	for a partition function with $F_\lambda$-like boundary
	conditions
	was proven. 
	This was done 
	by an Izergin--Korepin approach, that is, by showing
	that
	both the partition function 
	and the right-hand side of 
	\eqref{eq:intro_F_formula}
	satisfy the same list of properties which 
	uniquely 
	determine a function.
\end{remark}

Along with the Sergeev--Pragacz type formula,
$G_\lambda$ admits another explicit expression
based on 
the Cauchy identity and the 
\emph{inhomogeneous
biorthogonality} associated with the 
functions $F_\lambda$.
Here we present a single-variable version 
of this biorthogonality,
see \Cref{prop:F_F_star_orthogonality} in the text
for a multivariable statement involving determinants.
Define
\begin{equation*}
	\psi_k(x\mid \mathbf{y};\mathbf{s}):=
	\frac{y_{k+1}(s_{k+1}^{-2}-1)}{x-s^{-2}_{k+1}y_{k+1}}\,
	\prod_{j=1}^{k}
	\frac{x-y_j}{x-s_j^{-2}y_j}, \qquad k\ge0.
\end{equation*}
Then
we have for all $k,l\ge 0$
(\Cref{lemma:phi_psi_orthogonal} below):
\begin{equation}
	\label{eq:intro_orthogonality}
		\frac{1}{2\pi\mathbf{i}}\oint_{\gamma}
		\varphi_k(z\mid \mathbf{y}, \mathbf{s})\,
		\psi_l(z\mid \mathbf{y},\mathbf{s})\,dz
		=
		\begin{cases}
			1,&k=l;\\
			0,&k\ne l,
		\end{cases}
\end{equation}
where the simple closed contour $\gamma$
separates the sets
$\{y_j \}_{j\ge1}$ and $\{ s_j^{-2}y_j \}_{j\ge1}$
and goes around the $y_j$'s in the positive direction.

Using \eqref{eq:intro_orthogonality},
we can extract $G_\lambda$ as
the coefficient by $F_\lambda$
from the right-hand side of the 
Cauchy identity \eqref{eq:intro_Cauhy_identity}.
This leads to the following 
\emph{Jacobi--Trudy type formula}:

\begin{proposition}[\Cref{prop:nonskew_G_Jacobi_Trudi} in the text]
	\label{prop:intro_G_JT}
	Let $\lambda$ be a signature with $N$
	parts. Then we have
	\begin{equation*}
		G_\lambda(\mathbf{w};\mathbf{y};\boldsymbol\uptheta;\mathbf{s})
		=
		\prod_{1\le i<j\le N}\frac{s_i^{-2}y_i-y_j}{y_j-y_i}
		\,
		\det\left[ \mathsf{h}_{\lambda_i+N-i,\,j}
		(\mathbf{w};\mathbf{y};\boldsymbol\uptheta;\mathbf{s}) \right]_{i,j=1}^{N},
	\end{equation*}
	where
	\begin{equation*}
		\mathsf{h}_{k,m}(\mathbf{w};\mathbf{y};\boldsymbol\uptheta;\mathbf{s})
		=
		\frac{1}{2\pi\mathbf{i}}\oint_{\gamma'}dz\,
		\frac{\psi_k(z\mid \mathbf{y};\mathbf{s})}{y_m-z}
		\prod_{j=1}^{M}\frac{z-\theta_j^{-2}w_j}{z-w_j}.
	\end{equation*}
	Here the positively 
	oriented
	integration contour $\gamma'$
	surrounds
	all the points $y_j,w_i$ 
	and leaves out all the points
	$s_j^{-2}y_j$.
\end{proposition}
In \eqref{eq:intro_orthogonality} and 
\Cref{prop:intro_G_JT} we assume
that the parameters are chosen in such a way that the integration contours
$\gamma$ and $\gamma'$
exist.

\subsection{Determinantal processes}
\label{sub:intro_det_pp}

Dividing the Cauchy identity
of \Cref{thm:intro_Cauchy} by its right-hand side,
we define a
probability measure 
on the space of 
signatures with $N$ parts which we call an \emph{FG measure}:
\begin{equation}
	\label{eq:intro_FG_measure}
	\mathscr{M}(\lambda):=
	\frac{1}{Z}
	\,
	F_\lambda(\mathbf{x};\mathbf{y};\mathbf{r};\mathbf{s})\,
	G_\lambda(\mathbf{w};\mathbf{y};\boldsymbol\uptheta;\mathbf{s}),
\end{equation}
where $Z$ is the normalizing constant given by the right-hand
side of \eqref{eq:intro_Cauhy_identity},
and
$\mathbf{x}=(x_1,\ldots,x_N )$,
$\mathbf{r}=(r_1,\ldots,r_N )$, 
$\mathbf{w}=(w_1,\ldots,w_k )$, and
$\boldsymbol\uptheta=(\theta_1,\ldots,\theta_k)$.
This definition is analogous to that of Schur
measures introduced in \cite{okounkov2001infinite}.
Further, by analogy with Schur processes
\cite{okounkov2003correlation} and Macdonald processes
\cite{BorodinCorwin2011Macdonald},
we define (\emph{ascending}) \emph{FG processes}
which are probability measures on sequences of signatures
$\lambda^{(1)},\ldots,\lambda^{(T)} $ (each with $N$ parts)
defined as
\begin{equation}
	\label{eq:intro_FG_process}
	\mathscr{AP}(\lambda^{(1)},\lambda^{(2)},\ldots,\lambda^{(T)})
	=
	\frac1{Z}\,
	G_{\lambda^{(1)}}(w_1;\mathbf{y};\theta_1;\mathbf{s})
	\ldots
	G_{\lambda^{(T)}/\lambda^{(T-1)}}(w_T;\mathbf{y};\theta_T;\mathbf{s})
	F_{\lambda^{(T)}}(\mathbf{x};\mathbf{y};\mathbf{r};\mathbf{s}).
\end{equation}
Here
$Z$ is the same normalizing constant
(the right-hand side of \eqref{eq:intro_Cauhy_identity}),
and
$G_{\lambda/\mu}$ are skew versions of the functions $G_\lambda$
(see \Cref{rmk:intro_skew_functions} above,
or \Cref{def:G_function} below).
For any fixed $t$, the marginal distribution
of $\lambda^{(t)}$ under 
\eqref{eq:intro_FG_process}
is the FG measure 
\eqref{eq:intro_FG_measure}
with the same $\mathbf{x},\mathbf{r},\mathbf{y},\mathbf{s}$,
and with $\mathbf{w}=(w_1,\ldots,w_t ),\boldsymbol\uptheta=(\theta_1,\ldots,\theta_t )$.

Sufficient conditions
under which
formulas
\eqref{eq:intro_FG_measure} and \eqref{eq:intro_FG_process}
define probability distributions
with nonnegative probability weights
are \eqref{eq:intro_Cauchy_condition} (so that the probability
weights are normalizable, i.e.,
the series for $Z$ converges) and 
\begin{equation*}
	x_i<y_j<r_i^{-2}x_i<s_j^{-2}y_j
	\quad\textnormal{and}\quad
	w_i<y_j<\theta_i^{-2}w_i<s_j^{-2}y_j
	\quad
	\textnormal{for all $i,j$}.
\end{equation*}
The latter conditions imply that 
all vertex
weights
$W(i_1,j_1;i_2,j_2),\widehat{W}(i_1,j_1;i_2,j_2)$
are nonnegative,
hence $F_\lambda$,
$G_\lambda$ and the $G_{\lambda/\mu}$'s are nonnegative, too.

\medskip

We show that the probability measure
$\mathscr{AP}(\lambda^{(1)},\ldots,\lambda^{(T)})$
gives rise to a \emph{determinantal point process},
which also implies determinantal structure for the measure
$\mathscr{M}$ \eqref{eq:intro_FG_measure}.
We refer to \cite{Soshnikov2000},
\cite{peres2006determinantal},
\cite{Borodin2009}
for generalities on determinantal processes.
Let us adapt general definitions to our setting.
Let $(\lambda^{(1)},\ldots,\lambda^{(T)})$ be a 
sequence of random signatures with joint distribution
\eqref{eq:intro_FG_process}, and define a random point configuration
\begin{equation*}
	\mathcal{S}^{(T)}:=
	\bigcup_{t=1}^{T}
	\bigl\{(t,\lambda^{(t)}_1+N),
	(t,\lambda^{(t)}_2+N-1),\ldots,
	(t,\lambda^{(t)}_N+1)\bigr\}
	\subset \left\{ 1,\ldots,T  \right\}\times \mathbb{Z}_{\ge1}.
\end{equation*}
Let $A\subset 
\left\{ 1,\ldots,T  \right\}\times \mathbb{Z}_{\ge1}$
be a fixed finite subset. A 
\emph{correlation function} associated with $A$
is, by definition, the probability
$\mathbb{P}_{\mathscr{AP}}[A\subset \mathcal{S}^{(T)}]$.
We show that
this correlation function, for any $A$,
is given by an $|A|\times |A|$
determinant of 
a fixed \emph{correlation kernel} 
defined as
(here 
$1\le t,t'\le T$ and $a,a'\ge 1$):
\begin{equation}
	\label{eq:intro_K_AP}
	\begin{split}
		&K_{\mathscr{AP}}(t,a;t',a')
		=
		\frac{1}{(2\pi\mathbf{i})^2}
		\oint_{\Gamma_{y,\theta^{-2}w}}du
		\oint_{\Gamma_{y,w}}
		dv
		\,
		\frac{1}{u-v}
		\prod_{k=1}^{N}\frac{(u-y_k)(v-x_k)}{(u-x_k)(v-y_k)}
		\\
		&\hspace{20pt}
		\times
		\frac{y_{a}(1-s_{a}^{-2})}{v-s_{a}^{-2}y_{a}}
		\frac{1}{u-y_{a'}}
		\prod_{j=1}^{a-1}
		\frac{v-y_j}{v-s_j^{-2}y_j}
		\prod_{j=1}^{a'-1}
		\frac{u-s_j^{-2}y_j}{u-y_j}
		\prod_{d=1}^t \frac{v-\theta_d^{-2}w_d}{v-w_d}
		\prod_{c=1}^{t'}\frac{u-w_c}{u-\theta_c^{-2}w_c},
	\end{split}
\end{equation}
where the integration contours are 
positively oriented circles one inside the other
(the $u$ contour is outside for $t\le t'$ while
the $v$ contour is outside for $t>t'$);
the $u$ contour
encircles all the points $y_i,\theta_j^{-2}w_j$,
and not $x_k$;
and
the $v$ contour encircles all the points $y_i,w_j$,
and not $s_k^{-2}y_k$. Here we assume that the parameters
are such that the contours exist.

\begin{theorem}[\Cref{thm:ascending_FG_process_kernel} in the text]
	\label{thm:intro_kernel}
	The ascending FG process 
	\eqref{eq:intro_FG_process}
	is determinantal with the kernel
	$K_{\mathscr{AP}}$. That is,
	for any 
	$A=\left\{ (t_1,a_1),\ldots,(t_m,a_m)  \right\}
	\subset \left\{ 1,\ldots,T  \right\}\times \mathbb{Z}_{\ge1}$,
	we have
	\begin{equation*}
		\mathbb{P}_{\mathscr{AP}}\bigl[ A\subset \mathcal{S}^{(T)} \bigr]
		=
		\det\left[ K_{\mathscr{AP}}(t_i,a_i;t_j,a_j) \right]_{i,j=1}^{m}.
	\end{equation*}
\end{theorem}
When $t=t'$, the kernel
\eqref{eq:intro_FG_process}
becomes the correlation kernel 
$K_\mathscr{M}(a,a')=K_{\mathscr{AP}}(t,a;t,a')$
for the FG measure
$\mathscr{M}(\lambda)$
\eqref{eq:intro_FG_measure}, where $\lambda=\lambda^{(t)}$
and $\mathbf{w}=(w_1,\ldots,w_t )$, 
$\boldsymbol\uptheta=(\theta_1,\ldots,\theta_t )$.

We give two proofs of \Cref{thm:intro_kernel}.
The first proof 
(presented in \Cref{appB:Eynard_Mehta})
uses an 
Eynard--Mehta type approach
based on \cite{borodin2005eynard},
see also \cite{eynard1998matrices}.
This approach is parallel to how the 
kernel for the Schur measures
is computed in \cite{borodin2005eynard}.
The second proof,
presented in 
\Cref{sec:fermionic_operators,sec:Fock_and_FG_process},
is based on 
fermionic operators in a Fock space coming
from the Algebraic Bethe Ansatz row operators. 
We discuss the main features of the second approach
in \Cref{sub:intro_fermionic_operators} below.

In the horizontally homogeneous case
$y_j=y$, $s_j=s$ for all $j\ge1$,
the correlation kernel
$K_{\mathscr{M}}(a,a')$
turns into the kernel for a certain
Schur measure, see \Cref{sub:Schur_process_part_case}
in the text.
A certain inhomogeneous 
analogue of Schur processes (describing
continuous time 
particle dynamics in inhomogeneous space generalizing the 
push-block process from \cite{BorFerr2008DF})
was defined recently in 
\cite{theodoros2019_determ}.
It is likely that the probability measures
of \cite{theodoros2019_determ}
could arise as degenerations of our FG processes, 
but we do not address this question here.

\subsection{Random domino tilings}
\label{sub:intro_random_tilings}

We interpret ascending FG processes 
\eqref{eq:intro_FG_process}
as random tilings by $1\times 2$ dominoes
of an infinite half-strip,
in the spirit of the
steep tiling representation of Schur processes
\cite{bouttier2017aztec}.
(The connection between states of the free fermion six
vertex model and random domino tilings has been long known
before, cf.
Elkies--Kuperberg--Larsen--Propp
\cite{elkies1992alternating},
Zinn-Justin
\cite{zinn2000six},
Ferrari--Spohn
\cite{ferrari2006domino}.)
While our boundary conditions
are not as general as those in arbitrary steep tilings in the cited work,
we are able to consider dominoes with more general
weights which depend on the many parameters of the 
FG process.

\begin{figure}[ht]
	\centering
	\includegraphics[width=.88\textwidth]{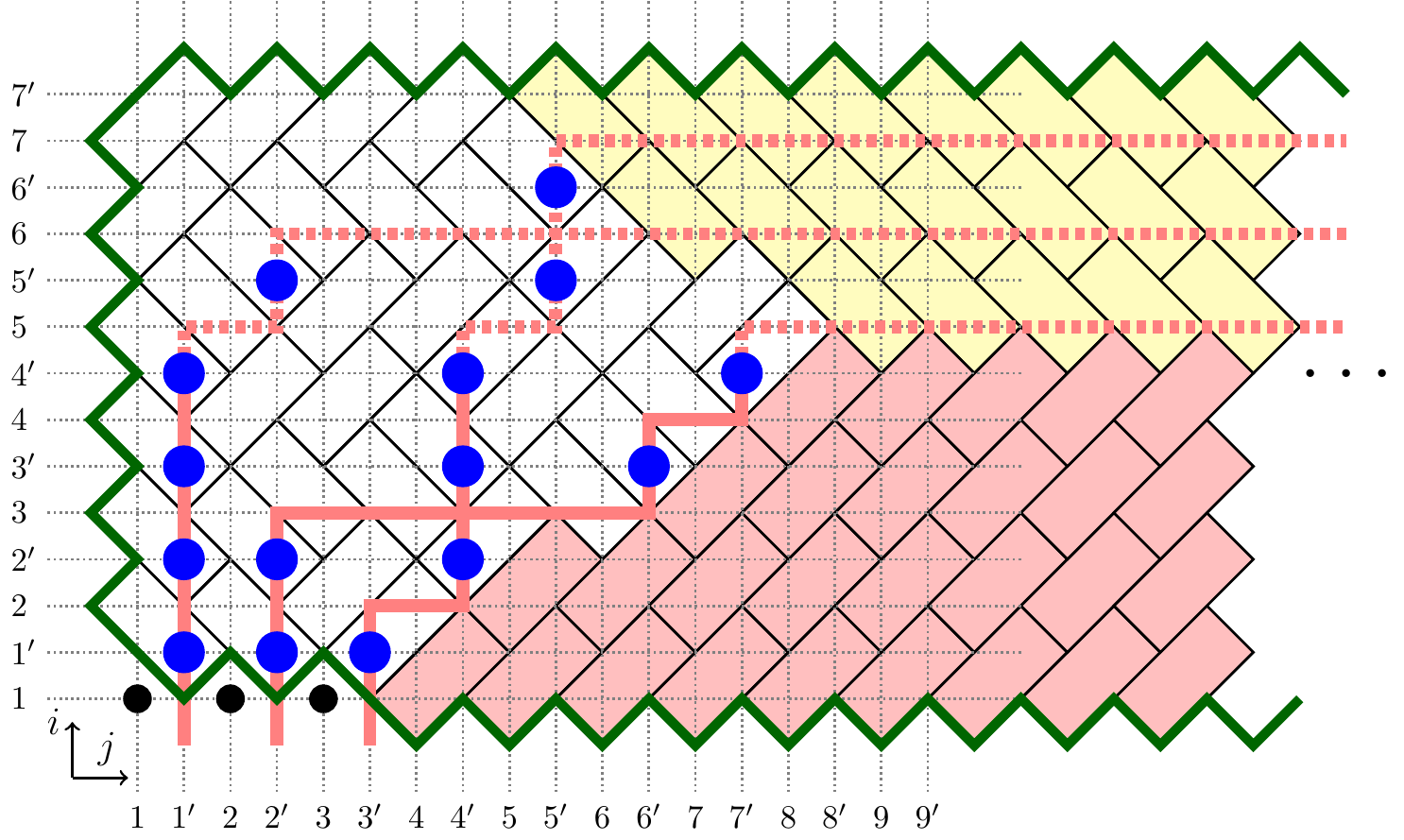}
	\caption{An example of a domino tiling 
	corresponding to the ascending FG process
	with $T=4$, $N=3$. 
	The dominoes which repeat infinitely many times
	far to the right are shaded.
	The large blue dots are particles associated to the tiling, i.e., the
	centers of the bottom
	squares of dominoes which have coordinates $(l',k')$.
	They are
	used to identify 
	the domino tiling with a sequence of signatures $\lambda^{(1)},\ldots,\lambda^{(T)}$.
	We also display the path ensembles
	leading to the partition
	functions $G_{\lambda^{(T)}}$ (in the bottom part)
	and 
	$F_{\lambda^{(T)}}$ (dashed paths in the top part).}
	\label{fig:intro_dimers}
\end{figure}

Recall that the ascending FG process
is associated with two integers, $N$ and $T$. 
Let the coordinates in the $\mathbb{Z}^2$ plane 
be numbered as 
$0'<1<1'<2<2'<\ldots $.
Consider the infinite half-strip with vertical coordinates
between $0'$ and $N+T+1$, and
with $N$ unit squares removed from the bottom left,
see \Cref{fig:intro_dimers}. 
We consider domino tilings of this strip
such that far to the right 
the dominoes stabilize to regular 
brick layers of two different
directions,
northeast and southeast 
in the regions $i\le T'$ and $i \ge T+1$, respectively.

Let us explain how a given domino tiling 
corresponds to a sequence
$\lambda^{(1)},\lambda^{(2)},\ldots,\lambda^{(T)}$
of signatures, each 
with $N$ parts.
Single out the dominoes
for which the center of the bottom unit square
has coordinates
of the form $(l',k')$.
There are only finitely many such dominoes,
and in \Cref{fig:intro_dimers}
we indicated the 
centers of the bottom squares.
Let us call these points the \emph{particles}
associated with the domino tiling.
More precisely, we have $N$ dominoes containing particles in the bottom
$T$ layers, 
and in the top $N$ layers 
there are $N-1,N-2,\ldots,1,0$ particles in each layer.
For $k=1,\ldots,T $, 
define the signature $\lambda^{(k)}$ 
so that the $N$ particles at layer $k$
have the horizontal coordinates 
\begin{equation}
	\label{eq:lambda_from_dominos}
	(\lambda^{(k)}_1+N)',(\lambda^{(k)}_{2}+N-1)',\ldots,(\lambda^{(k)}_{N-1}+2)', 
	(\lambda^{(k)}_N+1)'.
\end{equation}
For example, the sequence of signatures 
corresponding to the domino 
tiling in \Cref{fig:intro_dimers}
is
\begin{equation*}
	\lambda^{(1)}=(0,0,0),
	\qquad 
	\lambda^{(2)}=(1,0,0),
	\qquad 
	\lambda^{(3)}=(3,2,0),
	\qquad 
	\lambda^{(4)}=(4,2,0).
\end{equation*}

Let us now assign weights to dominoes
depending on the parameters $w_i,\theta_i,y_j,s_j$
in the top part, and 
$x_i,r_i,y_j,s_j$ in the bottom part,
as displayed in \Cref{fig:intro_dimer_weights}.
Note that the dominoes repeating infinitely
often (the shaded ones in 
\Cref{fig:intro_dimers,fig:intro_dimer_weights})
have weight $1$.
Assuming that the weights satisfy
\eqref{eq:intro_Cauchy_condition},
we see that the infinite series
for the 
normalizing constant 
of this probability measure on domino tilings converges.
Thus, the model of random domino tilings
is well-defined.

\begin{figure}[h]
	\centering
	\includegraphics[width=\textwidth]{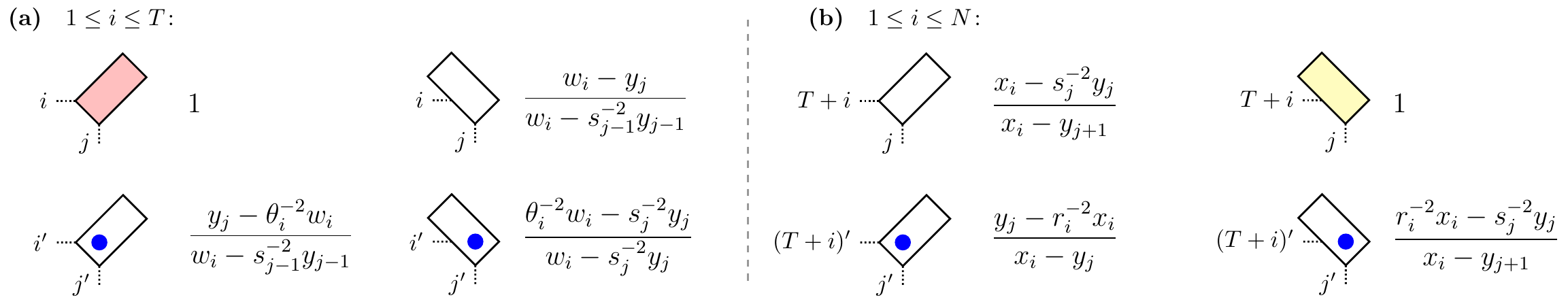}
	\caption{Domino weights
	in \Cref{fig:intro_dimers} leading
	to the ascending FG process.}
	\label{fig:intro_dimer_weights}
\end{figure}

In \Cref{sec:domino_tilings}
we establish the correspondence
between the random domino tiling
model just described,
and the ascending FG processes.
This correspondence is based on the 
known
mapping between the free fermion 
six vertex model
and a layered free fermion five vertex model,
e.g., see
\cite[Section 4.7]{wheeler2018hall}.

\begin{theorem}
	\label{thm:intro_dimers_dominoes}
	The joint distribution of the signatures
	$\lambda^{(1)},\ldots,\lambda^{(T)}$ (each with $N$ parts)
	associated via \eqref{eq:lambda_from_dominos} to the random
	domino tiling as in \Cref{fig:intro_dimers}
	with domino weights given in \Cref{fig:intro_dimer_weights}
	is described by the ascending FG process
	\eqref{eq:intro_FG_process}.
\end{theorem}

\begin{remark}
	One can also consider the joint distribution
	of all $T+N-1$ signatures arising from \Cref{fig:intro_dimers}.
	Namely, let $\mu^{(1)},\ldots,\mu^{(N)} $, 
	where $\mu^{(i)}$ has $i$ parts and $\mu^{(N)}=\lambda^{(T)}$,
	be constructed 
	as in \eqref{eq:lambda_from_dominos}
	from the 
	coordinates $(l',k')$ of the particles in the top $N$ rows of the tiling.
	Then the sequence of signatures
	$(\lambda^{(1)},\ldots,\lambda^{(T)}=\mu^{(N)},\mu^{(N-1)},\ldots,\mu^{(1)})$
	has the joint distribution of a general FG process
	defined
	in \Cref{sub:general_FG} in the text.
	For such FG processes (and their further generalizations)
	we obtain the correlation kernel 
	as a certain series coefficient using fermionic operators,
	see
	\Cref{thm:corr_kernel_no_contours} in the text.
	It	
	should be possible to rewrite the coefficient representation for
	the correlation kernel 
	of the general FG processes as a contour integral.
	We do not pursue this here for brevity
	and also because
	the 
	lattice (bulk) asymptotic behavior (discussed in \Cref{sub:intro_sine_kernel}
	below) 
	throughout the 
	whole domino tiling in \Cref{fig:intro_dimers}
	is expected to be the same, up to renaming the parameters.
\end{remark}

\subsection{Bulk asymptotics and the inhomogeneous discrete sine kernel}
\label{sub:intro_sine_kernel}

We study bulk asymptotics of the domino tiling model
described in 
\Cref{sub:intro_random_tilings}
as $N,T\to +\infty$
and $T\gg N$. Here ``bulk'' means that we
zoom around a global position $(\lfloor \alpha N \rfloor ,\lfloor \tau N \rfloor )$,
so that the lattice structure is preserved in the limit.
For simplicity of the asymptotic analysis,
we let the inhomogeneity parameters
of the domino tiling vary only in a finite
neighborhood of this global location.
After taking the limit, we send the size of the
finite neighborhood to infinity as well.

In this limit 
we observe a probability measure
on the space of domino tilings of the whole plane $\mathbb{Z}^2$.
This measure is determinantal.
Its correlation kernel,
which we call the 
(\emph{two-dimensional}) \emph{inhomogeneous discrete sine kernel}
and denote by $K^z_{\textnormal{2d}}$ (defined below in this subsection),
describes the bulk asymptotic joint distribution of the 
particles associated with a random domino tiling (as in \Cref{fig:intro_dimers}).
We take arbitrary inhomogeneity
parameters
around the global position 
$(\lfloor \alpha N \rfloor ,\lfloor \tau N \rfloor )$, so that
the limiting bulk 
kernel $K^z_{\textnormal{2d}}$ on $\mathbb{Z}^2$ 
is also inhomogeneous. Moreover, it depends on 
four sequences of parameters
$\mathbf{w}=\{w_i \}_{i\in \mathbb{Z}}$, $\boldsymbol\uptheta=\{\theta_i \}_{i\in \mathbb{Z}}$,
$\mathbf{y}=\{y_j \}_{j\in \mathbb{Z}}$, and $\mathbf{s}=\{s_j \}_{j\in \mathbb{Z}}$.
Here the indexing is by
$i,j\in \mathbb{Z}$ because
in the bulk limit the parameters vary 
in all $\mathbb{Z}^2$ directions around the global scaling position
which becomes the origin.
Along with these four sequences,
$K^z_{\textnormal{2d}}$ also depends on a 
point $z$ in the upper half complex
plane. 
In the homogeneous case, 
this point
is responsible for the 
\emph{slope} of the tiling,
i.e., the 
densities of the particles in the horizontal and the vertical directions.
The presence of the complex slope
is typical in 
homogeneous 
two-dimensional bulk 
lattice asymptotics \cite{okounkov2003correlation},
\cite{KOS2006}, \cite{OkounkovKenyon2007Limit}.
However, the dependence on four extra sequences of parameters is a
novel feature of our kernel 
$K^z_{\textnormal{2d}}$
that is a consequence of the
inhomogeneity of our model.

\begin{theorem}[\Cref{thm:bulk_limit} in the text]
	\label{thm:intro_limit_bulk}
	Fix $z$ in the open upper half complex plane.
	Then there exists 
	a choice of parameters
	of the ascending FG process
	together with
	a global location $(\alpha,\tau)$
	(detailed in \Cref{sub:scaling_and_parameter_assumptions}), 
	such that in the 
	limit 
	as $N,T\to +\infty$, $T\gg N$,
	the correlation kernel
	$K_{\mathscr{AP}}$ 
	\eqref{eq:intro_K_AP}
	of the ascending FG process
	admits the limit
	\begin{equation*}
		\lim_{N\to+\infty}K_{\mathscr{AP}}
		(
		{t}+\lfloor \tau N \rfloor ,
		{a}+\lfloor \alpha N \rfloor ;
		{t}'+\lfloor \tau N \rfloor ,
		{a}'+\lfloor \alpha N \rfloor
		)=
		K_{\textnormal{2d}}^z({t},{a};{t}',{a}'),
	\end{equation*}
	where $t,a,t',a'\in \mathbb{Z}$
	are fixed.
\end{theorem}

We establish \Cref{thm:intro_limit_bulk}
using the steepest descent method
for double contour integral correlation kernels
which essentially follows
\cite[Sections 3.1, 3.2]{Okounkov2002}.
This technique is quite standard, and we refer to 
\Cref{sec:asymptotics}
in the text for detailed formulations.

\medskip

Let us now proceed with the definition of
the inhomogeneous discrete sine kernel
$K^{z}_{\textnormal{2d}}$. First, we need some auxiliary notation.
For any two sequences
$\mathbf{b}=\{b_i \}_{i\in \mathbb{Z}}$ and
$\mathbf{c}=\{c_i \}_{i\in \mathbb{Z}}$,
define the following inhomogeneous analogues of the
power functions $U\mapsto U^n$, $n\in \mathbb{Z}$:
\begin{equation}
	\label{eq:intro_inhom_powers}
	\mathcal{P}_{n,n'}(u\mid \mathbf{b};\mathbf{c}):=
	\begin{cases}
		\displaystyle\prod_{j=n+1}^{n'}\frac{u-b_j}{u-c_j},&n<n';\\
		1,&n=n';\\
		\displaystyle\prod_{j=n'+1}^{n}\frac{u-c_j}{u-b_j},&n>n',
	\end{cases}
	\qquad \qquad n,n'\in \mathbb{Z}.
\end{equation}
Assume that the sequences satisfy
\begin{equation}
	\label{eq:intro_wyst_ordering}
		\sup\nolimits_i w_i< 
		\inf\nolimits_j y_j\le
		\sup\nolimits_j y_j<
		\inf\nolimits_i \theta_i^{-2}w_i\le 
		\sup\nolimits_i \theta_i^{-2}w_i<
		\inf\nolimits_j s_j^{-2}y_j.
\end{equation}
These ordering conditions are equivalent 
(as we show in \Cref{lemma:conditions_on_parameters_natural}
in the text)
to the 
fact that all the domino weights
given in \Cref{fig:intro_dimer_weights}, (a)
are positive and separated from zero and infinity.
We now define the 
two-dimensional inhomogeneous discrete sine kernel as
\begin{equation}
	\label{eq:intro_inhom_sine}
	K_{\textnormal{2d}}^{z}(t,a;t',a'):=
	-
	\frac{1}{2\pi\mathbf{i}}\int_{\bar z}^{z}
	\frac{y_a(1-s_a^{-2})}{(u-y_a)(u-s_{a'}^{-2}y_{a'})}
	\,\mathcal{P}_{a,a'}(u\mid \mathbf{s}^{-2}\mathbf{y};\mathbf{y})
	\,
	\mathcal{P}_{t,t'}(u\mid \mathbf{w};\boldsymbol\uptheta^{-2}\mathbf{w})
	\,
	du,
\end{equation}
where
$t,a,t',a'\in \mathbb{Z}$.
The integration contour 
is an arc
from $\bar z$ to $z$ which 
crosses the real line 
to the left of all
$w_i$ when $\Delta t=t'-t\ge 0$; and
between 
$\theta_i^{-2}w_i$ and $s_j^{-2}y_j$
when $\Delta t <0$.

From the fact 
that $K^z_{\textnormal{2d}}$
is a limit of $K_{\mathscr{AP}}$ (\Cref{thm:intro_limit_bulk}),
the correlation kernel of a 
determinantal random point process coming from the 
ascending FG process,
we deduce that $K^z_{\textnormal{2d}}$ 
with arbitrary inhomogeneity parameters 
$\mathbf{w},\boldsymbol\uptheta,\mathbf{y},\mathbf{s}$
satisfying the ordering 
\eqref{eq:intro_wyst_ordering}
has the following nonnegativity property:

\begin{theorem}[\Cref{thm:K_z_stochastic} in the text]
	\label{thm:intro_K_z}
	Under the above assumptions on the parameters and
	for any $z$ in the open upper half complex plane,
	the kernel $K^z_{\textnormal{2d}}$ 
	defines a determinantal random point process 
	on $\mathbb{Z}^2$. In particular,
	all symmetric minors 
	$\det[K^z_{\textnormal{2d}}(t_i,a_i;t_j,a_j)]$
	of any order are 
	between $0$ and $1$.
\end{theorem}

Indeed, each symmetric minor
$\det[K^z_{\textnormal{2d}}(t_i,a_i;t_j,a_j)]$
is the probability of the
correlation event $\{\textnormal{the random configuration contains
all the points $(t_i,a_i)$}\}$. This is nonnegative since the 
process $K^z_{\textnormal{2d}}$ is a limit of a 
\emph{bona fide} determinantal random point process.
We refer to \cite{Soshnikov2000},
\cite{peres2006determinantal},
\cite{Borodin2009}
for generalities on determinantal processes.

In \Cref{sec:inhom_sine_kernel}
we discuss specializations of the kernel 
$K^z_{\textnormal{2d}}$
leading to known determinantal
correlation kernels arising in 
bulk lattice limits: 
\begin{enumerate}[$\bullet$]
	\item the one-dimensional discrete sine kernel \cite{Borodin2000b};
	\item one-dimensional periodic and inhomogeneous generalizations of
		the discrete sine kernel 
		\cite{borodin2007periodic},
		\cite{borodin2010gibbs},
		\cite{Mkrtchyan2014Periodic}, \cite{Mkrtchyan2019}; 
	\item the bulk
		kernel arising from uniformly random domino tilings which 
		is a bulk
		limit of \cite[(2.21)]{Johansson2005arctic} but also follows
		from the general theory of \cite{KOS2006};
	\item the incomplete beta kernel \cite{okounkov2003correlation}
		giving rise to the unique family of ergodic translation invariant
		Gibbs measures on lozenge tilings of the whole plane
		\cite{Sheffield2008}, \cite{KOS2006} indexed by the complex slope
		$z$;
	\item 
		and $k\mathbb{Z}\times \mathbb{Z}$ periodic
		generalizations of the incomplete beta kernel
		\cite{borodin2010gibbs}, \cite{Mkrtchyan2014Periodic}, \cite{Mkrtchyan2019}.
\end{enumerate}

Our kernel $K^z_{\textnormal{2d}}$
admits a $k\mathbb{Z}\times m\mathbb{Z}$
periodic specialization for all $k,m\ge1$
by taking the 
parameters $(w_i,\theta_i)$ to be $m$-periodic in $i$,
and the parameters $(y_j,s_j)$ to be $k$-periodic in $j$.
For general $m\ge2$ 
the arc integral representation
\eqref{eq:intro_inhom_sine}
of such a periodic kernel is new. 
It would be interesting to match this 
arc integral to the 
two-dimensional torus integral representation
of the doubly periodic kernels
which follows from the general theory of 
\cite{KOS2006}, but we do not pursue this here.
		
We also remark that our kernel 
$K^z_{\textnormal{2d}}$ corresponds
only to the so-called liquid
phase of the domino tiling model.
It is known 
\cite{KOS2006},
\cite{chhita2016domino},
\cite{duits2017two},
\cite{berggren2021domino}
that 
doubly periodic domino weights may
lead to the appearance of 
gaseous phase. The gaseous phase is not present in our FG processes
because our domino weights
are \emph{not} fully generic and depend on their 
many parameters in quite a special way.
In particular,
in the $2\mathbb{Z}\times 2\mathbb{Z}$ periodic case
we have verified that
the domino weights are gauge equivalent
(in the sense of \cite[Section 3.10]{Kenyon2007Lecture}),
in a nontrivial way, to 
weights periodic in only one direction.

\subsection{Fermionic operators and correlation functions}
\label{sub:intro_fermionic_operators}

In this final part of the Introduction
we outline definitions and main properties
of 
fermionic operators acting in a Fock, or ``infinite wedge'', space.
Detailed definitions and statements are in 
\Cref{sec:fermionic_operators,sec:Fock_and_FG_process} below.

Our fermionic operators are combinations of
Algebraic Bethe Ansatz row 
operators constructed
from the vertex weights
$W$ \eqref{eq:intro_6v_weights}.
The fermionic operators
allow to compute certain
generating function type series
involving the correlation functions
of the FG processes.
The correlation functions are then extracted as series 
coefficients using inhomogeneous biorthogonality similar to 
\eqref{eq:intro_orthogonality}.

Fock space and fermionic operators coming from
Pieri rules for Schur functions
were used 
in \cite{okounkov2001infinite},
\cite{okounkov2003correlation}
to compute correlation 
kernels of Schur measures and processes.
Expressions for local operators and correlations
in various quantum integrable systems
through the row operators $A,B,C,D$ also
appear in, e.g., 
\cite{oota2003quantum}, \cite{kitanine2002spin}, but our model and formulas
are quite different from those.
It is also worth noting that
fermionic operators in the homogeneous
Fock space associated to the
free fermion six vertex model (and again leading to Schur functions)
were considered recently in 
\cite{korff2021cylindric}.
However, our inhomogeneous Fock space 
and
the fermionic operators acting in it 
arising from the free fermion six vertex model
seem to be new.

\medskip

A 
subset $\mathcal{T}\subset \mathbb{Z}$
is called
semi-infinite 
(or densely packed towards $-\infty$)
if
there exists $M=M(\mathcal{T})>0$
with $i\notin \mathcal{T}$ for all $i>M$ and $i\in \mathcal{T}$ for all $i< -M$.
For a semi-infinite subset, define its 
charge
$c(\mathcal{T}):=\#(\mathcal{T}\cap \mathbb{Z}_{>0})-
\#(\mathbb{Z}_{\le 0}\setminus \mathcal{T})$.
For example, all zero-charge semi-infinite subsets
are finite permutations of 
$\mathbb{Z}_{\le0}$.

Let $\mathscr{F}$ be the (\emph{fermionic})
\emph{Fock space} 
spanned by $e_{\mathcal{T}}$, where $\mathcal{T}$
runs over all semi-infinite subsets of $\mathbb{Z}$.
We view $\mathscr{F}$ as a 
subspace of the formal
infinite tensor product 
$\bigotimes_{m=-\infty}^{+\infty} V^{(m)} $,
where each $V^{(m)}$ is isomorphic to 
$\mathbb{C}^2$ with standard basis $e^{(m)}_0,e^{(m)}_1$.
This is done by interpreting
each $e_{\mathcal{T}}$ as a tensor product
\begin{equation*}
	e_{\mathcal{T}}=
	\bigotimes_{m=-\infty}^{+\infty}
	e_{k_m}^{(m)},
	\qquad k_m=k_m(\mathcal{T})=\mathbf{1}_{m\in \mathcal{T}}.
\end{equation*}
Here and below
$\mathbf{1}_{\cdots}$
is the indicator.
Sometimes in the literature the wedge product symbol
$
\bigwedge_{m=-\infty}^{+\infty}
e_{k_m}^{(m)}
$
is used instead of the tensor product,
with the same meaning.
In more detail, we never implicitly use the wedge commutation relation
$v\wedge w=-w\wedge v$, and whenever signs are required we insert them 
explicitly, as in, e.g., the creation and annihilation operators
$\psi_j,\psi_j^*$ in \eqref{eq:create_annih_intro} below.

We also need 
an inner product in $\mathscr{F}$
under which the $e_{\mathcal{T}}$'s form an orthonormal basis,
that is,
$\langle e_{\mathcal{T}},e_{\mathcal{R}} \rangle =\mathbf{1}_{\mathcal{T}=\mathcal{R}}$.
Let us decompose
$\mathscr{F}$ into subspaces with fixed charge:
\begin{equation*}
	\mathscr{F}=\bigoplus_{n\in \mathbb{Z}}\mathscr{F}_n,\qquad 
	\mathscr{F}_n=\mathop{\mathrm{span}}
	\left\{ e_{\mathcal{T}}\colon c(\mathcal{T})=n \right\}.
\end{equation*}

We are now in a position to define the row operators
$A^{\mathbb{Z}}=A^{\mathbb{Z}}(x,r)$,
$B^{\mathbb{Z}}=B^{\mathbb{Z}}(x,r)$,
$C^{\mathbb{Z}}=C^{\mathbb{Z}}(x,r)$, and
$D^{\mathbb{Z}}=D^{\mathbb{Z}}(x,r)$
acting in $\mathscr{F}$. They act in the following
way with respect to the charge:
\begin{equation*}
	A^{\mathbb{Z}}\colon \mathscr{F}_n\to \mathscr{F}_n,\qquad 
	B^{\mathbb{Z}}\colon \mathscr{F}_n\to \mathscr{F}_{n-1},\qquad 
	C^{\mathbb{Z}}\colon \mathscr{F}_n\to \mathscr{F}_{n+1},\qquad 
	D^{\mathbb{Z}}\colon \mathscr{F}_n\to \mathscr{F}_n.
\end{equation*}
We define these operators pictorially
through their matrix elements.
For $A^{\mathbb{Z}}$ we have
\begin{equation*}
	\langle A^{\mathbb{Z}}(x,r)e_{\mathcal{T}},e_{\mathcal{R}} \rangle =
	\frac{W
	\left(
		\scalebox{.7}{
		\begin{tikzpicture}
			[scale=1,thick,baseline=-9pt]
		\draw[densely dotted] (-5.5,0)--(5.5,0);
		\foreach \ii in {-5,...,5}
		{
			\draw[densely dotted] (\ii,.5)--++(0,-1) node [below] {$\ii$};
		}
		\draw[red,ultra thick] (-6,0)--++(4,0)--++(0,.5);
		\draw[red,ultra thick] (-5,-.5)--++(0,1);
		\draw[red,ultra thick] (-4,-.5)--++(0,1);
		\draw[red,ultra thick] (-3,-.5)--++(0,1);
		\draw[red,ultra thick] (-1,-.5)--++(0,1);
		\draw[red,ultra thick] (0,-.5)--++(0,.5)--(2,0)--++(0,.5);
		\draw[red,ultra thick] (1,-.5)--++(0,1);
		\draw[red,ultra thick] (3,-.5)--++(0,1);
		\draw[red,ultra thick] (4,-.5)--++(0,.5)--++(2,0);
		\node at (.5,-.4) {$\mathcal{T}$};
		\node at (.5,.3) {$\mathcal{R}$};
		\node at (-5.5,.25) {\large$\dots$};
		\node at (5.5,.25) {\large$\dots$};
		\end{tikzpicture}}
		\right)
	}
	{
	W
	\left(
		\scalebox{.7}{
		\begin{tikzpicture}
			[scale=1,thick,baseline=-9pt]
		\draw[densely dotted] (-5.5,0)--(5.5,0);
		\foreach \ii in {-5,...,5}
		{
			\draw[densely dotted] (\ii,.5)--++(0,-1) node [below] {$\ii$};
		}
		\draw[red,ultra thick] (-6,0)--++(12,0);
		\draw[red,ultra thick] (-5,-.5)--++(0,1);
		\draw[red,ultra thick] (-4,-.5)--++(0,1);
		\draw[red,ultra thick] (-3,-.5)--++(0,1);
		\draw[red,ultra thick] (-2,-.5)--++(0,1);
		\draw[red,ultra thick] (-1,-.5)--++(0,1);
		\draw[red,ultra thick] (0,-.5)--++(0,1);
		\node at (-5.5,.25) {\large$\dots$};
		\node at (5.5,.25) {\large$\dots$};
		\end{tikzpicture}}
		\right)
	}.
\end{equation*}
Here the numerator is a formal infinite product
of vertex weights $W(\cdots\mid x;y_j;r;s_j)$
over all $j\in \mathbb{Z}$,
where the bottom and top boundary conditions are
$\mathcal{T}$ and $\mathcal{R}$, respectively,
and the far left and far right boundary conditions are occupied.
By definition, the product of the weights $W(\cdots)$
is zero if there are no six vertex model configurations with these
boundary conditions.
The denominator in $A^{\mathbb{Z}}$ 
is the normalization factor which is also a 
formal infinite product. The ratio is well-defined
as 
the product of the ratios 
of the weights at each lattice site $j\in \mathbb{Z}$,
because this product involves only finitely 
many factors not equal to $1$.

Similarly we define the other three operators,
$B^{\mathbb{Z}}$ with boundary conditions empty and full 
at far left and far right,
$C^{\mathbb{Z}}$
with boundary conditions full and empty at far left and far right,
and $D^{\mathbb{Z}}$
with empty boundary conditions on both sides:
\begin{align*}
&
	\langle B^{\mathbb{Z}}(x,r)e_{\mathcal{T}},e_{\mathcal{R}} \rangle =
	\frac{W
	\left(
		\scalebox{.7}{
		\begin{tikzpicture}
			[scale=1,thick,baseline=-9pt]
		\draw[densely dotted] (-5.5,0)--(5.5,0);
		\foreach \ii in {-5,...,5}
		{
			\draw[densely dotted] (\ii,.5)--++(0,-1) node [below] {$\ii$};
		}
		\draw[red,ultra thick] (-5,-.5)--++(0,1);
		\draw[red,ultra thick] (-4,-.5)--++(0,1);
		\draw[red,ultra thick] (-3,-.5)--++(0,1);
		\draw[red,ultra thick] (-1,-.5)--++(0,1);
		\draw[red,ultra thick] (0,-.5)--++(0,.5)--(2,0)--++(0,.5);
		\draw[red,ultra thick] (1,-.5)--++(0,1);
		\draw[red,ultra thick] (3,-.5)--++(0,.5)--++(3,0);
		\node at (.5,-.4) {$\mathcal{T}$};
		\node at (.5,.3) {$\mathcal{R}$};
		\node at (-5.5,.25) {\large$\dots$};
		\node at (5.5,.25) {\large$\dots$};
		\end{tikzpicture}}
		\right)
	}
	{
	W
	\left(
		\scalebox{.7}{
		\begin{tikzpicture}
			[scale=1,thick,baseline=-9pt]
		\draw[densely dotted] (-5.5,0)--(0.5,0);
		\foreach \ii in {-5,...,0}
		{
			\draw[densely dotted] (\ii,.5)--++(0,-1) node [below] {$\ii$};
		}
		\draw[red,ultra thick] (-5,-.5)--++(0,1);
		\draw[red,ultra thick] (-4,-.5)--++(0,1);
		\draw[red,ultra thick] (-3,-.5)--++(0,1);
		\draw[red,ultra thick] (-2,-.5)--++(0,1);
		\draw[red,ultra thick] (-1,-.5)--++(0,1);
		\draw[red,ultra thick] (0,-.5)--++(0,1);
		\node at (-5.5,.25) {\large$\dots$};
		\end{tikzpicture}}
		\right)
	W
	\left(
		\scalebox{.7}{
		\begin{tikzpicture}
			[scale=1,thick,baseline=-9pt]
		\draw[densely dotted] (.5,0)--(5.5,0);
		\foreach \ii in {1,...,5}
		{
			\draw[densely dotted] (\ii,.5)--++(0,-1) node [below] {$\ii$};
		}
		\node at (5.5,.25) {\large$\dots$};
		\draw[red,ultra thick] (0,0)--++(6,0);
		\end{tikzpicture}}
		\right)
	},
	\\
	&\hspace{20pt}
	\langle C^{\mathbb{Z}}(x,r)e_{\mathcal{T}},e_{\mathcal{R}} \rangle =
	\frac{W
	\left(
		\scalebox{.7}{
		\begin{tikzpicture}
			[scale=1,thick,baseline=-9pt]
		\draw[densely dotted] (-5.5,0)--(5.5,0);
		\foreach \ii in {-5,...,5}
		{
			\draw[densely dotted] (\ii,.5)--++(0,-1) node [below] {$\ii$};
		}
		\draw[red,ultra thick] (-6,0)--++(4,0)--++(0,.5);
		\draw[red,ultra thick] (-5,-.5)--++(0,1);
		\draw[red,ultra thick] (-4,-.5)--++(0,1);
		\draw[red,ultra thick] (-3,-.5)--++(0,1);
		\draw[red,ultra thick] (-1,-.5)--++(0,1);
		\draw[red,ultra thick] (0,-.5)--++(0,.5)--(2,0)--++(0,.5);
		\draw[red,ultra thick] (1,-.5)--++(0,1);
		\draw[red,ultra thick] (3,-.5)--++(0,1);
		\node at (.5,-.4) {$\mathcal{T}$};
		\node at (.5,.3) {$\mathcal{R}$};
		\node at (-5.5,.25) {\large$\dots$};
		\node at (5.5,.25) {\large$\dots$};
		\end{tikzpicture}}
		\right)
	}
	{
	W
		\left(
		\scalebox{.7}{
		\begin{tikzpicture}
			[scale=1,thick,baseline=-9pt]
		\draw[densely dotted] (-5.5,0)--(0.5,0);
		\foreach \ii in {-5,...,0}
		{
			\draw[densely dotted] (\ii,.5)--++(0,-1) node [below] {$\ii$};
		}
		\draw[red,ultra thick] (-6,0)--++(6.5,0);
		\draw[red,ultra thick] (-5,-.5)--++(0,1);
		\draw[red,ultra thick] (-4,-.5)--++(0,1);
		\draw[red,ultra thick] (-3,-.5)--++(0,1);
		\draw[red,ultra thick] (-2,-.5)--++(0,1);
		\draw[red,ultra thick] (-1,-.5)--++(0,1);
		\draw[red,ultra thick] (0,-.5)--++(0,1);
		\node at (-5.5,.25) {\large$\dots$};
		\end{tikzpicture}}
		\right)
	},
	\\
	&\hspace{40pt}
	\langle D^{\mathbb{Z}}(x,r)e_{\mathcal{T}},e_{\mathcal{R}} \rangle =
	\frac{W
	\left(
		\scalebox{.7}{
		\begin{tikzpicture}
			[scale=1,thick,baseline=-9pt]
		\draw[densely dotted] (-5.5,0)--(5.5,0);
		\foreach \ii in {-5,...,5}
		{
			\draw[densely dotted] (\ii,.5)--++(0,-1) node [below] {$\ii$};
		}
		\draw[red,ultra thick] (-5,-.5)--++(0,1);
		\draw[red,ultra thick] (-4,-.5)--++(0,1);
		\draw[red,ultra thick] (-3,-.5)--++(0,.5)--++(1,0)--++(0,.5);
		\draw[red,ultra thick] (-1,-.5)--++(0,1);
		\draw[red,ultra thick] (0,-.5)--++(0,.5)--(2,0)--++(0,.5);
		\draw[red,ultra thick] (1,-.5)--++(0,1);
		\draw[red,ultra thick] (3,-.5)--++(0,1);
		\node at (.5,-.4) {$\mathcal{T}$};
		\node at (.5,.3) {$\mathcal{R}$};
		\node at (-5.5,.25) {\large$\dots$};
		\node at (5.5,.25) {\large$\dots$};
		\end{tikzpicture}}
		\right)
	}
	{
	W
	\left(
		\scalebox{.7}{
		\begin{tikzpicture}
			[scale=1,thick,baseline=-9pt]
		\draw[densely dotted] (-5.5,0)--(0.5,0);
		\foreach \ii in {-5,...,0}
		{
			\draw[densely dotted] (\ii,.5)--++(0,-1) node [below] {$\ii$};
		}
		\draw[red,ultra thick] (-5,-.5)--++(0,1);
		\draw[red,ultra thick] (-4,-.5)--++(0,1);
		\draw[red,ultra thick] (-3,-.5)--++(0,1);
		\draw[red,ultra thick] (-2,-.5)--++(0,1);
		\draw[red,ultra thick] (-1,-.5)--++(0,1);
		\draw[red,ultra thick] (0,-.5)--++(0,1);
		\node at (-5.5,.25) {\large$\dots$};
		\end{tikzpicture}}
		\right)
	}.
\end{align*}
We refer to \Cref{appC:normalized_ops,appC:fermionic_operators}
below for formal definitions of these operators in the Fock space
via a limiting procedure 
$m\to-\infty$, $n\to+\infty$
starting 
from finite segments $\left\{ m,m+1,\ldots,n-1,n  \right\}$.


\medskip

Thanks to the Yang--Baxter equation,
the row operators satisfy a number of commutation relations,
for example,
\begin{equation}
	\label{eq:intro_BD_relation}
	B^{\mathbb{Z}}(x,r)
	D^{\mathbb{Z}}(w,\theta)
	=
	\frac{x-\theta^{-2}w}{x-w}\,
	D^{\mathbb{Z}}(w,\theta)
	B^{\mathbb{Z}}(x,r),
\end{equation}
provided that $x,w$ satisfy \eqref{eq:intro_Cauchy_condition},
which is required to avoid diverging infinite series.
Other commutation relations are listed in
\Cref{prop:ABCD_infinite_volume}
in the text.

In fact, relation \eqref{eq:intro_BD_relation} is closely related
to the Cauchy identity
(\Cref{thm:intro_Cauchy}). 
Let $\lambda$ be a signature with $N$ parts,
and set
$\mathcal{T}=\{\lambda_1+N,\lambda_2+N-1,\ldots,\lambda_N+1 \}\cup \mathbb{Z}_{\le0}$.
Then
\begin{equation}
	\label{eq:intro_F_G_via_B_D}
	\begin{split}
	F_\lambda(x_1,\ldots,x_N;\mathbf{y};r_1,\ldots,r_N ;\mathbf{s} )
	&=
	\langle e_\mathcal{T} , B^{\mathbb{Z}}(x_N,r_N)\ldots
	B^{\mathbb{Z}}(x_1,r_1) e_{\mathbb{Z}_{\le0}} \rangle,
	\\
	G_\lambda(w_1,\ldots,w_M;\mathbf{y};\theta_1,\ldots,\theta_M ;\mathbf{s} )
	&=
	\langle e_\mathcal{T} , 
	D^{\mathbb{Z}}(w_M,\theta_M)\ldots
	 D^{\mathbb{Z}}(w_1,\theta_1)
	 e_{\mathbb{Z}_{\le N}} \rangle.
	\end{split}
\end{equation}
Applying
\eqref{eq:intro_BD_relation} several times, 
we have
\begin{equation}
	\label{eq:intro_BD_Cauchy}
	\begin{split}
		&
		\langle e_{\mathbb{Z}_{\le0}},
		B^{\mathbb{Z}}(x_N,r_N)\ldots
		B^{\mathbb{Z}}(x_1,r_1)
		D^{\mathbb{Z}}(w_M,\theta_M)\ldots
		D^{\mathbb{Z}}(w_1,\theta_1)
		e_{\mathbb{Z}_{\le N}} \rangle 
		\\&\hspace{7pt}=
		\prod_{i=1}^{N}\prod_{j=1}^{M}\frac{x_i-\theta_j^{-2}w_j}{x_i-w_j}\,
		\langle e_{\mathbb{Z}_{\le0}},
		D^{\mathbb{Z}}(w_M,\theta_M)\ldots
		D^{\mathbb{Z}}(w_1,\theta_1)
		B^{\mathbb{Z}}(x_N,r_N)\ldots
		B^{\mathbb{Z}}(x_1,r_1)
		e_{\mathbb{Z}_{\le N}} \rangle 
		\\&\hspace{7pt}=
		F_{(0,\ldots,0 )}(\mathbf{x};\mathbf{y};\mathbf{r};\mathbf{s})
		\prod_{i=1}^{N}\prod_{j=1}^{M}\frac{x_i-\theta_j^{-2}w_j}{x_i-w_j},
	\end{split}
\end{equation}
where in the last step we removed the $D^{\mathbb{Z}}$ operators
thanks to $e_{\mathbb{Z}_{\le 0}}$ in the 
inner product, and used the definition of $F$ again.
The first line of \eqref{eq:intro_BD_Cauchy}
is precisely the left-hand side of the Cauchy identity
\eqref{eq:intro_Cauhy_identity}. 
In the third line we need
an explicit product formula for 
$F_{(0,\ldots,0 )}$ which follows from \Cref{thm:intro_F_formula}
(see the proof of \Cref{thm:F_G_Cauchy_big} for this computation),
and the resulting expression becomes the right-hand side 
of the Cauchy identity 
\eqref{eq:intro_Cauhy_identity}.

\medskip

Using \eqref{eq:intro_F_G_via_B_D},
we may express the probability 
weight
$\mathscr{M}(\lambda)$ \eqref{eq:intro_FG_measure}
under the FG measure
as the following evaluated inner product in the Fock space:
\begin{equation}
	\label{eq:intro_M_lambda_via_B_D}
	\mathscr{M}(\lambda)=\frac{1}{Z}\,
	\langle 
	e_{\mathbb{Z}_{\le 0}}
	,
	B^{\mathbb{Z}}(x_N,r_N)\ldots
	B^{\mathbb{Z}}(x_1,r_1)
	\, I_\lambda
	D^{\mathbb{Z}}(w_M,\theta_M)\ldots
	D^{\mathbb{Z}}(w_1,\theta_1)
	e_{\mathbb{Z}_{\le N}} \rangle 
\end{equation}
Here and throughout this subsection $Z$ 
denotes the right-hand side of 
the Cauchy identity 
\eqref{eq:intro_Cauhy_identity}.
The operator
$I_\lambda$ is the 
rank one projection in $\mathscr{F}$
onto the semi-infinite subset corresponding to $\lambda$:
\begin{equation*}
	I_\lambda e_{\mathcal{R}}=
	\begin{cases}
		e_{\mathcal{R}},&
		\text{if $\mathcal{R}=\left\{ \lambda_1+N,\lambda_2+N-1,\ldots,\lambda_N+1,0,-1,-2,\ldots  \right\}$};\\
		0,&\text{otherwise},
	\end{cases}
\end{equation*}
for any semi-infinite subset $\mathcal{R}\subset\mathbb{Z}$.
Expressions similar to 
\eqref{eq:intro_M_lambda_via_B_D}
are available for FG processes as well, see
\Cref{sub:FG_meas_proc_through_ops} in the text.
For simplicity of notation, below in this subsection we stick to the case
of FG measures.

If instead of $I_\lambda$ 
we insert into 
\eqref{eq:intro_M_lambda_via_B_D}
a product of creation and annihilation operators in the Fock space 
$\mathscr{F}$,
we would get a correlation function
of the FG measure. 
Recall that the creation and annihilation operators are defined
as
\begin{equation}
	\label{eq:create_annih_intro}
	\begin{split}
		\psi_j e_{\mathcal{T}}
		=
		\begin{cases}
			(-1)^{
				\#\left\{ t\in \mathcal{T}\colon t>j \right\}
			}e_{\mathcal{T}\cup\left\{ j \right\}},
			&j\notin \mathcal{T};\\
			0,&j\in \mathcal{T},
		\end{cases}
		\qquad 
		\psi_j^*e_{\mathcal{T}}=
		\begin{cases}
			(-1)^{
				\#\left\{ t\in \mathcal{T}\colon t>j \right\}
			}
			e_{\mathcal{T}\setminus \left\{ j \right\}},&j\in \mathcal{T};\\
			0,&j\notin \mathcal{T}.
		\end{cases}
	\end{split}
\end{equation}
They satisfy the canonical anticommutation relations
\begin{equation*}
	\psi_k\psi_k^*+\psi_k^*\psi_k=1,\qquad 
	\psi_k\psi_\ell^*+\psi_\ell^*\psi_k
	=
	\psi_k^*\psi_\ell^*+\psi_\ell^*\psi_k^*
	=
	\psi_k\psi_\ell+\psi_\ell\psi_k=0,\qquad 
	k\ne \ell.
\end{equation*}

For any finite subset $A=\{a_1,\ldots,a_m\}\subset\mathbb{Z}_{\ge1}$
we have the following expression 
for the correlation function:
\begin{equation}
	\label{eq:intro_corr_f_via_creation_annihilation}
	\begin{split}
		&
		\mathbb{P}_{\mathscr{M}}\left[ A\subset\{\lambda_1+N,\lambda_2+N-1,
		\ldots,\lambda_N+1 \} \right]
		\\&
		\hspace{10pt}=
		\frac{1}{Z}\,
		\langle 
		e_{\mathbb{Z}_{\le 0}}
		,
		B^{\mathbb{Z}}(x_N,r_N)\ldots
		B^{\mathbb{Z}}(x_1,r_1)
		\,
		\psi_{a_m}
		\psi^*_{a_m}
		\ldots 
		\psi_{a_1}
		\psi^*_{a_1}
		D^{\mathbb{Z}}(w_M,\theta_M)\ldots
		D^{\mathbb{Z}}(w_1,\theta_1)
		e_{\mathbb{Z}_{\le N}} \rangle .
	\end{split}
\end{equation}
Inserting pairs of creation and annihilation operators
between the $D$
operators above produces correlation functions
of ascending FG processes.

\medskip

Instead of computing \eqref{eq:intro_corr_f_via_creation_annihilation}
directly, we replace the 
creation and annihilation operators
with certain 
generating series 
$\Psi(u),\Psi^*(v)$
containing 
creation and annihilation operators.
The series 
$\Psi(u),\Psi^*(v)$
themselves are operators which are expressed through
the row operators with special parameters.
Namely, 
define for any semi-infinite subset $\mathcal{T}$:
\begin{equation}
	\label{eq:intro_Psi_defn}
	\begin{split}
		\Psi(u,\xi)\,e_{\mathcal{T}}
		&
		:=
		D^{\mathbb{Z}}(u,\sqrt{u/\xi})C^{\mathbb{Z}}(\xi,\sqrt{\xi/u})
		(-1)^{c(\mathcal{T})}e_{\mathcal{T}}
		,\\
		\Psi^*(\zeta,v)\,e_{\mathcal{T}}
		&
		:=
		D^{\mathbb{Z}}(\zeta,\sqrt{\zeta/v})
		B^{\mathbb{Z}}(v,\sqrt{v/\zeta})\,e_{\mathcal{T}}.
	\end{split}
\end{equation}

The statement below 
could be viewed 
as an inhomogeneous 
analogue of the Boson--Fermion
correspondence, 
cf.
\cite[Theorem 14.10]{Kac1990InfiniteDim}
for the homogeneous version.
\begin{theorem}[\Cref{thm:Psi_PsiStar_Fock} in the text]
	\label{thm:intro_boson_fermion}
	As operators on $\mathscr{F}$, we have
	\begin{equation}
		\label{eq:intro_boson_fermion}
		\begin{split}
			\Psi(u,\xi)&=\sum_{j\in \mathbb{Z}}
			\frac{y_j(1-s_j^{-2})}{u-s_j^{-2}y_j}
			\, 
			\mathcal{P}_{0,j-1}(u\mid \mathbf{y},\mathbf{s}^{-2}\mathbf{y})
			\,
			\psi_j,\\
			\Psi^*(\zeta,v)&=\sum_{j\in \mathbb{Z}}
			\frac{v-\zeta}{v-y_j}
			\,
			\mathcal{P}_{0,j-1}(v\mid \mathbf{s}^{-2}\mathbf{y},\mathbf{y} )
			\,
			\psi_j^*,
		\end{split}
	\end{equation}
	where we use the notation of inhomogeneous
	powers 
	\eqref{eq:intro_inhom_powers}.
\end{theorem}

Let us set $\Psi(u):=\Psi(u,0)$ and $\Psi^*(v)=\Psi^*(0,v)$.
From 
\eqref{eq:intro_boson_fermion}
we see that 
$\Psi(u,\xi)$ does not depend on $\xi$,
and 
$\Psi^*(v)$ is 
well-defined
by specializing 
the second line of \eqref{eq:intro_boson_fermion}.
Thanks to \Cref{thm:intro_boson_fermion},
operators $\Psi(u),\Psi^*(v)$ satisfy the
Wick determinantal formula
for the ``vacuum average'':
\begin{equation}
	\label{eq:intro_Wick}
	\bigl\langle 
	e_{\mathbb{Z}_{\le0}},
	\Psi(u_1)\Psi^*(v_1)
	\ldots
	\Psi(u_m)\Psi^*(v_m)\,
	e_{\mathbb{Z}_{\le0}}
	\bigr\rangle 
	=
	\det\left[ \frac{v_\alpha}{u_{\alpha'}-v_\alpha} \right]_{\alpha,\alpha'=1}^m.
\end{equation}
See \Cref{prop:wick,prop:action_Psi_many} in the text for details
on the Wick determinantal formula.

\medskip

There are two steps remaining in the computation 
of the correlation functions \eqref{eq:intro_corr_f_via_creation_annihilation}
of the FG measure $\mathscr{M}$.
First, using commutation relations between the row operators
and 
\eqref{eq:intro_Wick},
we show the following.
\begin{proposition}[A particular case of \Cref{prop:correlations_gen_function} 
	in the text]
	\label{prop:intro_correlation_computation}
	Let $u_1,\ldots,u_m ,v_1,\ldots,v_m $
	be independent variables satisfying
	certain conditions
	(see \eqref{eq:correlations_gen_function_conditions}
	for details).
	Then we have
	\begin{equation}
		\label{eq:intro_correlation_computation}
		\begin{split}
			&
			\frac{1}{Z}\,
			\langle 
			e_{\mathbb{Z}_{\le 0}}
			,
			B^{\mathbb{Z}}(x_N,r_N)\ldots
			B^{\mathbb{Z}}(x_1,r_1)
			\Psi(u_m)\Psi^*(v_m)
			\ldots 
			\Psi(u_1)\Psi^*(v_1)
			\\&\hspace{220pt}\times
			D^{\mathbb{Z}}(w_M,\theta_M)\ldots
			D^{\mathbb{Z}}(w_1,\theta_1)
			e_{\mathbb{Z}_{\le N}} \rangle
			\\&\hspace{20pt}=
			\prod_{i=1}^{M}\prod_{\alpha=1}^{m}
			\frac{(v_\alpha-\theta_i^{-2}w_i)(u_\alpha-w_i)}
			{(v_\alpha-w_i)(u_\alpha-\theta_i^{-2}w_i)}
			\prod_{\alpha=1}^{m}
			\prod_{j=1}^{N}
			\frac{u_\alpha-y_j}{v_\alpha-y_j}
			\frac{v_\alpha-x_j}{u_\alpha-x_j}
			\det\left[ \frac{v_\alpha}{u_{\alpha'}-v_\alpha} \right]_{\alpha,\alpha'=1}^m.
		\end{split}
	\end{equation}
\end{proposition}

Using \eqref{eq:intro_boson_fermion},
we interpret \eqref{eq:intro_correlation_computation}
as an inhomogeneous generating series
of the correlation functions
\eqref{eq:intro_corr_f_via_creation_annihilation}
of the FG measure $\mathscr{M}$,
with the generating variables $u_i,v_j$.
The last step is to extract the
coefficients
\eqref{eq:intro_corr_f_via_creation_annihilation}
from this generating series.
The operation of a coefficient
extraction is linear, and we must apply
$2m$ such operations to the right-hand side of 
\eqref{eq:intro_correlation_computation}.
Due to the form of this right-hand side,
these operations may be put into the $m\times m$
determinant
(this is essentially the
Andr\'eief identity, cf.
\cite{forrester2019meet}).
This implies that the correlation functions have a determinantal
form.
In \Cref{thm:corr_kernel_no_contours}
in the text we write the resulting correlation kernel
of the general FG process as such a series coefficient.

Furthermore, the coefficient extraction can be 
done analytically by means of contour integrals. 
This is an extension of
the inhomogeneous biorthogonality 
\eqref{eq:intro_orthogonality},
to 
inhomogeneous powers indexed by both
positive and nonpositive integers,
see \Cref{lemma:Phi_extraction,lemma:PhiStar_extraction}
in the text.
The integration contours
for $u,v$ in the coefficient extraction must be chosen so that 
the commutation relations between the row operators
used to obtain \eqref{eq:intro_Wick}
and \eqref{eq:intro_correlation_computation}
are valid.
Using this, we 
determine the correct integration contours
for the correlation kernel $K_{\mathscr{AP}}$
\eqref{eq:intro_K_AP}
of the ascending FG process, 
see 
\Cref{thm:kernel_matching}
in the text. 
As a result, we have computed the
correlation kernel $K_{\mathscr{AP}}$
in two ways, via fermionic operators and 
via an Eynard--Mehta type approach, and
both computations led to the same expression.

\subsection*{Acknowledgments}

Amol Aggarwal 
was partially supported by a Clay Research Fellowship.
Alexei Borodin was partially supported by the NSF 
grants DMS-1664619, DMS-1853981, and the Simons Investigator program.
Leonid Petrov was partially supported by the NSF 
grant DMS-1664617, and the 
Simons Collaboration Grant for Mathematicians 709055. 
Michael Wheeler was partially supported by an Australian Research Council Future
Fellowship, grant FT200100981.
This material is based upon work supported by the National
Science Foundation under Grant No. DMS-1928930 while Aggarwal and Petrov
participated in program hosted by the Mathematical
Sciences Research institute in Berkeley, California, during
the Fall 2021 semester.
We are very grateful to the anonymous referees for numerous helpful remarks.

\newpage
\part{Symmetric Functions}
\label{partI}

In this part (accompanied by \Cref{appA:F_G_formula_proofs}) we develop
symmetric rational functions based on the free fermion six vertex model.

\section{Free fermion six vertex model}
\label{sec:vertex_weights_ff6v}

\subsection{Vertex weights}
\label{sub:vertex_weights_ff6v}

We consider the weights 
$w_{\mathrm{6V}}(i_1,j_1;i_2,j_2)$, $i_1,j_1,i_2,j_2
\in\left\{ 0,1 \right\}$,
of the asymmetric six vertex (square ice) 
model:
\begin{equation}
	\label{eq:six_vertex_weights_traditional}
	\begin{split}
		w_{\mathrm{6V}}(0,0;0,0)&=a_1,\qquad 
		w_{\mathrm{6V}}(1,1;1,1)=a_2,\qquad 
		w_{\mathrm{6V}}(1,0;1,0)=b_1,\\
		w_{\mathrm{6V}}(0,1;0,1)&=b_2,\qquad 
		w_{\mathrm{6V}}(1,0;0,1)=c_1,\qquad 
		w_{\mathrm{6V}}(0,1;1,0)=c_2,
	\end{split}
\end{equation}
see \Cref{fig:weights_6V} for the illustration.
By agreement, $w_{\mathrm{6V}}(i_1,j_1;i_2,j_2)$
is set to zero unless $i_1+j_1=i_2+j_2$, which corresponds to the
path preservation property: the number of paths coming 
into a vertex equals the number of paths coming out of it.
The notation
$a_1,a_2,b_1,b_2,c_1,c_2$ 
for the vertex weights follows the longstanding tradition, for example, see
\cite[Ch. 8]{baxter2007exactly}, \cite{reshetikhin2010lectures}.

\medskip
We further assume that our six vertex weights
obey the \emph{free fermion condition}
\begin{equation}
	\label{eq:6v_free_fermion_condition}
	a_1a_2+b_1b_2=c_1c_2.
\end{equation}
In other words, we impose the 
vanishing of the quantity
$\Delta=\dfrac{a_1a_2+b_1b_2-c_1c_2}{2\sqrt{a_1a_2b_1b_2}}$
associated with the six vertex weights.

\begin{figure}[htpb]
	\centering
	\includegraphics[width=\textwidth]{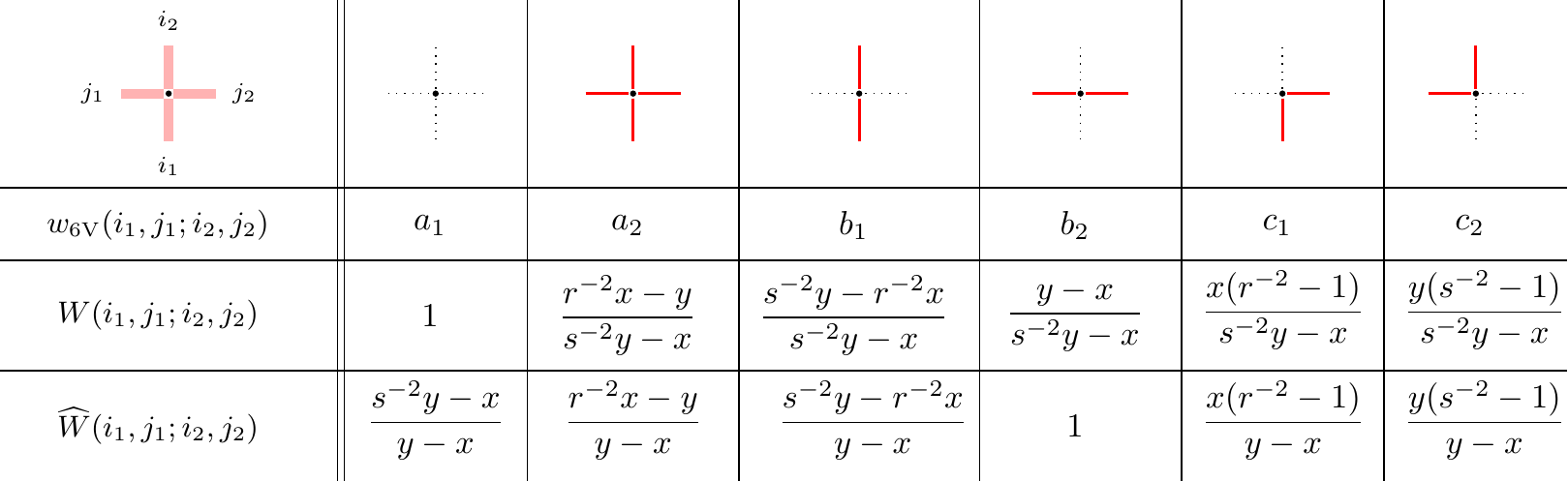}
	\caption{Weights \eqref{eq:six_vertex_weights_traditional}
		of the free fermion six vertex model, and their 
		parametrizations \eqref{eq:weights_W}, \eqref{eq:weights_W_hat} 
		with the four parameters $x,y,r,s$.}
	\label{fig:weights_6V}
\end{figure}

\medskip

The free fermion condition \eqref{eq:6v_free_fermion_condition}
leaves five out of six independent parameters. 
Furthermore, in order to build symmetric functions,
we normalize the weights so that 
one of them becomes equal to $1$, and we can repeat this type 
of vertices infinitely many times in a configuration.
The normalization leaves four independent parameters.
We make two different choices which of the weights to set to $1$:
\begin{enumerate}[$\bullet$]
	\item 
		Setting the weight $a_1$
		of the empty vertex $(0,0;0,0)$
		we get the weights
		\begin{equation}
			\label{eq:weights_W}
			\begin{split}
				W(0,0;0,0)=1,\quad 
				W(1,1;1,1)&=\frac{r^{-2}x-y}{s^{-2}y-x},\quad 
				W(1,0;1,0)=\frac{s^{-2}y-r^{-2}x}{s^{-2}y-x},\\
				W(0,1;0,1)=\frac{y-x}{s^{-2}y-x},\quad 
				W(1,0;0,1)&=\frac{x(r^{-2}-1)}{s^{-2}y-x},\quad 
				W(0,1;1,0)=\frac{y(s^{-2}-1)}{s^{-2}y-x}.
			\end{split}
		\end{equation}
	\item Setting the weight $b_2$ of the 
		vertex $(0,1;0,1)$,
		we get the weights
		\begin{equation}
			\label{eq:weights_W_hat}
			\begin{split}
				\widehat W(0,0;0,0)=\frac{s^{-2}y-x}{y-x},\quad 
				\widehat W(1,1;1,1)&=\frac{r^{-2}x-y}{y-x},\quad 
				\widehat W(1,0;1,0)=\frac{s^{-2}y-r^{-2}x}{y-x},\\
				\widehat W(0,1;0,1)=1,\quad 
				\widehat W(1,0;0,1)&=\frac{x(r^{-2}-1)}{y-x},\quad 
				\widehat W(0,1;1,0)=\frac{y(s^{-2}-1)}{y-x}.
			\end{split}
		\end{equation}
\end{enumerate}
See \Cref{fig:weights_6V} for an illustration.
The four parameters which they depend on are denoted by
$x,y,r,s$. Sometimes we will indicate this dependence explicitly
as 
\begin{equation*}
	W(i_1,j_1;i_2,j_2\mid x;y;r;s),\qquad 
	\widehat W(i_1,j_1;i_2,j_2\mid x;y;r;s).
\end{equation*}
The concrete choice of the parametrization
as in 
\eqref{eq:weights_W}--\eqref{eq:weights_W_hat}
is dictated by the form of the Yang--Baxter equation
(see the next \Cref{sub:YBE}), and 
by the overall goal of constructing symmetric functions.

One readily checks that the weights $W$ and $\widehat{W}$
satisfy the free fermion condition \eqref{eq:6v_free_fermion_condition}.
We also have
\begin{equation}
	\label{eq:W_What_renormalization}
	\widehat{W}(i_1,j_1;i_2,j_2)=
	\frac{W(i_1,j_1;i_2,j_2)}{W(0,1;0,1)}
\end{equation}
for all $i_1,j_1,i_2,j_2\in \left\{ 0,1 \right\}$, 
that is, the weights $W$ and $\widehat{W}$ differ only by normalization.

\begin{remark}
	\label{rmk:sl_11_symmetry}
	The free fermion six vertex weights 
	\eqref{eq:weights_W}--\eqref{eq:weights_W_hat}
	possess $U_q(\widehat{\mathfrak{sl}}(1|1))$ quantum affine superalgebra symmetry,
	and they are a special case of the higher rank 
	$U_q(\widehat{\mathfrak{sl}}(m|n))$
	integrable
	weights studied in
	the companion work
	\cite{agg-bor-wh2020-sl1n}.
\end{remark}

\subsection{Yang--Baxter equation}
\label{sub:YBE}

Define the following cross vertex weights:
\begin{equation}
	\label{eq:R_weights}
	\begin{split}
		R(0,0;0,0)=1,\quad
		R(1,1;1,1)&=\dfrac{x_2-x_1r_1^{-2}}{x_1-x_2r_2^{-2}},\quad
		R(1,0;1,0)=\dfrac{x_1r^{-2}_1-x_2r^{-2}_2}{x_1-x_2r^{-2}_2},\\
		R(0,1,0,1)=\dfrac{x_1-x_2}{x_1-x_2r^{-2}_2},\quad
		R(1,0;0,1)&=\dfrac{x_1(1-r^{-2}_1)}{x_1-x_2r^{-2}_2},\quad
		R(0,1;1,0)=\dfrac{x_2(1-r^{-2}_2)}{x_1-x_2r^{-2}_2}.
	\end{split}
\end{equation}
See \Cref{fig:cross_vertex_weights} for an illustration.
These weights, together with $W$ or $\widehat W$,
satisfy the Yang--Baxter equation:
\begin{figure}[htpb]
	\centering
	\includegraphics[width=\textwidth]{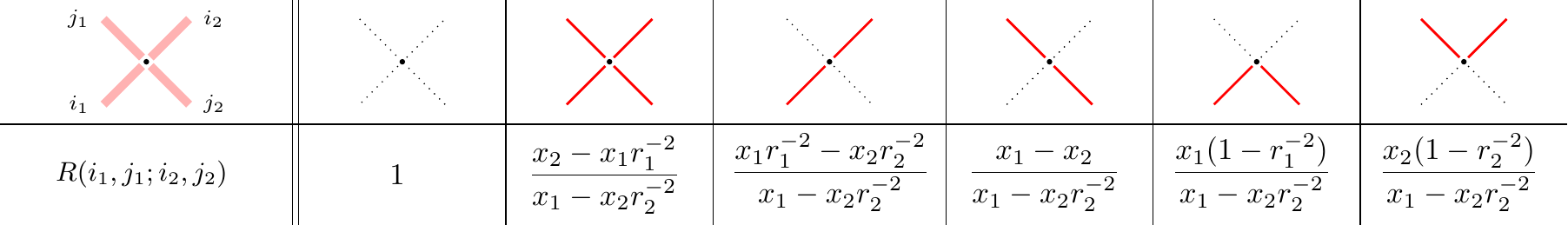}
	\caption{The cross vertex weights \eqref{eq:R_weights}.}
	\label{fig:cross_vertex_weights}
\end{figure}

\begin{proposition}[Yang--Baxter equation]
	\label{prop:YBE_first}
	For any fixed $i_1,i_2,i_3,j_1,j_2,j_3\in \left\{ 0,1 \right\}$ we have
	\begin{equation}
		\label{eq:YBE_first}
		\begin{split}
			&
			\sum_{k_1,k_2,k_3 \in \left\{ 0,1 \right\}}
			R(i_2,i_1;k_2,k_1)\,
			W(i_3,k_1;k_3,j_1\mid x_1;y;r_1;s)\,
			W(k_3,k_2;j_3,j_2\mid x_2;y;r_2;s)
			\\&\hspace{20pt}=
			\sum
			_{k_1',k_2',k_3'\in \left\{ 0,1 \right\}}
			R(k_2',k_1';j_2,j_1)\,
			W(i_3,i_2;k_3',k_2'\mid x_2;y;r_2;s)\,
			W(k_3',i_1;j_3,k_1'\mid x_1;y;r_1;s),
		\end{split}
	\end{equation}
	The same equation holds if we replace one or both of the weights
	\begin{equation*}
		W(\cdots\mid x_1;y;r_1;s),\qquad 
		W(\cdots\mid x_2;y;r_2;s)
	\end{equation*}
	everywhere
	by the corresponding $\widehat{W}$, which results in three other identities.
	See \Cref{fig:YBE_red} for a graphical illustration of the sums
	in both sides.
\end{proposition}
\begin{proof}
	Identity 
	\eqref{eq:YBE_first}
	or, 
	more precisely, the family of identities
	depending on the boundary conditions
	$i_1,i_2,i_3,j_1,j_2,j_3 $,
	is checked in a straightforward way.
	The three other families of identities involving the weights
	$\widehat{W}$ are multiples of \eqref{eq:YBE_first}, 
	thanks to \eqref{eq:W_What_renormalization}.
\end{proof}

\begin{remark}
	\label{rmk:YBE_from_slmn}
	The Yang--Baxter equation
	of \Cref{prop:YBE_first}
	is a consequence of 
	the more general 
	$U_q(\widehat{\mathfrak{sl}}(m|n))$ statement,
	see
	Proposition 5.1.4 and Example 8.1.1 
	in \cite{agg-bor-wh2020-sl1n}.
	This may be considered a conceptual
	reason 
	behind \Cref{prop:YBE_first}
	since our
	weights are obtained from the ones 
	in \cite{agg-bor-wh2020-sl1n} by fusion.
\end{remark}

\begin{figure}[htpb]
	\centering
	\includegraphics[width=.7\textwidth]{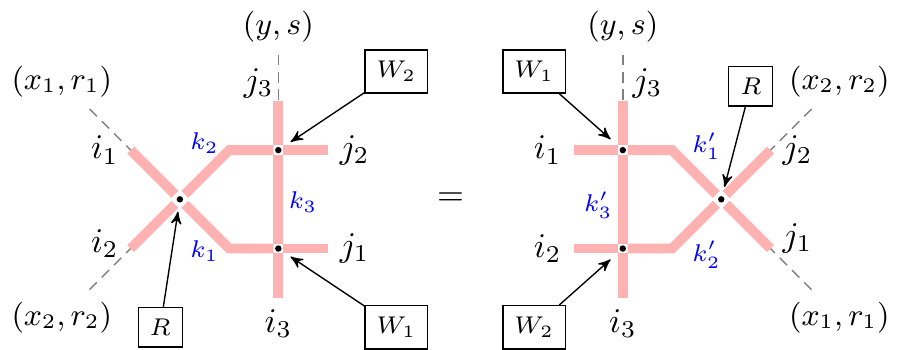}
	\caption{Illustration of the Yang--Baxter equation. 
		For fixed boundary values $i_1,i_2,i_3,j_1,j_2,j_3$,
		the partition functions in both sides are equal to each other.
		Here $W_j\equiv W(\cdots\mid x_j;y;r_j;s)$. The parameters
		$(x_1,r_1)$, $(x_2,r_2)$ and $(y,s)$ correspond to lines, which is also indicated.}
	\label{fig:YBE_red}
\end{figure}

\subsection{Row operators}
\label{sub:row_operators}

Based on the vertex weights, we define certain row operators
acting in tensor powers of $\mathbb{C}^2$. Thanks to the Yang--Baxter equation,
these operators satisfy certain commutation relations.

Let us fix sequences $\mathbf{s} = (s_1, s_2, \ldots ) \subset \mathbb{C}$ and $\mathbf{y}
= (y_1, y_2, \ldots ) \subset \mathbb{C}$ of complex numbers. For
any integer $k \ge 1$, we let $V^{(k)} \simeq \mathbb{C}^2$ denote
the two-dimensional complex vector space spanned by two vectors
$e_0^{(k)}$ and $e_1^{(k)}$. For notational convenience, we will
also set $e_j^{(k)} = 0$ for $j \notin \{ 0, 1 \}$. 

Next, for any complex numbers $x, r \in \mathbb{C}$, we define four
operators $A = A (x, r)$, $B = B (x, r)$, $C = C (x, r)$, and $D = D
(x, r)$ acting from the left on any $V^{(k)}$ through the equations
\begin{align}
\label{eq:abcdv}
\begin{aligned}
	& A e_i^{(k)} = W ( i, 1; i, 1 \mid x;y_k;r; s_k)\, e_i^{(k)}; 
	\qquad \qquad 
	B e_i^{(k)} = W (i, 0; i - 1, 1 \mid x;y_k;r; s_k )\, e_{i - 1}^{(k)}; \\
	& C e_i^{(k)} = W (i, 1; i + 1, 0 \mid x;y_k;r; s_k)\, e_{i + 1}^{(k)}; 
	\qquad 
	D e_i^{(k)} = W (i, 0; i, 0 \mid x;y_k;r; s_k)\, e_i^{(k)},
\end{aligned}
\end{align} 
where the weights $W$ are given in \eqref{eq:weights_W}.
	
We also define actions of these operators on tensor products
$V^{(k_1)} \otimes V^{(k_2)} \otimes \ldots \otimes V^{(k_n)}$. To do this
in the case $n = 2$, set 
\begin{align}
\label{eq:abcdv_n2} 
\begin{aligned}
& A \big( v_1 \otimes v_2) = C v_1 \otimes B v_2 + A v_1 \otimes A v_2; \\ 
& B (v_1 \otimes v_2) = D v_1 \otimes B v_2 + B v_1 \otimes A v_2; \\ 
& C (v_1 \otimes v_2) = C v_1 \otimes D v_2 + A v_1 \otimes C v_2; \\ 
& D (v_1 \otimes v_2) = D v_1 \otimes D v_2 + B v_1 \otimes C v_2,
\end{aligned} 
\end{align}
for all $v_1 \in V^{(k_1)}, v_2\in V^{(k_2)}$
(see \Cref{fig:tensor_ABCD} for an illustration). 
Then extend this action to $V^{(k_1)} \otimes V^{(k_2)}
\otimes \ldots \otimes V^{(k_n)}$ for $n > 2$ using the above relations
\eqref{eq:abcdv_n2} inductively on $n$.
The induction step consists of 
taking $v_1$ there to be
an element of $V^{(k_1)} \otimes V^{(k_2)} \otimes \ldots \otimes V^{(k_{n -1})}$,
and $v_2$ an element of~$V^{(k_n)}$. 
This action on tensor products
is associative.

\begin{figure}[htpb]
	\centering
	\includegraphics[width=.95\textwidth]{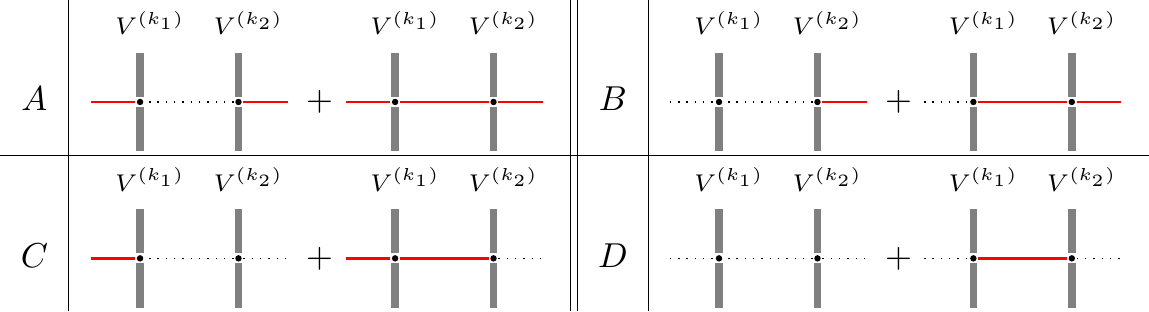}
	\caption{Graphical illustration of the action of the operators $A,B,C,D$
	on tensor products \eqref{eq:abcdv_n2}.
	The sums in \eqref{eq:abcdv_n2}
	correspond to various states of the inner horizontal edge. The vertical edges can have arbitrary states.}
	\label{fig:tensor_ABCD}
\end{figure}

The operators
$A$, $B$, $C$, and $D$
acting in any 
tensor product
$V^{(k_1)} \otimes V^{(k_2)} \otimes \ldots \otimes V^{(k_n)}$
satisfy the following commutation relations:
\begin{proposition}
	\label{prop:ABCD_YBE}
	For any $x_1, x_2, r_1, r_2 \in \mathbb{C}$, we have 
	\begin{align}
		\label{eq:A2A1}
	& A (x_2, r_2) A (x_1, r_1) = A (x_1, r_1) A (x_2, r_2); \\
	\label{eq:B2B1_commute}
	& B (x_2, r_2) B (x_1, r_1) = \displaystyle\frac{r^{-2}_1 x_1 - x_2}{r^{-2}_2 x_2 - x_1} B (x_1, r_1) B (x_2, r_2); \\
	\label{eq:C2C1_commute}
	& C (x_2, r_2) C (x_1, r_1) = \displaystyle\frac{r^{-2}_2 x_2 - x_1}{r^{-2}_1 x_1 - x_2} C (x_1, r_1) C (x_2, r_2); \\
	& D (x_2, r_2) D (x_1, r_1) = D (x_1, r_1) D (x_2, r_2); \label{eq:D2D1_commute}\\
	& \label{eq:B2D1}
	B (x_2, r_2) D (x_1, r_1) = \displaystyle\frac{r^{-2}_1 x_1 - x_2}{x_1 - x_2} D (x_1, r_1) B (x_2, r_2) + \displaystyle\frac{(1 - r^{-2}_2) x_2}{x_1 - x_2} D (x_2, r_2) B (x_1, r_1); \\
	& 
	\begin{aligned}
		\label{eq:D2_B1}
	&D (x_2, r_2) B (x_1, r_1)
	\\
	&\hspace{40pt}
	= \displaystyle\frac{r^{-2}_1 x_1 - x_2}{r^{-2}_1 x_1 - r^{-2}_2 x_2} B (x_1, r_1) D (x_2, r_2) + \displaystyle\frac{(1 - r^{-2}_1) x_1}{r^{-2}_1 x_1 - r^{-2}_2 x_2} B (x_2, r_2) D (x_1, r_1);
	\end{aligned}
	\\
	\label{eq:D2_C1}
	& 
	D (x_2, r_2) C (x_1, r_1) 
	= 
	\displaystyle\frac{r^{-2}_2 x_2 - x_1}{x_2 - x_1} C (x_1, r_1) D (x_2, r_2) + \displaystyle\frac{(1 - r^{-2}_2) x_2}{x_2 - x_1} C (x_2, r_2) D (x_1, r_1); 
	\\
	& 
	\begin{aligned}
		\label{eq:C2_D1}
		&
		C (x_2, r_2) D (x_1, r_1) 
		\\
		&\hspace{40pt}= \displaystyle\frac{r^{-2}_2 x_2 -
		x_1}{r^{-2}_2 x_2 - r^{-2}_1 x_1} D (x_1, r_1) C (x_2, r_2) +
		\displaystyle\frac{x_1 (1 - r^{-2}_1)}{r^{-2}_2 x_2 - r^{-2}_1 x_1}
		D(x_2, r_2) C (x_1, r_1); 
	\end{aligned}
	\\
	& A (x_2, r_2) C (x_1, r_1) = \displaystyle\frac{r^{-2}_2 x_2 - x_1}{x_1 - x_2} C (x_1, r_1) A (x_2, r_2) + \displaystyle\frac{x_2 (1 - r^{-2}_2)}{x_1 - x_2} C (x_2, r_2) A (x_1, r_1),
\end{align}
and
\begin{align}
	\begin{aligned}
		\label{eq:ADBC_1}
	\displaystyle\frac{x_1 (r^{-2}_1 - 1)}{r^{-2}_2 x_2 - x_1} & D (x_2, r_2) A (x_1, r_1) + \displaystyle\frac{r^{-2}_2 x_2 - r^{-2}_1 x_1}{r^{-2}_2 x_2 - x_1} C (x_2, r_2) B (x_1, r_1) \\
	& = \displaystyle\frac{x_1 (r^{-2}_ 1 - 1)}{r^{-2}_2 x_2 - x_1} D (x_1, r_1) A (x_2, r_2) + \displaystyle\frac{x_2 - x_1}{r^{-2}_2 x_2 - x_1} B (x_1, r_1) C (x_2, r_2);
	\end{aligned}\\
	\begin{aligned}
		\label{eq:ADBC_2}
	\displaystyle\frac{x_2 (r^{-2}_2 - 1)}{r^{-2}_2 x_2 - x_1} & A (x_2, r_2) D (x_1, r_1) + \displaystyle\frac{x_2 - x_1}{r^{-2}_2 x_2 - x_1} B (x_2, r_2) C (x_1, r_1) \\
& = \displaystyle\frac{x_2 (r^{-2}_2 - 1)}{r^{-2}_2 x_2 - x_1} A (x_1, r_1) D (x_2, r_2) + \displaystyle\frac{r^{-2}_2 x_2 - r^{-2}_1 x_1}{r^{-2}_2 x_2 - x_1} C (x_1, r_1) B (x_2, r_2).	
	\end{aligned}
\end{align}
\end{proposition}
\begin{proof}
	First, assume that the operators act in a single
	two-dimensional space $V^{(k)}\simeq \mathbb{C}^{2}$. 
	Then
	all of the desired commutation relations 
	follow from the Yang--Baxter equation
	of \Cref{prop:YBE_first}. 
	Let us show how to get one of these relations,
	say, \eqref{eq:B2D1}, the others are obtained in a similar way.
	Write the Yang--Baxter equation \eqref{eq:YBE_first}
	with the boundary conditions
	$i_1=i_2=0$, $i_3=j_1=1$, $j_2=j_3=0$ and with 
	the parameters $(x_1,r_1)$ and $(x_2,r_2)$ interchanged.
	In the operator form, this Yang--Baxter equation reads
	\begin{equation}\label{eq:YBE_commutation_proof}
		D(x_1,r_1)B(x_2,r_2)=
		\frac{x_2(1-r^{-2}_2)}{x_2-x_1r^{-2}_1}D(x_2,r_2)B(x_1,r_1)+
		\frac{x_2-x_1}{x_2-x_1r^{-2}_1}B(x_2,r_2)D(x_1,r_1).
	\end{equation}
	Note that in the product $D(x_1,r_1)B(x_2,r_2)$,
	the $B$ and $D$ operator corresponds to the bottom and, respectively, top, vertex in
	the left-hand side of
	\Cref{fig:YBE_red},
	and same for all other products in this proposition.
	In \eqref{eq:YBE_commutation_proof}
	we then move the term containing
	$D(x_2,r_2)B(x_1,r_1)$ into the left-hand side, and 
	divide the identity by the prefactor $(x_2-x_1)/(x_2-x_1r^{-2}_1)$ in front of $B(x_2,r_2)D(x_1,r_1)$.
	This gives~\eqref{eq:B2D1} for the case of the two-dimensional space $V^{(k)}\simeq \mathbb{C}^{2}$.

	To extend the relations to arbitrary tensor products
	$V^{(k_1)} \otimes V^{(k_2)} \otimes \ldots \otimes V^{(k_n)}$,
	we apply the standard ``zipper argument'' to establish the Yang--Baxter equation for 
	a horizontal chain of two-vertex configurations, see \Cref{fig:YBE_chain}. 
	This equation follows by sequentially applying
	the original equation \eqref{eq:YBE_first} to move the cross vertex and swap the parameters
	$(x_i,r_i)$.
	Then the Yang--Baxter equation corresponding to the space
	$V^{(k_1)} \otimes V^{(k_2)} \otimes \ldots \otimes V^{(k_n)}$ 
	implies all the desired commutation relations.	
\end{proof}

\begin{figure}[htpb]
	\centering
	\includegraphics[width=.9\textwidth]{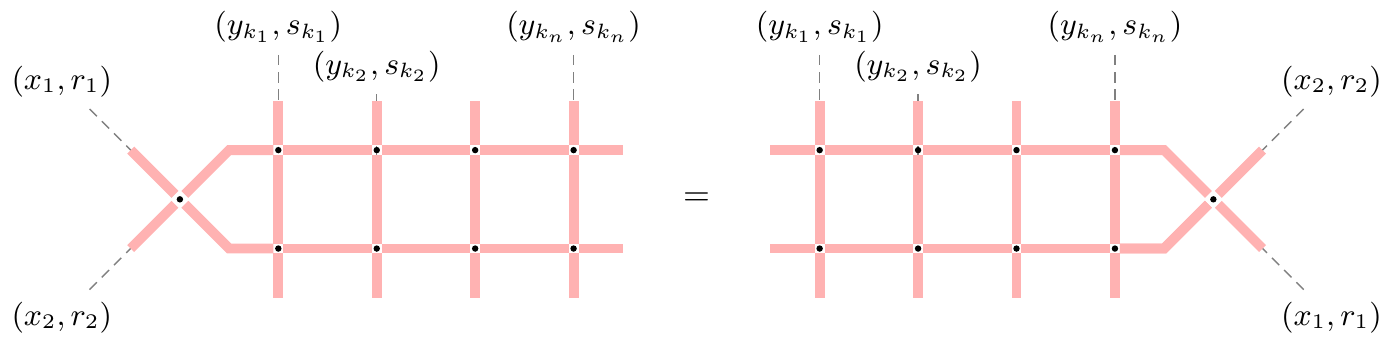}
	\caption{The Yang--Baxter equation for a horizontal chain of 
	two-vertex configurations.}
	\label{fig:YBE_chain}
\end{figure}

Using the weights $\widehat W$ \eqref{eq:weights_W_hat}, 
define four operators $\widehat{A}=\widehat{A}(x,r)$, 
$\widehat{B}=\widehat{B}(x,r)$, 
$\widehat{C}=\widehat{C}(x,r)$, 
and $\widehat{D}=\widehat{D}(x,r)$ 
acting from the right on each two-dimensional space $V^{(k)}$ as
\begin{equation}
	\label{eq:abcdv_hat}
	\begin{split}
		& e^{(k)}_i \, \widehat{A} = \widehat{W} (i, 1; i, 1 \mid x;y_k;r;s_k) \, e^{(k)}_i; 
		\qquad \qquad 
		e^{(k)}_i \, \widehat{B}  = \widehat{W} (i + 1, 0; i, 1\mid x;y_k;r;s_k) \, e^{(k)}_{i + 1}; 
		\\
		& e^{(k)}_i \, \widehat{C}  = \widehat{W} (i - 1, 1; i, 0\mid x;y_k;r;s_k) \, e^{(k)}_{i - 1};
		\qquad 
		e^{(k)}_i \, \widehat{D}  = \widehat{W} (i, 0; i, 0\mid x;y_k;r;s_k) \, e^{(k)}_i.
	\end{split}
\end{equation} 
Note the difference with the operators $A,B,C,D$ \eqref{eq:abcdv}
which read vectors $e_i^{(k)}$ corresponding to a vertex $(i_1,j_1;i_2,j_2)$ ``from 
bottom to top'' (i.e., map $e_{i_1}^{(k)}$ to $e_{i_2}^{(k)}$),
while $\widehat{A}, \widehat{B}, \widehat{C}, \widehat{D}$
read vectors ``from top to bottom''.
Note that the states of the left and right edges for the same-letter operators
in \eqref{eq:abcdv} and \eqref{eq:abcdv_hat} are same.

We extend the operators \eqref{eq:abcdv_hat} to tensor products
$V^{(k_1)}\otimes V^{(k_2)}\otimes \ldots \otimes V^{(k_n)}$ 
by first defining for $n=2$,
\begin{equation}
	\label{eq:abcdv_n2_hat}
	\begin{split}
		& (v_1 \otimes v_2)  \widehat{A} =  v_1 \widehat{C} \otimes v_2 \widehat{B} + v_1 \widehat{A} \otimes v_2 \widehat{A}; \\ 
		& (v_1 \otimes v_2)  \widehat{B} =  v_1 \widehat{D} \otimes v_2 \widehat{B} +  v_1 \widehat{B} \otimes v_2 \widehat{A}; \\ 
		& (v_1 \otimes v_2)  \widehat{C} =  v_1 \widehat{C} \otimes v_2 \widehat{D} + v_1 \widehat{A} \otimes v_2 \widehat{C}; \\ 
		& (v_1 \otimes v_2)  \widehat{D} =  v_1 \widehat{D} \otimes v_2 \widehat{D} + v_1 \widehat{B} \otimes v_2 \widehat{C},
	\end{split}
\end{equation}
and then for arbitrary $n$ by induction
similarly to 
\eqref{eq:abcdv_n2}.

The operators $\widehat{A},\widehat{B},\widehat{C},\widehat{D}$
acing in 
any tensor product
$V^{(k_1)} \otimes V^{(k_2)} \otimes \ldots \otimes V^{(k_n)}$
satisfy commutation relations which parallel the ones in \Cref{prop:ABCD_YBE}.
In the next proposition, let us list a few relations
which are
employed in proofs later in the paper.
They also follow from the Yang--Baxter equation (\Cref{prop:YBE_first})
and the ``zipper argument'', as in the proof of \Cref{prop:ABCD_YBE}.
\begin{proposition}
	\label{prop:ABCD_YBE_Hat}
	For any $x_1,x_2,r_1,r_2\in \mathbb{C}$, we have
	\begin{align}
		\label{eq:A2hatA1hat}
	& \widehat{A} (x_2, r_2) \widehat{A} (x_1, r_1) = \widehat{A} (x_1, r_1) \widehat{A} (x_2, r_2); \\
		\label{eq:B2hat_B1hat_commute}
	& \widehat{B}(x_2, r_2) \widehat{B}(x_1, r_1) = \displaystyle\frac{x_2 - r^{-2}_1 x_1}{x_1 - r^{-2}_2 x_2} \widehat{B} (x_1, r_1) \widehat{B}(x_2, r_2); \\
	\label{eq:BD_hat_relation}
	& \widehat{B} (x_2, r_2) \widehat{D}(x_1, r_1) = \displaystyle\frac{r^{-2}_1 x_1 - x_2}{x_1 - x_2} \widehat{D} (x_1, r_1) \widehat{B} (x_2, r_2) + \displaystyle\frac{(1 - r^{-2}_2) x_2}{x_1 - x_2} \widehat{D}(x_2, r_2) \widehat{B} (x_1, r_1); \\
	\label{eq:BA_hat_relation}
	& \widehat{B} (x_2, r_2) \widehat{A}(x_1, r_1) = \displaystyle\frac{r^{-2}_1 x_1 - x_2}{x_2 - x_1} \widehat{A}(x_1, r_1) \widehat{B} (x_2, r_2) + \displaystyle\frac{(1 - r^{-2}_2) x_2}{x_2 - x_1} \widehat{A}(x_2, r_2) \widehat{B} (x_1, r_1);\\
		\label{eq:D2hatD1hat}
		& \widehat{D} (x_2, r_2) \widehat{D} (x_1, r_1) = \widehat{D} (x_1, r_1) \widehat{D} (x_2, r_2).
	\end{align}
\end{proposition}

\section{Symmetric functions}
\label{sec:F_G_symm_funct}

Here we define symmetric functions
$F_\lambda$ and $G_\lambda$ 
(indexed by signatures $\lambda$) 
which are
partition functions of certain configurations
of the free fermion six vertex model.

\subsection{Signature states}
\label{sub:signature_states}

A (generalized) signature with $N$ parts is a nonincreasing integer sequence
\begin{equation*}
	\lambda=(\lambda_1\ge \ldots\ge \lambda_N ), \qquad \lambda_i\in \mathbb{Z}.
\end{equation*}
Denote $|\lambda|:=\lambda_1+\ldots+\lambda_N $.
We will mostly deal with nonnegative signatures, i.e., such that $\lambda_N\ge0$,
and will omit the word ``nonnegative''.
To a signature $\lambda$ with $N$ parts
we associate a configuration
\begin{equation}
	\label{eq:S_lambda_notation}
	\mathcal{S}(\lambda)=(\lambda_1+N,\lambda_2+N-1,\ldots,\lambda_{N-1}+2,\lambda_N+1 )\subset \mathbb{Z}_{\ge1}
\end{equation}
of distinct points
in the integer half-line.

Consider the (formal) infinite tensor product $V^{(1)}\otimes V^{(2)}\otimes\ldots $,
where
$V^{(k)}\simeq \mathbb{C}^2$ for all $k$
with basis $e_0^{(k)}, e_1^{(k)}$. Let $\mathscr{V}$ be the 
subset of the infinite tensor product
spanned by the following 
vectors:
\begin{equation}
	\label{eq:e_calT_notation}
	e_{\mathcal{T}}=
	e_{m_1}^{(1)}\otimes e_{m_2}^{(2)}\otimes e_{m_3}^{(3)}\otimes\ldots,
	\qquad m_i=\begin{cases}
		1,&i\in \mathcal{T};\\
		0,&\textnormal{otherwise},
	\end{cases}
\end{equation}
where $\mathcal{T}\subset \mathbb{Z}_{\ge1}$ runs over arbitrary finite sets.
These basis vectors are called \emph{finitary}. 
(We will sometimes use the same notation
\eqref{eq:e_calT_notation} for arbitrary subsets $\mathcal{T}\subseteq \mathbb{Z}_{\ge1}$.)
In particular, 
with each signature $\lambda$ we associate a \emph{signature state}
$e_{\mathcal{S}(\lambda)}$.
Note that all but finitely many of the tensor components of $e_{\mathcal{S}(\lambda)}$ are 
$e_0^{(k)}$.

By $\mathscr{V}_\ell$, $\ell=0,1,\ldots $, denote the
subspace of $\mathscr{V}$ spanned by $e_{\mathcal{T}}$ with $\mathcal{T}$
running over all $\ell$-element subsets of $\mathbb{Z}_{\ge1}$.
Let us extend some of the row operators defined in \Cref{sub:row_operators}
to act in the space~$\mathscr{V}$. Namely, 
thanks to $W(0,0;0,0)=1$, the operators $C(x,r)$ and $D(x,r)$ act in each of the $\mathscr{V}_\ell$'s, and 
\begin{equation}
	\label{eq:CD_half_infinite_operators}
	C(x,r)\colon \mathscr{V}_{\ell}\to \mathscr{V}_{\ell+1},
	\qquad 
	D(x,r)\colon \mathscr{V}_{\ell}\to \mathscr{V}_{\ell}.
\end{equation}
Similarly, thanks to $\widehat{W}(0,1;0,1)=1$, 
the operators $\widehat{A}(x,r)$ and $\widehat{B}(x,r)$
act as
\begin{equation}
	\label{eq:hat_AB_half_infinite_operators}
	\widehat{A}(x,r)\colon \mathscr{V}_{\ell}\to \mathscr{V}_{\ell},
	\qquad 
	\widehat{B}(x,r)\colon \mathscr{V}_{\ell}\to \mathscr{V}_{\ell+1}.
\end{equation}
These operators satisfy the commutation relations
\eqref{eq:A2hatA1hat}, \eqref{eq:B2hat_B1hat_commute}, and \eqref{eq:BA_hat_relation}.
Note that while the spaces $\mathscr{V}_{\ell}$ involve infinite tensor products,
in the action of the operators $\widehat A,\widehat B,C,D$
the boundary condition at infinity is always determined uniquely.

\begin{remark}
	Later in \Cref{sec:fermionic_operators} 
	we employ a similar infinite tensor product over the whole lattice
	$\mathbb{Z}$, or, more precisely, a corresponding Fock space,
	to compute a generating function 
	for correlations of certain probability distributions
	based on our symmetric functions.
\end{remark}

\subsection{Symmetric functions as partition functions}
\label{sub:def_F_G_as_partition_functions}

Here we define the functions 
$F_{\lambda/\mu}(\mathbf{x};\mathbf{y};\mathbf{r};\mathbf{s})$ and
$G_{\lambda/\mu}(\mathbf{x};\mathbf{y};\mathbf{r};\mathbf{s})$ as certain partition functions
of the free fermion six vertex model. 
\begin{definition}
	\label{def:G_function}
	Fix a positive integer $k$, two signatures $\lambda,\mu$
	with the same number of parts,
	and parameters $\mathbf{x}=(x_1,\ldots,x_k )$, 
	$\mathbf{y}=(y_1,y_2,\ldots )$, $\mathbf{r}=(r_1,\ldots,r_k )$,
	$\mathbf{s}=(s_1,s_2,\ldots )$.

	Consider the following boundary data in the half-infinite rectangle 
	$\mathbb{Z}_{\ge1}\times \{1,\ldots, k \}$.
	A path vertically enters the rectangle from the bottom 
	at each $m\in \mathcal{S}(\mu)$;
	a path vertically exits the rectangle at the top at each $\ell\in \mathcal{S}(\lambda)$;
	the left and right boundaries of the rectangle, as well
	as all the other boundary edges on top and bottom, are empty
	(see \Cref{fig:F_G_partition_functions}, left, for an illustration).

	Let the vertex weight at each 
	$(i,j)$ in the rectangle be $W(i_1,j_1;i_2,j_2\mid x_j;y_i;r_j;s_i)$ \eqref{eq:weights_W}.
	That is, the parameters $(x,r)$ are constant along the horizontal, 
	and $(y,s)$ are constant along the vertical direction.
	Denote the partition function of thus defined vertex model in the half-infinite rectangle
	by $G_{\lambda/\mu}=G_{\lambda/\mu}(\mathbf{x};\mathbf{y};\mathbf{r};\mathbf{s})$.
	Even though the domain is infinite,
	this partition function is well-defined thanks to $W(0,0;0,0)=1$.

	In the particular case $\mu=(0,\ldots,0 )$ (with the same number of parts as in $\lambda$), 
	we abbreviate $G_{\lambda}=G_{\lambda/\mu}$.
\end{definition}

\begin{definition}
	\label{def:F_function}
	Within the notation of \Cref{def:G_function}, 
	let now the number of parts in $\lambda$ be $N+k$ and the number of parts in $\mu$ be $N$ 
	for some $N\in \mathbb{Z}_{\ge0}$.
	Consider the following (different) boundary data for the half-infinite rectangle $\mathbb{Z}_{\ge1}\times
	\left\{ 1,\ldots,k  \right\}$. Let a path enter vertically from the bottom
	at each $\ell \in \mathcal{S}(\lambda)=(\lambda_1+N+k,\lambda_2+N+k-1,\ldots,\lambda_{N+k}+1)$; 
	a path exit vertically at the top
	at each $m\in \mathcal{S}(\mu)=(\mu_1+N,\mu_2+N-1,\ldots,\mu_N+1)$; and a path 
	exit the rectangle far to the right at each horizontal layer. Let 
	all the other boundary edges
	of the rectangle be empty (see \Cref{fig:F_G_partition_functions}, right, for an illustration).

	Let the vertex weight at each 
	$(i,j)$ in the rectangle be $\widehat{W}(i_1,j_1;i_2,j_2\mid x_j;y_i;r_j;s_i)$ \eqref{eq:weights_W_hat}.
	Denote the partition function of this vertex model by $F_{\lambda/\mu}=F_{\lambda/\mu}(\mathbf{x};\mathbf{y};\mathbf{r};\mathbf{s})$.
	This partition function is well-defined because $\widehat{W}(0,1;0,1)=1$.

	In the particular case $N=0$ (i.e., when $\mu=\varnothing$
	is the empty signature with no parts),
	we abbreviate $F_\lambda=F_{\lambda/\mu}$.
\end{definition}
We extend the definitions of $G_{\lambda/\mu}$ and $F_{\lambda/\mu}$ to arbitrary
pairs of signatures $\lambda,\mu$ (without the restrictions on the number of parts),
by setting these functions to zero if there are no path configurations with the 
prescribed boundary data.

\begin{figure}[htpb]
	\centering
	\includegraphics[width=\textwidth]{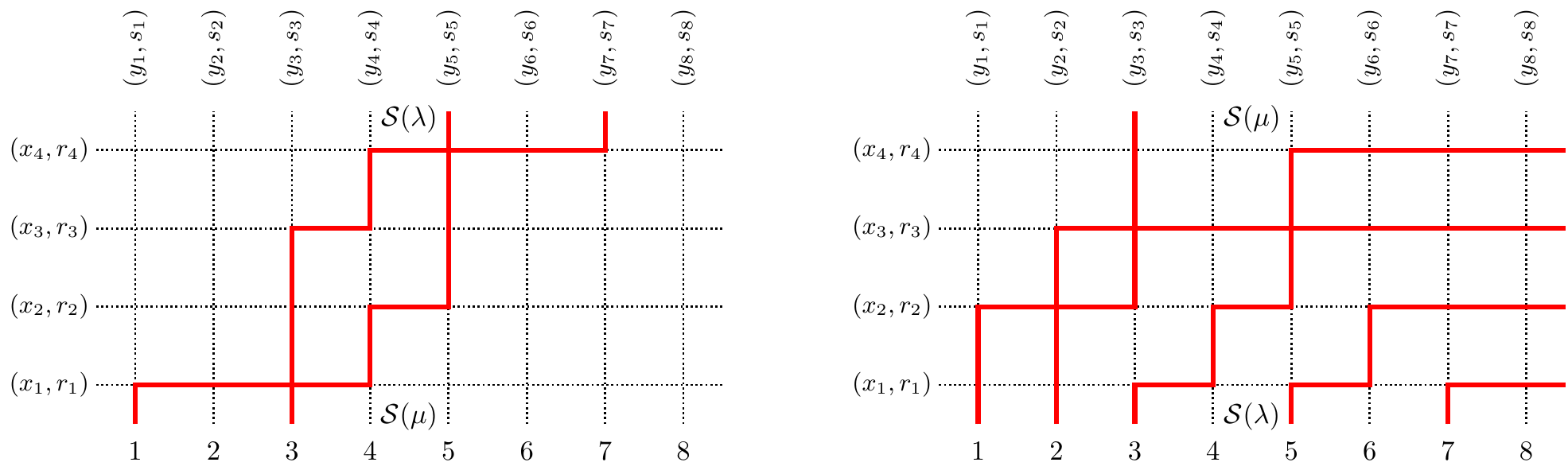}
	\caption{Examples of path configurations 
		contributing to the
		partition functions in \Cref{def:G_function,def:F_function}. 
		Left: $G_{\lambda/\mu}$
		with 
		$\lambda=(5,4)$ and
		$\mu=(1,0)$.
		Right: $F_{\lambda/\mu}$
		with $\lambda=(2,1,0,0,0)$ and
		$\mu=(2)$.}
	\label{fig:F_G_partition_functions}
\end{figure}

From the definitions it follows that
the partition functions $G_{\lambda/\mu}$ and $F_{\lambda/\mu}$ are
written in terms of the row operators from \Cref{sub:row_operators}
acting in the subspace $\mathscr{V}$
of the infinite tensor product (as explained in \Cref{sub:signature_states}):
\begin{proposition}
	\label{prop:F_G_row_op}
	The function
	$G_{\lambda/\mu}(x_1,\ldots,x_k;\mathbf{y};r_1,\ldots,r_k
	;\mathbf{s} )$ is equal to the coefficient 
	by $e_{\mathcal{S}(\lambda)}$ in 
	$D(x_k,r_k)\ldots D(x_1,r_1) e_{\mathcal{S}(\mu)}$.
	Here $e_{\mathcal{S}(\lambda)},e_{\mathcal{S}(\mu)}$
	belong to the same subspace $\mathscr{V}_N$, and the product of the $D(x_i,r_i)$'s preserves
	this subspace by \eqref{eq:CD_half_infinite_operators}.

	Similarly, the function
	$F_{\lambda/\mu}(x_1,\ldots,x_k;\mathbf{y};r_1,\ldots,r_k
	;\mathbf{s} )$ is equal to the coefficient 
	by $e_{\mathcal{S}(\lambda)}$ in 
	$e_{\mathcal{S}(\mu)}  \widehat{B}(x_k,r_k)\ldots \widehat{B}(x_1,r_1) $.
	Here $e_{\mathcal{S}(\mu)}\in \mathscr{V}_N$ and $e_{\mathcal{S}(\lambda)}\in \mathscr{V}_{N+k}$,
	and the product of the $\widehat{B}(x_i,r_i)$'s maps 
	$\mathscr{V}_N$ to $\mathscr{V}_{N+k}$, see \eqref{eq:hat_AB_half_infinite_operators}.
\end{proposition}

\subsection{Symmetry and branching}
\label{sub:symm_branching}

Let us derive a few basic properties of the functions $F,G$
defined in the previous subsection.
For each $i\ge1$, let $\sigma_i$ denote the elementary transposition of the indices 
$i\leftrightarrow i+1$, and define its action on 
(various)
sets of variables as
\begin{equation*}
	\sigma_i(x_1,\ldots,x_{i-1},x_i,x_{i+1},x_{i+2}\ldots  )=
	(x_1,\ldots,x_{i-1},x_{i+1},x_{i},x_{i+2},\ldots  ),
\end{equation*}
and similarly for $\mathbf{r},\mathbf{y},\mathbf{s}$.

\begin{proposition}[Symmetries]
	\label{prop:F_G_symmetry}
	For each $k\ge1$ and $i \in \left\{ 1,\ldots,k-1  \right\}$ we have
	the following symmetry properties:
	\begin{align*}
		G_{\lambda/\mu}(\sigma_i(\mathbf{x});\mathbf{y};\sigma_i(\mathbf{r});\mathbf{s})&=
		G_{\lambda/\mu}(\mathbf{x};\mathbf{y};\mathbf{r};\mathbf{s});\\
		F_{\lambda/\mu}(\sigma_i(\mathbf{x});\mathbf{y};\sigma_i(\mathbf{r});\mathbf{s})&=
		\frac{x_i-x_{i+1}r^{-2}_{i+1}}{x_{i+1}-x_i r^{-2}_i}\,
		F_{\lambda/\mu}(\mathbf{x};\mathbf{y};\mathbf{r};\mathbf{s}).
	\end{align*}
	In other words,
	the functions $G_{\lambda/\mu}(\mathbf{x};\mathbf{y};\mathbf{r};\mathbf{s})$ and 
	$\displaystyle F_{\lambda/\mu}(\mathbf{x};\mathbf{y};\mathbf{r};\mathbf{s})
	\prod_{1\le i<j\le k}( x_i-x_jr^{-2}_j )$
	are symmetric under simultaneous permutations of the variables
	$(x_j,r_j)$.
\end{proposition}
\begin{proof}
	The symmetry properties of the functions $G$ and $F$
	follow, respectively, from the commutation relations
	\eqref{eq:D2D1_commute} and \eqref{eq:B2hat_B1hat_commute}
	for the operators acting in the subspace $\mathscr{V}$
	of the infinite tensor product (see \Cref{sub:signature_states}).
\end{proof}

\begin{proposition}[Branching]
	\label{prop:F_G_branching}
	Fix integers $M,N\ge1$ and sets of complex variables
	\begin{align*}
		&\mathbf{x}=(x_1,\ldots,x_M ), \quad 
		\mathbf{x}'=(x_{M+1},\ldots,x_{M+N} ),
		\quad 
		\mathbf{r}=(r_1,\ldots,r_M ), \quad \mathbf{r}'=(r_{M+1},\ldots,r_{M+N} ),
		\\
		&\mathbf{y}=(y_1,y_2,y_3,\ldots, ),\qquad \mathbf{s}=(s_1,s_2,s_3,\ldots ).
	\end{align*}
	Define 
	$\mathbf{x}\cup \mathbf{x}'=(x_1,\ldots,x_{M+N} )$, 
	$\mathbf{r}\cup \mathbf{r}'=(r_1,\ldots,r_{M+N} )$.
	Then for any signatures $\lambda,\mu$ we have
	\begin{align*}
		\sum_{\nu}G_{\lambda/\nu}(\mathbf{x};\mathbf{y};\mathbf{r};\mathbf{s})\,G_{\nu/\mu}(\mathbf{x}';\mathbf{y};\mathbf{r}';\mathbf{s})
		&=
		G_{\lambda/\mu}
		(\mathbf{x}\cup \mathbf{x}';\mathbf{y};\mathbf{r}\cup \mathbf{r}';\mathbf{s});\\
		\sum_{\nu}F_{\lambda/\nu}(\mathbf{x};\mathbf{y};\mathbf{r};\mathbf{s})\, F_{\nu/\mu}(\mathbf{x}';\mathbf{y};\mathbf{r}';\mathbf{s})
		&=
		F_{\lambda/\mu}
		(\mathbf{x}\cup \mathbf{x}';\mathbf{y};\mathbf{r}\cup \mathbf{r}';\mathbf{s}).
	\end{align*}
\end{proposition}
\begin{proof}
	These identities follow from \Cref{def:G_function,def:F_function}, respectively.
	For example, for the first identity
	consider the vertex model in $\mathbb{Z}_{\ge1}\times \left\{ 1,\ldots,M+N  \right\}$
	for the right-hand side $G_{\lambda/\mu}$. Encode the configuration of the vertical arrows
	between rows $M$ and $M+1$ by a signature $\nu$. Then the bottom and the top
	vertex models in 
	$\mathbb{Z}_{\ge1}\times\left\{ 1,\ldots,M  \right\}$
	and $\mathbb{Z}_{\ge1}\times\left\{ M+1,\ldots,N  \right\}$
	have partition functions
	$G_{\nu/\mu}$ and
	$G_{\lambda/\nu}$, respectively.
	Summing over $\nu$ leads to the desired identity (note that this sum is finite).
	The second identity is proven in the same way.
\end{proof}

\subsection{Cauchy identity}
\label{sub:F_G_Cauchy_identity}

The functions $F,G$ satisfy Cauchy-type summation identities
which follow from the Yang--Baxter equation.

\begin{proposition}
	[Skew Cauchy identity]
	\label{prop:F_G_Cauchy_skew}
	Fix two signatures $\lambda,\mu$, integers $M,N\ge1$, and 
	sequences of complex variables
	\begin{align*}
		\mathbf{x}=(x_1,\ldots,x_M ),\qquad \mathbf{r}&=(r_1,\ldots,r_M );\\
		\mathbf{w}=(w_1,\ldots,w_N ),\qquad \boldsymbol\uptheta&=(\theta_1,\ldots,\theta_N);\\
		\mathbf{y}=(y_1,y_2,y_3,\ldots, ),\qquad \mathbf{s}&=(s_1,s_2,s_3,\ldots ),
	\end{align*}
	satisfying
	\begin{equation}
		\label{eq:skew_Cauchy_analytic_condition}
		\left|
		\frac{y_k-s^2_k x_i}{y_k-x_i}
		\frac{y_k-w_j}{y_k-s^2_k w_j}
		\right|<1-\delta<1
		\qquad \textnormal{for all $1\le i\le M$, $1\le j\le N$, and $k\ge1$}.
	\end{equation}
	Then we have
	\begin{equation}
		\begin{split}
			&
			\sum_{\nu}
			G_{\nu/\lambda}(\mathbf{w};\mathbf{y};\boldsymbol\uptheta;\mathbf{s})
			F_{\nu/\mu}(\mathbf{x};\mathbf{y};\mathbf{r};\mathbf{s})
			\\&\hspace{80pt}=
			\prod_{i=1}^{M}\prod_{j=1}^{N}
			\frac{x_i-\theta_j^{-2}w_j}{x_i-w_j}
			\sum_{\varkappa}
			G_{\mu/\varkappa}(\mathbf{w};\mathbf{y};\boldsymbol\uptheta;\mathbf{s})
			F_{\lambda/\varkappa}(\mathbf{x};\mathbf{y};\mathbf{r};\mathbf{s}).
		\end{split}
		\label{eq:skew_Cauchy_F_G}
	\end{equation}
\end{proposition}
\begin{proof}
	With the help of 
	of \Cref{prop:F_G_branching} and induction on $M$ and $N$, 
	it suffices to prove 
	\eqref{eq:skew_Cauchy_F_G}
	for $M=N=1$.

	Fix $\lambda,\mu$ with 
	$K+1$ and $K$ 
	parts, respectively,
	for some $K\ge 0$ 
	(other choices for the numbers of parts 
	of $\lambda$ and $\mu$ lead to triviality of both sides).
	Interpret
	$
	\sum_{\varkappa}
	G_{\mu/\varkappa}(w;\mathbf{y};\theta;\mathbf{s})
	F_{\lambda/\varkappa}(x;\mathbf{y};r;\mathbf{s})
	$
	as a partition function of a vertex model in the half-infinite rectangle
	$\mathbb{Z}_{\ge1}\times \left\{ 1,2 \right\}$,
	with the boundary conditions $\mathcal{S}(\lambda)$ at the bottom, $\mathcal{S}(\mu)$ at the top,
	empty on the left, and $\left\{ \text{full},\text{empty} \right\}$ on the far right.
	The weights at vertices $(k,1)$ and $(k,2)$, $k\ge1$,  are $\widehat{W}(\cdots\mid x;y_k;r;s_k)$
	and $W(\cdots\mid w;y_k;\theta;s_k)$, respectively.
	Due to this choice of the weights, the partition function is well-defined
	(all vertices which are repeated infinitely often have weight $1$).
	See \Cref{fig:skew_Cauchy} for an illustration.

	\begin{figure}[ht]
		\centering
		\includegraphics{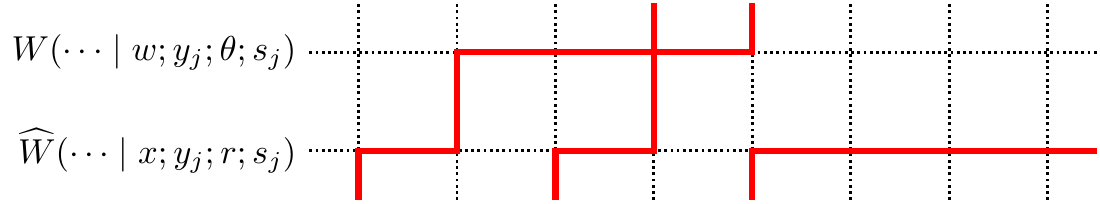}
		\caption{The partition function corresponding to the 
			sum over $\varkappa$ in the
			right-hand
			side of the skew Cauchy identity \eqref{eq:skew_Cauchy_F_G}.
			Adding the empty cross vertex on the left and dragging it to the right
			with the help of the Yang--Baxter equation
			produces the sum over $\nu$.}
		\label{fig:skew_Cauchy}
	\end{figure}

	Add the empty cross vertex 
	with weight $R(0,0;0,0)=1$
	to the left of 
	$\mathbb{Z}_{\ge1}\times \left\{ 1,2 \right\}$,
	and use the Yang--Baxter equation (\Cref{prop:YBE_first})
	to move it to the right. 
	After $L\ge \max(\lambda_1+K+1,\mu_1+K)$ steps we get the identity
	\begin{align*}
		&\sum_{\varkappa}
		G_{\mu/\varkappa}(w;\mathbf{y};\theta;\mathbf{s})
		F_{\lambda/\varkappa}(x;\mathbf{y};r;\mathbf{s})
		=
		R(0,1;0,1)
		\sum_{\nu\colon \nu_1+K+1\le L}
		G_{\nu/\lambda}(w;\mathbf{y};\theta;\mathbf{s})
		F_{\nu/\mu}(x;\mathbf{y};r;\mathbf{s})
		\\&\hspace{20pt}+
		\tilde Z\cdot
		R(1,0;0,1)
		\prod
		_{k=\max(\lambda_1+K+1,\,\mu_1+K)+1}^{L}
		W(0,1;0,1\mid w;y_k;\theta;s_k)\,
		\widehat{W}(0,0;0,0\mid x;y_k;r;s_k).
	\end{align*}
	Here $R(0,1;0,1)=\dfrac{x-w}{x-\theta^{-2}w}$, and
	$\tilde Z$
	is a quantity independent of $L$.
	Sending $L\to +\infty$, we see that 
	the product over $k$ in the second summand goes to zero thanks 
	to \eqref{eq:skew_Cauchy_analytic_condition}, 
	while in the first summand the restriction $\nu_1+K+1\le L$
	disappears. Dividing both sides by $R(0,1;0,1)$ produces the desired identity
	\eqref{eq:skew_Cauchy_F_G}.
\end{proof}

As a corollary of \Cref{prop:F_G_Cauchy_skew} and 
the determinantal formula for $F$ \eqref{eq:F_det_formula_in_theorem}
from the next \Cref{sub:symm_formulas_F_G},
we also get the following identity:

\begin{theorem}[Cauchy identity, \Cref{thm:intro_Cauchy} from Introduction]
	\label{thm:F_G_Cauchy_big}
	In the setting of \Cref{prop:F_G_Cauchy_skew}
	we have
	\begin{equation}
	\label{eq:Full_Cauchy_Identity}
		\begin{split}
			&\sum_{\nu}
			G_{\nu}(\mathbf{w};\mathbf{y};\boldsymbol\uptheta;\mathbf{s})\,
			F_{\nu}(\mathbf{x};\mathbf{y};\mathbf{r};\mathbf{s})
			\\&\hspace{30pt}=
			\prod_{i=1}^{M}x_i(r^{-2}_i-1)\,
			\frac{\prod_{1\le i<j\le M}(r_i^{-2}x_i-x_j)(s_i^{-2}y_i-y_j)}
			{\prod_{i,j=1}^M(y_i-x_j)}
			\prod_{i=1}^{M}\prod_{j=1}^{N}
			\frac{x_i-\theta_j^{-2}w_j}{x_i-w_j},
		\end{split}
	\end{equation}
	where the summation in the left-hand side 
	is over all signatures $\nu$ with $M$ parts.
\end{theorem}
\begin{proof}
	This is a particular case 
	of \Cref{prop:F_G_Cauchy_skew}
	when $\lambda=0^M$ (this notation stands for the signature $(0,\ldots,0)$ 
	with $0$ repeated $M$ times)
	and $\mu=\varnothing$.
	The sum in the right-hand side of \eqref{eq:skew_Cauchy_F_G}
	reduces to a single term with $\varkappa=\varnothing$,
	and $G_{\varnothing/\varnothing}=1$.
	We thus get
	\begin{equation}
		\label{eq:F_G_Cauchy_big_with_F_0_RHS}
		\sum_{\nu}
		G_{\nu}(\mathbf{w};\mathbf{y};\boldsymbol\uptheta;\mathbf{s})\,
		F_{\nu}(\mathbf{x};\mathbf{y};\mathbf{r};\mathbf{s})
		=
		F_{0^M}(\mathbf{x};\mathbf{y};\mathbf{r};\mathbf{s})
		\prod_{i=1}^{M}\prod_{j=1}^{N}
		\frac{x_i-\theta_j^{-2}w_j}{x_i-w_j}.
	\end{equation}
	Using \Cref{thm:F_formula} formulated below, we have
	\begin{equation*}
		F_{0^M}(\mathbf{x};\mathbf{y};\mathbf{r};\mathbf{s})
		=
		\Biggl(
			\prod_{i=1}^{M}x_i(r^{-2}_i-1)
			\prod_{1\le i<j\le M}\frac{r_i^{-2}x_i-x_j}{x_i-x_j}
		\Biggr)
		\det
		\biggl[  
			\frac{1}{y_{M-j+1}-x_i}
			\prod_{m=1}^{M-j}
			\frac{y_m-s_m^2x_i}{s_m^2(y_m-x_i)}
		\biggr]_{i,j=1}^{M}.
	\end{equation*}
	This determinant has an explicit product form:
	\begin{equation}
		\label{eq:fully_deformed_Cauchy_determinant}
		\det
		\biggl[  
			\frac{1}{y_{M-j+1}-x_i}
			\prod_{m=1}^{M-j}
			\frac{y_m-s_m^2x_i}{s_m^2(y_m-x_i)}
		\biggr]_{i,j=1}^{M}=
		\frac{\prod_{i<j}(x_i-x_j)(s_i^{-2}y_i-y_j)}
		{\prod_{i,j=1}^M(y_i-x_j)}.
	\end{equation}
	This can be established 
	by induction on $M$
	using the Desnanot--Jacobi 
	identity for determinants:
	\begin{equation}
		\label{eq:Desnanot_Jacobi}
		\det(A)=\frac{\det(A^1_1)\det(A^M_M)-\det(A^M_1)\det(A^1_M)}{\det(A^{1,M}_{1,M})},
	\end{equation}
	where $A$ is the $M\times M$ matrix in the left-hand side of 
	\eqref{eq:fully_deformed_Cauchy_determinant},
	and
	$A^\mathcal{I}_\mathcal{J}$, $\mathcal{I},\mathcal{J}\subset\left\{ 1,\ldots,M  \right\}$,
	$|\mathcal{I}| = |\mathcal{J}|$,
	is the matrix obtained from $A$ by deleting rows and columns
	indexed
	by $\mathcal{I}$ and $\mathcal{J}$, respectively.
	Each of the matrices in the right-hand side of sizes $M-1$ and $M-2$ 
	are essentially the ones in the left-hand side of \eqref{eq:fully_deformed_Cauchy_determinant},
	up to shifts in some of the parameters $x_i,y_i,s_i$, and diagonal factors.
	More precisely, denote the matrix elements by
	\begin{equation*}
		a_{ij}^{(M)}=\frac{1}{y_{M-j+1}-x_i}
		\prod_{m=1}^{M-j}
		\frac{y_m-s_m^2x_i}{s_m^2(y_m-x_i)}.
	\end{equation*}
	One can readily check that
	\begin{align*}
		(A^M_M)_{ij}\,
		\frac{x_i-y_1}{x_i-s_1^{-2}y_1}
		&=
		a_{ij}^{(M-1)}\Big\vert_{y_i\to y_{i+1},\ s_i\to s_{i+1}};\qquad 
		(A^1_1)_{ij}
		=
		a_{ij}^{(M-1)}\Big\vert_{x_i\to x_{i+1}};\\
		(A^1_M)_{ij}\,
		\frac{x_{i+1}-y_1}{x_{i+1}-s_1^{-2}y_1}
		&=
		a_{ij}^{(M-1)}\Big\vert_{x_i\to x_{i+1},\ y_i\to y_{i+1},\ s_i\to s_{i+1}};\qquad 
		(A^M_1)_{ij}=a_{ij}^{(M-1)};
		\\
		(A^{1,M}_{1,M})_{ij}\, \frac{x_i-y_1}{x_i-s_1^{-2}y_1}&=a_{ij}^{(M-2)}
		\Big\vert_{x_i\to x_{i+1},\ y_i\to y_{i+1},\ s_i\to s_{i+1}}.
	\end{align*}
	Thus, by the induction hypothesis, all the determinants in the right-hand side of
	\eqref{eq:Desnanot_Jacobi} are expressed as certain products.
	The induction step then follows by matching the combination \eqref{eq:Desnanot_Jacobi}
	of these products
	with the desired right-hand side of \eqref{eq:fully_deformed_Cauchy_determinant}.
	Putting all together produces the Cauchy identity \eqref{eq:Full_Cauchy_Identity}.
\end{proof}

\subsection{Determinantal and Sergeev--Pragacz type formulas}
\label{sub:symm_formulas_F_G}

The function $F_\lambda$
admits an explicit formula involving a determinant of the
single-variable functions~$F_{(m)}$. 
The function $G_\lambda$ 
admits a more complicated (yet still explicit)
formula in terms of 
a summation over the product of two symmetric groups.
In this section we formulate both expressions,
and their proofs based on commutation relations for the row operators
(\Cref{sub:row_operators})
are postponed to \Cref{appA:F_G_formula_proofs}.
See also the next \Cref{sec:particular_cases}
for simpler proofs in some particular cases.

For sequences of complex numbers
$\mathbf{s}=(s_1,s_2,\ldots )$, $\mathbf{y}=(y_1,y_2,\ldots )$ 
and any integer $k\ge0$
denote
\begin{equation}
	\label{eq:phi_def}
	\varphi_k(x)=
	\varphi_k(x\mid \mathbf{y};\mathbf{s}):=
	\frac{1}{y_{k+1}-x}
	\prod_{j=1}^{k}
	\frac{x-s_j^{-2}y_j}{x-y_j}.
\end{equation}
In particular, $\varphi_0(x\mid \mathbf{y};\mathbf{s})=1/(y_1-x)$.
From 
\Cref{def:F_function}
and the formula for the vertex weights 
$\widehat{W}$~\eqref{eq:weights_W_hat} we have
$F_{(k)}(x;\mathbf{y};r;\mathbf{s})=x(r^{-2}-1)\,\varphi_k(x\mid \mathbf{y};\mathbf{s})$
for all $k\ge0$.

\begin{theorem}[Determinantal formula for $F_\lambda$,
	\Cref{thm:intro_F_formula} from Introduction]
	\label{thm:F_formula}
	For any $N\ge1$, complex variables
	\begin{equation*}
		\mathbf{x}=(x_1,\ldots,x_N ),\qquad \mathbf{r}=(r_1,\ldots,r_N ),\qquad 
		\mathbf{y}=(y_1,y_2,\ldots ),\qquad \mathbf{s}=(s_1,s_2,\ldots ),
	\end{equation*}
	and 
	any signature $\lambda=(\lambda_1\ge \ldots\ge \lambda_N )$ we 
	have
	\begin{equation}
		\label{eq:F_det_formula_in_theorem}
		F_\lambda(\mathbf{x};\mathbf{y};\mathbf{r};\mathbf{s})=
		\Biggl(
			\prod_{i=1}^{N}x_i(r^{-2}_i-1)
			\prod_{1\le i<j\le N}\frac{r_i^{-2}x_i-x_j}{x_i-x_j}
		\Biggr)
		\det\left[ \varphi_{\lambda_j+N-j}(x_i\mid \mathbf{y};\mathbf{s}) \right]_{i,j=1}^{N},
	\end{equation}
	where $\varphi_k$ are defined by \eqref{eq:phi_def}.
\end{theorem}
This theorem is proven in \Cref{appA:F}.

\medskip

In the next theorem and throughout the text,
$\mathfrak{S}_m$ denotes the group of all permutations of $\left\{ 1,\ldots,m  \right\}$.
Let $\lambda=(\lambda_1\ge \ldots\ge \lambda_N )$ be an arbitrary signature with $N$ parts.
Let $d = d(\lambda) \ge 0$ denote
the integer such that $\lambda_d \ge d$ and $\lambda_{d + 1} < d + 1$.
Denote by $\ell_j=\lambda_j+N-j+1$, $j=1,\ldots,N $, the elements of the set $\mathcal{S}(\lambda)$.
Let 
\begin{equation}
	\label{eq:mu_depending_on_lambda_notation}
	\mu = (\mu_1< \mu_2< \ldots < \mu_d) =
	\{1,\ldots,N \}  \setminus \big( \mathcal{S}(\lambda) \cap \{1,\ldots,N \} \big).
\end{equation}

\begin{theorem}[Sergeev--Pragacz type formula for $G_\lambda$]
	\label{thm:G_formula}
	Let $M, N > 0$ denote integers. 
	For complex variables
	\begin{equation*}
		\mathbf{x} = (x_1, x_2, \ldots , x_M),
		\qquad 
		\mathbf{r} = (r_1, r_2, \ldots , r_M),
		\qquad 
		\mathbf{y} =
		(y_1, y_2, \ldots ),\qquad 
		\mathbf{s} = (s_1, s_2, \ldots ),
	\end{equation*}
	with the notation above, we have
	\begin{equation}
		\label{eq:G_SP_formula_in_theorem}
		\begin{split}
			&G_{\lambda} (\mathbf{x}; \mathbf{y}; \mathbf{r}; \mathbf{s}) 
			=
			\prod_{j=1}^{M}\prod_{k=1}^{N}
			\frac{y_k-s_k^2r_j^{-2}x_j}{y_k-s_k^2 x_j}
			\sum_{\substack{\mathcal{I},\mathcal{J}\subseteq \left\{ 1,\ldots,M  \right\}\\ |\mathcal{I}|=|\mathcal{J}|=d}}
			\prod_{\substack{i\in \mathcal{J}^c\\j\in \mathcal{J}}}\frac{1}{x_i-x_j}
			\prod_{\substack{i,j\in \mathcal{J}\\i<j}}\frac1{x_j-x_i}
			\\
			&\hspace{50pt}
			\times
			\prod_{\substack{i\in \mathcal{I}\\ 1\le j\le M}}(r_i^{-2}x_i-x_j)
			\prod_{\substack{i\in \mathcal{I}^c\\j\in \mathcal{J}}}(r_i^{-2}x_i-x_j)
			\prod_{\substack{i\in \mathcal{I}\\j\in \mathcal{I}^c}}\frac{1}{r_i^{-2}x_i-r_j^{-2}x_j}
			\prod_{\substack{i,j\in \mathcal{I}\\i<j}}\frac{1}{r_i^{-2}x_i-r_j^{-2}x_j}
			\\
			&\hspace{50pt}
			\times
			\sum_{\sigma,\rho\in \mathfrak{S}_d}
			\mathop{\mathrm{sgn}}(\sigma\rho)
			\prod_{h = 1}^d \biggl( \frac{y_{\ell_h}
			\big( 1 - s_{\ell_h}^2 \big)}{y_{\ell_h} - s_{\ell_h}^2
			x_{j_{\rho (h)}}} \prod_{i = N + 1}^{\ell_h - 1}
			\frac{s_i^2 \big(y_i - x_{j_{\rho (h)}} \big)}{y_i - s_i^2
			x_{j_{\rho (h)}}} \biggr)
			\\&\hspace{100pt}\times
			\prod_{m=1}^d
			\biggl(
				\frac{s_{\mu_m}^2 }{y_{\mu_m} - s_{\mu_m}^2
				r^{-2}_{i_{\sigma (m)}} x_{i_{\sigma (m)}}} 
				\prod_{k =
				\mu_m + 1}^N 
				\frac{s_k^2 \big(r^{-2}_{i_{\sigma (m)}}
				x_{i_{\sigma (m)}} - y_k \big)}{y_k - s_k^2 r^{-2}_{i_{\sigma (m)}}
				x_{i_{\sigma (m)}}}
			\biggr).
		\end{split}
	\end{equation} 	
	where
	$\mathcal{I}= (i_1< i_2< \ldots < i_d)$
	and $\mathcal{J}= (j_1< j_2< \ldots < j_d)$.
	Note that both sides of \eqref{eq:G_SP_formula_in_theorem}
	vanish if $d(\lambda)>M$, as it should be.
\end{theorem}

This theorem is proven in \Cref{appA:G}.
The name ``Sergeev--Pragacz type formula'' 
for \eqref{eq:G_SP_formula_in_theorem}
is suggested by the connection between $G_\lambda$ (in the horizontally homogeneous case)
and supersymmetric Schur functions.
Even though we could not relate 
\eqref{eq:G_SP_formula_in_theorem} to the Sergeev--Pragacz
formula (e.g., see \cite[(5)]{hamel1995lattice}) itself,
the form of our formula is sufficiently similar to justify the name.
See
\Cref{sub:homogeneous_F_G} and in particular \Cref{rmk:SP_formula_comparison}
below for a detailed discussion.

\section{Specializations to known symmetric functions}
\label{sec:particular_cases}

Here we discuss how our functions 
$F$ and $G$ degenerate to certain known Schur-type 
symmetric functions. 
This leads to independent proofs of 
\Cref{thm:F_formula,thm:G_formula}
in some particular cases.

\subsection{Five vertex model and factorial Schur polynomials}
\label{sub:factorial_Schur}

We begin with a five vertex degeneration, when the 
weight of the vertex of type $(1,1;1,1)$ vanishes, and 
connect the functions $F_\lambda$ to the 
\emph{factorial Schur polynomials} (also sometimes called double 
Schur polynomials, cf. \cite{molev2009comultiplication}).
We adopt the definition 
from \cite[6th variation]{macdonald1992schur_Theme}. 
The factorial Schur polynomials
are indexed by signatures
$\lambda=(\lambda_1\ge \ldots\ge \lambda_N \ge0)$,
depend on the variables $\mathbf{x}=(x_1,\ldots,x_N )$ and on
a sequence
$\mathbf{y}=(y_1,y_2,\ldots )$ of complex parameters:
\begin{equation}
	\label{eq:factorialSchur_det_det}
	s_\lambda(\mathbf{x}\mid \mathbf{y})=\frac{\det\left[ (x_i\mid \mathbf{y})^{\lambda_j+N-j} \right]_{i,j=1}^{N}}
	{\prod_{1\le i<j\le N}(x_i-x_j)},\qquad 
	(x\mid \mathbf{y})^k:=(x+y_1)\ldots(x+y_k).
\end{equation}

These polynomials can also be represented as sums over semistandard Young tableaux.
We use the language of
Young tableaux only in this subsection, and so refer to, e.g., \cite[Ch. I]{Macdonald1995}
for the relevant definitions.
We have
\cite[(6.16)]{macdonald1992schur_Theme}
\begin{equation}
	\label{eq:factorialSchur_tableaux}
	s_\lambda(\mathbf{x}\mid \mathbf{y})=\sum_{T}(\mathbf{x}\mid \mathbf{y})^{T},\qquad 
	(\mathbf{x}\mid \mathbf{y})^{T}:=\prod_{(i,j)\in \lambda}\left(x_{T(i,j)}+y_{T(i,j)+j-i}\right).
\end{equation}
Here the sum is taken over all semistandard Young tableaux
of shape $\lambda$ filled with numbers from $1$ to $N$, 
the product is over all boxes $(i,j)$
in $\lambda$, where $i$ and $j$ are the row and column coordinates of the box, and
$T(i,j)$ is the tableau entry in this box.
See
an example in \Cref{fig:factorialSchur_tableaux}.

\begin{remark}
	\label{rmk:ordinary_Schur}
	In the particular case
	$\mathbf{y}=(0,0,\ldots )$,
	the factorial Schur polynomials $s_\lambda(\mathbf{x}\mid \mathbf{y})$
	turn into the ordinary Schur polynomials
	$s_\lambda(x_1,\ldots,x_N )$.
\end{remark}

\begin{figure}[htpb]
	\centering
	\includegraphics[width=\textwidth]{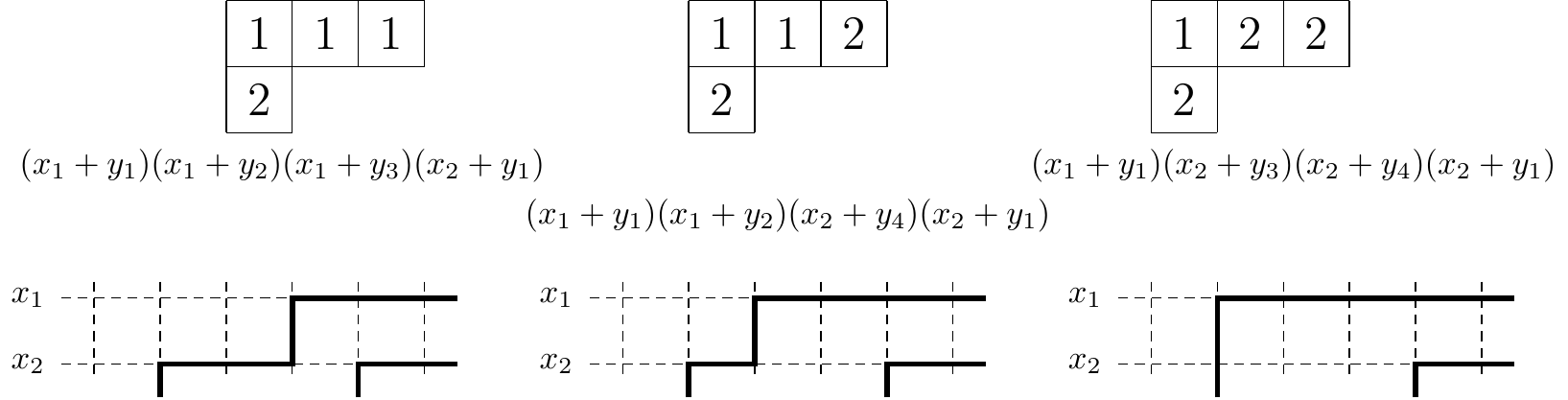}
	\caption{All three semistandard Young tableaux
	for $\lambda=(3,1)$ and $N=2$, the corresponding 
	summands in $s_{\lambda}(\mathbf{x}\mid \mathbf{y})$ given by formula \eqref{eq:factorialSchur_tableaux},
	and the corresponding vertex model configurations in the proof of \Cref{lemma:factorialSchur_as_vertex_model}.}
	\label{fig:factorialSchur_tableaux}
\end{figure}

The tableaux formula \eqref{eq:factorialSchur_tableaux} 
can be translated into the vertex model language.
This fact
is contained in either of 
\cite{lascoux20076}, \cite{mcnamara2009factorial},
\cite{bump2011factorial},
but for consistency we give a proof here.

\begin{lemma}
	\label{lemma:factorialSchur_as_vertex_model}
	There is a one-to-one correspondence between 
	(see \Cref{fig:factorialSchur_tableaux} for an example)
	\begin{enumerate}
		\item 
			Semistandard Young tableaux
			of shape $\lambda$ filled with numbers from $1$ to $N$, and
		\item 
			Path configurations of the six vertex model
			in the half-infinite rectangle $\mathbb{Z}_{\ge1}\times\{1,\ldots,N \}$
			with paths entering from the bottom at locations
			$\mathcal{S}(\lambda)=\{\lambda_j+N-j+1 \}_{1\le j\le N}$
			and exiting through the right boundary, such that the left and the top boundaries
			are empty. The paths must satisfy an additional \emph{five vertex
			condition} that the vertex $(1,1;1,1)$ is not present.
	\end{enumerate}
	Moreover, with the vertex weights at each $(i,j)\in\mathbb{Z}_{\ge1}\times\left\{ 1,\ldots,N  \right\}$ 
	equal to
	\begin{equation}
		\label{eq:fSchur_weights}
		\begin{split}
			&
			\widehat{W}_{\mathrm{fSchur}}(0,0;0,0)=x_{N+1-j}+y_i,
			\qquad 
			\widehat{W}_{\mathrm{fSchur}}(1,1;1,1)=0,
			\\
			&
			\widehat{W}_{\mathrm{fSchur}}(1,0;1,0)=
			\widehat{W}_{\mathrm{fSchur}}(0,1;0,1)=
			\widehat{W}_{\mathrm{fSchur}}(1,0;0,1)=
			\widehat{W}_{\mathrm{fSchur}}(0,1;1,0)=1,
		\end{split}
	\end{equation}
	the partition function in the half-infinite rectangle
	described above is equal to the factorial Schur polynomial
	$s_\lambda(\mathbf{x}\mid \mathbf{y})$.
\end{lemma}
A similar bijection holds between semistandard tableaux of a skew shape $\lambda/\mu$, and 
five vertex model
configurations with top and bottom boundary conditions 
given by
$\mathcal{S}(\mu),\mathcal{S}(\lambda)$.
Here the boundary conditions are the same as
for the functions $F_{\lambda/\mu}$, see \Cref{def:F_function}.
\begin{proof}[Proof of \Cref{lemma:factorialSchur_as_vertex_model}]
	Take a semistandard tableau $T$ of shape $\lambda$, and let $\nu$ be the shape
	formed by numbers from $1$ to $N-1$. These signatures interlace:
	\begin{equation*}
		\lambda_1\ge\nu_1\ge\lambda_2\ge\nu_2\ge \ldots\ge\lambda_{N-1}\ge\nu_{N-1} \ge \lambda_N,
	\end{equation*}
	which in the language of the configurations 
	$\mathcal{S}(\lambda)=\left\{ \lambda_j+N-j+1 \right\}$ 
	and $\mathcal{S}(\nu)=\left\{ \nu_j+N-j \right\}$
	translates into
	\begin{equation}
		\label{eq:factorialSchur_as_vertex_model_proof}
		\lambda_1+N>\nu_1+N-1\ge\lambda_2+N-1>\nu_2+N-2\ge \ldots\ge\lambda_{N-1}+2>\nu_{N-1}+1 \ge \lambda_N+1.
	\end{equation}
	There is a unique one-layer six vertex model configuration
	in $\mathbb{Z}_{\ge1}$ with boundary conditions $\mathcal{S}(\lambda)$ at the bottom
	and $\mathcal{S}(\nu)$ at the top, with no paths entering from the left, and 
	a path exiting through the right boundary.
	The strict inequalities in \eqref{eq:factorialSchur_as_vertex_model_proof}
	imply that this configuration satisfies the five vertex condition.

	To construct the full desired bijection, look at shapes in the tableau formed by 
	numbers from $1$ to $j$ for every $j\le N-1$, and argue in the same manner as above.

	Finally, observe that under this bijection, for each tableau $T$
	the product $(\mathbf{x}\mid \mathbf{y})^T$ from \eqref{eq:factorialSchur_tableaux}
	is the same as the product of the weights
	\eqref{eq:factorialSchur_as_vertex_model_proof}
	in the corresponding five vertex model configuration. Indeed, 
	with the notation $\nu$ as above, observe that the
	variable $x_N$ enters the configuration weight only through empty vertices $(0,0;0,0)$,
	which gives
	\begin{equation*}
		\prod_{m=\nu_1+N}^{\lambda_1+N-1}(x_N+y_m)
		\prod_{m=\nu_2+N-1}^{\lambda_2+N-2}(x_N+y_m)
		\ldots 
		\prod_{m=1}^{\lambda_N}(x_N+y_m).
	\end{equation*}
	This is readily seen to be the same contribution as in \eqref{eq:factorialSchur_tableaux}.
	The argument similarly extends to all other variables $x_i$.
\end{proof}
\begin{remark}
	\label{rmk:symmetry_in_factorialSchur}
	The weight $x_{N+1-j}+y_i$ 
	in \eqref{eq:fSchur_weights}
	may be replaced by $x_j+y_i$ because the
	polynomial $s_\lambda(\mathbf{x}\mid \mathbf{y})$ is symmetric in the $x_j$'s.
\end{remark}

We can now specialize
$F_\lambda$ given by \Cref{def:F_function} into
$s_\lambda(\mathbf{x}\mid \mathbf{y})$:

\begin{proposition}
	\label{prop:from_F_lambda_to_factorialSchur}
	Let $\lambda$ be a signature with $N$ parts. 
	Take complex parameters
	\begin{equation}
		\label{eq:from_F_lambda_to_factorialSchur}
		s^{-2}\mathbf{x}^{-1}=(s^{-2}x_1^{-1},\ldots,s^{-2}x_N^{-1}),\ \, 
		-\mathbf{y}^{-1}=(-y_1^{-1},-y_2^{-1},\ldots),\ \, 
		\mathbf{r}=(r,\ldots,r ),\ \, 
		\mathbf{s}=(s,s,\ldots ).
	\end{equation}
	Then we have
	\begin{equation}
		\label{eq:from_F_lambda_to_factorialSchur1}
		\lim_{s\to0,\ r\to+\infty}F_\lambda(s^{-2}\mathbf{x}^{-1};-\mathbf{y}^{-1};\mathbf{r};\mathbf{s})=
		\frac{x_1^{N-1}x_2^{N-2}\ldots x_{N-1}}
		{
			\prod_{i\ge 1}
			y_i^{\# \{k\in \mathcal{S}(\lambda)\colon k>i \}}
		}
		\,
		s_\lambda(\mathbf{x}\mid \mathbf{y}).
	\end{equation}
\end{proposition}
The factorial Schur polynomial
$s_\lambda(\mathbf{x}\mid \mathbf{y})$ is symmetric in the $x_i$'s, and
observe that the deformed symmetry in the $x_i$'s of the right-hand side of
\eqref{eq:from_F_lambda_to_factorialSchur1}
agrees with \Cref{prop:F_G_symmetry}.
\begin{proof}[Proof of \Cref{prop:from_F_lambda_to_factorialSchur}]
	Taking the weights $\widehat{W}$ \eqref{eq:weights_W_hat} 
	with
	the 
	parameters \eqref{eq:from_F_lambda_to_factorialSchur}
	and applying the limits $s\to0$, $r\to+\infty$ leads to the vertex weights
	at each $(i,j)\in \mathbb{Z}_{\ge1}\times\left\{ 1,\ldots,N  \right\}$
	(listed in the same order as in \eqref{eq:fSchur_weights} and \Cref{fig:weights_6V}):
	\begin{equation}
		\label{eq:from_F_lambda_to_factorialSchur_proof1}
		\left( 1+x_j/y_i,0,x_j/y_i,1,1,x_j/y_i \right).
	\end{equation}
	These weights are almost the same as $\widehat{W}_{\mathrm{fSchur}}$ 
	\eqref{eq:fSchur_weights}, and after a certain renormalization we can get 
	$s_\lambda(\mathbf{x}\mid \mathbf{y})$. Namely, denote by $Z$ the left-hand side of
	\eqref{eq:from_F_lambda_to_factorialSchur1}. Then 
	\begin{equation}
		\label{eq:from_F_lambda_to_factorialSchur_proof12}
		\frac{Z}{x_1^{N-1}x_2^{N-2}\ldots x_{N-1}}\prod_{i=1}^{\lambda_1+N-1}y_i^{N}
	\end{equation}
	is the partition function like $F_\lambda$ with the vertex weights
	\begin{equation}
		\label{eq:from_F_lambda_to_factorialSchur_proof2}
		\left( x_j+y_i,0,1,y_i,y_i,1 \right)
	\end{equation}
	in the part $\left\{ 1,2,\ldots,\lambda_1+N-1  \right\}\times \left\{ 1,\ldots,N  \right\}$ 
	of the half-infinite rectangle. 
	Indeed, multiplying by the $y_i$'s clears the all denominators, while the number of the extra $x_j$
	factors in \eqref{eq:from_F_lambda_to_factorialSchur_proof1} coming from the weights
	$(1,0;1,0)$ and $(0,1;1,0)$
	is $N-j$, independently of the path configuration.
	Note also that to the right of $\lambda_1+N-1$, every path configuration contains
	only vertices of the type
	$(0,1;0,1)$, and so the remaining weight is equal to $1$.

	Finally, one can readily check that the number of the extra $y_i$ factors 
	coming from the vertices of types $(0,1;0,1)$ and $(1,0;0,1)$
	corresponding to the weights \eqref{eq:from_F_lambda_to_factorialSchur_proof2}
	is also independent of the path configuration, and is equal to 
	$\prod_{i=1}^{\lambda_1+N-1}
	y_i^{\# \left\{ k\in \mathcal{S}(\lambda)\setminus\{\lambda_1+N \}\colon k\le i \right\}}$.
	This allows to remove the extra $y_i$ factors, and turn
	the weights \eqref{eq:from_F_lambda_to_factorialSchur_proof2}
	into $\widehat{W}_{\mathrm{fSchur}}$ \eqref{eq:fSchur_weights},
	up to the permutation of the $x_j$'s allowed by \Cref{rmk:symmetry_in_factorialSchur}.
	Combining all the extra factors leads to the result.
\end{proof}

Let us now present an alternative derivation of the determinantal 
formula \eqref{eq:F_det_formula_in_theorem}
for $F_\lambda$
in the special factorial Schur case.
This argument is simpler 
than in the general case discussed in \Cref{appA:F}.
\begin{corollary}
	\label{cor:simpler_proof_F_formula_factorial_Schur}
	The determinantal formula 
	for $F_\lambda$
	of \Cref{thm:F_formula}
	holds in the
	factorial Schur specialization 
	\eqref{eq:from_F_lambda_to_factorialSchur}--\eqref{eq:from_F_lambda_to_factorialSchur1}.
\end{corollary}
\begin{proof}[Alternative proof]
	\Cref{prop:from_F_lambda_to_factorialSchur} 
	was proven using only
	the partition function definition of $F_\lambda$ and
	the tableau formula for $s_\lambda(\mathbf{x}\mid \mathbf{y})$,
	to compare the two functions.
	We can now use the determinantal formula for $s_\lambda(\mathbf{x}\mid \mathbf{y})$
	\eqref{eq:factorialSchur_det_det}, and it remains to check that the formula of
	\Cref{thm:F_formula} specializes to \eqref{eq:factorialSchur_det_det}.
	We have
	\begin{equation*}
		\lim_{s\to0}
		s^{-2}
		\varphi_k(s^{-2}x^{-1}\mid -\mathbf{y}^{-1};\mathbf{s})
		=
		\lim_{s\to0}
		\frac{s^{-2}}{-y_{k+1}^{-1}-s^{-2}x^{-1}}
		\prod_{j=1}^{k}
		\frac{-y_j^{-1}-x^{-1}}{s^2(-y_j^{-1}-s^{-2}x^{-1})}
		=
		-x
		\prod_{j=1}^{k}
		\frac{x+y_j}{y_j}.
	\end{equation*}
	Therefore, 
	in the specialization of \Cref{prop:from_F_lambda_to_factorialSchur},
	the determinantal formula for
	$F_\lambda$
	turns into
	\begin{equation*}
		\begin{split}
			&
			\lim_{s\to0,\,r\to+\infty}
			\Biggl(
				\prod_{i=1}^{N}x_i^{-1}(r^{-2}-1)
				\prod_{1\le i<j\le N}\frac{r^{-2}x_j-x_i}{x_j-x_i}
			\Biggr)
			\det\left[ s^{-2}\varphi_{\lambda_j+N-j}(s^{-2}x_i^{-1}\mid -\mathbf{y}^{-1};\mathbf{s}) \right]_{i,j=1}^{N}
			\\&\hspace{20pt}
			=
			\Biggl(
				\prod_{i=1}^{N}(-x_i)^{-1}
				\prod_{1\le i<j\le N}\frac{x_i}{x_i-x_j}
			\Biggr)
			\det\left[ 
				-x_i
				\prod_{m=1}^{\lambda_j+N-j}
				\frac{x_i+y_m}{y_m}
			\right]_{i,j=1}^{N},
		\end{split}
	\end{equation*}
	which produces \eqref{eq:factorialSchur_det_det}
	up to the same prefactor as in \eqref{eq:from_F_lambda_to_factorialSchur1}.
\end{proof}

Let us now turn to the functions $G_\lambda$, and denote
\begin{equation}
	\label{eq:G_lambda_factorialSchur_limit}
	\check{s}_\lambda(\mathbf{x} \mid \mathbf{y}):=
	\frac{1}{\prod_{i\ge 1}
	y_i^{\# \{k\in \mathcal{S}(\lambda)\colon k>i \}}}
	\frac{y_1^{N-1}y_2^{N-2}\ldots y_{N-1} }{(x_1 \ldots x_M )^{N}
	}
	\lim_{s\to 0,\,\theta\to+\infty}
	G_\lambda(s^{-2}\mathbf{x}^{-1};-\mathbf{y}^{-1};\boldsymbol\uptheta;\mathbf{s}),
\end{equation}
where $\lambda$ is a signature with $N$ parts, and the parameters are
\begin{equation*}
	s^{-2}\mathbf{x}^{-1}=(s^2x_1^{-1},\ldots,s^2x_M^{-1} ),\quad
	-\mathbf{y}^{-1}=(y_1^{-1},y_2^{-1},\ldots ),\quad
	\boldsymbol\uptheta=(\theta,\ldots, \theta),\quad
	\mathbf{s}=(s,s,\ldots ).
\end{equation*}

\begin{proposition}
	\label{prop:from_G_lambda_to_factorialSchur}
	The limit \eqref{eq:G_lambda_factorialSchur_limit} exists. It is equal to the 
	five vertex partition function 
	in $\mathbb{Z}_{\ge1}\times \left\{ 1,\ldots,M  \right\}$
	with boundary conditions as for $G_\lambda$ (\Cref{def:G_function}),
	and the following vertex weight at each lattice point $(i,j)$:
	\begin{equation}
		\label{eq:fSchur_weights_W_for_G}
		\begin{split}
			&
			W_{\mathrm{fSchur}}(0,0;0,0)=1,
			\qquad 
			W_{\mathrm{fSchur}}(1,1;1,1)=0,
			\\
			&
			W_{\mathrm{fSchur}}(1,0;1,0)=
			W_{\mathrm{fSchur}}(0,1;0,1)
			\\&\hspace{50pt}=
			W_{\mathrm{fSchur}}(1,0;0,1)=
			W_{\mathrm{fSchur}}(0,1;1,0)=\frac1{x_j+y_i}.
		\end{split}
	\end{equation}
\end{proposition}
\begin{proof}
	Take the weights $W$ \eqref{eq:weights_W}
	used in the definition of the functions $G_\lambda$, 
	and apply the limit transition as in \eqref{eq:G_lambda_factorialSchur_limit}.
	We obtain the following weights (listed in the same order as in the claim):
	\begin{equation*}
		\left( 1,0,\frac{x_j}{x_j+y_i},\frac{y_i}{x_j+y_i},\frac{y_i}{x_j+y_i},\frac{x_j}{x_j+y_i} \right),
	\end{equation*}
	which are almost the same as the desired $W_{\mathrm{fSchur}}$ \eqref{eq:fSchur_weights_W_for_G}.
	The extra factors $x_j$ and $y_i$ can be taken out analogously 
	to the proof of \Cref{prop:from_F_lambda_to_factorialSchur},
	which results in the prefactor in \eqref{eq:G_lambda_factorialSchur_limit}.
\end{proof}

\begin{remark}
	The weights $W_{\mathrm{fSchur}}$
	\eqref{eq:fSchur_weights_W_for_G} 
	and their partition functions (coinciding with our
	$\check s_\lambda$ for special $\lambda$)
	appeared in 
	\cite{ikeda2009excited},
	\cite{MPP2}
	(see also
	\cite{MPP3},
	\cite{Pak_F_Petrov2020})
	in connection
	with enumeration of skew standard Young tableaux.
\end{remark}

\begin{proposition}
	\label{prop:factorialSchur_Cauchy}
	The functions
	$s_\lambda(x_1,\ldots,x_N \mid \mathbf{y})$
	and 
	$\check{s}_\lambda(w_1,\ldots,w_M \mid \mathbf{y})$ 
	satisfy the Cauchy summation identity
	\begin{equation}
		\label{eq:factorialSchur_Cauchy}
		\sum_{\lambda=(\lambda_1\ge \ldots \ge\lambda_N\ge0 )}
		s_\lambda(x_1,\ldots,x_N \mid \mathbf{y})\,
		\check{s}_\lambda(w_1,\ldots,w_M  \mid \mathbf{y})
		=
		\prod_{i=1}^{N}\prod_{j=1}^{M}
		\frac{1}{w_j-x_i},
	\end{equation}
	where
	$|x_i|< |w_j|$ for all $i,j$.
\end{proposition}
\begin{proof}
	This is a combination of the Cauchy identity of
	\Cref{thm:F_G_Cauchy_big} and the limit transitions
	from \Cref{prop:from_F_lambda_to_factorialSchur,prop:from_G_lambda_to_factorialSchur}.
\end{proof}
When $y_j\equiv 0$, the 
functions
$s_\lambda(\mathbf{x} \mid \mathbf{y})$
become the usual Schur polynomials $s_\lambda(\mathbf{x})$,
while for $\check{s}_\lambda$ we have
\begin{equation*}
	\check{s}_\lambda(\mathbf{w}\mid 0)=
	(w_1\ldots w_M)^{-N}s_\lambda(w_1^{-1}\ldots w_M^{-1} ).
\end{equation*}
This agrees with the fact that
for $y_j\equiv 0$, identity 
\eqref{eq:factorialSchur_Cauchy} reduces to the usual Cauchy identity for Schur functions
\cite[Ch. I, (4.3)]{Macdonald1995}.

\begin{remark}
	\label{rmk:Molev}
	There is another Cauchy identity involving the
	polynomials 
	$s_\lambda(\mathbf{x}\mid \mathbf{y})$
	together with the dual Schur functions $\widehat{s}_\lambda(\mathbf{w} \,\|\,\mathbf{y})$
	\cite[Theorem 3.1]{molev2009comultiplication}
	(in \cite{olshanski2019interpolation} a particular case of the dual Schur
	functions is included into the family of dual interpolation Macdonald functions).
	One can check (by comparing the Cauchy identities or 
	using the explicit determinantal formula for $\widehat{s}_\lambda$)	
	that the $\check{s}_\lambda$'s are 
	different from the dual Schur functions $\widehat{s}_\lambda$.
\end{remark}

After this paper was posted, \cite{gunna2022integrable} 
studied combinatorial properties of the functions 
$\check{s}_\lambda$ (in their notation, these are multiples of $E^\lambda$)
defined through the Cauchy identity with factorial Schur polynomials.

\subsection{Horizontally homogeneous model}
\label{sub:homogeneous_F_G}

Throughout this subsection we set all the column parameters (in the sense of 
\Cref{fig:F_G_partition_functions}) to be constant:
\begin{equation}
	\label{eq:homogeneous_specialization}
	s_j\equiv s,\qquad y_j\equiv 1,\qquad \textnormal{for all $j=1,2,\ldots $}.
\end{equation}
Denote this specialization by $(\mathbf{y};\mathbf{s})=(1;s)$.
In this special case we can relate the functions $F_\lambda$ and $G_\lambda$ to 
the ordinary and the supersymmetric Schur functions, respectively.

The supersymmetric Schur functions $s_\lambda(\mathbf{a}/\mathbf{b})$, where $\mathbf{a}=(a_1,\ldots,a_M )$, 
$\mathbf{b}=(b_1,\ldots,b_M )$ are two sequences of variables,
may be defined as coefficients in the following Cauchy identity involving the ordinary
Schur polynomials $s_\lambda(t_1,\ldots,t_N )$, where $N$ is an arbitrary large enough
integer:
\begin{equation}
	\label{eq:supersymm_Schur_definition}
	\sum_{\lambda=(\lambda_1\ge \ldots\lambda_N\ge0 )}
	s_\lambda(\mathbf{a}/\mathbf{b})
	\,
	s_\lambda(t_1,\ldots,t_N )=
	\prod_{i=1}^{N}\prod_{j=1}^{M}\frac{1+t_ib_j}{1-t_ia_j}.
\end{equation}
The supersymmetric Schur functions are related to the factorial Schur polynomials
which appeared in \Cref{sub:factorial_Schur},
but we do not need this connection here. See 
\cite{BereleRegev}, \cite[(6.19)]{macdonald1992schur_Theme} for details.
The next statement is independent of the explicit formulas of
\Cref{thm:F_formula,thm:G_formula} (while may also be derived as a corollary
of \Cref{thm:F_formula}, see \Cref{cor:F_G_formulas_simpler} below).

\begin{proposition}[\Cref{prop:intro_supersymm} from Introduction]
	\label{prop:F_G_homogeneous_through_Schur}
	Let $\lambda$ be a signature with $N$ parts. 
	Under the homogeneous specialization \eqref{eq:homogeneous_specialization},
	we have
	\begin{equation}
		\label{eq:F_homogeneous_as_Schur}
		F_\lambda(x_1,\ldots,x_N;1;\mathbf{r};s)
		=
		\det\left[ \left( \frac{1-s^2x_i}{s^2(1-x_i)} \right)^{\lambda_j+N-j} \right]_{i,j=1}^{N}
		\,
		\prod_{i=1}^{N}\frac{(r_i^{-2}-1)x_i}{1-x_i}
		\prod_{1\le i<j\le N}\frac{r_i^{-2}x_i-x_j}{x_i-x_j}
	\end{equation}
	and, in particular,
	\begin{equation}
		\label{eq:F_homogeneous_as_Schur1}
		\frac{F_\lambda(x_1,\ldots,x_N;1;\mathbf{r};s)}{F_{0^N}(x_1,\ldots,x_N;1;\mathbf{r};s)}=
		s_\lambda\left( \frac{1-s^2x_1}{s^2(1-x_1)},\ldots, \frac{1-s^2x_N}{s^2(1-x_N)} \right).
	\end{equation}
	For $G_\lambda$, we have the following expression:
	\begin{equation}
		\label{eq:G_homogeneous_as_Schur}
		\begin{split}
			&
			G_\lambda(x_1,\ldots,x_M;1;\mathbf{r};s)
			\\&\hspace{30pt}
			=
			s_\lambda
			\left( 
				\biggl\{ \frac{s^2(1-x_j)}{1-s^2 x_j} \biggr\}_{j=1}^M
				\,\middle/\,
				\biggl\{ \frac{s^2(x_jr_j^{-2}-1)}{1-s^2r_j^{-2} x_j} \biggr\}_{j=1}^M
			\right)\,
			\prod_{i=1}^{M}\left( \frac{1-s^2r_i^{-2} x_i }{1-s^2x_i} \right)^N.
		\end{split}
	\end{equation}
\end{proposition}
Both identities \eqref{eq:F_homogeneous_as_Schur1} and \eqref{eq:G_homogeneous_as_Schur}
readily extend to skew functions thanks to the 
branching rules (\Cref{prop:F_G_branching})
and the fact that Schur and supersymmetric Schur functions form bases.
\begin{proof}[Proof of \Cref{prop:F_G_homogeneous_through_Schur}]
	First, observe that \eqref{eq:F_homogeneous_as_Schur1} immediately
	follows from \eqref{eq:F_homogeneous_as_Schur} and the 
	formula for the Schur polynomial as a ratio of two determinants.

	The claim \eqref{eq:F_homogeneous_as_Schur} about $F_\lambda$ follows from 
	the results of \cite{brubaker2011schur}. 
	This theorem deals with the 
	free fermion six vertex model
	whose weights 
	in the $k$-th row are given by 
	\begin{equation}
		\label{eq:abc_weights_for_BBF}
		\begin{split}
			& w (0, 0; 0, 0) = a_1^{(k)}; \qquad w (1, 0; 1, 0) = b_1^{(k)}; \qquad w (1, 0; 0, 1) = c_2^{(k)}; \\
			& w(1, 1; 1, 1) = a_2^{(k)}; \qquad w (0, 1; 0, 1) = b_2^{(k)}; \qquad w (0, 1; 1, 0) = c_1^{(k)},
		\end{split}
	\end{equation}
	such that 
	$a_1^{(k)}a_2^{(k)}+b_1^{(k)}b_2^{(k)}=c_1^{(k)}c_2^{(k)}$
	for all $k$. Note the swap $c_1 \leftrightarrow c_2$
	in \eqref{eq:abc_weights_for_BBF} compared to our usual conventions 
	\eqref{eq:six_vertex_weights_traditional}
	which is needed to match with \cite{brubaker2011schur}.

	By \cite[Theorem 9]{brubaker2011schur}
	the partition function with the same boundary conditions as for 
	$F_\lambda$ (\Cref{def:F_function}) 
	with the weights \eqref{eq:abc_weights_for_BBF}
	is equal to 
	\begin{equation}
		\label{eq:F_homogeneous_as_Schur_proof1}
		s_{\mu}\left( \frac{b_2^{(1)}}{a_1^{(1)}},\ldots,\frac{b_2^{(N)}}{a_1^{(N)}}  \right)
		\prod_{k=1}^{N}( a_1^{(k)} )^{\mu_1+\lambda_N}c_2^{(k)}
		\prod_{1\le i<j\le N}(a_1^{(j)}a_2^{(i)}+b_1^{(i)}b_2^{(j)}),
	\end{equation}
	where $\mu_i=\lambda_1-\lambda_{N+1-i}$.
	Note that here 
	we needed to flip both the horizontal and the vertical directions 
	for $F_\lambda$ compared to the boundary conditions $\mathfrak{S}_\lambda$
	in \cite{brubaker2011schur}. The extra factors $(a_1^{(k)})^{\lambda_N}$
	come from the fact that our lattice weights for $F_\lambda$ start from 
	horizontal position $1$ on the left boundary.

	Let us rewrite \eqref{eq:F_homogeneous_as_Schur_proof1}
	in a determinantal form, and specialize the weights
	\eqref{eq:abc_weights_for_BBF} to our $\widehat W$ given by \eqref{eq:weights_W_hat}.
	We obtain
	\begin{equation*}
		\begin{split}
			&
			F_\lambda(x_1,\ldots,x_N;1;\mathbf{r};s)
			\\&\hspace{20pt}=
			\det
			\left[ (b_2^{(i)}/a_1^{(i)})^{\lambda_1-\lambda_{N+1-j}+N-j} \right]
			\prod_{k=1}^{N}( a_1^{(k)} )^{\lambda_1}c_2^{(k)}
			(a_1^{(k)})^{N-1}
			\prod_{1\le i<j\le N}
			\frac{a_1^{(j)}a_2^{(i)}+b_1^{(i)}b_2^{(j)}}
			{a_1^{(j)}b_2^{(i)}-a_1^{(i)}b_2^{(j)}}
			\\&\hspace{20pt}=
			\det
			\left[ \left( \frac{1-s^2x_i}{s^2(1-x_i)} \right)^{\lambda_{N+1-j}+j-1} \right]
			\prod_{i=1}^{N}\frac{x_i(r_i^{-2}-1)}{1-x_i}
			\prod_{1\le i<j\le N}
			\frac{x_j-r_i^{-2}x_i}{x_i-x_j}.
		\end{split}
	\end{equation*}
	Replacing $\lambda_{N+1-j}+j-1$ by $\lambda_j+N-j$ in the determinant
	amounts to flipping the signs in all the factors $x_i-x_j$ in the denominator,
	which leads to the desired formula
	\eqref{eq:F_homogeneous_as_Schur}.

	\medskip

	To establish the claim \eqref{eq:G_homogeneous_as_Schur} about $G_\lambda$, take
	identity \eqref{eq:F_G_Cauchy_big_with_F_0_RHS}
	used in the proof of \Cref{thm:F_G_Cauchy_big}:
	\begin{equation*}
		\sum_{\lambda}
		G_{\lambda}(w_1,\ldots,w_M ;1;\boldsymbol\uptheta;s)\,
		\frac{F_{\lambda}(x_1,\ldots,x_N ;1;\mathbf{r};s)}{F_{0^N}(x_1,\ldots,x_N ;1;\mathbf{r};s)}
		=	
		\prod_{i=1}^{N}\prod_{j=1}^{M}
		\frac{x_i-\theta_j^{-2}w_j}{x_i-w_j}.
	\end{equation*}
	Observe that this does not use the explicit formula for $F_\lambda$
	from \Cref{thm:F_formula}. Employing \eqref{eq:F_homogeneous_as_Schur1}, 
	write this identity as
	\begin{equation*}
		\begin{split}
			\sum_{\lambda}G_\lambda(w_1,\ldots,w_M;1;\boldsymbol\uptheta;s)\,s_\lambda(t_1,\ldots,t_N )
			&=
			\prod_{i=1}^{N}\prod_{j=1}^{M}
			\frac
			{1+s^2 ((t_i-1) w_j\theta_j^{-2}- t_i)}
			{1+s^2 (t_i (w_j-1)-w_j)}
			\\&
			=
			\prod_{j=1}^{M}\left( \frac{1-s^2w_j\theta_j^{-2}}{1-s^2w_j} \right)^{N}
			\prod_{i=1}^{N}\prod_{j=1}^{M}
			\frac{1+t_i\frac{s^2(w_j\theta_j^{-2}-1)}{1-s^2w_j\theta_j^{-2}}}{1-t_i\frac{s^2(1-w_j)}{1-s^2w_j}}
			,
		\end{split}
	\end{equation*}
	where we have denoted
	\begin{equation*}
		t_i=\frac{1-s^2 x_i}{s^2(1-x_i)},\qquad \textnormal{so that}
		\qquad 
		x_i=\frac{1-s^2t_i}{s^2(1-t_i)}
	\end{equation*}
	(the map $t_i \leftrightarrow x_i$ is an involution).
	Comparing the previous summation identity with \eqref{eq:supersymm_Schur_definition}
	and using linear independence of the Schur polynomials (i.e., the fact that the 
	coefficients by $s_\lambda(t_1,\ldots,t_N )$ 
	are uniquely determined by the right-hand side),
	we get the claim about $G_\lambda$.
\end{proof}
\begin{corollary}
	\label{cor:F_G_formulas_simpler}
	In the horizontally homogeneous case \eqref{eq:homogeneous_specialization}
	the function $F_\lambda$ is given by the determinantal formula
	\eqref{eq:F_det_formula_in_theorem} of
	\Cref{thm:F_formula}.
\end{corollary}
\begin{proof}[Alternative proof]
	This proof does not rely on 
	\Cref{thm:F_formula} proven in \Cref{appA:F}.
	Under \eqref{eq:homogeneous_specialization} we have
	\begin{equation*}
		\varphi_k(x\mid \mathbf{y};\mathbf{s})=\frac{1}{1-x}\left( \frac{1-s^2x}{s^2(1-x)} \right)^k.
	\end{equation*}
	Thus, the claim follows 
	from the first part of
	\Cref{prop:F_G_homogeneous_through_Schur}.
\end{proof}

\begin{remark}
	It should be also possible to derive 
	the determinantal formulas 
	of \Cref{prop:F_G_homogeneous_through_Schur}
	using the Wick formula and Hamiltonian operators 
	for the free fermion six vertex model
	with horizontally homogeneous weights,
	recently studied in 
	\cite{hardt2021lattice}.
\end{remark}

Let us rewrite the explicit formula for $G_\lambda$ 
from \Cref{thm:G_formula} in the homogeneous case. 
Let $\lambda$ be a signature with $N$ parts.
Recall the integer $d=d(\lambda)$ for which $\lambda_d\ge d$
and $\lambda_{d+1}<d+1$.
Let $\tau=(\lambda_1-d,\ldots,\lambda_d-d )$ and $\eta=(\lambda_{d+1},\ldots,\lambda_N )$
be two auxiliary signatures. Let $\eta'$ denote the transposed signature
corresponding to the reflection of the Young diagram of $\eta$ with respect to the diagonal,
cf. \cite[I.(1.3)]{Macdonald1995}. Observe that both $\tau$ and $\eta'$ have $d$ parts.

\begin{proposition}
	\label{prop:G_homogeneous}
	With the notation given before this proposition, we have
	\begin{equation}
		\label{eq:G_homogeneous}
		\begin{split}
			&\prod_{j=1}^{M}
			\left( 
			\frac{1-s^2r_j^{-2}x_j}{1-s^2 x_j}
			\right)^{-N}
			G_{\lambda} (\mathbf{x}; 1; \mathbf{r}; s) 
			\\&\hspace{20pt}
			=
			\sum_{\substack{\mathcal{I},\mathcal{J}\subseteq \left\{ 1,\ldots,M  \right\}\\ |\mathcal{I}|=|\mathcal{J}|=d}}
			s_\tau(\mathsf{x}_\mathcal{J})
			s_{\eta'}(\mathsf{y}_{\mathcal{I}})
			\prod_{\substack{i\in \mathcal{J}^c\\j\in \mathcal{J}}}
			\frac1{(\mathsf{x}_j-\mathsf{x}_i)}
			\prod_{\substack{i\in \mathcal{I}\\j\in \mathcal{I}^c}}
			\frac{1}{(\mathsf{y}_i-\mathsf{y}_j)}
			\prod_{\substack{i\in \mathcal{I}\\ 1\le j\le M}}(\mathsf{x}_j+\mathsf{y}_i)
			\prod_{\substack{i\in \mathcal{I}^c\\j\in \mathcal{J}}}(\mathsf{x}_j+\mathsf{y}_i),
		\end{split}
	\end{equation}
	where
	\begin{equation}\label{eq:x_y_slf_notations_for_G_homogeneous}
		\mathsf{x}_i=\frac{s^2(1-x_i)}{1-s^2x_i},\qquad 
		\mathsf{y}_i=\frac{s^2(x_ir_i^{-2}-1)}{1-s^2r_i^{-2}x_i},
	\end{equation}
	and $\mathsf{x}_\mathcal{J}$, $\mathsf{y}_{\mathcal{I}}$ are subsets of 
	the variables with indices belonging to $\mathcal{J}$ and $\mathcal{I}$, respectively.
\end{proposition}
\begin{proof}
	Using notation \eqref{eq:x_y_slf_notations_for_G_homogeneous},
	we have
	\begin{align*}
		\mathsf{x}_i-\mathsf{x}_j&=
		\frac{s^2(s^2-1)(x_i-x_j)}{(1-s^2 x_i)(1-s^2 x_j)}
		,
		\\
		\mathsf{y}_i-\mathsf{y}_j&=
		\frac{s^2(1-s^2)(r_i^{-2}x_i-r_j^{-2}x_j)}{(1-s^2r_i^{-2}x_i)(1-s^2r_j^{-2}x_j)}
		,
		\\
		\mathsf{x}_i+\mathsf{y}_j&=\frac{s^2(s^2-1)(x_i-r_j^{-2}x_j)}{(1-s^2x_i)(1-s^2r_j^{-2}x_j)}.
	\end{align*}
	With the help of these formulas, we can express all products in \eqref{eq:G_SP_formula_in_theorem}
	through $\mathsf{x}_i-\mathsf{x_j}$, $\mathsf{y_i}-\mathsf{y}_j$, and
	$\mathsf{x}_j+\mathsf{y}_i$. The remaining sum over $\sigma,\rho\in \mathsf{S}_d$ 
	turns into a product of determinants leading to Schur polynomials in $d$ variables.
	This is due to the facts that in the homogeneous case \eqref{eq:homogeneous_specialization}
	we have
	\begin{equation*}
		\prod_{i = N + 1}^{\ell_h - 1}
			\frac{s_i^2 \big(y_i - x_{j_{\rho (h)}} \big)}{y_i - s_i^2
			x_{j_{\rho (h)}}}
			=
			\mathsf{x}_{j_{\rho(h)}}^{\lambda_h-h}=
			\mathsf{x}_{j_{\rho(h)}}^{\tau_h+d-h}
			,
			\quad 
			\prod_{k =
			\mu_m + 1}^N 
			\frac{s_k^2 \big(r^{-2}_{i_{\sigma (m)}}
			x_{i_{\sigma (m)}} - y_k \big)}{y_k - s_k^2 r^{-2}_{i_{\sigma (m)}}
			x_{i_{\sigma (m)}}}
			=
			\mathsf{y}_{i_{\sigma(m)}}^{N-\mu_m}
			=
			\mathsf{y}_{i_{\sigma(m)}}^{\eta_m'+d-m}.
	\end{equation*}
	In the last equality we used the notation $\mu$ \eqref{eq:mu_depending_on_lambda_notation} and
	an observation that $N-\mu_m=\lambda_m'-m=\eta_m'+d-m$.
	This completes the proof.
\end{proof}

Combining the second part of \Cref{prop:F_G_homogeneous_through_Schur} with
\Cref{prop:G_homogeneous}, we arrive at the following formula
for supersymmetric Schur polynomials:
\begin{corollary}
	\label{cor:super_Schur_formula}
	With the notation given before \Cref{prop:G_homogeneous}, we have
	\begin{equation}
		\label{eq:super_Schur_formula}
		\begin{split}
			&s_\lambda(\mathsf{x}_1,\ldots,\mathsf{x}_M/\mathsf{y}_1,\ldots,\mathsf{y}_M  )
			\\&\hspace{10pt}=
			\sum_{\substack{\mathcal{I},\mathcal{J}\subseteq \left\{ 1,\ldots,M  \right\}\\ |\mathcal{I}|=|\mathcal{J}|=d}}
			s_\tau(\mathsf{x}_\mathcal{J})
			s_{\eta'}(\mathsf{y}_{\mathcal{I}})
			\prod_{\substack{i\in \mathcal{J}^c\\j\in \mathcal{J}}}
			\frac1{(\mathsf{x}_j-\mathsf{x}_i)}
			\prod_{\substack{i\in \mathcal{I}\\j\in \mathcal{I}^c}}
			\frac{1}{(\mathsf{y}_i-\mathsf{y}_j)}
			\prod_{\substack{i\in \mathcal{I}\\ 1\le j\le M}}(\mathsf{x}_j+\mathsf{y}_i)
			\prod_{\substack{i\in \mathcal{I}^c\\j\in \mathcal{J}}}(\mathsf{x}_j+\mathsf{y}_i).
		\end{split}
	\end{equation}
\end{corollary}

For $d=M$, formula \eqref{eq:super_Schur_formula} coincides with the Berele--Regev formula 
\cite{BereleRegev}. The latter provides an expression for supersymmetric Schur polynomials
in this special case $d(\lambda)=M$:
\begin{equation*}
	s_\lambda(\mathsf{x}_1,\ldots,\mathsf{x}_M \,/\, \mathsf{y}_1,\ldots,\mathsf{y}_M  )
	=
	s_\tau(\mathsf{x}_1,\ldots,\mathsf{x}_M )\,
	s_{\eta'}(\mathsf{y}_1,\ldots,\mathsf{y}_M )
	\prod_{i,j=1}^{M}(\mathsf{x}_i+\mathsf{y}_j).
\end{equation*}

\begin{remark}
	\label{rmk:SP_formula_comparison}
	In the general case $d(\lambda)<M$, a formula like
	\eqref{eq:super_Schur_formula} for factorial Grothendieck
	polynomials was proven using integrable lattice models in
	\cite{motegi2020integrability}. See also
	\cite{feher2012equivariant}, \cite{guo2019identities} for
	special cases. We remark that our identity
	\eqref{eq:super_Schur_formula} generalizes the ones in
	\cite{feher2012equivariant}, \cite{guo2019identities} in a
	different direction than what is shown in
	\cite{motegi2020integrability}.

	Moreover, there does not seem to be a direct
	relation between
	\Cref{cor:super_Schur_formula} and other known formulas for supersymmetric Schur polynomials, 
	including the Sergeev--Pragacz formula (e.g., see \cite[(5)]{hamel1995lattice}) and 
	the Moens--Van der Jeugt determinantal formula
	\cite{moens2003determinantal}.
\end{remark}

\section{Biorthogonality and contour integral formulas}
\label{sec:biorthogonality}

Here we discuss torus-like biorthogonality property for
the functions $F_\lambda$, and employ it to derive 
integral and determinantal formulas for the functions $G_\lambda$.

\subsection{Biorthogonality}
\label{sub:F_G_orthogonality}

The functions $F_\lambda(\mathbf{x};\mathbf{y};\mathbf{r};\mathbf{s})$ satisfy a certain biorthogonality 
property with respect to contour integration in the $\mathbf{x}$ variables.
This biorthogonality extends the torus orthogonality of Schur polynomials
as irreducible characters of unitary groups. 
To get the biorthogonality, we use the determinantal formula for $F_\lambda$
of \Cref{thm:F_formula}. 

Fix an integer $N\ge1$ and sequences of complex 
parameters $\mathbf{y}=(y_1,y_2,\ldots )$, $\mathbf{s}=(s_1,s_2,\ldots )$,
$\mathbf{x}=(x_1,\ldots,x_N )$, $\mathbf{r}=(r_1,\ldots,r_N )$.
Recall the functions
$\varphi_k(x\mid \mathbf{y};\mathbf{s})$ defined in \eqref{eq:phi_def}.
Let us also define
\begin{equation}
	\label{eq:psi_def}
	\psi_k(x)=\psi_k(x\mid \mathbf{y};\mathbf{s}):=
	\frac{y_{k+1}(s_{k+1}^2-1)}{y_{k+1}-s^2_{k+1}x}\,
	\prod_{j=1}^{k}
	\frac{s_j^2(y_j-x)}{y_j-s_j^2x}, \qquad k\ge0.
\end{equation}
In particular, $\psi_0(x\mid \mathbf{y},\mathbf{s})=\frac{y_1(s_1^2-1)}{y_1-s_1^2x}$.

\begin{lemma}
	\label{lemma:phi_psi_orthogonal}
	We have for all $k,l\ge 0$:
	\begin{equation}
		\label{eq:phi_psi_orthogonal}
		\frac{1}{2\pi\mathbf{i}}\oint_{\gamma}
		\varphi_k(z\mid \mathbf{y}, \mathbf{s})\,
		\psi_l(z\mid \mathbf{y},\mathbf{s})\,dz
		=
		\mathbf{1}_{k=l},
	\end{equation}
	where 
	$\gamma$ is a
	closed counterclockwise simple contour
	in the complex plane containing the points $y_j$ for all $j\ge1$
	and not $y_js_j^{-2}$ for all $j\ge1$,
	and
	$\mathbf{1}_{k=l}$ is the indicator that $k=l$ (i.e., the Kronecker delta).
\end{lemma}
\begin{remark}
	\label{rmk:formal_integrals}
	Here and below in this section 
	we assume that the parameters (here, sequences $\mathbf{y}$ and $\mathbf{s}$)
	are such that the integration contour exists.
	Alternatively,
	one may also think of the integration \emph{formally} as the 
	sum of residues at all the points $y_j$, $j\ge1$, which the contour
	must encircle.
\end{remark}
\begin{proof}[Proof of \Cref{lemma:phi_psi_orthogonal}]
	First, observe that at 
	$z=\infty$ both $\varphi_k(z),\psi_l(z)$
	behave as $\mathrm{const}\cdot z^{-1}$
	for all $k,l\ge0$. 
	For $k<l$, 
	the product $\varphi_k(z)\psi_l(z)$ 
	has only factors of the form $y_j-s_j^2z$ in the denominator, and thus
	has no poles 
	inside the integration contour $\gamma$. Therefore, the 
	integral \eqref{eq:phi_psi_orthogonal} vanishes for $k<l$.
	For $k>l$, the 
	product $\varphi_k(z)\psi_l(z)$ 
	has at least two factors of the form $z-y_j$ 
	and no factors of the form $y_j-s_j^2z$ in the denominator.
	Therefore, the integrand is regular outside the contour, so
	the 
	integral \eqref{eq:phi_psi_orthogonal} vanishes for $k>l$ as well.
	Finally, for $k=l$ we have
	\begin{equation*}
		\varphi_k(z)\psi_k(z)=
		\frac{1}{z-y_{k+1}}\cdot \frac{y_{k+1}(1-s_{k+1}^2)}{y_{k+1}-s_{k+1}^2z},
	\end{equation*}
	and the claim immediately follows.
\end{proof}

Using \Cref{lemma:phi_psi_orthogonal}, define the following counterparts of the functions
$F_\lambda$:
\begin{equation}
	\label{eq:F_star_det_definition}
	F_\lambda^*(\mathbf{x};\mathbf{y};\mathbf{r};\mathbf{s}):=
	\det\left[ \psi_{\lambda_j+N-j}(x_i\mid \mathbf{y};\mathbf{s}) \right]_{i,j=1}^{N}
	\prod_{N\ge i>j\ge 1}\frac{x_j-r_i^{-2}x_i}{x_j-x_i},
\end{equation}
where $\lambda$ is a signature with $N$ parts.
\begin{remark}
	In the horizontally homogeneous case $s_j \equiv s$, $y_j\equiv 1$,
	the functions $F_\lambda^*$ are almost the same as the $F_\lambda$'s, up to a factor and a 
	change of variables:
	\begin{multline*}
		F_\lambda^*(s^{-2}/x_1,\ldots,s^{-2}/x_N;1;\mathbf{r};s)
		\\=
		\frac{(1-s^2)^N (s^2)^{|\lambda|+N(N-1)/2}}{\prod_{i=1}^{N}(r_i^{-2}-1)}
		\prod_{1\le i<j\le N}
		\frac{r_j^{-2}x_i-x_j}{r_i^{-2}x_i-x_j}
		\,
		F_\lambda(x_1,\ldots, x_N;\mathbf{y};\mathbf{r};\mathbf{s}).
	\end{multline*}
	However, in general the $F_\lambda^*$'s cannot be expressed through the $F_\lambda$'s.
\end{remark}

\begin{proposition}
	\label{prop:F_F_star_orthogonality}
	For any signatures $\lambda,\mu$ with $N$ parts we have
	\begin{equation*}
		\frac{1}{N!(2\pi\mathbf{i})^{N}}
		\oint_\gamma dz_1
		\ldots 
		\oint_\gamma dz_N
		\,
		\frac{\prod_{1 \le i\ne j\le N}(z_i-z_j)}{\prod_{i,j=1}^{N}(r_i^{-2}z_i-z_j)}
		\,
		F_\lambda(\mathbf{z};\mathbf{y};\mathbf{r};\mathbf{s} )
		\,
		F_\mu^*(\mathbf{z};\mathbf{y};\mathbf{r};\mathbf{s} )
		=\mathbf{1}_{\lambda=\mu},
	\end{equation*}
	where the integration is over the variables
	$\mathbf{z}=(z_1,\ldots,z_N )$ belonging to the torus $\gamma^N$,
	and $\gamma$ is a contour around $y_j$ not encircling $y_js_j^{-2}$, $j\ge1$.
	Note that the integrand has no poles at $z_j=r_i^{-2}z_i$ for any $i,j$.
\end{proposition}
\begin{proof}
	This proof is similar to the well-known proof of torus orthogonality of the Schur polynomials. 
	Cancel out the prefactors, and
	expand the determinants in $F_\lambda$ and $F_\mu^*$
	as sums over permutations:
	\begin{equation*}
		\begin{split}
			&\frac{1}{N!(2\pi\mathbf{i})^{N}}
			\oint_\gamma dz_1
			\ldots 
			\oint_\gamma dz_N
			\,
			\frac{\prod_{1\le i\ne j\le N}(z_i-z_j)}{\prod_{i,j=1}^{N}(r_i^{-2}z_i-z_j)}
			\,
			F_\lambda(\mathbf{z};\mathbf{y};\mathbf{r};\mathbf{s} )
			\,
			F_\mu^*(\mathbf{z};\mathbf{y};\mathbf{r};\mathbf{s} )
			\\
			&\hspace{20pt}=
			\frac{1}{N!(2\pi\mathbf{i})^{N}}
			\oint_\gamma dz_1
			\ldots 
			\oint_\gamma dz_N
			\,
			\det\left[ \varphi_{\lambda_j+N-j}(z_i\mid \mathbf{y};\mathbf{s}) \right]_{i,j=1}^{N}
			\det\left[ \psi_{\mu_j+N-j}(z_i\mid \mathbf{y};\mathbf{s}) \right]_{i,j=1}^{N}
			\\&\hspace{20pt}=
			\frac{1}{N!}
			\sum_{\sigma,\tau\in \mathfrak{S}_N}
			\mathop{\mathrm{sgn}}(\sigma\tau)
			\prod_{i=1}^{N}
			\frac{1}{2\pi\mathbf{i}}
			\oint_\gamma
			\varphi_{\lambda_{\sigma(i)}+N-\sigma(i)}(z_i)\,
			\psi_{\mu_{\tau(i)}+N-\tau(i)}(z_i)\,
			dz_i.
		\end{split}
	\end{equation*}
	Using \Cref{lemma:phi_psi_orthogonal}, we see that the product of the integrals is 
	nonzero only if $\sigma=\tau$ and $\lambda=\mu$. When the integral is nonzero, it is equal to $1$.
	Summing
	$N!$ terms
	corresponding to each $\sigma=\tau\in \mathfrak{S}_N$, we get the result.
\end{proof}

\subsection{Contour integral formula for \texorpdfstring{$F_{\lambda/\mu}$}{F}}
\label{sub:contour_int_F}

Using the branching rule (\Cref{prop:F_G_branching})
and the biorthogonality (\Cref{prop:F_F_star_orthogonality}),
we are able to get contour integral formulas for the functions $F_{\lambda/\mu}$.

Fix $N,M\ge1$, and signatures $\lambda$ with $N+M$ parts and $\mu$ with $M$ parts. 
Furthermore, fix sequences of complex numbers
\begin{equation*}
	\mathbf{x}=(x_1,\ldots,x_N ),
	\qquad 
	\mathbf{r}=(r_1,\ldots,r_N ),
	\qquad 
	\mathbf{y}=(y_1,y_2,\ldots, ),
	\qquad 
	\mathbf{s}=(s_1,s_2,\ldots ).
\end{equation*}

\begin{proposition}
	\label{prop:skew_F_contour}
	With the above notation, we have
	\begin{equation}
		\label{eq:skew_F_contour}
		\begin{split}
			&
			F_{\lambda/\mu}(\mathbf{x};\mathbf{y};\mathbf{r};\mathbf{s})
			=
			\prod_{i=1}^{N}x_i(r_i^{-2}-1)
			\prod_{1\le i<j\le N}\frac{r_i^{-2}x_i-x_j}{x_i-x_j}
			\\&\hspace{40pt}\times\frac{1}{M!(2\pi\mathbf{i})^{M}}
			\oint_\gamma dz_1
			\ldots
			\oint_\gamma dz_M
			\prod_{i=1}^{N}\prod_{j=1}^{M}\frac{z_j-r_i^{-2}x_i}{z_j-x_i}
			\\&\hspace{100pt}\times
			\det[\varphi_{\lambda_j+N+M-j}
			( (\mathbf{x}\cup \mathbf{z})_i\mid \mathbf{y};\mathbf{s})]_{i,j=1}^{N+M}
			\det\left[ \psi_{\mu_j+M-j}(z_i\mid \mathbf{y};\mathbf{s}) \right]_{i,j=1}^{M},
		\end{split}
	\end{equation}
	where 
	\begin{equation*}
		(\mathbf{x}\cup\mathbf{z})_i=\begin{cases}
			x_i,&1\le i\le M;\\
			z_{i-M},&M+1\le i\le M+N,
		\end{cases}
	\end{equation*}
	the integration is over 
	the torus $\gamma^M$,
	and $\gamma$ is a contour around $y_j$ not encircling $y_js_j^{-2}$, $j\ge1$.
	Per \Cref{rmk:formal_integrals}, 
	we either assume that the contour $\gamma$ exists, or 
	treat the integral
	formally.
\end{proposition}
\begin{proof}
	Throughout the proof we use the notation 
	$\mathbf{z}=(z_1,\ldots,z_M)$
	and 
	$\boldsymbol\uptheta=(\theta_1,\ldots,\theta_M )$.
	We have 
	from the branching rule (\Cref{prop:F_G_branching})
	\begin{equation*}
		F_\lambda(\mathbf{x}\cup \mathbf{z};\mathbf{y};
		\mathbf{r}\cup \boldsymbol\uptheta;\mathbf{s})=
		\sum_{\nu}
		F_{\lambda/\nu}(\mathbf{x};\mathbf{y};\mathbf{r};\mathbf{s})
		\,
		F_\nu(\mathbf{z};\mathbf{y};\boldsymbol\uptheta;\mathbf{s})
		,
	\end{equation*}
	where the sum is over all signatures with $M$ parts.
	Multiply this (finite) sum by
	\begin{equation*}
		\begin{split}
			&
			\frac{1}{M!(2\pi\mathbf{i})^{M}}\,
			F_\mu^*(\mathbf{z};\mathbf{y};\boldsymbol\uptheta;\mathbf{s})\,
			\frac{\prod_{1 \le i\ne j\le M}(z_i-z_j)}{\prod_{i,j=1}^{M}(\theta_i^{-2}z_i-z_j)}
			\\&\hspace{90pt}=
			\frac{1}{M!(2\pi\mathbf{i})^{M}}\,
			\frac{\prod_{i<j}(z_i-z_j)}{\prod_{i\le j}(\theta_i^{-2}z_i-z_j)}
			\det\left[ \psi_{\mu_j+M-j}(z_i\mid \mathbf{y};\mathbf{s}) \right]_{i,j=1}^{M}
		\end{split}
	\end{equation*}
	and integrate over $\mathbf{z}=(z_1,\ldots,z_M )\in\gamma^M$,
	where $\gamma$ is a positively oriented contour 
	around all $y_j$ and not encircling $y_js_j^{-2}$, $j\ge1$.
	The integration extracts the single coefficient by
	$F_\nu$ with $\mu=\nu$, which together with 
	the determinantal formulas for $F_\lambda$ \eqref{eq:F_det_formula_in_theorem}
	and for $F_\mu^*$ \eqref{eq:F_star_det_definition}
	produces the desired expression.
\end{proof}

\subsection{Contour integral formula for \texorpdfstring{$G_{\nu/\lambda}$}{G}}
\label{sub:contour_int_G}

Using the skew Cauchy identity (\Cref{prop:F_G_Cauchy_skew})
and the biorthogonality (\Cref{prop:F_F_star_orthogonality}),
we can get contour integral formulas for the 
functions $G_{\nu/\lambda}$ and $G_\nu$.

Fix integers $N,M\ge1$ and signatures $\lambda,\nu$ with $N$ parts.
Also fix sequences of complex numbers
\begin{equation}
	\label{eq:wythetas_for_G_skew_contour}
	\mathbf{w}=(w_1,\ldots,w_M ),\qquad 
	\mathbf{y}=(y_1,y_2,\ldots ),\qquad 
	\boldsymbol\uptheta=(\theta_1,\ldots,\theta_N ),\qquad 
	\mathbf{s}=(s_1,s_2,\ldots ).
\end{equation}

\begin{proposition}
	\label{prop:G_integral_formula}
	With the above notation, 
	we have the following contour integral representation for the symmetric 
	functions $G_{\nu/\lambda}$:
	\begin{equation}
		\label{eq:skew_G_contour_integral}
		\begin{split}
			&
			G_{\nu/\lambda}(\mathbf{w};\mathbf{y};\boldsymbol\uptheta;\mathbf{s})=
			\frac{1}{N!(2\pi \mathbf{i})^N}
			\oint_{\gamma'} dz_1 \ldots \oint_{\gamma'} dz_N 
			\prod_{i=1}^{N}\prod_{j=1}^M\frac{z_i-\theta_j^{-2}w_j}{z_i-w_j}
			\\&\hspace{150pt}\times
			\det\left[ \varphi_{\lambda_i+N-i}(z_j\mid \mathbf{y};\mathbf{s}) \right]_{i,j=1}^N
			\det\left[ \psi_{\nu_i+N-i}(z_j\mid \mathbf{y};\mathbf{s}) \right]_{i,j=1}^N
			.
		\end{split}
	\end{equation}
	In particular, for $\lambda=\varnothing$ we have
	\begin{equation}
		\label{eq:nonskew_G_contour_integral}
		\begin{split}
			&G_\nu(\mathbf{w};\mathbf{y};\boldsymbol\uptheta;\mathbf{s})=
			\frac{\prod_{1\le i<j\le N}(s_i^{-2}y_i-y_j)}{N!(2\pi \mathbf{i})^N}
			\oint_{\gamma'} dz_1 \ldots \oint_{\gamma'} dz_N \,
			\det\left[ \psi_{\nu_i+N-i}(z_j\mid \mathbf{y};\mathbf{s}) \right]_{i,j=1}^N
			\\&\hspace{220pt}\times
			\frac{\prod_{1\le i<j\le N}(z_i-z_j)}{\prod_{i,j=1}^N(y_i-z_j)}
			\,
			\prod_{i=1}^{N}\prod_{j=1}^M\frac{z_i-\theta_j^{-2}w_j}{z_i-w_j}
			.
		\end{split}
	\end{equation}
	In both formulas \eqref{eq:skew_G_contour_integral} and \eqref{eq:nonskew_G_contour_integral}
	the contour $\gamma'$ encircles 
	all $y_j$, $j\ge1$, and $w_i$, $i=1,\ldots,M $, 
	and leaves out all $y_js_j^{-2}$, $j\ge1$.
	Per \Cref{rmk:formal_integrals}, 
	we either assume the contour $\gamma'$ exists, or 
	treat the integrals in 
	\eqref{eq:skew_G_contour_integral}--\eqref{eq:nonskew_G_contour_integral} formally.
\end{proposition}
\begin{proof}
	Since $G_\nu=G_{\nu/0^N}$, identity \eqref{eq:nonskew_G_contour_integral} follows from
	\eqref{eq:skew_G_contour_integral} and the product formula 
	\eqref{eq:fully_deformed_Cauchy_determinant}
	for 
	$F_{0^N}$.
	Next, by 
	\Cref{def:G_function}, for fixed $\lambda,\nu$ 
	the partition function $G_{\nu/\lambda}$ is a rational 
	function in $\mathbf{w}, \boldsymbol\uptheta$, 
	as well as in a finite subfamily of the parameters $\mathbf{y}$ and $\mathbf{s}$.
	The integral in the right-hand side of \eqref{eq:skew_G_contour_integral}
	is also a rational function of these parameters. Therefore, in proving the proposition
	we are allowed to impose any open conditions on the parameters, and then identity
	\eqref{eq:skew_G_contour_integral}
	would hold in general thanks to analytic continuation.

	Take the skew Cauchy identity 
	\eqref{eq:skew_Cauchy_F_G}
	with $\mu=\varnothing$ (and hence $\varkappa=\varnothing$ in the right-hand side, 
	which eliminates the summation):
	\begin{equation}
		\label{eq:skew_G_contour_integral_proof}
		\sum_{\nu}
		G_{\nu/\lambda}(\mathbf{w};\mathbf{y};\boldsymbol\uptheta;\mathbf{s})
		F_{\nu}(\mathbf{z};\mathbf{y};\mathbf{r};\mathbf{s})
		=
		F_{\lambda}(\mathbf{z};\mathbf{y};\mathbf{r};\mathbf{s})
		\prod_{i=1}^{M}\prod_{j=1}^{N}
		\frac{z_i-\theta_j^{-2}w_j}{z_i-w_j}.
	\end{equation}
	In fact, both sides of \eqref{eq:skew_G_contour_integral_proof}
	contain the same factor depending on $\mathbf{r}$ which can be canceled out,
	see the determinantal formula 
	for $F_\lambda$ \eqref{eq:F_det_formula_in_theorem}.
	In other words, \eqref{eq:skew_G_contour_integral_proof} essentially does not depend 
	on~$\mathbf{r}$.
	
	In \eqref{eq:skew_G_contour_integral_proof},
	we assume that
	$\mathbf{w},\mathbf{y},\mathbf{s}$, and $\mathbf{z}$ are such that
	\begin{enumerate}
		\item All $z_j$ belong to some contour $\gamma'$ encircling all $w_i$ and all $y_i$;
		\item For all $z\in \mathbf{\gamma}'$ and all
			$j,k$
			we have
			$\displaystyle\left|
			\frac{y_k-s^2_k z}{y_k-z}
			\frac{y_k-w_j}{y_k-s^2_k w_j}
			\right|<1-\delta<1$.
			This is the condition which implies convergence in \eqref{eq:skew_G_contour_integral_proof}, see
			\Cref{prop:F_G_Cauchy_skew}.
	\end{enumerate}
	One readily sees that these restrictions on $\mathbf{w},\mathbf{y}$, and $\mathbf{s}$
	place them into a nonempty open set, which is sufficient for analytic continuation.

	Now, multiply both sides
	of \eqref{eq:skew_G_contour_integral_proof}
	by 
	\begin{equation*}
		\displaystyle
		\frac{1}{N!(2\pi\mathbf{i})^{N}}\,
		F_\nu^*(\mathbf{z};\mathbf{y};\mathbf{r};\mathbf{s})\,
		\frac{\prod_{1 \le i\ne j\le N}(z_i-z_j)}{\prod_{i,j=1}^{N}(r_i^{-2}z_i-z_j)}
	\end{equation*}
	and integrate over $\mathbf{z}\in (\gamma')^N $. 
	Since the sum in the right-hand side of \eqref{eq:skew_G_contour_integral_proof}
	converges 
	uniformly
	on the contours, we can interchange summation and integration
	and use the biorthogonality of \Cref{prop:F_F_star_orthogonality}
	to extract the coefficient $G_{\nu/\lambda}(\mathbf{w};\mathbf{y};\boldsymbol\uptheta;\mathbf{s})$.
	After simplification with the help
	of determinantal formulas
	\eqref{eq:F_det_formula_in_theorem} and \eqref{eq:F_star_det_definition}, 
	the integration of the right-hand side of \eqref{eq:skew_G_contour_integral_proof}
	yields the right-hand side of the desired identity
	\eqref{eq:skew_G_contour_integral}. Observe that the dependence on the $r_i$'s 
	disappears, as it should be.
	Analytic continuation then allows to remove the restrictions stated above in the proof, and 
	we arrive at the result.
\end{proof}

\subsection{Jacobi--Trudy type formulas for \texorpdfstring{$G_{\nu/\lambda}$ and $G_\nu$}{G}}
\label{sub:G_det_formula}

Using the contour integral representation for $G_{\nu/\lambda}$
from \Cref{prop:G_integral_formula}, it is possible to derive a 
Jacobi--Trudy type determinantal formula for these symmetric functions.

For $m\ge1$ define the shift operator $\mathrm{sh}_m$ acting on the sequences $\mathbf{y},\mathbf{s}$
as
$(\mathrm{sh}_m \mathbf{y})_i=y_{m+i}$, 
$(\mathrm{sh}_m \mathbf{s})_i=s_{m+i}$. 
Also define for all $l\in \mathbb{Z}$:
\begin{equation}
	\label{eq:tilde_h_function_for_JT_G}
	\widetilde{\mathsf{h}}_{l}(\mathbf{w};\mathbf{y};\boldsymbol\uptheta;\mathbf{s}):=
	\frac{\mathbf{1}_{l\ge0}}{2\pi \mathbf{i}}\oint_{\gamma'}
		\frac{(s_{l+1}^2-1)y_{l+1}}{(y_{l+1}-s_{l+1}^2z)(y_1-z)}
		\prod_{j=1}^{l}\frac{s_j^2(y_j-z)}{y_j-s_j^2z}
		\prod_{j=1}^{M}\frac{z-\theta_j^{-2}w_j}{z-w_j}\,
		dz,
\end{equation}
where the integration is over a contour $\gamma'$ around all the points
$y_j$ and $w_i$,
leaving outside the points $y_js_j^{-2}$.
Observe that \eqref{eq:tilde_h_function_for_JT_G} is
symmetric under simultaneous
permutations of $(w_i,\theta_i)$.
Also denote 
\begin{equation}
	\label{eq:g_skew_better_notation}
	\mathsf{g}_{l/k}(\mathbf{w};\mathbf{y};\boldsymbol\uptheta;\mathbf{s})
	:=
	\widetilde{\mathsf{h}}_{l-k}(\mathbf{w};\mathrm{sh}_k \mathbf{y};\boldsymbol\uptheta;\mathrm{sh}_k \mathbf{s}).
\end{equation}

\begin{proposition}
	\label{prop:skew_G_Jacobi_Trudi}
	Fix $N\ge1$.
	For any signatures $\lambda,\nu$ with $N$ parts, 
	and sequences of complex numbers
	$\mathbf{w},\mathbf{y},\boldsymbol\uptheta,\mathbf{s}$ as in 
	\eqref{eq:wythetas_for_G_skew_contour},
	we have
	\begin{equation}
		\label{eq:skew_G_Jacobi_Trudi}
		G_{\nu/\lambda}(\mathbf{w};\mathbf{y};\boldsymbol\uptheta;\mathbf{s})=
		\det
		\Bigl[ 
			\mathsf{g}_{(\nu_i+N-i)/(\lambda_j+N-j)}
			(\mathbf{w};\mathbf{y};
			\boldsymbol\uptheta;\mathbf{s}) 
		\Bigr]_{i,j=1}^{N}.
	\end{equation}
\end{proposition}
\begin{remark}
	\label{rmk:9th_variation}
	The shifts of the indices in $\mathbf{y}$ and $\mathbf{s}$
	in \eqref{eq:skew_G_Jacobi_Trudi} (see \eqref{eq:g_skew_better_notation})
	are the same as in \cite[9th variation]{macdonald1992schur_Theme} (see also
	\cite{nakagawa2001tableau}),
	which makes our functions $G_{\nu/\lambda}$ a particular case of the 
	Macdonald's
	9-th variation of the Schur functions.
\end{remark}

\begin{proof}[Proof of \Cref{prop:skew_G_Jacobi_Trudi}]
	Observe that 
	\begin{equation*}
		\frac{1}{2\pi \mathbf{i}}\oint_{\gamma'}
		dz\,
		\varphi_k(z\mid \mathbf{y};\mathbf{s})\psi_l(z\mid \mathbf{y};\mathbf{s})
		\prod_{j=1}^{M}\frac{z-\theta_j^{-2}w_j}{z-w_j}
		=\widetilde{\mathsf{h}}_{l-k}(\mathbf{w};\mathrm{sh}_k \mathbf{y};\boldsymbol\uptheta;\mathrm{sh}_k \mathbf{s}).
	\end{equation*}
	Thus, the proposition follows by applying Andr\'eief identity 
	(cf. \cite{forrester2019meet})
	\begin{equation}
		\begin{split}
			\label{eq:Andreief}
			&
			\frac{1}{(2\pi\mathbf{i})^N}
			\oint_{\gamma'}\ldots\oint_{\gamma'} 
			\det[f_i(z_j)]_{i,j=1}^N
			\det[g_i(z_j)]_{i,j=1}^N\,dz_1\ldots dz_N \\&\hspace{220pt}=
			N!\,
			\det\left[ \frac{1}{2\pi\mathbf{i}}\oint_{\gamma'}f_i(z)g_j(z)\,dz \right]_{i,j=1}^N
		\end{split}
	\end{equation}
	to the contour integral formula for $G_{\nu/\lambda}$
	\eqref{eq:skew_G_contour_integral}. Indeed, here
	we can take $f_i(z)=\varphi_{\lambda_i+N-i}(z\mid \mathbf{y};\mathbf{s})$ 
	and $\displaystyle
	g_j(z)=\psi_{\nu_j+N-j}(z\mid \mathbf{y};\mathbf{s})\prod_{m=1}^M\frac{z-\theta_m^{-2}w_m}{z-w_m}$.
\end{proof}

\begin{remark}
	Using \Cref{prop:skew_G_Jacobi_Trudi},
	one can check that 
	in the case of horizontally homogeneous
	parameters
	$y_j\equiv 1$, $s_j\equiv s$
	the skew functions 
	$G_{\lambda/\nu}(\mathbf{w};\mathbf{y};\boldsymbol\uptheta;\mathbf{s})$ 
	turn (up to a simple product factor) into the supersymmetric skew
	Schur functions in the variables
	$\frac{w_i-1}{w_i-s^{-2}}\Big/
	\frac{1-\theta_i^{-2}w_i}{\theta_i^{-2}w_i-s^{-2}}$.
	In other words, identity
	\eqref{eq:G_homogeneous_as_Schur}
	extends from the 
	non-skew case to the skew one.
\end{remark}

For non-skew functions $G_\nu$
there is a 
simplification of the formula of \Cref{prop:skew_G_Jacobi_Trudi}.
Define
\begin{equation}
	\label{eq:nontilde_h_function_for_JT_G_nonskew}
	\mathsf{h}_{k,p}(\mathbf{w};\mathbf{y};\boldsymbol\uptheta;\mathbf{s})
	:=
	\frac{1}{2\pi\mathbf{i}}\oint_{\gamma'}dz\,
	\frac{\psi_k(z\mid \mathbf{y};\mathbf{s})}{y_p-z}
	\prod_{j=1}^{M}\frac{z-\theta_j^{-2}w_j}{z-w_j},
\end{equation}
where $\psi_k$ is given by \eqref{eq:psi_def},
and the integration contour $\gamma'$ surrounds
all the points $y_j,w_i$ and leaves out all
$y_js_j^{-2}$. Comparing \eqref{eq:tilde_h_function_for_JT_G} and \eqref{eq:nontilde_h_function_for_JT_G_nonskew},
we see that 
$\widetilde{\mathsf{h}}_l = \mathsf{h}_{l,1}$
for $l\ge0$.
\begin{proposition}[\Cref{prop:intro_G_JT} from Introduction]
	\label{prop:nonskew_G_Jacobi_Trudi}
	Fix $N\ge1$. For any signature $\nu$ with $N$ parts, and 
	sequences of complex numbers
	$\mathbf{w},\mathbf{y},\boldsymbol\uptheta,\mathbf{s}$ as in 
	\eqref{eq:wythetas_for_G_skew_contour},
	we have
	\begin{equation}
		\label{eq:nonskew_G_Jacobi_Trudi}
		G_\nu(\mathbf{w};\mathbf{y};\boldsymbol\uptheta;\mathbf{s})
		=
		\prod_{1\le i<j\le N}\frac{s_i^{-2}y_i-y_j}{y_j-y_i}
		\,
		\det\left[ \mathsf{h}_{\nu_i+N-i,\,j}(\mathbf{w};\mathbf{y};\boldsymbol\uptheta;\mathbf{s}) \right]_{i,j=1}^{N}.
	\end{equation}
\end{proposition}
\begin{proof}
	The integrand in the contour integral formula 
	for $G_\nu$
	\eqref{eq:nonskew_G_contour_integral}
	contains the terms which can be rewritten as a Cauchy determinant:
	\begin{equation*}
		\frac{\prod_{1\le i<j\le N}(z_i-z_j)}{\prod_{i,j=1}^{N}(y_i-z_j)}=
		\frac{1}{\prod_{1\le i<j\le N}(y_j-y_i)}\,
		\det\left[ \frac{1}{y_i-z_j} \right]_{i,j=1}^N.
	\end{equation*}
	Combining this with the other determinant 
	$\det[\psi_{\nu_i+N-i}(z_j)]$ in \eqref{eq:nonskew_G_contour_integral}
	and applying Andr\'eief identity \eqref{eq:Andreief},
	we arrive at the desired formula.
\end{proof}

\newpage
\part{Determinantal processes}
\label{partII}

In this part (accompanied by \Cref{appB:Eynard_Mehta}) we 
develop determinantal point processes based on the symmetric functions
$F_\lambda,G_\lambda$ from \Cref{partI}.
By analogy with Schur and Macdonald processes
\cite{okounkov2003correlation}, \cite{BorodinCorwin2011Macdonald},
we call them the \emph{FG processes}. We compute the 
correlation kernel for ascending FG processes (a particular subclass of FG processes) in a double contour integral 
form.

\section{FG measures and processes}
\label{sec:FG_measures}

\subsection{Specializations}
\label{sub:FG_spec}

Fix the parameter sequences 
\begin{equation*}
	\mathbf{y}=(y_1,y_2,\ldots ),
	\quad
	\mathbf{s}=(s_1,s_2,\ldots ),
\end{equation*}
for which there exists $\varepsilon>0$ such that
\begin{equation}
	\label{eq:ys_FG_parameters}
	\varepsilon<y_j<\varepsilon^{-1},\quad
	\varepsilon<s_j<1-\varepsilon\quad 
	\textnormal{for all $j$}.
\end{equation}

Under suitable restrictions on the other parameters, the 
values of $F_\lambda, G_\lambda$ become nonnegative. This leads to the 
following definition:

\begin{definition}
	\label{def:spec}
	Let $N\in \mathbb{Z}_{\ge1}$ and let
	$\mathbf{x}=(x_1,\ldots,x_N )$, $\mathbf{r}=(r_1,\ldots,r_N )$
	be such that 
	\begin{equation}
		\label{eq:x_r_spec_parameters}
		0< x_i<y_j<r_{i}^{-2}x_i<s_j^{-2}y_j,\quad
		\textnormal{for all $i,j$}.
	\end{equation}
	Under \eqref{eq:ys_FG_parameters} and \eqref{eq:x_r_spec_parameters}, 
	one readily sees that all the vertex weights 
	$W,\widehat W$ \eqref{eq:weights_W}, \eqref{eq:weights_W_hat} 
	are nonnegative. This implies that the values of the symmetric functions
	$F_{\lambda/\mu}(\mathbf{x};\mathbf{y};\mathbf{r};\mathbf{s})$
	and
	$G_{\lambda/\mu}(\mathbf{x};\mathbf{y};\mathbf{r};\mathbf{s})$
	are nonnegative.

	We call $(\mathbf{x};\mathbf{r})$ a \emph{nonnegative specialization}, 
	and denote this by $\rho\in \mathsf{Spec}_N$ (where $N$ indicates the number of 
	variables). 
	When 
	convenient, we denote the values of our symmetric functions at~$\rho$ by 
	$F_{\lambda/\mu}(\rho)$, $G_{\lambda/\mu}(\rho)$, and omit 
	explicitly specifying 
	their overall dependence on 
	$\mathbf{y},\mathbf{s}$.
\end{definition}

\begin{remark}
	\label{rmk:nonnegative_condition_on_x_for_convenience}
	The vertex weights 
	$W,\widehat W$ \eqref{eq:weights_W}, \eqref{eq:weights_W_hat} 
	depend only on differences between the parameters $x,y,r^{-2}x,s^{-2}y$.
	Therefore, conditions $x_i,y_j>0$ 
	in \eqref{eq:ys_FG_parameters}--\eqref{eq:x_r_spec_parameters}
	may be dropped, but we keep them 
	throughout this \Cref{partII}
	for convenience of dealing with various inequalities on the parameters.
\end{remark}

The empty specialization $\rho=\varnothing\in \mathsf{Spec}_0$ is nonnegative, and
\begin{equation}
	\label{eq:F_G_at_empty}
	F_{\lambda/\mu}(\varnothing)
	=\mathbf{1}_{\lambda=\mu},\qquad 
	G_{\lambda/\mu}(\varnothing)=\mathbf{1}_{\lambda=\mu}.
\end{equation}
For the function $G_{\lambda/\mu}$, 
we also get the same delta function by substituting the zero variables,
namely, 
$G_{\lambda/\mu}(0,\ldots,0;\mathbf{y};\mathbf{r} ;\mathbf{s})=\mathbf{1}_{\lambda=\mu}$. This is evident by looking at the vertex weights
$W$ \eqref{eq:weights_W}, as the weight of the vertex 
$(1,0;0,1)$ vanishes.

For two specializations 
$\rho=(\mathbf{x};\mathbf{r})$ and 
$\rho'=(\mathbf{x}';\mathbf{r}')$,
denote by $\rho\cup \rho'$ their union (concatenation)
with variables
$(\mathbf{x}\cup \mathbf{x}';\mathbf{r}\cup \mathbf{r}')$,
as in \Cref{prop:F_G_branching}.

In order to use the Cauchy identities, we need to make sure that the 
corresponding infinite series converge:

\begin{definition}
	\label{def:compatible_spec}
	Two nonnegative specializations $\rho=(\mathbf{x};\mathbf{r})\in \mathsf{Spec}_N$ 
	and
	$\rho'=(\mathbf{w},\boldsymbol\uptheta)\in \mathsf{Spec}_M$
	are called \emph{compatible}
	(notation $(\rho;\rho')\in \mathsf{Comp}$)
	if there exists $\delta>0$ such that
	\begin{equation}
		\label{eq:compatible_specializations}
		\left|
		\frac{s_k^{-2}y_k-x_i}{y_k-x_i}
		\frac{y_k-w_j}{s_k^{-2}y_k- w_j}
		\right|<1-\delta<1
		\qquad 
		\textnormal{for all $i,j$ and all sufficiently large $k>0$}.
	\end{equation}
	Compatibility depends on 
	the parameters $(\mathbf{y};\mathbf{s})$, which are assumed fixed.
	Note also that the relation 
	$(\rho;\rho')\in \mathsf{Comp}$
	is not symmetric
	in $\rho,\rho'$.
\end{definition}

Let us denote, for 
$\rho=(\mathbf{x};\mathbf{r})\in \mathsf{Spec}_N$,
$\rho'=(\mathbf{w},\boldsymbol\uptheta)\in \mathsf{Spec}_M$,
\begin{equation}
	\label{eq:Pi_Z_notation_}
	\begin{split}
		\Pi(\rho;\rho')&:=
		\prod_{i=1}^{N}\prod_{j=1}^{M}
		\frac{x_i-\theta_j^{-2}w_j}{x_i-w_j},
		\\
		Z(\rho)&:=
		\prod_{i=1}^{N}x_i(r^{-2}_i-1)\,
		\frac{\prod_{1\le i<j\le N}(r_i^{-2}x_i-x_j)(s_i^{-2}y_i-y_j)}
		{\prod_{i,j=1}^N(y_i-x_j)}.
	\end{split}
\end{equation}
Thus, the Cauchy identities (\Cref{prop:F_G_Cauchy_skew,thm:F_G_Cauchy_big})
take the following form for two compatible
specializations $\rho,\rho'$:
\begin{equation}
	\label{eq:Cauchy_identities_via_specs}
	\begin{split}
		\sum_{\nu}
		G_{\nu/\lambda}(\rho')
		F_{\nu/\mu}(\rho)&=
		\Pi(\rho;\rho')
		\sum_{\varkappa}
		G_{\mu/\varkappa}(\rho')
		F_{\lambda/\varkappa}(\rho),
		\\
		\sum_\nu
		F_\nu(\rho)G_\nu(\rho')&=\Pi(\rho;\rho')\,Z(\rho).
	\end{split}
\end{equation}

\subsection{Probability distributions from Cauchy identities}
\label{sub:probability_from_Cauchy}

Let $T\ge1,N\ge0$ be integers, and pick a nonnegative specialization
\begin{equation*}
	\rho=(x_1,\ldots,x_N;r_1,\ldots,r_N  )\in \mathsf{Spec}_N
\end{equation*}
and variables
$(w_1,\theta_1),\ldots,(w_T,\theta_T)$
such that each $(w_i,\theta_i)\in \mathsf{Spec}_1$ is also a nonnegative specialization
in the sense of \Cref{def:spec}.
Assume that these specializations are compatible
in the sense of \Cref{def:compatible_spec},
that is, 
\begin{equation}
	\label{eq:Eynard_asc_FG_condition}
	\left|
	\frac{x_i-s_k^{-2}y_k}{x_i-y_k}
	\frac{w_j-y_k}{w_j-s_k^{-2}y_k}
	\right|<1-\delta<1 \qquad \textnormal{for all $i,j$ and sufficiently large $k>0$}.
\end{equation}
\begin{definition}
	\label{def:FG_asc}
	The \emph{ascending FG process} is a probability measure on sequences 
	of signatures $\lambda^{(1)},\ldots,\lambda^{(T)} $
	(each with $N$ parts)
	defined by 
	\begin{equation}
		\label{eq:ascending_process}
		\begin{split}
			&	
			\mathscr{AP}(\lambda^{(1)},\lambda^{(2)},\ldots,\lambda^{(T)})
			\\&\hspace{20pt}=
			\frac1{Z}\,
			G_{\lambda^{(1)}}(w_1;\mathbf{y};\theta_1;\mathbf{s})
			G_{\lambda^{(2)}/\lambda^{(1)}}(w_2;\mathbf{y};\theta_2;\mathbf{s})
			\ldots
			G_{\lambda^{(T)}/\lambda^{(T-1)}}(w_T;\mathbf{y};\theta_T;\mathbf{s})
			F_{\lambda^{(T)}}(\rho),
		\end{split}
	\end{equation}
	where the normalizing constant is equal to 
	\begin{equation}
		\label{eq:Z_ascending_process}
		\begin{split}
			Z
			&=
			Z(\rho)
			\prod_{j=1}^{T}\Pi(\rho;(w_j,\theta_j))
			\\&=
			\prod_{i=1}^{N}x_i(r^{-2}_i-1)\,
			\frac{\prod_{1\le i<j\le N}(r_i^{-2}x_i-x_j)(s_i^{-2}y_i-y_j)}
			{\prod_{i,j=1}^N(y_i-x_j)}
			\prod_{i=1}^N
			\prod_{j=1}^{T}
			\frac{x_i-\theta_j^{-2}w_j}{x_i-w_j}.
		\end{split}
	\end{equation}
\end{definition}
This definition is parallel to a particular case of 
Schur processes \cite{okounkov2003correlation}, 
see also \cite{BorodinCorwin2011Macdonald}.
Later in \Cref{sec:domino_tilings} we connect ascending FG processes to a certain
dimer model.

\begin{remark}
	\label{rmk:no_r_in_ascending}
	From the explicit formula for $F_\lambda$ (\Cref{thm:F_formula})
	it is evident that $F_{\lambda^{(T)}}$ divided by 
	$Z$ \eqref{eq:Z_ascending_process} does not depend on the parameters $r_j$,
	and hence the whole ascending FG process is independent of these parameters, too.
\end{remark}

The marginal distribution 
of each $\lambda^{(j)}$ under the FG process has the following form:

\begin{definition}
	\label{def:FG_measure}
	Let 
	the parameters $(\mathbf{y};\mathbf{s})$
	satisfying \eqref{eq:ys_FG_parameters}
	be fixed.
	Let $M,N\ge1$, and take nonnegative specializations
	$\rho=(\mathbf{x};\mathbf{r})\in\mathsf{Spec}_N$, 
	$\rho'=(\mathbf{w};\boldsymbol\uptheta)\in \mathsf{Spec}_M$
	such that $(\rho,\rho')\in \mathsf{Comp}$.
	The \emph{FG measure} is a probability distribution on 
	signatures $\lambda$ with $N$ parts with probability weights 
	\begin{equation}
		\label{eq:FG_measure}
		\mathscr{M}(\lambda):=\frac{F_\lambda(\rho)G_\lambda(\rho')}
		{Z(\rho)\Pi(\rho;\rho')}.
	\end{equation}
\end{definition}

\Cref{def:FG_measure} is parallel to the definition of the Schur measure
\cite{okounkov2001infinite}. 
Using branching (\Cref{prop:F_G_branching}) and the first relation of 
\eqref{eq:Cauchy_identities_via_specs}, we see that 
for every $j=1,\ldots,T $, the signature
$\lambda^{(j)}$ (with $N$ parts) is distributed as the FG measure
with specializations $\rho=(\mathbf{x};\mathbf{r})\in \mathsf{Spec}_N$ 
and $\rho'=(w_1,\ldots,w_j;\theta_1,\ldots,\theta_j )\in \mathsf{Spec}_j$.

\subsection{Determinantal correlation kernel}
\label{sub:det_process_kernel}

Take a random sequence of signatures 
$\lambda^{(1)},\ldots,\lambda^{(T)} $
distributed according to the ascending FG process \eqref{eq:ascending_process}, and 
associate to it a random point configuration
\begin{equation}
	\label{eq:S_T_notation}
	\mathcal{S}^{(T)}:=\bigcup_{t=1}^{T}
	\bigl(\left\{ t \right\}\times \mathcal{S}(\lambda^{(t)})\bigr)
\end{equation}
in 
$\left\{ 1,\ldots,T  \right\}\times \mathbb{Z}_{\ge1}$,
where we use notation $\mathcal{S}(\lambda)$ from \eqref{eq:S_lambda_notation}.

We are interested in \emph{correlation functions} of the ascending FG process, which are 
defined for
any fixed finite subset 
$A\subset \left\{ 1,\ldots,T  \right\}\times \mathbb{Z}_{\ge1}$
as the probabilities $\mathbb{P}_{\mathscr{AP}}\left[ A\subset \mathcal{S}^{(T)} \right]$.
We show that the ascending FG process is \emph{determinantal}, that is, 
all its correlation functions are determinants of a certain
\emph{correlation kernel} $K_{\mathscr{AP}}(t,a;t',a')$, $1\le t,t'\le T$, $a,a'\ge 1$. 
It has the following form:
\begin{equation}
	\label{eq:ascending_FG_process_kernel_text}
	\begin{split}
		&K_{\mathscr{AP}}(t,a;t',a')
		=
		\frac{1}{(2\pi\mathbf{i})^2}
		\oint_{\Gamma_{y,\theta^{-2}w}}du
		\oint_{\Gamma_{y,w}}
		dv
		\,
		\frac{1}{u-v}
		\prod_{k=1}^{N}\frac{(u-y_k)(v-x_k)}{(u-x_k)(v-y_k)}
		\\
		&\hspace{20pt}
		\times
		\frac{y_{a}(1-s_{a}^{-2})}{v-s_{a}^{-2}y_{a}}
		\frac{1}{u-y_{a'}}
		\prod_{j=1}^{a-1}
		\frac{v-y_j}{v-s_j^{-2}y_j}
		\prod_{j=1}^{a'-1}
		\frac{u-s_j^{-2}y_j}{u-y_j}
		\prod_{d=1}^t \frac{v-\theta_d^{-2}w_d}{v-w_d}
		\prod_{c=1}^{t'}\frac{u-w_c}{u-\theta_c^{-2}w_c},
	\end{split}
\end{equation}
where the integration contours are 
positively oriented circles one inside the other
(the $u$ contour is outside for $t\le t'$ while
the $v$ contour is outside for $t>t'$);
the $u$ contour
encircles all the points $y_i,\theta_j^{-2}w_j$
and not $x_k$;
and
the $v$ contour encircles all the points $y_i,w_j$
and not $s_k^{-2}y_k$.
Observe that $K_{\mathscr{AP}}$ is independent of the $r_j$'s,
which agrees with \Cref{rmk:no_r_in_ascending}.

\begin{theorem}[\Cref{thm:intro_kernel} from Introduction]
	\label{thm:ascending_FG_process_kernel}
	The random point configuration $\mathcal{S}^{(T)}$ 
	constructed from the ascending FG process
	is a determinantal point
	process with the kernel $K_{\mathscr{AP}}$ 
	\eqref{eq:ascending_FG_process_kernel_text}:
	\begin{equation}
		\label{eq:corr_function_ascending_FG_in_text}
		\mathbb{P}_{\mathscr{AP}}\bigl[ A\subset \mathcal{S}^{(T)} \bigr]
		=
		\det\left[ K_{\mathscr{AP}}(t_i,a_i;t_j,a_j) \right]_{i,j=1}^{m}
	\end{equation}
	for any
	$A=\left\{ (t_1,a_1),\ldots,(t_m,a_m)  \right\}
	\subset \left\{ 1,\ldots,T  \right\}\times \mathbb{Z}_{\ge1}$.
\end{theorem}

We prove \Cref{thm:ascending_FG_process_kernel}
using an
Eynard--Mehta type approach
(e.g., see \cite{borodin2005eynard})
which is possible due to determinantal formulas for our symmetric functions
from
\Cref{sec:biorthogonality}. This approach is quite standard and is 
deferred to 
\Cref{appB:Eynard_Mehta}.

Moreover, in \Cref{sec:Fock_and_FG_process}
below we 
discuss
more general 
FG processes (having the structure similar to the general Schur processes of \cite{okounkov2003correlation})
and connect them to 
fermionic operators in the Fock space (developed in \Cref{sec:fermionic_operators}). 
We employ this connection 
to obtain a generating function for the correlation
kernel. In \Cref{sub:ascending_and_Fock}
we check that the Fock space approach leads to the same correlation kernel in the ascending case.

\begin{corollary}
	\label{cor:FG_measure_kernel}
	The FG measure (\Cref{def:FG_measure}) gives
	rise to a determinantal point
	process 
	$\mathcal{S}(\lambda)$ on $\mathbb{Z}_{\ge1}$
	with the correlation kernel
	\begin{equation}
		\label{eq:FG_measure_kernel}
		\begin{split}
			K_{\mathscr{M}}(a,a')
			&
			=
			\frac{1}{(2\pi\mathbf{i})^2}
			\oint_{\Gamma_{y,\theta^{-2}w}}du
			\oint_{\Gamma_{y,w}}
			dv
			\prod_{k=1}^{N}\frac{(u-y_k)(v-x_k)}{(u-x_k)(v-y_k)}
			\prod_{i=1}^M 
			\frac{(u-w_i)(v-\theta_i^{-2}w_i)}{(v-w_i)(u-\theta_i^{-2}w_i)}
			\\
			&\hspace{80pt}
			\times
			\frac{1}{u-v}
			\frac{y_{a}(1-s_{a}^{-2})}{v-s_{a}^{-2}y_{a}}
			\frac{1}{u-y_{a'}}
			\prod_{j=1}^{a-1}
			\frac{v-y_j}{v-s_j^{-2}y_j}
			\prod_{j=1}^{a'-1}
			\frac{u-s_j^{-2}y_j}{u-y_j}
			,
		\end{split}
	\end{equation}
	with the integration contours
	are the same as in \eqref{eq:ascending_FG_process_kernel_text}, 
	and the $u$ contour is outside the $v$ contour.
\end{corollary}

\subsection{Horizontally homogeneous model and Schur measure}
\label{sub:Schur_process_part_case}

In the horizontally homogeneous case 
$s_i=s$, $y_i=1$ for all $i\ge1$,
thanks to \Cref{prop:F_G_homogeneous_through_Schur},
the FG measure reduces to the Schur measure 
\begin{equation}
	\label{eq:compare_FG_measure_and_Schur_measure}
	\mathbb{P}_{\mathrm{Schur}}(\lambda)=\frac{1}{Z}
	s_\lambda\left( \frac{1-s^2x_1}{s^2(1-x_1)},\ldots, \frac{1-s^2x_N}{s^2(1-x_N)} \right)
	s_\lambda
	\left( 
		\biggl\{ \frac{s^2(1-x_j)}{1-s^2 x_j} \biggr\}_{j=1}^M
		\,\middle/\,
		\biggl\{ \frac{s^2(w_j\theta_j^{-2}-1)}{1-s^2\theta_j^{-2} w_j} \biggr\}_{j=1}^M
	\right)
\end{equation}
for a suitable normalization constant $Z$,
and
$\bigl|\frac{1-s^2x_i}{1-x_i}\frac{1-w_j}{1-s^2w_j}\bigr|<1-\delta<1$ for all $i,j$
(this condition follows from \Cref{def:compatible_spec}).
Therefore, by \cite{okounkov2001infinite}, \cite{borodin2005eynard}
its determinantal correlation kernel has a double contour integral form. 
Let us compare that expression with \Cref{cor:FG_measure_kernel}.

First, we recall the correlation kernel of the Schur measure from 
\cite{okounkov2001infinite}, \cite{borodin2005eynard}.
A function that enters the kernel is
\begin{equation*}
	\Phi_{\mathrm{Schur}}(U)
	=
	\underbrace{\prod_{i=1}^{N}\frac{1}{1-\frac{1-s^2x_i}{s^2(1-x_i)}U}}_{H_1(U)}
	\underbrace{\prod_{j=1}^{M}\frac{1-\frac{s^2(1-w_j)}{1-s^2w_j}U^{-1}}
	{1+\frac{s^2(\theta_j^{-2}w_j-1)}{1-s^2\theta_j^{-2}w_j}U^{-1}}}_{1/H_2(U^{-1})}
	,
\end{equation*}
where $H_1,H_2$ are the generating functions associated with the
two
specializations of the Schur functions in
\eqref{eq:compare_FG_measure_and_Schur_measure}.
The correlation kernel is then given by
\begin{equation}
	\label{eq:Schur_measure_kernel}
	\begin{split}
		K_{\mathrm{Schur}}(a,a')=
		\frac{1}{(2\pi\mathbf{i})^2}
		\oint
		\oint
		\frac{dU dV}{U-V}\frac{V^{a'-N-1}}{U^{a-N}}
		\frac
		{\Phi_{\mathrm{Schur}}(U)}
		{\Phi_{\mathrm{Schur}}(V)},\qquad 
		a,a'\in \mathbb{Z}_{\ge1}.
	\end{split}
\end{equation}
The shifts by $N+1$ in $a,a'$ come from 
the fact that our
encoding of the 
particle configurations in $\mathbb{Z}$ is different compared to the 
Schur measures. The integration contours are such that 
$|V| < |U|$ and the Taylor expansions of $H_{1}(U),H_1(V),
H_2(U^{-1}),H_2(V^{-1})$ on the contours are into suitable generating series in 
$U$ and $V$, respectively:
\begin{equation*}
	\left|\frac{1-s^2x_i}{s^2(1-x_i)}\,U\right|<1,\qquad 
	\left|\frac{s^2(\theta_j^{-2}w_j-1)}{1-s^2\theta_j^{-2}w_j}\,U^{-1}\right|<1,\qquad 
	\left|\frac{s^2(1-w_j)}{1-s^2 w_j}\,V^{-1}\right|<1,
\end{equation*}
for all $i,j$.

Let us change the variables in \eqref{eq:Schur_measure_kernel}
as 
\begin{equation*}
	U=\frac{u-1}{u-s^{-2}},\qquad V=\frac{v-1}{v-s^{-2}},
	\qquad 
	\frac{dUdV}{U-V}
	=
	\frac{s^2(s^2-1)}{(1-s^2u)(1-s^2v)}\frac{du dv}{u-v},
\end{equation*}
which yields
\begin{equation*}
	\begin{split}
		K_{\mathrm{Schur}}(a,a')
		&=
		\frac{1}{(2\pi\mathbf{i})^{2}}
		\oint\oint
		\frac{du dv}{u-v}
		\frac{(1-s^{-2})}{(u-s^{-2})(v-1)}
		\left( \frac{u-s^{-2}}{u-1} \right)^{a}
		\left( \frac{v-1}{v-s^{-2}} \right)^{a'}
		\\&
		\hspace{120pt}
		\times
		\left( \frac{u-1}{v-1} \right)^{N}\,
		\prod_{i=1}^{N}
		\frac{v-x_i}{u-x_i}
		\prod_{j=1}^{M}
		\frac
		{u-w_j}
		{v-w_j}
		\frac
		{v-\theta_j^{-2}w_j}
		{u-\theta_j^{-2}w_j},
	\end{split}
\end{equation*}
over the contours such that 
\begin{equation*}
	\begin{split}
		\left|\frac{v-1}{v-s^{-2}}
		\frac{u-s^{-2}}{u-1}\right|
		<1,
		\qquad
		&
		\left|\frac{s^{-2}-x_i}{1-x_i}\frac{u-1}{u-s^{-2}}\right|<1,
		\\
		\left|\frac{\theta_j^{-2}w_j-1}{s^{-2}-\theta_j^{-2}w_j}\frac{u-s^{-2}}{u-1}\right|<1,
		\qquad 
		&
		\left|\frac{1-w_j}{s^{-2}-w_j}\frac{v-s^{-2}}{v-1}\right|<1.
	\end{split}
\end{equation*}
One can check that these conditions hold on the contour $v$ around $1$ and $w$,
and the contour $u$ containing the $v$ contour and also encircling
$w/\theta^2$.
Thus, our \Cref{cor:FG_measure_kernel} 
reduces 
(up to the swap $a\leftrightarrow a'$ which does not affect the determinantal point process)
to the known kernel of the particular Schur measure \eqref{eq:compare_FG_measure_and_Schur_measure}.

\section{Fermionic operators}
\label{sec:fermionic_operators}

In this section
we develop fermionic operators in the Fock space 
which serve as inhomogeneous analogues of the operators
employed in studying Schur measures and processes in 
\cite{okounkov2001infinite}, \cite{okounkov2003correlation}.

\subsection{Simplified commutation relations}
\label{appC:simplified_relations}

Let $V^{(k)}$, $k\in \mathbb{Z}$, be the two-dimensional complex space
with basis $e_0^{(k)}$, $e_1^{(k)}$. We will consider tensor products of the form
\begin{equation}
	\label{eq:tensor_product_space_VMN}
	V^{[M,N]}:=V^{(M)}\otimes V^{(M+1)}\otimes \ldots\otimes V^{(N)},\qquad M\le N.
\end{equation}
As usual, when working with row operators 
(see the beginning of \Cref{sub:row_operators}),
we think that each $V^{(k)}$ carries two parameters 
$(y_k,s_k)$, $k\in \mathbb{Z}$. 
\begin{remark}
	\label{rmk:negative_indices_of_parameters}
	This is the first time when we allow the indices of the parameters $(y_k,s_k)$
	to be nonpositive. However, when applying our computations to actual probability measures,
	the indices of $(y_j,s_j)$ will always satisfy $j\in \mathbb{Z}_{\ge1}$.
\end{remark}

Recall the operators $A,B,C,D$ \eqref{eq:abcdv} acting in each $V^{(k)}$. 
They depend on $x,r$, and also on the parameters $(y_k,s_k)$ 
attached to $V^{(k)}$. We omit the latter in the notation, and write $A=A(x,r)$, and so on.
Via \eqref{eq:abcdv_n2}, these operators also act on any tensor products 
of the form $V^{([M,N])}$.
Our first observation is that with special values of the parameters $(x,r)$, 
the operators $A,B,C,D$ satisfy certain simplified relations:

\begin{proposition}
	\label{prop:simplified_commutation_relations_first_lemma}
	For any $x,z,t\in \mathbb{C}$ we have\footnote{All square roots involved in identities
	in this proposition and throughout the section are always squared in the 
	action of the operators, so we do not need to specify the branches.}
	\begin{equation}
		\label{eq:simplified_commutation_relations_first_lemma_1}
		\begin{split}
			& 
			B (x, t) B (z,\sqrt{z/x}) = 0 = 
			B(z,t)
			B (x,\sqrt{x/z}) 
			;
			\\ & 
			C(x,\sqrt{x/z})C(z,t)
			=0=
			C(z,\sqrt{z/x})C(x,t)
			,
		\end{split}
	\end{equation}
	and
	\begin{equation}
		\label{eq:simplified_commutation_relations_first_lemma_2}
		\begin{split}
			&
			B(z,\sqrt{z/x})
			D(x,\sqrt{x/z})
			+
			D(z,\sqrt{z/x})
			B(x,\sqrt{x/z})=0
			;
			\\
			&
			D(z,\sqrt{z/x})C(x,\sqrt{x/z})=C(z,\sqrt{z/x})D(x,\sqrt{x/z})
			;
			\\ 
			&
			D(x,\sqrt{x/z})A(z,\sqrt{z/x})-C(x,\sqrt{x/z})B(z,\sqrt{z/x})
			\\&\hspace{100pt}=
			D(z,\sqrt{z/x})A(x,\sqrt{x/z})+B(z,\sqrt{z/x})C(x,\sqrt{x/z})
			;
			\\ 
			& 
			A(x,\sqrt{x/z})D(z,\sqrt{z/x})
			-B(x,\sqrt{x/z})C(z,\sqrt{z/x})
			\\&\hspace{100pt}=
			A(z,\sqrt{z/x})D(x,\sqrt{x/z})+C(z,\sqrt{z/x})B(x,\sqrt{x/z})
			.
		\end{split}
	\end{equation}
\end{proposition}
\begin{proof}
	For \eqref{eq:simplified_commutation_relations_first_lemma_1}, we use 
	relations 
	\eqref{eq:B2B1_commute},
	\eqref{eq:C2C1_commute}.
	In particular, 
	to get the first identity in \eqref{eq:simplified_commutation_relations_first_lemma_1},
	take
	$(x_1,r_1)=(z,\sqrt{z/x})$ and $(x_2,r_2)=(x,t)$,
	which implies
	$(t^{-2}x-z)B(x,t)B(z,\sqrt{z/x})=0$.
	All other identities in \eqref{eq:simplified_commutation_relations_first_lemma_1}
	are established in a similar way.

	Let us now turn to \eqref{eq:simplified_commutation_relations_first_lemma_2}.
	For the first identity, use
	\eqref{eq:B2D1},
	\begin{equation*}
		(x_1-x_2)B (x_2, r_2) D (x_1, r_1) = (r^{-2}_1 x_1 -
		x_2)
		D (x_1, r_1) B (x_2, r_2) + 
		x_2 (1 - r^{-2}_2) D (x_2, r_2) B (x_1, r_1)
	\end{equation*}
	with $(x_1,r_1)=(x,\sqrt{x/z})$ and $(x_2,r_2)=(z,\sqrt{z/x})$, which yields
	\begin{equation*}
		(x-z)B (z,\sqrt{z/x}) D (x,\sqrt{x/z})=
		z(1-x/z)D(z,\sqrt{z/x})B(x,\sqrt{x/z}),
	\end{equation*}
	and thus we obtain the first identity from 
	\eqref{eq:simplified_commutation_relations_first_lemma_2}
	The second identity is analogous with the help of \eqref{eq:D2_C1}.
	The last two identities follow in a similar way from
	\eqref{eq:ADBC_1} and \eqref{eq:ADBC_2}, respectively.
\end{proof}

Recall that each subset $\mathcal{T}\subseteq\left\{ M,M+1,\ldots,N  \right\}$
(where $M\le N$) corresponds to a vector $e_{\mathcal{T}}\in V^{[M,N]}$ 
defined as
\begin{equation*}
	e_{\mathcal{T}}=e^{(M)}_{k_M}\otimes e^{(M+1)}_{k_{M+1}}\otimes
	\ldots\otimes e_{k_N}^{(N)},\qquad k_i=k_i(\mathcal{T})=\mathbf{1}_{i\in \mathcal{T}}.
\end{equation*}
Also recall the inner product $\langle \cdot,\cdot \rangle $ on tensor 
powers of $\mathbb{C}^2$ such as $V^{[M,N]}$, under which the vectors of the form
$e_{\mathcal{T}}$ are orthonormal.

\Cref{prop:fermionic_first_correlation,prop:normalized_Phi_action,prop:normalized_Phi_star_action}
below show that matrix elements 
of
$D(x,\sqrt{x/z})B(z,\sqrt{z/x})$
and 
$D(x,\sqrt{x/z})C(z,\sqrt{z/x})$
can be used
to detect if two subsets of $\left\{ M,\ldots,N  \right\}$
are different by a single element. 
This is summarized in 
\Cref{thm:Psi_PsiStar_Fock} below.

\begin{proposition}
	\label{prop:fermionic_first_correlation}
	Fix nonzero $x, z\in \mathbb{C}$; integers $m \ge
	0$ and $M \le N$; and two integer sets
	\begin{equation*}
		\mathcal{R}=(r_1<r_2<\ldots<r_m ),
		\quad 
		\mathcal{T}=(t_1<t_2<\ldots<t_m<t_{m+1} ),
		\quad 
		\mathcal{R},\mathcal{T}\subset\left\{ M,M+1,\ldots,N  \right\}.
	\end{equation*}
	If $\mathcal{R}$ is not a subset of $\mathcal{T}$, then 
	\begin{flalign}
	\label{eq:RT_subset_zero}
		\bigl\langle e_{\mathcal{R}}, D ( x, \sqrt{x/z}) B ( z, \sqrt{z/x} )
		e_{\mathcal{T}} \bigr\rangle = 0 = \bigl\langle e_{\mathcal{T}}, D ( x,
		\sqrt{x/z} ) C ( z, \sqrt{z/x} ) e_{\mathcal{R}} \bigr\rangle.
	\end{flalign}
\end{proposition}
\begin{proof}
	We only establish the first equality in \eqref{eq:RT_subset_zero},
	the second equality follows in a similar manner.
	For convenience of notation, we set $r_{m+1}=t_{m+2}=+\infty$ throughout the proof.

	Since $\mathcal{R}$ is not a subset of $\mathcal{T}$, there exists an index 
	$1\le n\le m$ such that either $t_n<r_n<t_{n+1}$, 
	or $t_{n+1}<r_n<t_{n+2}$. Fix such $n$.
	Define sets $\mathcal{R}'=\mathcal{R}\cap (-\infty,r_n-1]$
	and $\mathcal{R}''=\mathcal{R}\cap [r_n,+\infty)$,
	and similarly $\mathcal{T}',\mathcal{T}''$.

	First, assume that $t_{n+1}<r_n<t_{n+2}$. Then
	since $|\mathcal{R}'| = n-1 = |\mathcal{T}'| - 2$, 
	we have 
	using $V^{[M,N]}=V^{[M,r_n-1]}\otimes V^{[r_n,N]}$
	and
	\eqref{eq:abcdv_n2}, picking 
	$B$ twice for the left tensor product:
	\begin{multline*}
		\bigl\langle 
		e_{\mathcal{R}},
		D(x,\sqrt{x/z})B(z,\sqrt{z/x})
		e_{\mathcal{T}}
		\bigr\rangle 
		\\=
		\bigl\langle 
		e_{\mathcal{R}'},
		B(x,\sqrt{x/z})B(z,\sqrt{z/x})
		e_{\mathcal{T}'}
		\bigr\rangle 
		\bigl\langle 
		e_{\mathcal{R}''},
		C(x,\sqrt{x/z})A(z,\sqrt{z/x})
		e_{\mathcal{T}''}
		\bigr\rangle .
	\end{multline*}
	The latter expression vanishes by 
	identity \eqref{eq:simplified_commutation_relations_first_lemma_1}
	from \Cref{prop:simplified_commutation_relations_first_lemma}.

	Now let us assume that $t_n<r_n<t_{n+1}$. Define sets
	$\mathcal{R}_0'=\mathcal{R}\cap(-\infty,r_n]$, 
	$\mathcal{R}_0''=\mathcal{R}\cap [r_n+1,+\infty)$,
	and similarly $\mathcal{T}_0',\mathcal{T}_0''$. We have
	$|\mathcal{R}_0'| =|\mathcal{T}_0'| =n$, so
	with
	$V^{[M,N]}=V^{[M,r_n]}\otimes V^{[r_n+1,N]}$,
	in the expansion
	\eqref{eq:abcdv_n2} we need to take the $D$ operator twice. We have
	\begin{multline*}
		\bigl\langle 
		e_{\mathcal{R}}, 
		D ( x, \sqrt{x/z}) B ( z, \sqrt{z/x} )
		e_{\mathcal{T}} 
		\bigr\rangle
		\\=
		\bigl\langle 
		e_{\mathcal{R}_0'}, 
		D ( x, \sqrt{x/z}) D ( z, \sqrt{z/x} )
		e_{\mathcal{T}_0'} 
		\bigr\rangle
		\bigl\langle 
		e_{\mathcal{R}_0''}, 
		D ( x, \sqrt{x/z}) B ( z, \sqrt{z/x} )
		e_{\mathcal{T}_0''}
		\bigr\rangle.
	\end{multline*}
In the first factor we apply \eqref{eq:abcdv_n2} 
with
$V^{[M,r_n]}=V^{[M,r_n-1]}\otimes V^{(r_n)}$.
Observe that $|\mathcal{R}'| = n-1 = |\mathcal{T}'| -1$, so we obtain
\begin{multline*}
	\bigl\langle 
	e_{\mathcal{R}_0'}, 
	D ( x, \sqrt{x/z}) D ( z, \sqrt{z/x} )
	e_{\mathcal{T}_0'} 
	\bigr\rangle
	\\=
	\bigl\langle 
	e_{\mathcal{R}'}, 
	D ( x, \sqrt{x/z}) B ( z, \sqrt{z/x} )
	e_{\mathcal{T}'} 
	\bigr\rangle
	\bigl\langle 
	e_1^{(r_n)},
	D ( x, \sqrt{x/z}) C ( z, \sqrt{z/x} )
	e_0^{(r_n)}
	\bigr\rangle
	\\+
	\bigl\langle 
	e_{\mathcal{R}'}, 
	B ( x, \sqrt{x/z}) D ( z, \sqrt{z/x} )
	e_{\mathcal{T}'} 
	\bigr\rangle
	\bigl\langle 
	e_1^{(r_n)},
	C ( x, \sqrt{x/z}) D ( z, \sqrt{z/x} )
	e_0^{(r_n)}
	\bigr\rangle.
\end{multline*}
This expression vanishes thanks to the first two identities in 
\eqref{eq:simplified_commutation_relations_first_lemma_2}.
\end{proof}

\subsection{Normalized operators}
\label{appC:normalized_ops}

Let us now introduce normalizations of our operators
$A,B,C,D$, which allow to 
take the limit as $M\to -\infty$, $N\to+\infty$
without running into infinite products:

\begin{definition}
	\label{def:normalized_ops}
	Fix $x, r \in \mathbb{C}$. For $M \le 0\le N$, define
	the normalized operators 
	$A^{[M, N]} (x, r)$, $B^{[M, N]} (x, r)$,
	$C^{[M, N]} (x, r)$, and $D^{[M, N]} (x, r)$ acting on $V^{[M, N]}$ by
	\begin{align*}
		A^{[M,N]}(x,r)
		&=
		\frac{A(x,r)}{\prod_{i=M}^0 W_i(1,1;1,1)\prod_{j=1}^N W_j(0,1;0,1)}
		=
		A(x,r)
		\prod_{i=M}^0
		\frac{r^2(y_i-s_i^2x)}{s_i^2(x-r^2 y_i)}
		\prod_{j=1}^N
		\frac{y_j-s_j^2x}{s_j^2(y_j-x)}
		;
		\\
		B^{[M,N]}(x,r)
		&=
		\frac{B(x,r)}{\prod_{i=M}^0 W_i(1,0;1,0)\prod_{j=1}^N W_j(0,1;0,1)}
		=
		B(x,r)
		\prod_{i=M}^0
		\frac{y_i-s_i^2x}{y_i-s_i^2r^{-2}x}
		\prod_{j=1}^N
		\frac{y_j-s_j^2x}{s_j^2(y_j-x)}
		;
		\\
		C^{[M,N]}(x,r)
		&=
		\frac{C(x,r)}{\prod_{i=M}^0 W_i(1,1;1,1)}
		=
		C(x,r)\prod_{i=M}^0
		\frac{r^2(y_i-s_i^2x)}{s_i^2(x-r^2 y_i)}
		;
		\\
		D^{[M,N]}(x,r)
		&=
		\frac{D(x,r)}{\prod_{i=M}^0 W_i(1,0;1,0)}
		=
		D(x,r)
		\prod_{i=M}^0
		\frac{y_i-s_i^2x}{y_i-s_i^2r^{-2}x}
		.
	\end{align*}
	Here $W_j$ are vertex weights \eqref{eq:weights_W}
	with the parameters 
	$W_j(\cdots)=W(\cdots\mid x;y_j;r;s_j)$.
  %
\end{definition}

\begin{definition}
	\label{def:Phi_PhiStar}
	Let us define expressions $\Phi_j,\Phi_j^*$ for $j\in \mathbb{Z}$ as 
	follows:
	\begin{align*}
		\Phi_j(x,z)&=
		\begin{cases}
			\displaystyle
			\frac{y_j (1-s_j^{-2})}{x-s_j^{-2}y_j} 
			\prod_{k = 1}^{j - 1} 
			\frac{x-y_k}{x-s_k^{-2}y_k},&j>0;\\
			\displaystyle
			\frac{y_j (1-s_j^{-2})}{x-s_j^{-2}y_j} 
			\prod_{k = j}^0 
			\frac{x-s_k^{-2}y_k}{x-y_k},&j\le 0,
		\end{cases}
		\\
		\Phi^*_j(x,z)&=
		\begin{cases}
			\displaystyle
			\frac{z-x}{z-y_j} 
			\prod_{k = 1}^{j - 1} 
			\frac{z-s_k^{-2}y_k}{z-y_k}
			,&j>0
			;\\
			\displaystyle
			\frac{z-x}{z-y_j} 
			\prod_{k = j}^0 
			\frac{z-y_k}{z-s_k^{-2}y_k}
			,&j\le 0.
		\end{cases}
	\end{align*}
\end{definition}
Note that while $\Phi_j(x,z)$ does not depend on $z$,
it is convenient to keep the
$\Phi_j,\Phi_j^*$
notation uniform.

\begin{proposition}
	\label{prop:normalized_Phi_star_action}
	Fix nonzero $x, z\in \mathbb{C}$; integers $m \ge
	0$ and $M \le N$; and two integer sets
	\begin{equation*}
		\mathcal{R}=(r_1<r_2<\ldots<r_m ),
		\quad 
		\mathcal{T}=(t_1<t_2<\ldots<t_m<t_{m+1} ),
		\quad 
		\mathcal{R},\mathcal{T}\subset\left\{ M,M+1,\ldots,N  \right\}.
	\end{equation*}
	If for some $j\in \left\{ 1,\ldots,m+1  \right\}$ we have
	$\mathcal{R}=\mathcal{T}\setminus\left\{ t_j \right\}$, then
	\begin{equation}
		\label{eq:normalized_Phi_star_action}
		\bigl\langle  
		e_{\mathcal{R}}, 
		D^{[M, N]} ( x, \sqrt{x/z}) 
		B^{[M, N]} ( z, \sqrt{z/x}) 
		e_{\mathcal{T}} 
		\bigr\rangle 
		=
		(-1)^{m - j + 1} \Phi_{t_j}^* (x,z).
	\end{equation}
\end{proposition}
\begin{proof}
	Observe that by \Cref{def:normalized_ops},
	\begin{equation}
		\label{eq:normalized_Phi_star_action_proof_norm}
		D^{[M, N]} ( x, \sqrt{x/z}) 
		B^{[M, N]} ( z, \sqrt{z/x}) 
		=
		D( x, \sqrt{x/z}) 
		B( z, \sqrt{z/x}) 
		\prod_{j=1}^{N}\frac{y_j-s_j^2 z}{s_j^2(y_j-z)}.
	\end{equation}
	Therefore, it suffices to evaluate $D(x,\sqrt{x/z})B(z,\sqrt{z/x})$. 
	In the action of this operator in 
	the tensor product of the spaces $V^{(k)}$, whenever we see 
	$D(x,\sqrt{x/z})D(z,\sqrt{z/x})$, we have
	\begin{equation}
		\label{eq:normalized_Phi_star_action_DD_action}
		\bigl\langle
		e_1^{(k)},
		D(x,\sqrt{x/z})D(z,\sqrt{z/x})
		e_1^{(k)}
		\bigr\rangle 
		=
		\bigl\langle
		e_0^{(k)},
		D(x,\sqrt{x/z})D(z,\sqrt{z/x})
		e_0^{(k)}
		\bigr\rangle 
		=1.
	\end{equation}
	This means that nontrivial
	contributions to the left-hand side of \eqref{eq:normalized_Phi_star_action}
	can only 
	come from the configuration to the right of $t_j$.
	Without loss of the generality, we may assume that $j=1$,
	and $r_i=t_{i+1}$ for $i=1,\ldots,m $.

	For any $k$, define $\mathcal{R}_k=\mathcal{R}\cap[M,k]$, and similarly for
	$\mathcal{T}_k$.
	If $k=t_i$ for some $i=2,\ldots,m+1 $, we have by \eqref{eq:abcdv_n2}:
	\begin{equation}
		\label{eq:normalized_Phi_star_action_proof1}
		\begin{split}
			&
			\bigl\langle e_{\mathcal{R}_k},
			D(x,\sqrt{x/z})B(z,\sqrt{z/x})
			e_{\mathcal{T}_k} \bigr\rangle 
			\\
			&\hspace{20pt}=
			\bigl\langle 
			e_{\mathcal{R}_{k-1}},
			D(x,\sqrt{x/z})B(z,\sqrt{z/x})
			e_{\mathcal{T}_{k-1}}
			\bigr\rangle
			\bigl\langle 
			e_1^{(k)},
			D(x,\sqrt{x/z})A(z,\sqrt{z/x})
			e_1^{(k)}
			\bigr\rangle 
			\\
			&\hspace{40pt}+
			\bigl\langle 
			e_{\mathcal{R}_{k-1}},
			B(x,\sqrt{x/z})D(z,\sqrt{z/x})
			e_{\mathcal{T}_{k-1}}
			\bigr\rangle
			\bigl\langle 
			e_1^{(k)},
			C(x,\sqrt{x/z})B(z,\sqrt{z/x})
			e_1^{(k)}
			\bigr\rangle.
		\end{split}
	\end{equation}
	The first identity in \eqref{eq:simplified_commutation_relations_first_lemma_2}
	states that the $D$ and $B$ operators can be swapped, 
	producing a negative sign. 
	Applying this to the second summand in the right-hand side of 
	\eqref{eq:normalized_Phi_star_action_proof1},
	we see that
	\begin{equation*}
		\begin{split}
			\eqref{eq:normalized_Phi_star_action_proof1}
			&=
			\bigl\langle 
			e_{\mathcal{R}_{k-1}},
			D(x,\sqrt{x/z})B(z,\sqrt{z/x})
			e_{\mathcal{T}_{k-1}}
			\bigr\rangle
			\\&\hspace{80pt}\times
			\bigl\langle 
			e_1^{(k)},
			\left( 
				D(x,\sqrt{x/z})A(z,\sqrt{z/x})-
				C(x,\sqrt{x/z})B(z,\sqrt{z/x})
			\right)
			e_1^{(k)}
			\bigr\rangle 
			\\&=
			\bigl\langle 
			e_{\mathcal{R}_{k-1}},
			D(x,\sqrt{x/z})B(z,\sqrt{z/x})
			e_{\mathcal{T}_{k-1}}
			\bigr\rangle
			\\&\hspace{80pt}\times
			\bigl\langle 
			e_1^{(k)},
			\left( 
				D(z,\sqrt{z/x})A(x,\sqrt{x/z})+
				B(z,\sqrt{z/x})C(x,\sqrt{x/z})
			\right)
			e_1^{(k)}
			\bigr\rangle,
		\end{split}
	\end{equation*}
	where in the second equality we applied the third identity
	in \eqref{eq:simplified_commutation_relations_first_lemma_2}. 
	Now observe that the operator $C$ maps $e_1^{(k)}$ to $0$,
	and we can continue
	(evaluating the eigenaction of
	$D(z,\sqrt{z/x})A(x,\sqrt{x/z})$ on the vector $e_1^{(k)}$):
	\begin{equation*}
		\eqref{eq:normalized_Phi_star_action_proof1}=
		\frac{s_k^2(z-y_k)}{y_k-s_k^2 z}\,
			\bigl\langle 
			e_{\mathcal{R}_{k-1}},
			D(x,\sqrt{x/z})B(z,\sqrt{z/x})
			e_{\mathcal{T}_{k-1}}
			\bigr\rangle.
	\end{equation*}

	Now let us compute the same quantity if 
	$k\notin \mathcal{T}$. Then we have by \eqref{eq:abcdv_n2}:
	\begin{equation*}
		\begin{split}
			&
			\bigl\langle e_{\mathcal{R}_k},
			D(x,\sqrt{x/z})B(z,\sqrt{z/x})
			e_{\mathcal{T}_k} \bigr\rangle 
			\\
			&\hspace{20pt}=
			\bigl\langle 
			e_{\mathcal{R}_{k-1}},
			D(x,\sqrt{x/z})B(z,\sqrt{z/x})
			e_{\mathcal{T}_{k-1}}
			\bigr\rangle
			\bigl\langle 
			e_0^{(k)},
			D(x,\sqrt{x/z})A(z,\sqrt{z/x})
			e_0^{(k)}
			\bigr\rangle 
			\\
			&\hspace{20pt}=
			\frac{s_k^2(y_k-z)}{y_k-s_k^2z}\,
			\bigl\langle 
			e_{\mathcal{R}_{k-1}},
			D(x,\sqrt{x/z})B(z,\sqrt{z/x})
			e_{\mathcal{T}_{k-1}}
			\bigr\rangle.
		\end{split}
	\end{equation*}

	We can now compute the action of 
	$D(x,\sqrt{x/z})B(z,\sqrt{z/x})$
	by successively splitting off the tensor factors,
	each time we obtain the factor $\pm \dfrac{s_k^2(y_k-z)}{y_k-s_k^2z}$.
	The overall number of negative signs is $(-1)^{m}$, which translates
	to $(-1)^{m-j+1}$ when dropping the assumption $j=1$.
	In the last step of the splitting, at $k=t_j$, we have
	\begin{equation*}
		\bigl\langle e_0^{(k)},
		D(x,\sqrt{x/z})B(z,\sqrt{z/x})
		e_1^{(k)}
		\bigr\rangle =
		\frac{s_k^2(x-z)}{y_k-s_k^2z}.
	\end{equation*}
	Recalling normalization 
	\eqref{eq:normalized_Phi_star_action_proof_norm}, 
	we get the desired identity.
\end{proof}

\begin{proposition}
	\label{prop:normalized_Phi_action}
	Fix nonzero $x, z\in \mathbb{C}$; integers $m \ge
	0$ and $M \le N$; and two integer sets
	\begin{equation*}
		\mathcal{R}=(r_1<r_2<\ldots<r_m ),
		\quad 
		\mathcal{T}=(t_1<t_2<\ldots<t_m<t_{m+1} ),
		\quad 
		\mathcal{R},\mathcal{T}\subset\left\{ M,M+1,\ldots,N  \right\}.
	\end{equation*}
	If for some $j\in \left\{ 1,\ldots,m+1  \right\}$ we have
	$\mathcal{R}=\mathcal{T}\setminus\left\{ t_j \right\}$, then
	\begin{equation}
		\label{eq:normalized_Phi_action}
		\bigl\langle 
		e_{\mathcal{T}},
		D^{[M,N]}(x,\sqrt{x/z})C^{[M,N]}(z,\sqrt{z/x})
		e_{\mathcal{R}}
		\bigr\rangle
		=
		(-1)^{M-j}
		\Phi_{t_j}(x,z).
	\end{equation}
\end{proposition}
\begin{proof}
	The proof follows along the same lines as the proof of the previous
	\Cref{prop:normalized_Phi_star_action}.
	First we observe that (cf. 
	\Cref{def:normalized_ops})
	\begin{equation}
		\label{eq:normalized_Phi_action_proof0}
		D^{[M,N]}(x,\sqrt{x/z})C^{[M,N]}(z,\sqrt{z/x})
		=
		(-1)^{M+1}
		D(x,\sqrt{x/z})C(z,\sqrt{z/x})
		\prod_{i=M}^{0}\frac{y_i-s_i^2x}{s_i^2(y_i-x)}.
	\end{equation}
	Thus, it suffices to evaluate
	$D(x,\sqrt{x/z})C(z,\sqrt{z/x})$.
	In the action of this operator in 
	the tensor product of the spaces $V^{(k)}$, whenever we see 
	$D(x,\sqrt{x/z})D(z,\sqrt{z/x})$, we may use
	\eqref{eq:normalized_Phi_star_action_DD_action}.
	This means that nontrivial
	contributions to the left-hand side of 
	\eqref{eq:normalized_Phi_action} 
	can only come from the 
	configuration to the left of $t_j$. 
	Without loss of the generality, we may assume that $j=m+1$, 
	and $r_i=t_i$ for $i=1,\ldots,m $.

	For any $k$, define $\mathcal{R}_k=[k,N]\cap \mathcal{R}$, and similarly 
	for $\mathcal{T}_k$.
	First, assuming that $k=t_i$ for some $i=1,\ldots,m $,
	we have by 
	\eqref{eq:abcdv_n2} and
	the second identity 
	in \eqref{eq:simplified_commutation_relations_first_lemma_2}:
	\begin{equation}
		\label{eq:normalized_Phi_action_proof1}
		\begin{split}
			&
			\bigl\langle e_{\mathcal{T}_k},
			D(x,\sqrt{x/z})C(z,\sqrt{z/x})e_{\mathcal{R}_k}\bigr\rangle 
			=
			\bigl\langle e_{\mathcal{T}_k},
			C(x,\sqrt{x/z})D(z,\sqrt{z/x})e_{\mathcal{R}_k}\bigr\rangle 
			\\&
			\hspace{10pt}
			=
			\bigl\langle e_{1}^{(k)},
			C(x,\sqrt{x/z})B(z,\sqrt{z/x})e_{1}^{(k)}\bigr\rangle 
			\bigl\langle e_{\mathcal{T}_{k+1}},
			D(x,\sqrt{x/z})C(z,\sqrt{z/x})e_{\mathcal{R}_{k+1}}\bigr\rangle 
			\\&
			\hspace{40pt}
			+
			\bigl\langle e_{1}^{(k)},
			A(x,\sqrt{x/z})D(z,\sqrt{z/x})e_{1}^{(k)}\bigr\rangle 
			\bigl\langle e_{\mathcal{T}_{k+1}},
			C(x,\sqrt{x/z})D(z,\sqrt{z/x})e_{\mathcal{R}_{k+1}}\bigr\rangle.
		\end{split}
	\end{equation}
	We next have by \eqref{eq:simplified_commutation_relations_first_lemma_2}:
	\begin{equation*}
		\begin{split}
			\eqref{eq:normalized_Phi_action_proof1}&=
			\bigl\langle e_{\mathcal{T}_{k+1}},
			C(x,\sqrt{x/z})D(z,\sqrt{z/x})e_{\mathcal{R}_{k+1}}\bigr\rangle
			\\&
			\hspace{40pt}\times
			\bigl\langle e_1^{(k)},
			\left( A(x,\sqrt{x/z})D(z,\sqrt{z/x})+C(x,\sqrt{x/z})B(z,\sqrt{z/x}) \right)
			e_1^{(k)}
			\bigr\rangle 
			\\&
			=
			\bigl\langle e_{\mathcal{T}_{k+1}},
			C(x,\sqrt{x/z})D(z,\sqrt{z/x})e_{\mathcal{R}_{k+1}}\bigr\rangle
			\\&
			\hspace{40pt}\times
			\bigl\langle e_1^{(k)},
			\left( A(z,\sqrt{z/x})D(x,\sqrt{x/z})-B(z,\sqrt{z/x})C(x,\sqrt{x/z}) \right)
			e_1^{(k)}
			\bigr\rangle 
			\\&=
			\bigl\langle e_{\mathcal{T}_{k+1}},
			C(x,\sqrt{x/z})D(z,\sqrt{z/x})e_{\mathcal{R}_{k+1}}\bigr\rangle
			\bigl\langle e_1^{(k)},
			A(z,\sqrt{z/x})D(x,\sqrt{x/z})
			e_1^{(k)}
			\bigr\rangle
			\\&=
			\frac{s_k^2(x-y_k)}{y_k-s_k^2x}\,
			\bigl\langle e_{\mathcal{T}_{k+1}},
			D(x,\sqrt{x/z})C(z,\sqrt{z/x})e_{\mathcal{R}_{k+1}}\bigr\rangle
			.
		\end{split}
	\end{equation*}
	Here we used the fact that the $C$ operator maps $e_1^{(k)}$ to $0$,
	and for the last equality we evaluated the eigenaction on
	$e_1^{(k)}$.

	Now assume that $k\notin \mathcal{T}$. Then we have
	\begin{equation*}
		\begin{split}
			&\bigl\langle e_{\mathcal{T}_k},
			D(x,\sqrt{x/z})C(z,\sqrt{z/x})e_{\mathcal{R}_k}\bigr\rangle
			=
			\bigl\langle e_{\mathcal{T}_k},
			C(x,\sqrt{x/z})D(z,\sqrt{z/x})e_{\mathcal{R}_k}\bigr\rangle
			\\&\hspace{20pt}=
			\bigl\langle e_0^{(k)},
			A(x,\sqrt{x/z})D(z,\sqrt{z/x})
			e_0^{(k)}
			\bigr\rangle 
			\bigl\langle e_{\mathcal{T}_{k+1}},
			C(x,\sqrt{x/z})D(z,\sqrt{z/x})e_{\mathcal{R}_{k+1}}\bigr\rangle
			\\&\hspace{20pt}=
			\frac{s_k^2(y_k-x)}{y_k-s_k^2x}\,
			\bigl\langle e_{\mathcal{T}_{k+1}},
			D(x,\sqrt{x/z})C(z,\sqrt{z/x})e_{\mathcal{R}_{k+1}}\bigr\rangle.
		\end{split}
	\end{equation*}
	We can now compute the action of $D(x,\sqrt{x/z})C(z,\sqrt{z/x})$
	by successively splitting off the tensor factors. Each time we 
	obtain the factor $\pm\dfrac{s_k^2(y_k-x)}{y_k-s_k^2x}$,
	and the overall number of negative signs is $(-1)^{m}$, which translates into
	$(-1)^{j-1}$ upon passing to the general case not assuming $j=m+1$.
	In the last splitting, at $k=t_{m+1}$, we have
	\begin{equation*}
		\bigl\langle e_1^{(k)},
		D(x,\sqrt{x/z})C(z,\sqrt{z/x})
		e_0^{(k)}
		\bigr\rangle 
		=\frac{y_k(1-s_k^2)}{y_k-s_k^2x}.
	\end{equation*}
	Recalling normalization 
	\eqref{eq:normalized_Phi_action_proof0},
	we get the desired identity.
\end{proof}

\subsection{Fermionic operators in the Fock space}
\label{appC:fermionic_operators}

We now pass to the infinite volume limit
as $M\to-\infty$ and $N\to+\infty$.
As a result, from the spaces $V^{[M,N]}$ we get the \emph{Fock space} $\mathscr{F}$.
By definition, $\mathscr{F}$ is spanned by the vectors $e_{\mathcal{T}}$,
\begin{equation*}
	e_{\mathcal{T}}=
	\bigotimes_{m=-\infty}^{+\infty}
	e_{k_m}^{(m)},
	\qquad k_m=k_m(\mathcal{T})=\mathbf{1}_{m\in \mathcal{T}},
\end{equation*}
where $\mathcal{T}$ runs over \emph{semi-infinite} subsets of $\mathbb{Z}$.
A subset $\mathcal{T}$ is called semi-infinite (also sometimes
referred to as densely packed towards $-\infty$)
if there exists $C=C(\mathcal{T})>0$
such that $i\notin \mathcal{T}$ for all $i>C$ and $i\in \mathcal{T}$ for all $i< -C$.

In $\mathscr{F}$, 
we can define an
inner product 
under which the $e_{\mathcal{T}}$'s form an orthonormal basis:
\begin{equation*}
	\langle e_{\mathcal{T}},e_{\mathcal{R}} \rangle =\mathbf{1}_{\mathcal{T}=\mathcal{R}}.
\end{equation*}
We do not need to consider convergence in $\mathscr{F}$ as all 
computations below are done in terms of this inner product. For example,
$\langle v,e_{\mathcal{T}} \rangle $
may be viewed as an operation of picking a 
coefficient of $e_{\mathcal{T}}$
in $v$, a formal infinite linear combination of the $e_{\mathcal{R}}$'s.

For a semi-infinite subset $\mathcal{T}$, define
\begin{equation}
	\label{eq:charge_height}
	c(\mathcal{T}):=\#(\mathcal{T}\cap \mathbb{Z}_{>0})-
	\#(\mathbb{Z}_{\le 0}\setminus \mathcal{T}),\qquad 
	h_j(\mathcal{T}):=
	\#\left\{ t\in \mathcal{T}\colon t>j \right\}.
\end{equation}
The quantity $c(\mathcal{T})$ is called \emph{charge}.
Define the change operator $\mathfrak{c}\colon \mathscr{F}\to\mathscr{F}$ on 
the basis by 
$\mathfrak{c}(e_{\mathcal{T}})=c(\mathcal{T})e_{\mathcal{T}}$, and then extend by linearity.

Clearly, $c(\mathcal{T})$ can be any integer, and we have the decomposition of 
$\mathscr{F}$ into subspaces with fixed charge:
\begin{equation*}
	\mathscr{F}=\bigoplus_{n\in \mathbb{Z}}\mathscr{F}_n,\qquad 
	\mathscr{F}_n=\mathop{\mathrm{span}}
	\left\{ e_{\mathcal{T}}\colon c(\mathcal{T})=n \right\}.
\end{equation*}

The normalized operators $A^{[M,N]}, B^{[M,N]}, C^{[M,N]}$, and $D^{[M,N]}$
from \Cref{def:normalized_ops} admit matrix element-wise
infinite volume limits as
$M\to-\infty$, $N\to+\infty$.
We denote the limiting operators by 
$A^{\mathbb{Z}}(x,r), B^{\mathbb{Z}}(x,r), C^{\mathbb{Z}}(x,r), D^{\mathbb{Z}}(x,r)$.
These operators act in the Fock space $\mathscr{F}$, more precisely,
\begin{equation*}
	A^{\mathbb{Z}}\colon \mathscr{F}_n\to \mathscr{F}_n,\qquad 
	B^{\mathbb{Z}}\colon \mathscr{F}_n\to \mathscr{F}_{n-1},\qquad 
	C^{\mathbb{Z}}\colon \mathscr{F}_n\to \mathscr{F}_{n+1},\qquad 
	D^{\mathbb{Z}}\colon \mathscr{F}_n\to \mathscr{F}_n.
\end{equation*}
With this understanding, it is clear that the
matrix elements like $\langle e_{\mathcal{T}},A^{\mathbb{Z}}e_{\mathcal{R}} \rangle$,
and so on, are well-defined for all possible values of 
the parameters $(x;\mathbf{y};r;\mathbf{s})$, simply as products 
(=~suitable partition functions)
of 
normalized weights as in \Cref{def:normalized_ops}, only finitely many of which 
differ from $1$.
See also \Cref{sub:intro_fermionic_operators} in Introduction
for a 
pictorial definition of these operators.

%
The operators $A^{\mathbb{Z}}, B^{\mathbb{Z}}, C^{\mathbb{Z}}, D^{\mathbb{Z}}$
satisfy a number of commutation relations:

\begin{proposition}
	\label{prop:ABCD_infinite_volume}
	We have
	\begin{align}
		\label{eq:BB_infty}
		B^{\mathbb{Z}}(x_1,r_1)B^{\mathbb{Z}}(x_2,r_2)&=
		\frac{r_2^{-2}x_2-x_1}{r_1^{-2}x_1-x_2}\,
		B^{\mathbb{Z}}(x_2,r_2)
		B^{\mathbb{Z}}(x_1,r_1);
		\\
		\label{eq:DD_infty}
		D^{\mathbb{Z}}(x_1,r_1)
		D^{\mathbb{Z}}(x_2,r_2)
		&=
		D^{\mathbb{Z}}(x_2,r_2)
		D^{\mathbb{Z}}(x_1,r_1)
		;
		\\
		\label{eq:BD_infty}
		B^{\mathbb{Z}}(x_1,r_1)D^{\mathbb{Z}}(x_2,r_2)&=
		\frac{r_2^{-2}x_2-x_1}{x_2-x_1}\,
		D^{\mathbb{Z}}(x_2,r_2)B^{\mathbb{Z}}(x_1,r_1)
		;
		\\
		\label{eq:CD_infty}
		C^{\mathbb{Z}}(x_1,r_1)D^{\mathbb{Z}}(x_2,r_2)&=
		\frac{r_1^{-2}x_1-x_2}{r_1^{-2}x_1-r_2^{-2}x_2}\,
		D^{\mathbb{Z}}(x_2,r_2)C^{\mathbb{Z}}(x_1,r_1)
		;
		\\
		\label{eq:BC_infty}
		B^{\mathbb{Z}}(x_1,r_1)C^{\mathbb{Z}}(x_2,r_2)&=
		\frac{r_2^{-2}x_2-r_1^{-2}x_1}{x_2-x_1}\,
		C^{\mathbb{Z}}(x_2,r_2)B^{\mathbb{Z}}(x_1,r_1)
		.
	\end{align}
	All identities are understood in the sense of matrix elements, for example,
	for \eqref{eq:BB_infty} we have
	\begin{equation*}
		\langle 
		e_{\mathcal{T}}, B^{\mathbb{Z}}(x_1,r_1)B^{\mathbb{Z}}(x_2,r_2)
		e_{\mathcal{R}}\rangle 
		=
		\frac{r_2^{-2}x_2-x_1}{r_1^{-2}x_1-x_2}\,
		\langle 
		e_{\mathcal{T}},
		B^{\mathbb{Z}}(x_2,r_2)
		B^{\mathbb{Z}}(x_1,r_1)
		e_{\mathcal{R}}\rangle,\quad 
		e_{\mathcal{R}}\in \mathscr{F}_{n},\ e_{\mathcal{T}}\in \mathscr{F}_{n-2}.
	\end{equation*}
	Identities \eqref{eq:BB_infty}, \eqref{eq:DD_infty} hold for arbitrary values of the 
	parameters, but the other ones require the following restrictions.
	For \eqref{eq:BD_infty}, we assume
	\begin{equation}
		\label{eq:oplus_condition}
		\oplus_{x_1;x_2}\colon \qquad 
		\left|\frac{(s_j^{-2}y_j-x_1)(y_j-x_2)}{(s_j^{-2}y_j-x_2)(y_j-x_1)}\right|<1-\delta<1
		\qquad \textnormal{for sufficiently large $j>0$}.
	\end{equation}
	For \eqref{eq:CD_infty}, we assume
	\begin{equation}
		\label{eq:ominus_condition}
		\ominus_{r_1^{-2}x_1;r_2^{-2}x_2}\colon \qquad 
		\left|
		\frac{(s_j^{-2}y_j-r_1^{-2}x_1)(y_j-r_2^{-2}x_2)}
		{(s_j^{-2}y_j-r_2^{-2}x_2)(y_j-r_1^{-2}x_1)}
		\right|<1-\delta<1
		\qquad \textnormal{for sufficiently small $j\le0$}.
	\end{equation}
	Finally, \eqref{eq:BC_infty} holds under 
	both conditions
	$\oplus_{x_1;x_2}$ and $\ominus_{r_2^{-2}x_2;r_1^{-2}x_1}$.
\end{proposition}
\begin{proof}
	All the desired identities follow from the Yang--Baxter
	equation (\Cref{prop:ABCD_YBE})
	applied to the operators
	$A^{[M,N]},B^{[M,N]},C^{[M,N]},D^{[M,N]}$ (\Cref{def:normalized_ops}), after 
	taking the limit 
	$M\to-\infty$, $N\to+\infty$.
	This limit is straightforward for 
	identities
	\eqref{eq:BB_infty} and \eqref{eq:DD_infty}
	using \eqref{eq:B2B1_commute} and \eqref{eq:D2D1_commute}, respectively.
	Let us explain how to obtain the remaining identities.

	For \eqref{eq:BD_infty}, we use
	\eqref{eq:B2D1} to write
	\begin{multline*}
		B^{[M,N]} (x_1, r_1) D^{[M,N]} (x_2, r_2) = 
		\frac{r^{-2}_2 x_2 -
		x_1}{x_2 - x_1} \,
		D^{[M,N]} (x_2, r_2) B^{[M,N]} (x_1, r_1) 
		\\+
		\frac{(1 - r^{-2}_1) x_1}{x_2 - x_1}\,
		D^{[M,N]} (x_1, r_1) 
		B^{[M,N]} (x_2, r_2)
		\prod_{j=1}^N
		\frac{y_j-s_j^2x_1}{y_j-x_1}
		\frac
		{y_j-x_2}
		{y_j-s_j^2x_2}.
	\end{multline*}
	Thanks to the assumptions, the second term 
	(more precisely, its pairing with 
	two arbitrary vectors $\langle e_{\mathcal{T}},(\cdots)e_{\mathcal{R}} \rangle $)
	can be bounded in the absolute value by $C(1-\delta)^N$, and 
	hence vanishes as $N\to+\infty$.

	For \eqref{eq:CD_infty}, we write using \eqref{eq:C2_D1}:
	\begin{multline*}
		C^{[M,N]} (x_1, r_1) D^{[M,N]} (x_2, r_2) 
		= \frac{r^{-2}_1 x_1 -
		x_2}{r^{-2}_1 x_1 - r^{-2}_2 x_2}\, D^{[M,N]} (x_2, r_2) C^{[M,N]} (x_1, r_1) 
		\\
		+
		\frac{x_2 (1 - r^{-2}_2)}{r^{-2}_1 x_1 - r^{-2}_2 x_2}\,
		D^{[M,N]}(x_1, r_1) C^{[M,N]} (x_2, r_2)
		\prod_{i=M}^{0}
			\frac{y_i-s_i^2r_1^{-2}x_1}{r_1^{-2}x_1-y_i}
			\frac
			{r_2^{-2}x_2-y_i}
			{y_i-s_i^2r_2^{-2}x_2}
		.
	\end{multline*}
	Thanks to our assumptions, the second term vanishes in the infinite volume limit.

	Finally, for \eqref{eq:BC_infty}, we write using \eqref{eq:ADBC_1}:
	\begin{multline*}
		B^{[M,N]} (x_1, r_1) C^{[M,N]}(x_2, r_2)=
		\frac{r^{-2}_2 x_2 - r^{-2}_1 x_1}{x_2 - x_1}\,
		C^{[M,N]}(x_2, r_2) B^{[M,N]} (x_1, r_1) 
		\\
		+
		\frac{x_1 (r^{-2}_1 - 1)}{x_2 - x_1}
		\Biggl( 
		D^{[M,N]} (x_2, r_2) A^{[M,N]} (x_1,r_1)
		\prod_{i=M}^0
		\frac{y_i-s_i^2r_2^{-2} x_2}{y_i-s_i^2x_2}
		\frac{y_i-x_1/r_1^2}{x_1-y_i/s_i^2}
		\prod_{j=1}^N
		\frac{s_j^2(y_j-x_1)}{y_j-s_j^2x_1}
		\\-
		D^{[M,N]} (x_1, r_1) A^{[M,N]}(x_2, r_2)
		\prod_{i=M}^0
		\frac{y_i-s_i^2r_1^{-2} x_1}{y_i-s_i^2x_1}
		\frac{y_i-x_2/r_2^2}{x_2-y_i/s_i^2}
		\prod_{j=1}^N
		\frac{s_j^2(y_j-x_2)}{y_j-s_j^2x_2}
		\Biggr)
		\\\times
		\prod_{i=M}^0
		\frac{y_i-s_i^2x_1}{y_i-s_i^2r_1^{-2}x_1}
		\frac{x_2-y_i/s_i^2}{y_i-x_2/r_2^2}
		\prod_{j=1}^N
		\frac{y_i-s_i^2x_1}{s_i^2(y_i-x_1)}
		.
	\end{multline*}
	Again, thanks to our assumptions, 
	the terms involving the operators $A,D$ vanish in the infinite
	volume limit.
\end{proof}

\begin{definition}
	\label{def:psi_psi_star}
	For each $j\in \mathbb{Z}$, define the 
	fermionic \emph{creation and annihilation}
	operators $\psi_j,\psi_j^*\colon \mathscr{F}\to\mathscr{F}$ on the basis by 
	(and then extending by linearity):
	\begin{equation*}
		\begin{split}
			\psi_j e_{\mathcal{T}}
			=
			\begin{cases}
				(-1)^{h_j(\mathcal{T})}e_{\mathcal{T}\cup\left\{ j \right\}},
				&j\notin \mathcal{T};\\
				0,&j\in \mathcal{T},
			\end{cases}
			\qquad 
			\psi_j^*e_{\mathcal{T}}=
			\begin{cases}
				(-1)^{h_j(\mathcal{T})}
				e_{\mathcal{T}\setminus \left\{ j \right\}},&j\in \mathcal{T};\\
				0,&j\notin \mathcal{T}.
			\end{cases}
		\end{split}
	\end{equation*}
\end{definition}
Clearly, $\psi_j\colon \mathscr{F}_n\to\mathscr{F}_{n+1}$
and $\psi_j^*\colon \mathscr{F}_n\to\mathscr{F}_{n-1}$.
The operators $\psi_i,\psi_j^*$ satisfy
the anticommutation relations
(that are easy to check directly)
\begin{equation*}
	\psi_k\psi_k^*+\psi_k^*\psi_k=1,\qquad 
	\psi_k\psi_\ell^*+\psi_\ell^*\psi_k
	=
	\psi_k^*\psi_\ell^*+\psi_\ell^*\psi_k^*
	=
	\psi_k\psi_\ell+\psi_\ell\psi_k=0,\qquad 
	k\ne \ell.
\end{equation*}
Moreover,
\begin{equation}
	\label{eq:small_psi_correlation_property}
	\psi_j\psi_j^* e_{\mathcal{T}}=
	\mathbf{1}_{j\in \mathcal{T}}\,
	e_{\mathcal{T}}.
\end{equation}

Define the operators
\begin{equation}
	\label{eq:Psi_Psi_star_operators}
	\Psi(x,z)
	=
	D^{\mathbb{Z}}(x,\sqrt{x/z})
	C^{\mathbb{Z}}(z,\sqrt{z/x})
	(-1)^{\mathfrak{c}} 
	,
	\qquad 
	\Psi^*(x,z)
	=
	D^{\mathbb{Z}}(x,\sqrt{x/z})
	B^{\mathbb{Z}}(z,\sqrt{z/x}),
\end{equation}
where $\mathfrak{c}$ is the charge operator.
Clearly, $\Psi(x,z)\colon \mathscr{F}_n\to\mathscr{F}_{n+1}$ and 
$\Psi^*(x,z)\colon \mathscr{F}_n\to\mathscr{F}_{n-1}$.
Let us record several relations for 
$\Psi,\Psi^*$:
\begin{proposition}
	\label{prop:Psi_Psi_Star_relations_with_ABCD}
	We have
	\begin{align}
		\label{eq:BPsi}
		B^{\mathbb{Z}}(x,r)
		\Psi(u,\zeta)
		=-\frac{u-r^{-2}x}{u-x}\,
		\Psi(u,\zeta)
		B^{\mathbb{Z}}(x,r)
		&\quad\textnormal{under}\ 
		\oplus_{x;u},\oplus_{x;\zeta},\ominus_{u;r^{-2}x}
		;
		\\
		\label{eq:PsiD}
		\Psi(u,\zeta)
		D^{\mathbb{Z}}(w,\theta)
		=
		\frac{u-w}{u-\theta^{-2}w}\,
		D^{\mathbb{Z}}(w,\theta)
		\Psi(u,\zeta)
		&\quad
		\textnormal{under}\ \ominus_{u;\theta^{-2}w}
		;
		\\
		\label{eq:BPsi_Star}
		B^{\mathbb{Z}}(x,r)
		\Psi^*(\kappa,v)
		=
		\frac{v-x}{r^{-2}x-v}\,
		\Psi^*(\kappa,v)
		B^{\mathbb{Z}}(x,r)
		&\quad
		\textnormal{under}\ \oplus_{x;\kappa}
		;
		\\
		\label{eq:PsiStar_D}
		\Psi^*(\kappa,v)
		D^{\mathbb{Z}}(w,\theta)=
		\frac{\theta^{-2}w-v}{w-v}\,
		D^{\mathbb{Z}}(w,\theta)
		\Psi^*(\kappa,v)
		&\quad
		\textnormal{under}\ \oplus_{v;w}
		;
		\\
		\label{eq:Psi_PsiStar}
		\Psi^*(\kappa,v)
		\Psi(u,\zeta)
		=
		-
		\Psi(u,\zeta)
		\Psi^*(\kappa,v)
		&\quad
		\textnormal{under}\ 
		\oplus_{v;u},
		\oplus_{v;\zeta},
		\ominus_{u;\kappa},
		\ominus_{u;v}
		.
	\end{align}
	These identities 
	are again understood in the sense of matrix elements in the standard
	basis. For example,
	we have
	$\langle e_{\mathcal{T}}, \Psi^*(\kappa,v)
	\Psi(u,\zeta)e_{\mathcal{R}} \rangle 
	=-
	\langle e_{\mathcal{T}}
	\Psi(u,\zeta)
	\Psi^*(\kappa,v)
	e_{\mathcal{R}}
	\rangle $ for all $e_{\mathcal{T}},e_{\mathcal{R}}\in \mathscr{F}_n$, $n\in \mathbb{Z}$.
	Conditions $\oplus$ and $\ominus$ are defined
	in \eqref{eq:oplus_condition}--\eqref{eq:ominus_condition}.
\end{proposition}
\begin{proof}
	These identities and the corresponding conditions
	immediately
	follow from \Cref{prop:ABCD_infinite_volume}.
	Note the extra minus signs in \eqref{eq:BPsi}
	and \eqref{eq:Psi_PsiStar}
	which arise from commuting $(-1)^{\mathfrak{c}}$ with $B^{\mathbb{Z}}$.
\end{proof}

The next statement is one of the key 
results on the operators 
$\Psi$ and $\Psi^*$
in the Fock space:
\begin{theorem}[\Cref{thm:intro_boson_fermion} from Introduction]
	\label{thm:Psi_PsiStar_Fock}
	As operators on $\mathscr{F}$, we have
	\begin{equation*}
		\Psi(x,z)=\sum_{j\in \mathbb{Z}}
		\Phi_j(x,z)\psi_j,\qquad 
		\Psi^*(x,z)=\sum_{j\in \mathbb{Z}}
		\Phi_j^*(x,z)\psi_j^*,
	\end{equation*}
	where the expressions $\Phi_j,\Phi_j^*$ are given in \Cref{def:Phi_PhiStar}.
	In particular, $\Psi(x,z)$ is independent of $z$, which is evident from the formula
	for $\Phi_j(x,z)$.
\end{theorem}
\begin{proof}
	This follows from \Cref{prop:fermionic_first_correlation,prop:normalized_Phi_action,prop:normalized_Phi_star_action}.
	Indeed, for $\psi_{t_j}^*$ in \Cref{prop:normalized_Phi_star_action}
	we have $(-1)^{m+1-j}=(-1)^{h_{t_j}(\mathcal{T})}$,
	and for $\psi_{t_j}$ in \Cref{prop:normalized_Phi_action} we have
	$(-1)^{M-j}=(-1)^{\mathfrak{c}(\mathcal{R})+h_{t_j}(\mathcal{R})}$.
	Here we used the fact that the charge of any $m$-subset of
	$\left\{ M,M+1,\ldots,N  \right\}$ for sufficiently small $M$ and large
	$N$ is equal to $m-(M+1)$.
	This completes the proof.
\end{proof}

The next statement follows from the previous
\Cref{thm:Psi_PsiStar_Fock} and is a key 
ingredient in getting determinantal correlation functions.
We refer to \cite[Lemma 1]{okounkov2003correlation} or
\cite[Lemma B.1]{betea2019periodic}
for its proof which dates back to at least \cite{gaudin1960demonstration}.

\begin{proposition}[Wick's determinant]
	\label{prop:wick}
	Fix an integer $k\ge1$ and two sequences of complex numbers
	$\left\{ a_{ij} \right\}, \left\{ b_{ij} \right\}$.
	Define the operators
	$A_i=\sum_{j\in \mathbb{Z}} a_{ij}\psi_j$ and 
	$B_i=\sum_{j\in \mathbb{Z}} b_{ij}\psi_j^*$
	(note that the operators
	$A_i\colon \mathscr{F}_n\to \mathscr{F}_{n+1}$
	and 
	$B_j\colon \mathscr{F}_n\to\mathscr{F}_{n-1}$
	are well-defined for arbitrary
	coefficients $a_{ij},b_{kl}$).
	Then
	\begin{equation*}
		\bigl\langle e_{\mathbb{Z}_{\le0}},A_1 B_1 A_2B_2 \ldots A_k B_k
		e_{\mathbb{Z}_{\le0}} \bigr\rangle 
		=
		\det\left[ 
			M_{ij}
		\right]_{i,j=1}^{k},\qquad 
		M_{ij}=\begin{cases}
			\bigl\langle e_{\mathbb{Z}_{\le0}},A_i B_j e_{\mathbb{Z}_{\le0}} \bigr\rangle,
			&i\ge j;\\
			-\bigl\langle e_{\mathbb{Z}_{\le0}},B_jA_i e_{\mathbb{Z}_{\le0}} \bigr\rangle,
			&i<j.
		\end{cases}
	\end{equation*}
\end{proposition}

\subsection{Action of the \texorpdfstring{$\Psi$}{Psi} operators}
\label{appC:action_Psi}

We now compute matrix elements
of various 
products of the operators $\Psi(u,\zeta)$ and $\Psi^*(\kappa,v)$
on the vectors $e_{\mathbb{Z}_{\le0}}, e_{\mathbb{Z}_{\le N}}$ from 
the Fock space $\mathscr{F}$.

\begin{lemma}
	\label{prop:action_Psi_one}
	Under $\ominus_{u;v}$ \eqref{eq:ominus_condition}, we have
	\begin{equation*}
		\bigl\langle e_{\mathbb{Z}_{\le0}},
		\Psi(u,\zeta)\Psi^*(\kappa,v)
		e_{\mathbb{Z}_{\le0}}
		\bigr\rangle
		=
		\frac{v-\kappa}{u-v}.
	\end{equation*}
	Under $\oplus_{v;u}$, we have
	\begin{equation*}
		-\bigl\langle e_{\mathbb{Z}_{\le0}},
		\Psi^*(\kappa,v)
		\Psi(u,\zeta)
		e_{\mathbb{Z}_{\le0}}
		\bigr\rangle
		=
		\frac{v-\kappa}{u-v}.
	\end{equation*}
\end{lemma}
\begin{proof}
	We only prove the first identity, the second is analogous.
	Observe that 
	\begin{equation*}
		\bigl\langle e_{\mathbb{Z}_{\le0}},\psi_i\psi_j^* e_{\mathbb{Z}_{\le0}} \bigr\rangle 
		=\mathbf{1}_{i=j}\mathbf{1}_{i\le 0}.
	\end{equation*}
	Therefore, using \Cref{thm:Psi_PsiStar_Fock}
	and \Cref{def:Phi_PhiStar} we have
	\begin{equation*}
		\begin{split}
			&\bigl\langle e_{\mathbb{Z}_{\le0}},
			\Psi(u,\zeta)\Psi^*(\kappa,v)
			e_{\mathbb{Z}_{\le0}}
			\bigr\rangle
			=\sum_{j=-\infty}^0 \Phi_j(u,\zeta)\Phi_j^*(\kappa,v)
			\\&\hspace{20pt}=
			\sum_{j=-\infty}^{0}
			\frac{y_j (1 - s_j^2)}{y_j - s_j^2 u} 
			\frac{\kappa - v}{y_j - v} 
			\prod_{k = j}^0 \frac{y_k - s_k^2 u}{s_k^2 (y_k - u)}
			\frac{s_k^2 (y_k - v)}{y_k - s_k^2v}
			\\&\hspace{20pt}=
			\frac{v-\kappa}{u-v}
			\sum_{j=-\infty}^{0}
			\left( 1-\frac{(y_j-v)(y_j-s_j^2 u)}{(y_j-u)(y_j-s_j^2v)} \right)
			\prod_{k = j+1}^0 
			\frac{y_k - s_k^2 u}{y_k - u}
			\frac{y_k - v}{y_k - s_k^2v}
			=
			\frac{v-\kappa}{u-v}.
		\end{split}
	\end{equation*}
	In the last equality we used the fact that the infinite sum 
	telescopes to 1 under $\ominus_{u;v}$.
\end{proof}

\begin{proposition}
	\label{prop:action_Psi_many}
	Fix an integer $m\ge1$, and let $u_i,v_j$ satisfy
	\begin{equation*}
		\begin{cases}
			\ominus_{u_i;v_j},&i\ge j;\\
			\oplus_{v_j;u_i},&i<j.
		\end{cases}
	\end{equation*}
	Then we have
	\begin{multline*}
		\bigl\langle 
		e_{\mathbb{Z}_{\le0}},
		\Psi(u_1,\zeta_1)\Psi^*(\kappa_1,v_1)
		\ldots
		\Psi(u_m,\zeta_m)\Psi^*(\kappa_m,v_m)\,
		e_{\mathbb{Z}_{\le0}}
		\bigr\rangle 
		\\=
		\prod_{i=1}^{m}(\kappa_i-v_i)
		\prod_{1\le i<j\le m}(v_j-v_i)(u_i-u_j)
		\prod_{i,j=1}^m \frac{1}{v_j-u_i}.
	\end{multline*}
\end{proposition}
Note that $\Psi(u_i,\zeta_i)$ does not depend on $\zeta_i$, 
and the $\zeta_i$'s are not present in the right-hand side, as it should be.
\begin{proof}[Proof of \Cref{prop:action_Psi_many}]
	Employing \Cref{prop:wick} and \Cref{prop:action_Psi_one}, we have 
	\begin{equation*}
		\bigl\langle 
		e_{\mathbb{Z}_{\le0}},
		\Psi(u_1,\zeta_1)\Psi^*(\kappa_1,v_1)
		\ldots
		\Psi(u_m,\zeta_m)\Psi^*(\kappa_m,v_m)\,
		e_{\mathbb{Z}_{\le0}}
		\bigr\rangle 
		=
		\det\left[ \frac{v_j-\kappa_j}{u_i-v_j} \right]_{i,j=1}^m.
	\end{equation*}
	The determinant in the right-hand side 
	factorizes thanks to the Cauchy determinant formula,
	and we arrive at the desired identity.
\end{proof}

\section{Correlation kernel via fermionic operators}
\label{sec:Fock_and_FG_process}

In this section we study a generalization of the 
ascending FG process introduced in 
\Cref{sub:probability_from_Cauchy}, and compute a generating function type series
for its correlation kernel $K_{\mathscr{P}}$
using fermionic operators in the Fock space developed in \Cref{sec:fermionic_operators} above.

\subsection{General FG processes}
\label{sub:general_FG}

The following definition is parallel to the definition of the Schur process introduced in
\cite{okounkov2003correlation}.

Assume that the parameters $(\mathbf{y};\mathbf{s})$
satisfying \eqref{eq:ys_FG_parameters}
are fixed.
Fix $T\ge1$ and variables $(x_i;r_i)$ and $(w_i;\theta_i)$, $i=1,\ldots,T$,
such that these specializations are nonnegative in the sense of \Cref{def:spec}
and are compatible as in \Cref{def:compatible_spec},
i.e., the variables satisfy \eqref{eq:Eynard_asc_FG_condition}.
The \emph{FG process} with this data is a probability measure
on sequences of signatures 
$\lambda^{(1)},\mu^{(1)},\lambda^{(2)},\ldots,\mu^{(T-1)},\lambda^{(T)}$
with probability weights
\begin{equation}
	\label{eq:FG_process}
	\begin{split}
		&\mathscr{P}
		(\lambda^{(1)},\mu^{(1)},\lambda^{(2)},\mu^{(2)},\ldots,\mu^{(T-1)},\lambda^{(T)})
		:=
		\frac{1}{Z}
			G_{\lambda^{(1)}}(w_1;\theta_1)
			F_{\lambda^{(1)}/\mu^{(1)}}(x_1;r_1)
			\\&\hspace{20pt}\times
			G_{\lambda^{(2)}/\mu^{(1)}}(w_2;\theta_2)
			\ldots
			F_{\lambda^{(T-1)}/\mu^{(T-1)}}(x_{T-1};r_{T-1})
			G_{\lambda^{(T)}/\mu^{(T-1)}}(w_T;\theta_T)
			F_{\lambda^{(T)}}(x_T;r_T),
	\end{split}
\end{equation}
where the normalizing constant is computed using multiple applications of \eqref{eq:Cauchy_identities_via_specs}:
\begin{equation}\label{eq:Z_FG_appC_normalizing}
	Z=
	\prod_{i=1}^{T}x_i(r^{-2}_i-1)\,
	\frac{\prod_{1\le i<j\le T}(r_i^{-2}x_i-x_j)(s_i^{-2}y_i-y_j)}
	{\prod_{i,j=1}^T(y_i-x_j)}
	\prod_{1\le i\le j\le T}
	\frac{x_j-\theta_i^{-2}w_i}{x_j-w_i}.
\end{equation}
Note that the number of parts in the signatures in
\eqref{eq:FG_process} is fixed: $\lambda^{(1)}$ has $T$ parts, and 
both $\mu^{(j)}$ and $\lambda^{(j+1)}$
have $T-j$ parts,
$j=1,\ldots,T-1$.
In \eqref{eq:FG_process} and below when convenient
we omit 
the notation $\mathbf{y},\mathbf{s}$
in the functions 
$F_{\lambda/\mu}(x_i;\mathbf{y};r_i;\mathbf{s})$,
$G_{\nu/\varkappa}(w_i;\mathbf{y};\theta_i;\mathbf{s})$.

\begin{remark}
	\label{rmk:reduction_to_ascending}
	The FG process \eqref{eq:FG_process} reduces to the ascending FG process
	from \Cref{def:FG_asc}
	as follows. Fix some $1\le a\le T-1$, set
	$w_{a+1}=w_{a+2}=\ldots=w_T=0$, and replace each of the 
	specializations $(x_1;r_1),\ldots,(x_a;r_a)$ by $\varnothing$, 
	the empty specialization (see \eqref{eq:F_G_at_empty}).
	We discuss the reduction to the ascending case in \Cref{sub:ascending_and_Fock} below.
\end{remark}

\subsection{FG process via Fock space}
\label{sub:FG_meas_proc_through_ops}

Let us now express the probability weights
under the FG process
\eqref{eq:FG_process}
through matrix elements of
our operators acting in the Fock space $\mathscr{F}$.
Fix a 
signature
$\lambda=(\lambda_1\ge \ldots\ge \lambda_{m}\ge0 )$,
and let $I_\lambda$ be the 
rank one projection in $\mathscr{F}$
onto the semi-infinite subset corresponding to $\lambda$,
that is, which acts as
\begin{equation*}
	I_\lambda e_{\mathcal{T}}=
	\begin{cases}
		e_{\mathcal{T}},&
		\text{if $\mathcal{T}=\left\{ \lambda_1+m,\lambda_2+m-1,\ldots,\lambda_m+1,0,-1,-2,\ldots  \right\}$};\\
		0,&\text{otherwise},
	\end{cases}
\end{equation*}
for any semi-infinite $\mathcal{T}\subset \mathbb{Z}$.

\begin{lemma}
	\label{lemma:FG_process_via_operators}
	The probability weights of the FG process have the form
	\begin{equation*}
		\begin{split}
			&\mathscr{P}
			(\lambda^{(1)},\mu^{(1)},\lambda^{(2)},\mu^{(2)},\ldots,\mu^{(T-1)},\lambda^{(T)})
			=
			\frac{1}{Z}
			\bigl\langle e_{\mathbb{Z}_{\le 0}},
			B^{\mathbb{Z}}(x_T,r_T) I_{\lambda^{(T)}}
			\\&\hspace{40pt}\times
			D^{\mathbb{Z}}(w_T,\theta_T)I_{\mu^{(T-1)}}
			\ldots 
			I_{\lambda^{(2)}}
			D^{\mathbb{Z}}(w_2,\theta_2) I_{\mu^{(1)}}
			B^{\mathbb{Z}}(x_1,r_1) I_{\lambda^{(1)}}
			D^{\mathbb{Z}}(w_1,\theta_1)
			e_{\mathbb{Z}_{\le T }}
			\bigr\rangle .
		\end{split}
	\end{equation*}
	The normalizing constant \eqref{eq:Z_FG_appC_normalizing} is
	\begin{equation*}
		Z
		=
		\bigl\langle e_{\mathbb{Z}_{\le 0}},
		B^{\mathbb{Z}}(x_T,r_T) 
		D^{\mathbb{Z}}(w_T,\theta_T)
		B^{\mathbb{Z}}(x_{T-1},r_{T-1})
		\ldots 
		D^{\mathbb{Z}}(w_2,\theta_2) 
		B^{\mathbb{Z}}(x_1,r_1) 
		D^{\mathbb{Z}}(w_1,\theta_1)
		e_{\mathbb{Z}_{\le T }}
		\bigr\rangle .
	\end{equation*}
\end{lemma}
\begin{proof}
	The matrix element 
	$\langle e_{\mathbb{Z}_{\le0}},A e_{\mathbb{Z}_{\le N}} \rangle $,
	where $A$ is a product of $B^{\mathbb{Z}},D^{\mathbb{Z}}$, and $I_\lambda$'s,
	can be nonzero only if
	in the partition function representation of it
	the vertical arrows in each position $j\le 0$ 
	move vertically straight. The normalization of the 
	operators $B^{\mathbb{Z}}$, $D^{\mathbb{Z}}$ on $\mathbb{Z}_{\le0}$
	(\Cref{def:normalized_ops})
	ensures that this straight movement contributes the total weight $1$.
	Next, in the positive half-line $\mathbb{Z}_{\ge1}$, the 
	operators $D^{\mathbb{Z}}$ are not normalized and thus 
	yield vertex configurations for the functions $G_{\lambda/\mu}$.
	The normalization of the operators 
	$B^{\mathbb{Z}}$ is equivalent to passing from the 
	vertex weights $W$ to the weights $\widehat{W}$, see \eqref{eq:W_What_renormalization}.
	Therefore, the action of the operators $B^{\mathbb{Z}}$ yield
	vertex configurations for the functions $F_{\lambda/\mu}$.
	This shows the desired expressions for the probability
	weights of the FG process and measures.
\end{proof}

\begin{remark}
	\label{rmk:FG_process_independent_from_negative}
	The matrix element representation
	for the probability weights of the FG process
	in \Cref{lemma:FG_process_via_operators}
	is independent
	of $(y_k,s_k)$, $k\le0$, as it should be.
\end{remark}

\subsection{Extracting series coefficients}
\label{sub:series_coeffs}

In the rest of this section, we abbreviate 
\begin{equation*}
	\Psi(u):=\Psi(u,0),\qquad 
	\Psi^*(v):=\Psi^*(0,v).
\end{equation*}
Note that the operator $\Psi^*(v)$ is well-defined by setting $\kappa=0$ in 
the expansion of \Cref{thm:Psi_PsiStar_Fock}:
\begin{equation*}
	\Psi^*(v)=\sum_{j=1}^{+\infty}
	\left(
	\frac{v}{v-y_j} 
	\prod_{k = 1}^{j - 1} \frac{y_k - s_k^2v}{s_k^2 (y_k - v)}
	\right)\psi^*_j
	+
	\sum_{j=-\infty}^{0}
	\left(
	\frac{v}{v-y_j} \prod_{k = j}^0 \frac{s_k^2 (y_k - v)}{y_k - s_k^2v}
	\right)\psi^*_j.
\end{equation*}
Therefore, we have for any 
semi-infinite subset $\mathcal{T}\subset \mathbb{Z}$ and
$j\ge1$:
\begin{equation}
	\label{eq:correlation_as_a_coefficient}
	\psi_j\psi_j^* e_{\mathcal{T}}=
	\mathbf{1}_{j\in \mathcal{T}}\,
	e_{\mathcal{T}}
	.
\end{equation}
The quantity $\mathbf{1}_{j\in \mathcal{T}}\, e_{\mathcal{T}}$ can also be 
written as
$\bigl[\Phi_j(u,0)\Phi_j^*(0,v)\bigr]
\Psi(u)\Psi^*(v)\,e_{\mathcal{T}}$,
where the notation $\bigl[\cdots \bigr]$ means extracting the coefficient
in the generating series (recall \Cref{def:Phi_PhiStar}). 
The next two lemmas clarify what it means to extract such a 
coefficient.
\begin{lemma}
	\label{lemma:Phi_extraction}
	Let $f(u)=\sum_{i\in \mathbb{Z}}c_i \Phi_i(u,0)$, where $c_i\in \mathbb{C}$.
	Then for any $i\in \mathbb{Z}$ we have
	\begin{equation*}
		c_i=\frac{1}{2\pi \mathbf{i}}\oint 
		\frac{f(u)\,du}{u-y_i}\prod_{k=1}^{i-1}\frac{u-s_k^{-2}y_k}{u-y_k},\quad 
		i\ge 1;
		\qquad 
		c_i=\frac{1}{2\pi \mathbf{i}}\oint 
		\frac{f(u)\,du}{u-s_i^{-2}y_i}\prod_{k=i+1}^{0}
		\frac
		{u-y_k}
		{u-s_k^{-2}y_k}
		,\quad 
		i\le 0.
	\end{equation*}
	In both integrals the integration contour separates the families of the 
	points $\{y_k \}_{k\in \mathbb{Z}}$ and $\{ s_k^{-2}y_k \}_{k\in \mathbb{Z}}$,
	and 
	goes around the $y_k$'s in the positive direction.
	Moreover, the 
	series for $f(u)$
	must converge uniformly
	on the contour.
\end{lemma}
\begin{proof}
	This is essentially
	the single-variable biorthogonality 
	(\Cref{lemma:phi_psi_orthogonal}).
	Indeed, the hypothesis of \Cref{lemma:Phi_extraction}
	allows to interchange summation and integration.
	Then one can show that for all $m\in \mathbb{Z}$
	and $i\ge1$ we have
	\begin{equation*}
		\frac{1}{2\pi \mathbf{i}}\oint 
		\frac{\Phi_m(u,0)\,du}{u-y_i}
		\prod_{k=1}^{i-1}\frac{u-s_k^{-2}y_k}{u-y_k}=\mathbf{1}_{m=i}.
	\end{equation*}
	Note that when $m$ is nonpositive
	(the case not covered by \Cref{lemma:phi_psi_orthogonal}), the 
	$u$ contour has no poles outside 
	(as all the poles are of the form $u=y_k$ and are inside),
	so the integral vanishes. 
	When $i\le 0$, we similarly have for all $m\in \mathbb{Z}$:
	\begin{equation*}
		\frac{1}{2\pi \mathbf{i}}\oint 
		\frac{\Phi_m(u,0)\,du}{u-s_i^{-2}y_i}\prod_{k=i+1}^{0}
		\frac
		{u-y_k}
		{u-s_k^{-2}y_k}=\mathbf{1}_{m=i}.
	\end{equation*}
	This completes the proof.
\end{proof}

\begin{lemma}
	\label{lemma:PhiStar_extraction}
	Let $g(v)=\sum_{j\in \mathbb{Z}}d_j \Phi_j^*(0,v)$, where $d_j\in \mathbb{C}$.
	Then for $j\ge1$ and $j\le 0$ we have, respectively,
	\begin{equation*}
		d_j=
		\frac{y_j(s_j^{-2}-1)}{2\pi \mathbf{i}}\oint 
		\frac{v^{-1}g(v)\,dv}{v-s_j^{-2}y_j}
		\prod_{k=1}^{j-1}
		\frac{v-y_k}{v-s_k^{-2}y_k}
		;\quad 
		d_j=
		\frac{y_j(s_j^{-2}-1)}{2\pi \mathbf{i}}\oint 
		\frac{v^{-1}g(v)\,dv}{v-y_j}
		\prod_{k=j+1}^{0}\frac{v-s_k^{-2}y_k}{v-y_k}
		.
	\end{equation*}
	In both integrals the integration contour separates the families of the 
	points $\{y_k \}_{k\in \mathbb{Z}}$ and $\{ s_k^{-2}y_k \}_{k\in \mathbb{Z}}$,
	and 
	goes around the $s_k^{-2}y_k$'s in the positive direction.
	Moreover, the 
	series for $g(v)$
	must converge uniformly
	on the contour.
\end{lemma}
\begin{proof}
	This is proven in the same way as \Cref{lemma:Phi_extraction}.
\end{proof}

\begin{remark}
	\label{rmk:linear_independence_Phi_PhiStar}	
	\Cref{lemma:Phi_extraction,lemma:PhiStar_extraction} 
	imply linear independence of the products
	$\Phi_i(u,0)\Phi_j^*(0,v)$ for all $i,j\in \mathbb{Z}$.
	Therefore, the
	operation of extracting the coefficient 
	\begin{equation*}
		\bigl[\Phi_j(u,0)\Phi_j^*(0,v)\bigr]
		\Psi(u)\Psi^*(v)\,e_{\mathcal{T}}
	\end{equation*}
	discussed before
	\Cref{lemma:Phi_extraction}
	is indeed well-defined
	and can be realized by \Cref{lemma:Phi_extraction,lemma:PhiStar_extraction}.
\end{remark}


\subsection{Correlation generating function}
\label{appC:corr_FG_via_fermionic}

We are now in a position to 
compute a generating series type expression for the correlations
of the FG process.
Denote
\begin{equation*}
	\boldsymbol\Psi(\mathbf{u};\mathbf{v})
	:=
	\Psi(u_k)\Psi^*(v_k)\ldots\Psi(u_1)\Psi^*(v_1).
\end{equation*}
Applying
$\boldsymbol\Psi(\mathbf{u};\mathbf{v})$ to a
vector $e_{\mathcal{T}}$ (with semi-infinite $\mathcal{T}$)
produces a linear combination of 
terms $\mathbf{1}_{\{j_1,\ldots,j_k  \}\in \mathcal{T}}\,e_{\mathcal{T}}$
(corresponding to the desired $k$-point correlations),
together with some other terms.
More precisely, each desired term 
$\mathbf{1}_{\{j_1,\ldots,j_k  \}\in \mathcal{T}}\,e_{\mathcal{T}}$,
with $(j_1,\ldots,j_k )\in \mathbb{Z}_{\ge1}^{k}$ distinct and ordered,
arises 
from $\psi_{j_k}\psi_{j_k}^*\ldots\psi_{j_1}\psi_{j_1}^*$,
where each index $j_m$ comes from the pair 
of the generating functions $\Psi(u_m)\Psi^*(v_m)$.
All the other terms are ``parasite'' and should be excluded 
by extracting only the appropriate coefficients 
as in
\Cref{sub:series_coeffs}.
The generating function with all the terms put together
has an explicit product form:

\begin{proposition}
	\label{prop:correlations_gen_function}
	Take the following sequences of parameters for the generating series:
	\begin{equation*}
		\mathbf{u}^j=(u_1^j,\ldots,u_{k_j}^j ),\qquad
		\mathbf{v}^j=(v_1^j,\ldots,v_{k_j}^j ),\qquad
		k_j\ge0,
		\qquad 
		j=1,\ldots,T.
	\end{equation*}
	Assume that for all possible indices 
	$m=1,\ldots,T$, and
	$i,j,\alpha,\beta$, the parameters satisfy:
	\begin{equation}
		\label{eq:correlations_gen_function_conditions}
		\begin{split}
			&
			\oplus_{x_i;w_j},
			\ 
			\oplus_{v^j_\alpha;w_i},
			\  
			\oplus_{v^i_\beta;r_j^{-2}x_j},
			\ 
			\oplus_{x_j;u^i_\alpha},
			\ 
			\oplus_{v^j_\beta;y_m},
			\ 
			\oplus_{x_i;y_m},
			\ 
			\ominus_{u^j_\alpha;\theta_i^{-2}w_i},
			\ 
			\ominus_{u^i_\beta;r_j^{-2}x_j},
			\ 
			\ominus_{u^i_\beta;x_j},
			\ 
			\\&
			\begin{cases}
				\ominus_{y_j;x_i},
				\ \ominus_{y_j;r_i^{-2}x_i},& 1\le i\le j\le T;\\
				\oplus_{x_i;r_j^{-2}x_j},
				& 1\le j<i\le T,
			\end{cases}
			\qquad
			\begin{cases}
			\ominus_{u^i_\alpha;v^j_\beta},& \textnormal{$i>j$ or $i=j,\alpha\ge\beta$};
			\\
			\oplus_{v^j_\beta;u^i_\alpha},&\textnormal{$i<j$ or $i=j,\alpha<\beta$},
			\end{cases}
		\end{split}
	\end{equation}
	see \eqref{eq:oplus_condition},
	\eqref{eq:ominus_condition} for the notation.
	Then we have
	\begin{align*}
			&
			\frac{1}{Z}
			\bigl\langle e_{\mathbb{Z}_{\le 0}},
			B^{\mathbb{Z}}(x_T,r_T)
			\Psi(\mathbf{u}^T;\mathbf{v}^T)
			D^{\mathbb{Z}}(w_T,\theta_T)
			\\&\hspace{20pt}\times
			B^{\mathbb{Z}}(x_{T-1},r_{T-1}) 
			\ldots 
			\Psi(\mathbf{u}^{2};\mathbf{v}^{2})
			D^{\mathbb{Z}}(w_2,\theta_2) 
			B^{\mathbb{Z}}(x_1,r_1)
			\Psi(\mathbf{u}^1;\mathbf{v}^1)
			D^{\mathbb{Z}}(w_1,\theta_1)
			e_{\mathbb{Z}_{\le T}}
			\bigr\rangle
			\\&\hspace{2pt}
			=
			\prod_{1\le i\le j\le T}\prod_{\alpha=1}^{k_j}
			\frac{(v^j_\alpha-\theta_i^{-2}w_i)(u^j_\alpha-w_i)}{(v^j_\alpha-w_i)(u^j_\alpha-\theta_i^{-2}w_i)}
			\prod_{1\le i< j\le T}
			\prod_{\alpha=1}^{k_j}
			\frac{(u^j_\alpha-x_i)(v^j_\alpha-r_i^{-2}x_i)}{(u^j_\alpha-r_i^{-2}x_i)(v^j_\alpha-x_i)}
			\\&\hspace{14pt}\times
			\prod_{m,i=1}^{T}\prod_{\alpha=1}^{k_i}
			\frac{(u^i_\alpha-y_m)(v^i_\alpha-x_m)}{(v^i_\alpha-y_m)(u^i_\alpha-x_m)}
			\\&\hspace{14pt}\times
			\prod_{i=1}^{T}\biggl(\prod_{\alpha=1}^{k_i}
			\frac{v^i_\alpha}{u^i_\alpha-v^i_\alpha}
			\prod_{1\le \alpha<\beta\le k_i}
			\frac{(u^i_\alpha-u^i_\beta)(v^i_\alpha-v^i_\beta)}
			{(v^i_\alpha-u^i_\beta)(u^i_\alpha-v^i_\beta)}
			\biggr)
			\prod_{1\le i<j\le T}\biggl(
				\prod_{\alpha=1}^{k_i}\prod_{\beta=1}^{k_j}
				\frac{(u^i_\alpha-u^j_\beta)(v^i_\alpha-v^j_\beta)}
				{(v^i_\alpha-u^j_\beta)(u^i_\alpha-v^j_\beta)}
			\biggr)
			,
	\end{align*}
	where $Z$ is given by \eqref{eq:Z_FG_appC_normalizing}.
	If $k_j=0$ for some $j$, we
	omit the operator $\Psi(\mathbf{u}^j;\mathbf{v}^j)$ in the left-hand side; and
	in the right-hand side, the products $\prod_{\alpha=1}^{k_j}$
	are equal to $1$, by agreement.
\end{proposition}
\begin{proof}
	Throughout the proof, 
	conditions
	\eqref{eq:correlations_gen_function_conditions}
	arise from recording all the required commutations 
	of the operators which
	are obtained in
	\Cref{appC:fermionic_operators}.
	Moreover, we use Wick's determinant (\Cref{prop:action_Psi_many})
	which leads to the last condition 
	on $\mathbf{u}^i,\mathbf{v}^j$
	in \eqref{eq:correlations_gen_function_conditions}.

	Observe that
	\begin{equation*}
		\Psi(y_j)e_{\mathbb{Z}_{\le j-1}}=
		e_{\mathbb{Z}_{\le j}}
		\prod_{i=1}^{j-1}\frac{s_i^2(y_i-y_j)}{y_i-s_i^2y_j},
	\end{equation*}
	where $j\ge1$ (note the specific choice of the argument in $\Psi(\cdot)$).
	Therefore, 
	\begin{equation}
		\label{eq:correlations_gen_function_proof_e}
		e_{\mathbb{Z}_{\le T}}=
		\prod_{1\le i<j\le T}\frac{y_i-s_j^2y_j}{s_i^2(y_i-y_j)}
		\Psi(y_T)\ldots 
		\Psi(y_1)
		e_{\mathbb{Z}_{\le 0}}.
	\end{equation}

	First, we move each $D^{\mathbb{Z}}(w_i,\theta_i)$
	to the left of
	$\boldsymbol\Psi(\mathbf{u}^i;\mathbf{v}^i)$ and
	$B^{\mathbb{Z}}(x_i,r_i)$, $j\ge i$.
	Then we move each 
	$B^{\mathbb{Z}}(x_j,r_j)$ to the right of 
	$\boldsymbol\Psi(\mathbf{u}^i;\mathbf{v}^i)$, $i\le j$.
	This leads, by \Cref{prop:Psi_Psi_Star_relations_with_ABCD}, to 
	\begin{equation}
		\label{eq:correlations_gen_function_1}
		\begin{split}
			&
			B^{\mathbb{Z}}(x_T,r_T)
			\Psi(\mathbf{u}^T;\mathbf{v}^T)
			D^{\mathbb{Z}}(w_T,\theta_T)
			\ldots 
			\Psi(\mathbf{u}^{2};\mathbf{v}^{2})
			D^{\mathbb{Z}}(w_2,\theta_2) 
			B^{\mathbb{Z}}(x_1,r_1)
			\Psi(\mathbf{u}^1;\mathbf{v}^1)
			D^{\mathbb{Z}}(w_1,\theta_1)
			\\&\hspace{10pt}=
			D^{\mathbb{Z}}(w_T,\theta_T)
			\ldots 
			D^{\mathbb{Z}}(w_1,\theta_1)
			\Psi(\mathbf{u}^T;\mathbf{v}^T)
			\ldots 
			\Psi(\mathbf{u}^1;\mathbf{v}^1)
			B^{\mathbb{Z}}(x_T,r_T)
			\ldots 
			B^{\mathbb{Z}}(x_1,r_1)
			\\&\hspace{30pt}\times
			\prod_{1\le i\le j\le T}
			\biggl(
				\frac{\theta_i^{-2}w_i-x_j}{w_i-x_j}
				\prod_{\alpha=1}^{k_j}
				\frac{\theta_i^{-2}w_i-v_\alpha^j}{w_i-v_\alpha^j}
				\frac{w_i-u_\alpha^j}{\theta_i^{-2}w_i-u_\alpha^j}
				\frac{u_\alpha^j-r_i^{-2}x_i}{u_\alpha^j-x_i}
				\frac{v_\alpha^j-x_i}{v_\alpha^j-r_i^{-2} x_i}
			\biggr).
		\end{split}
	\end{equation}

	Now, note that 
	$\langle 
	e_{\mathbb{Z}_{\le0}},D^{\mathbb{Z}}(w,\theta)e_{\mathcal{T}}
	\rangle =
	\mathbf{1}_{\mathcal{T}=\mathbb{Z}_{\le0}}$.
	Thus, we may replace the $D$ operators on the left
	by any other $D$ operators. So we have, using 
	\eqref{eq:correlations_gen_function_proof_e},
	\begin{equation}
		\label{eq:matrix_element_in_marked}
		\begin{split}
			& 
			\bigl\langle 
			e_{\mathbb{Z}_{\le0}},
			D^{\mathbb{Z}}(w_T,\theta_T)
			\ldots 
			D^{\mathbb{Z}}(w_1,\theta_1)
			\Psi(\mathbf{u}^T;\mathbf{v}^T)
			\ldots 
			\Psi(\mathbf{u}^1;\mathbf{v}^1)
			B^{\mathbb{Z}}(x_T,r_T)
			\ldots 
			B^{\mathbb{Z}}(x_1,r_1)
			e_{\mathbb{Z}_{\le T }}
			\bigr\rangle
			\\=
			&
			\prod_{1\le i<j\le T}\frac{y_i-s_i^2y_j}{s_i^2(y_i-y_j)}
			\,
			\bigl\langle 
			e_{\mathbb{Z}_{\le0}},
			D(r_T^{-2}x_T,r_T^{-1})
			\ldots 
			D(r_1^{-2}x_1,r_1^{-1})
			\Psi(\mathbf{u}^T;\mathbf{v}^T)
			\ldots 
			\Psi(\mathbf{u}^1;\mathbf{v}^1)
			\\&\hspace{180pt}\times
			B^{\mathbb{Z}}(x_T,r_T)
			\ldots 
			B^{\mathbb{Z}}(x_1,r_1)
			\Psi(y_T)\ldots\Psi(y_1) 
			e_{\mathbb{Z}_{\le 0 }}
			\bigr\rangle,
		\end{split}
	\end{equation}
	We chose the new $D$ operators such that together 
	with $B(x_i,r_i)$ they will lead to the operators $\Psi^*$, 
	cf. \eqref{eq:Psi_Psi_star_operators}.
	Now we commute again and have,
	using \eqref{eq:PsiD}, \eqref{eq:PsiStar_D}, and
	\eqref{eq:BD_infty}:
	\begin{align*}
			&
			D(r_T^{-2}x_T,r_T^{-1})
			\ldots 
			D(r_1^{-2}x_1,r_1^{-1})
			\Psi(\mathbf{u}^T;\mathbf{v}^T)
			\ldots 
			\Psi(\mathbf{u}^1;\mathbf{v}^1)
			B^{\mathbb{Z}}(x_T,r_T)
			\ldots 
			B^{\mathbb{Z}}(x_1,r_1)
			\\&=
			\Psi(\mathbf{u}^T;\mathbf{v}^T)
			\ldots 
			\Psi(\mathbf{u}^1;\mathbf{v}^1)
			D(r_T^{-2}x_T,r_T^{-1})
			\ldots 
			D(r_1^{-2}x_1,r_1^{-1})
			B^{\mathbb{Z}}(x_T,r_T)
			\ldots 
			B^{\mathbb{Z}}(x_1,r_1)
			\\&\hspace{30pt}\times
			\prod_{i,j=1}^T
			\prod_{\alpha=1}^{k_j}
				\frac{u_\alpha^j-x_i}{u_\alpha^j-r_i^{-2}x_i }
			\frac
			{v_\alpha^j-r_i^{-2}x_i }
			{v_\alpha^j-x_i}
			\\&=
			\Psi(\mathbf{u}^T;\mathbf{v}^T)
			\ldots 
			\Psi(\mathbf{u}^1;\mathbf{v}^1)
			\prod_{i=1}^T
			\underbrace{
				D^{\mathbb{Z}}(r_i^{-2}x_i,r_i^{-1})
				B^{\mathbb{Z}}(x_i,r_i)
			}
			_{\Psi^*(r_i^{-2}x_i,x_i)}
			\\&\hspace{30pt}\times
			\prod_{i,j=1}^T
			\prod_{\alpha=1}^{k_j}
				\frac{u_\alpha^j-x_i}{u_\alpha^j-r_i^{-2}x_i }
			\frac
			{v_\alpha^j-r_i^{-2}x_i }
			{v_\alpha^j-x_i}
			\prod_{1\le i<j\le T}
			\frac{r_i^{-2}x_i-x_j}{x_i-x_j}.
	\end{align*}
	
	Now in the matrix element \eqref{eq:matrix_element_in_marked} we have
	a total of $T $ operators
	$\Psi^*(r_i^{-2}x_i,x_i)$
	in front of the same number of operators
	$\Psi(y_j)$. We can commute these operators through each other
	to form pairs of the operators as $\Psi\Psi^*$. Thanks to
	\eqref{eq:Psi_PsiStar}, this only produces the sign 
	$(-1)^{T(T+1)/2}$.
	Putting this together, for the computation of the 
	matrix element $\bigl\langle e_{\mathbb{Z}_{\le 0}}, 
	(\cdots)e_{\mathbb{Z}_{\le0}} \bigr\rangle$,
	we apply \Cref{prop:action_Psi_many} 
	with the variables
	\begin{equation*}
		\begin{split}
			\left\{ u_i \right\}
			&=
			\left\{ u^i_\alpha\colon 1\le i\le T,\, 1\le \alpha\le k_i \right\}
			\cup \left\{ y_m\colon 1\le m\le T \right\}
			,\qquad 
			\zeta_i\equiv 0;\\
			\left\{ \kappa_i \right\}&=
			\left\{ 0\colon 1\le i\le T,\, 1\le \alpha\le k_i \right\}\cup
			\left\{ r_i^{-2}x_i\colon 1\le i\le T \right\};
			\\
			\left\{ v_i \right\}&=
			\left\{ v^i_\alpha\colon 1\le i\le T,\,1\le \alpha\le k_i \right\}
			\cup
			\left\{ x_i\colon 1\le i\le T \right\}.
		\end{split}
	\end{equation*}
	This produces the following expression for the final
	matrix element
	$\bigl\langle e_{\mathbb{Z}_{\le 0}}, 
	(\cdots)e_{\mathbb{Z}_{\le0}} \bigr\rangle$:
	\begin{align*}
		&
		\prod_{i=1}^T
			v_i
			x_i(1-r_i^{-2})
			\prod_{i,j=1}^{T}
		\biggl(
			\prod_{\alpha=1}^{k_i}\prod_{\beta=1}^{k_j}
			\frac{1}{v^j_\beta-u^i_\alpha}
			\prod_{\alpha=1}^{k_i}
		\frac{1}{x_j-u^i_\alpha}
		\biggr)
		\\&\hspace{30pt}\times
		\prod_{m,i=1}^T
		\biggl(
			\frac{1}{x_i-y_m}
		\prod_{\alpha=1}^{k_i}
		\frac{1}{v^i_\alpha-y_m}
		\biggr)
		\prod_{1\le i<j\le T}(y_i-y_j)(x_j-x_i)
		\\&\hspace{30pt}\times
		\prod_{1\le i<j\le T}
		\biggl(
			\prod_{\alpha=1}^{k_i}
			\prod_{\beta=1}^{k_j}
			(u^i_\alpha-u^j_\beta)(v^j_\beta-v^i_\alpha)
		\biggr)
		\prod_{m,i=1}^{T}\prod_{\alpha=1}^{k_i}
		(u^i_\alpha-y_m)
		\\&\hspace{30pt}\times
		\prod_{i=1}^{T}
		\biggl(
			\prod_{1\le \alpha<\beta\le k_i}(u^i_\alpha-u^i_\beta)(v^i_\beta-v^i_\alpha)
		\biggr)
		\prod_{i,j=1}^{T}
		\prod_{\alpha=1}^{k_i}
		(x_j-v^i_\alpha).
	\end{align*}
	Combining this with all the factors resulting from 
	commutations at previous steps of the proof,
	and with the denominator \eqref{eq:Z_FG_appC_normalizing},
	we get the desired identity.
\end{proof}	

\begin{remark}
	\label{rmk:correlation_functions_gen_function_conditions_nonempty}
	Recall that we assume that
	for some fixed $\varepsilon>0$, we have
	$\varepsilon<y_i<\varepsilon^{-1}$
	and $\varepsilon<s_i<1-\varepsilon$ for all $i$.
	One can check 
	that there exist parameters for which
	all conditions \eqref{eq:correlations_gen_function_conditions} hold
	and the FG process is well-defined 
	(see \Cref{def:spec,def:compatible_spec}).
	For example, we may take the following parameters:
	\begin{equation}
		\label{eq:concrete_parameters}
		\begin{split}
			&
			y_i\approx 1, \ i\ge1;\quad y_i\approx 0.9,\ i\le 0;\quad 
			s_i\approx 0.25, \ i\ge1;\quad 
			s_i\approx 0.95,\ i\le 0;\
			\\
			&
			r_i\approx 0.84;\quad \theta_i\approx 0.84;
			\quad 
			x_i\approx 0.8;\quad
			w_i\approx 0.85.
		\end{split}
	\end{equation}
	Here ``$\approx$'' means that the parameters are very
	close to the corresponding values (for all $i$), 
	but are allowed to be distinct.
	Given \eqref{eq:concrete_parameters}, 
	one readily sees that $u_i,v_j$ satisfying
	\eqref{eq:correlations_gen_function_conditions} 
	also exist.
\end{remark}

\subsection{Correlation kernel}
\label{sub:corr_kernel_general_FG}

We can now compute the correlation kernel for the FG process.

\begin{theorem}
	\label{thm:corr_kernel_no_contours}
	The point process $\mathcal{S}^{(T)}$ \eqref{eq:S_T_notation}
	corresponding to the FG process \eqref{eq:FG_process}
	is determinantal. That is, for any finite
	$A=\left\{ (t,\mathsf{a}^t_\alpha)\colon 1\le t\le T,\,1\le \alpha\le k_t \right\}\subset
	\left\{ 1,\ldots,T  \right\}\times \mathbb{Z}_{\ge1}$ we have
	\begin{equation}
		\label{eq:corr_A_P_equals_det_K_P}
		\mathbb{P}_{\mathscr{P}}
		\bigl[
			A\subset \mathcal{S}^{(T)}
		\bigr]
		=\det
		\left[ K_{\mathscr{P}}(t,\mathsf{a}^t_\alpha;t',\mathsf{a}^{t'}_{\alpha'}) \right].
	\end{equation}
	Here the determinant is of size $k_1+\ldots+k_T$ corresponding to $1\le t,t'\le T$,
	$1\le \alpha\le k_t$, $1\le \alpha'\le k_{t'}$.
	The kernel $K_{\mathscr{P}}$ has the form
	\begin{equation}
		\label{eq:corr_kernel_no_contours}
		\begin{split}	
		&
		K_{\mathscr{P}}(t,a;t',a')=
		\bigl[\Phi_{a'}(u,0)\Phi_{a}^*(0,v)\bigl]
			\,\frac{v}{u-v}
			\prod_{m=1}^{T}
			\frac{(u-y_m)(v-x_m)}{(v-y_m)(u-x_m)}
			\\&\hspace{100pt}\times
			\prod_{i=1}^{t'}
			\frac{u-w_i}{u-\theta_i^{-2}w_i}
			\prod_{i=1}^{t}
			\frac{v-\theta_i^{-2}w_i}{v-w_i}
			\prod_{i=1}^{t'-1}
			\frac{u-x_i}{u-r_i^{-2}x_i}
			\prod_{i=1}^{t-1}
			\frac{v-r_i^{-2}x_i}{v-x_i}.
		\end{split}
	\end{equation}
\end{theorem}
\begin{proof}
	From \Cref{prop:correlations_gen_function} we see that 
	$\mathbb{P}_{\mathscr{P}} \bigl[ A\subset \mathcal{S}^{(T)} \bigr]$
	is the coefficient 
	\begin{align*}
		&
		\left[ \prod_{t=1}^T\prod_{\alpha=1}^{k_t}
		\Phi_{\mathsf{a}^t_\alpha}(u^t_\alpha,0)\Phi^*_{\mathsf{a}^t_\alpha}(0,v^t_\alpha) \right]
			\prod_{1\le i\le j\le T}\prod_{\alpha=1}^{k_j}
			\frac{(v^j_\alpha-\theta_i^{-2}w_i)(u^j_\alpha-w_i)}{(v^j_\alpha-w_i)(u^j_\alpha-\theta_i^{-2}w_i)}
			\prod_{1\le i< j\le T}
			\frac{(u^j_\alpha-x_i)(v^j_\alpha-r_i^{-2}x_i)}{(u^j_\alpha-r_i^{-2}x_i)(v^j_\alpha-x_i)}
			\\&\hspace{60pt}\times
			\prod_{m,i=1}^{T}\prod_{\alpha=1}^{k_i}
			\frac{(u^i_\alpha-y_m)(v^i_\alpha-x_m)}{(v^i_\alpha-y_m)(u^i_\alpha-x_m)}
			\det\left[ \frac{v^{t}_{\alpha}}{u^{t'}_{\alpha'}-v^{t}_{\alpha}} \right],
	\end{align*}
	where we used the Cauchy determinantal formula, and the 
	last determinant is of the same size as in \eqref{eq:corr_A_P_equals_det_K_P}.
	The dependence of the remaining expression 
	is of a product form in the $u^j_\alpha$'s and $v^j_\beta$'s, 
	and we may put this product expression into the determinant. 
	Finally, the operation of extracting the series coefficient
	may also be placed inside the determinant
	thanks to Andr\'eief identity \eqref{eq:Andreief},
	see also \cite{forrester2019meet}.
\end{proof}

\subsection{Specialization to ascending FG processes}
\label{sub:ascending_and_Fock}

Let us now specialize the results for the general FG process 
\eqref{eq:FG_process}
obtained in this 
section to the case of the ascending process
(\Cref{def:FG_asc}). Recall that
in the latter case the correlation kernel is computed
via an Eynard--Mehta type approach (\Cref{thm:ascending_FG_process_kernel}
proven in \Cref{appB:Eynard_Mehta}).
Our aim is to establish the following result
whose proof occupies the rest of this subsection:

\begin{theorem}
	\label{thm:kernel_matching}
	Specialize the correlation kernel for the general FG process
	(given by \Cref{thm:corr_kernel_no_contours} 
	as a generating series coefficient)
	to the case of an ascending FG process.
	Then the series coefficient can be extracted 
	with the help of a double contour integration, 
	which results in the same expression 
	\eqref{eq:ascending_FG_process_kernel_text}
	for the correlation kernel
	$K_{\mathscr{AP}}$
	as the one obtained using 
	the Eynard--Mehta type approach.
\end{theorem}

To match the notation, let us rename the parameter 
$T$ in the general FG process \eqref{eq:FG_process} to 
$N+T$, make the specializations
$(x_1;r_1),\ldots,(x_{T};r_{T})$ empty,
rename
$(x_{T+1};r_{T+1}),\ldots(x_{T+N};r_{T+N})$ to
$(x_{1};r_{1}),\ldots(x_{N};r_{N})$,
and set $w_{T+1}=\ldots=w_{T+N-1}=w_{T+N}=0$.
Furthermore, in \Cref{prop:correlations_gen_function} 
let us take $k_{T+N}=\ldots=k_{T+1}=0$.
One readily sees as in the proof of 
\Cref{prop:correlations_gen_function}
that the correlation generating function becomes
\begin{equation}
	\label{eq:ascending_correlation_genfunc}
	\begin{split}
		&
		\frac{1}{Z}
		\bigl\langle e_{\mathbb{Z}_{\le 0}},
		B^{\mathbb{Z}}(x_N,r_N)\ldots 
		B^{\mathbb{Z}}(x_{1},r_{1}) 
		\Psi(\mathbf{u}^{T};\mathbf{v}^{T})
		D^{\mathbb{Z}}(w_T,\theta_T) 
		\\
		&\hspace{60pt}
		\times
		\Psi(\mathbf{u}^{T-1};\mathbf{v}^{T-1})
		D^{\mathbb{Z}}(w_{T-1},\theta_{T-1})
		\ldots
		D^{\mathbb{Z}}(w_{2},\theta_{2})
		\Psi(\mathbf{u}^{1};\mathbf{v}^{1})
		D^{\mathbb{Z}}(w_{1},\theta_{1})
		e_{\mathbb{Z}_{\le N}}
		\bigr\rangle
		\\&\hspace{2pt}
		=
		\det\left[ \frac{v^{t}_{\alpha}}{u^{t'}_{\alpha'}-v^{t}_{\alpha}} \right]
		\prod_{1\le i\le j\le T}\prod_{\alpha=1}^{k_j}
		\frac{(v^j_\alpha-\theta_i^{-2}w_i)(u^j_\alpha-w_i)}
		{(v^j_\alpha-w_i)(u^j_\alpha-\theta_i^{-2}w_i)}
		\prod_{i=1}^{T}\prod_{\alpha=1}^{k_i}
		\prod_{m=1}^{N}
		\frac{u^i_\alpha-y_m}{v^i_\alpha-y_m}
		\frac{v^i_\alpha-x_m}{u^i_\alpha-x_m}
		,
	\end{split}
\end{equation}
where $Z$ is now given by \eqref{eq:Z_ascending_process},
and the determinant
is 
the same as in \Cref{sub:corr_kernel_general_FG}, that is,
of size $k_1+\ldots+k_T$ such that $1\le t,t'\le T$, 
$1\le \alpha\le k_t$, 
$1\le \alpha'\le k_{t'}$.

Identity \eqref{eq:ascending_correlation_genfunc} holds 
under assumptions \eqref{eq:correlations_gen_function_conditions}
on the parameters which are quite restrictive. In fact, 
some of these assumptions are artifacts of our proof of 
\Cref{prop:correlations_gen_function}
and can be removed:
\begin{lemma}
	\label{lemma:remove_conditions}
	Identity \eqref{eq:ascending_correlation_genfunc}
	holds under the weaker assumptions
	\begin{equation}
		\label{eq:ascending_weaker_conditions}
			\oplus_{x_i;w_j},
			\ 
			\oplus_{v^j_\alpha;w_i},
			\  
			\oplus_{x_j;u^i_\alpha},
			\ 
			\ominus_{u^j_\alpha;\theta_i^{-2}w_i},
			\
			\begin{cases}
			\ominus_{u^i_\alpha;v^j_\beta},& \textnormal{$i>j$ or $i=j,\alpha\ge\beta$};
			\\
			\oplus_{v^j_\beta;u^i_\alpha},&\textnormal{$i<j$ or $i=j,\alpha<\beta$},
			\end{cases}
	\end{equation}
	where we use notation 
	\eqref{eq:oplus_condition},
	\eqref{eq:ominus_condition}.
\end{lemma}
\begin{proof}
	Since the right-hand side of \eqref{eq:ascending_correlation_genfunc}
	is rational, it suffices to show that 
	under \eqref{eq:ascending_weaker_conditions} the
	left-hand side of \eqref{eq:ascending_correlation_genfunc} converges.
	After establishing this, we may drop the unnecessary conditions
	from \eqref{eq:correlations_gen_function_conditions}.

	Observe that possible infinite summations in the left-hand side of
	\eqref{eq:ascending_correlation_genfunc} may arise in two cases.
	Either one of the operators $D^{\mathbb{Z}}$ or $\Psi$ adds a vertical arrow at some $L\ge 1$,
	and then one of the following operators 
	$B^{\mathbb{Z}}$ or $\Psi^*$ removes it;
	or one of the operators $D^{\mathbb{Z}}$
	or $\Psi^*$
	removes a vertical arrow at some $L\le 0$,
	and one of the following operators $\Psi$ adds it back.
	There are no operators $\Psi$ to the left of $B^\mathbb{Z}$, so
	removals of the arrows at $L\le 0$ by $B^{\mathbb{Z}}$
	cannot be compensated and thus do not contribute to the left-hand
	side of \eqref{eq:ascending_correlation_genfunc}.

	We now use \Cref{def:normalized_ops} and 
	$W$ \eqref{eq:weights_W} for $B^{\mathbb{Z}}, D^{\mathbb{Z}}$ and 
	\Cref{thm:Psi_PsiStar_Fock} for $\Psi,\Psi^*$.
	We see that 
	at sufficiently large $L\ge1$, the combination of 
	$D^{\mathbb{Z}}(w,\theta)$ and $B^{\mathbb{Z}}(x,r)$
	produces a factor $\prod_{i=m}^{L}
	\frac{x-s_i^{-2}y_i}{x-y_i}\frac{w-s_i^{-2}y_i}{w-y_i}$, where $m$
	is fixed and $L$ grows. This factor is summable over $L$ under 
	$\oplus_{x;w}$. All other pairs of operators
	are considered similarly. Namely, for $L\ge1$, pairs of operators
	lead to conditions as follows:
	\begin{equation*}
		\begin{pmatrix}
			D^{\mathbb{Z}}(w,\theta),\, \Psi^*(v)\\
			\Psi(u),\,B^{\mathbb{Z}}(x,r)\\
			\Psi(u),\,\Psi^*(v)
		\end{pmatrix}
		\quad\textnormal{leads to}\quad
		\begin{pmatrix}
			\oplus_{v;w}\\
			\oplus_{x;u}\\
			\oplus_{v;u}
		\end{pmatrix}
	\end{equation*}
	And for $L\le0$, we have
	\begin{equation*}
		\begin{pmatrix}
			D^{\mathbb{Z}}(w,\theta),\,\Psi(u)\\
			\Psi^*(v),\,\Psi(u)\\
		\end{pmatrix}
		\quad\textnormal{leads to}\quad
		\begin{pmatrix}
			\ominus_{u;\theta^{-2}w}\\
			\ominus_{u;v}
		\end{pmatrix}.
	\end{equation*}
	Finally, note that the two different 
	cases in \eqref{eq:ascending_weaker_conditions}
	are due to the fact that
	when $\Psi$ comes before or after $\Psi^*$, only the case $L\ge1$ or $L\le0$, respectively, may lead to 
	infinite sums. 
\end{proof}

Fix a finite set
$A=\left\{ (t,\mathsf{a}^t_\alpha)
\colon 1\le t\le T,\,1\le \alpha\le k_t \right\}$, and denote
$k=k_1+\ldots+k_T$ (this is the size of $A$).
Arguing as in the proof of \Cref{thm:corr_kernel_no_contours},
we see that
$\mathbb{P}_{\mathscr{AP}}[A\subset \mathcal{S}^{(T)}]$,
the correlation function 
of the ascending FG process, is equal to the coefficient by 
$\prod_{t=1}^T\prod_{\alpha=1}^{k_t}
\Phi_{\mathsf{a}^t_\alpha}(u^t_\alpha,0)
\Phi^*_{\mathsf{a}^t_\alpha}(0,v^t_\alpha)$
in the expansion of the right-hand side of 
\eqref{eq:ascending_correlation_genfunc}.
By \Cref{lemma:Phi_extraction,lemma:PhiStar_extraction},
this coefficient can be extracted with the help of 
a $2k$-fold contour integral
\begin{equation}
	\label{eq:2k_contour_integral}
	\begin{split}
		&
		\frac{1}{(2\pi \mathbf{i})^{2k}}
		\Biggl(
		\prod_{i=1}^{T}\prod_{\alpha=1}^{k_i}
			\oint_{\Gamma_{y,*}}du^{i}_\alpha
			\oint_{\Gamma_{s^{-2}y,*}}dv^{i}_\alpha
		\Biggr)
		\prod_{t=1}^{T}\prod_{\alpha=1}^{k_t}
		\Biggr(
		\frac{y_{\mathsf{a}^t_\alpha}
		(s_{\mathsf{a}^t_\alpha}^{-2}-1)}
		{v-s_{\mathsf{a}^t_\alpha}^{-2}y_{\mathsf{a}^t_\alpha}}
		\frac{1}{u-y_{\mathsf{a}^t_\alpha}}
		\prod_{j=1}^{\mathsf{a}^t_\alpha-1}
		\frac{v-y_j}{v-s_j^{-2}y_j}
		\frac{u-s_j^{-2}y_j}{u-y_j}
		\Biggr)
		\\&\hspace{40pt}\times
		\det\left[ \frac{1}{u^{t'}_{\alpha'}-v^{t}_{\alpha}} \right]
		\prod_{i=1}^{T}\prod_{\alpha=1}^{k_i}
		\prod_{m=1}^{N}
		\frac{u^i_\alpha-y_m}{v^i_\alpha-y_m}
		\frac{v^i_\alpha-x_m}{u^i_\alpha-x_m}
		\prod_{1\le i\le j\le T}\prod_{\alpha=1}^{k_j}
		\frac{(v^j_\alpha-\theta_i^{-2}w_i)(u^j_\alpha-w_i)}
		{(v^j_\alpha-w_i)(u^j_\alpha-\theta_i^{-2}w_i)}
		.
	\end{split}
\end{equation}
Here each contour
$u^i_\alpha$ goes around all $y_k$ in the positive direction and 
leaves all $s_k^{-2}y_k$ outside, while each contour
$v^j_\beta$ encircles all $s_k^{-2}y_k$
and leaves all $y_k$ outside. Moreover, the contours
might encircle some of the other poles 
$u^i_\alpha=v^j_\beta$, $u^i_\alpha=x_k$, 
$u^i_\alpha=\theta_k^{-2}w_k$,  or $v^j_\beta=w_k$
of the integrand.
These additional residues are not yet specified
because
\Cref{lemma:Phi_extraction,lemma:PhiStar_extraction}
involve series expansions and not actual rational functions.
Therefore, 
we need to
determine 
which of these additional poles the contours in \eqref{eq:2k_contour_integral}
encircle. This is done in the next statement.
\begin{proposition}
	\label{prop:final_K_AP_proof}
	The correlation function 
	$\mathbb{P}_{\mathscr{AP}}[ A\subset \mathcal{S}^{(T)} ]$
	is equal to the $2k$-fold contour integral 
	\eqref{eq:2k_contour_integral},
	where:
	\begin{enumerate}[$\bullet$]
		\item the integration contour
			for each $u^i_\alpha$ is positively oriented, encircles
			all $y_k,\theta_k^{-2}w_k$, and does not encircle any of $x_k$;
		\item 
			the integration contour for
			each $v^j_\beta$ is 
			negatively oriented, encircles 
			all $y_k,w_k$, and does not encircle any of $s_k^{-2}y_k$;
		\item 
			the contour $u^i_\alpha$ contains the contour
			$v^j_\beta$
			for $i>j$ or $i=j$, $\alpha\ge \beta$;
			and 
			the
			$v^j_\beta$
			contains $u^i_\alpha$ otherwise.
	\end{enumerate}
\end{proposition}
\begin{proof}
	First, let us take
	the parameters close to each other as follows:
	\begin{equation}
		\label{eq:particular_parameters_analytic_continuation}
		\begin{split}
			&
			y_i\approx y, \ i\in \mathbb{Z};\qquad 
			s_i\approx s, \ i\in \mathbb{Z};\qquad 
			r_i\approx r;\quad \theta_i\approx \theta;\quad
			x_i\approx x;\quad
			w_i\approx w,
		\end{split}
	\end{equation}
	where
	\begin{equation*}
		\left|\frac{x-s^{-2}y}{x-y}\right|<\left|\frac{w-s^{-2}y}{w-y}\right|,
	\end{equation*}
	and the nonnegativity of the specializations
	(\Cref{def:spec}) holds. 
	In \eqref{eq:particular_parameters_analytic_continuation},
	``$\approx$'' means that the parameters are very
	close to the corresponding values (for all $i$), 
	but are all distinct.
	One can check that such a choice of 
	$x,r,w,\theta,y,s$
	exists, for example,
	$y=1,s=0.25,x=0.3,r=0.5,w=0.43,\theta=0.64$.

	Under
	\eqref{eq:particular_parameters_analytic_continuation}, 
	conditions $\oplus_{a;b}$ and $\ominus_{a;b}$ are essentially the same since 
	there is no difference between $(y_j,s_j)$ with negative and positive indices.
	Consider the map (and its inverse)
	\begin{equation*}
		U\mapsto \Xi=\Xi(U):=\frac{U-s^{-2}y}{U-y},
		\qquad 
		\Xi\mapsto U=U(\Xi)=\frac{y(s^{-2}-\Xi)}{1-\Xi}.
	\end{equation*}
	Clearly, $\Xi$ maps $y_j$ close to $\infty$, and $s_j^{-2}y_j$ 
	close to $0$.
	We also see that conditions \eqref{eq:ascending_weaker_conditions}
	are satisfied if the variables $u^i_\alpha,v^j_\beta$ are chosen so that
	$|\Xi(x)|< |\Xi(u^i_\alpha)|< |\Xi(w)|$,
	$|\Xi(x)|< |\Xi(v^j_\beta)|< |\Xi(w)|$,
	$|\Xi(u^i_\alpha)|< |\Xi(\theta^{-2}w)|$,
	and
	\begin{equation*}
		\begin{cases}
			|\Xi(u^i_\alpha)|<|\Xi(v^j_\beta)|,& \textnormal{$i>j$ or $i=j,\alpha\ge\beta$};
			\\
			|\Xi(v^j_\beta)|<|\Xi(u^i_\alpha)|,&\textnormal{$i<j$ or $i=j,\alpha<\beta$}.
		\end{cases}
	\end{equation*}
	We claim that the integration 
	contours for $u^i_\alpha$, $v^j_\beta$ satisfying all the required conditions
	exist. Indeed, one can simply take 
	$u^i_\alpha=U(c^i_{\alpha}e^{-\mathbf{ i}t})$, 
	$v^j_\beta=U(d^j_{\beta}e^{\mathbf{ i}t})$, where
	$0<t<2\pi$
	and
	\begin{equation*}
		\begin{cases}
			0<c^i_\alpha<d^j_\beta,& \textnormal{$i>j$ or $i=j,\alpha\ge\beta$};
			\\
			0<d^j_\beta<c^i_\alpha
			,&\textnormal{$i<j$ or $i=j,\alpha<\beta$}.
		\end{cases}
	\end{equation*}
	In particular, the radii $c^{i}_\alpha,d^j_\beta$ are interlacing
	in a certain way.
	Since these radii can be arbitrarily close to each other, this can be achieved.
	Note the different orientation of the $u$ and the $v$ contours which is due to the 
	fact that $\Xi$ maps $y$ to infinity. 
	This agrees with \Cref{lemma:Phi_extraction,lemma:PhiStar_extraction}
	in that the
	$u^i_\alpha$ contours must go around all $y_j$ in the positive direction, 
	and 
	the $v^j_\beta$ contours must go around all $s_j^{-2}y_j$ in the negative
	direction. Both families of contours should separate
	$\{y_k \}$ from $\{ s_k^{-2}y_k \}$.

	Moreover, inequalities for 
	$|\Xi(u^i_\alpha)|$ and 
	$|\Xi(v^j_\beta)|$
	listed above imply that the integration contours
	are as described in the claim of the proposition.
	Therefore, by \Cref{lemma:Phi_extraction,lemma:PhiStar_extraction},
	in the case when all the similarly named parameters
	are close to each other as in 
	\eqref{eq:particular_parameters_analytic_continuation}, we may 
	extract the desired coefficient 
	$\mathbb{P}_{\mathscr{AP}}[ A\subset \mathcal{S}^{(T)} ]$ from the 
	right-hand side of \eqref{eq:ascending_correlation_genfunc}
	by means of integration over the $2k$
	contours described above in the proof.

	\medskip

	To complete the proof in the general case, we use analytic continuation.
	First, 
	a straightforward a priori argument
	(like in \cite[Lemma 8.10]{BorodinPetrov2016inhom})
	shows that under
	$\oplus_{x_i;w_j}$ (for all $i,j$),
	any correlation
	function
	$\mathbb{P}_{\mathscr{AP}}[ A\subset \mathcal{S}^{(T)} ]$
	of the FG process is a rational function depending on
	a finite subset of the parameters of the process (the size of the subset
	depends on~$A$). 
	Second, the $2k$-fold contour integral 
	\eqref{eq:2k_contour_integral} over the contours
	described above in the proof
	is also a rational function
	because it is a finite sum of residues of the integrand.
	These two rational functions are equal on an open 
	full-dimensional 
	subset in the finite-dimensional space of the 
	parameters that they depend on. Therefore, these
	functions are equal in general,
	provided that the correlation function 
	$\mathbb{P}_{\mathscr{AP}}[ A\subset \mathcal{S}^{(T)} ]$
	is well-defined (i.e., under $\oplus_{x_i;w_j}$)
	and the $2k$-fold integral is taken over the same residues
	as before the analytic continuation.
\end{proof}

By applying
Andr\'eief identity \eqref{eq:Andreief}
(see also \cite{forrester2019meet}),
the $2k$-fold contour integral 
\eqref{eq:2k_contour_integral}
(over the contours described in \Cref{prop:final_K_AP_proof})
is rewritten as a determinant
\begin{equation*}
	\mathbb{P}_{\mathscr{AP}}\bigl[ A\subset \mathcal{S}^{(T)} \bigr]
	=
	\det
	\bigl[
		K_{\mathscr{AP}}(t,\mathsf{a}^t_\alpha;t',\mathsf{a}^{t'}_{\alpha'}) 
	\bigr]
\end{equation*}
of the correlation kernel $K_{\mathscr{AP}}(t,a;t',a')$
given by the double contour integral \eqref{eq:ascending_FG_process_kernel_text}.
The determinant is of size $k$, indexed by
$1\le t,t'\le T$ with
$1\le \alpha\le k_t$, 
$1\le \alpha'\le k_{t'}$.
Let us add two remarks:
\begin{enumerate}[$\bullet$]
	\item 
		The sign difference in the term
		$1-s_a^{-2}$ between \eqref{eq:2k_contour_integral}
		and the kernel 
		$K_{\mathscr{AP}}$ \eqref{eq:ascending_FG_process_kernel_text},
		is due to reversing the direction of the $v$ contour.
	\item 
		The conditions
		$\ominus_{u^i_\alpha;v^j_\beta}$ 
		and
		$\oplus_{v^j_\beta;u^i_\alpha}$ 
		in \eqref{eq:ascending_weaker_conditions}
		depending on the relative order of the indices
		$(i,\alpha)$ and $(j,\beta)$
		translate to 
		$\ominus_{u;v}$ for $t'\ge t$ 
		and $\oplus_{v;u}$ for $t'<t$
		in the double contour integral kernel
		$K_{\mathscr{AP}}(t,a;t',a')$.
\end{enumerate}

Overall, we see that 
for the ascending FG process,
the fermionic operator approach
developed in 
\Cref{sec:fermionic_operators,sec:Fock_and_FG_process}
and the Eynard--Mehta type approach from
\Cref{appB:Eynard_Mehta}
produce the same correlation kernel $K_{\mathscr{AP}}$.
This completes the proof of \Cref{thm:kernel_matching}.

\newpage

\part{Random Tilings}
\label{partIII}

In this part we represent determinantal point processes from
\Cref{partII} as a certain inhomogeneous dimer model (which can also be viewed as 
a model of random domino tilings), and study the bulk
asymptotic behavior of the model.

\section{Dimers and domino tilings}
\label{sec:domino_tilings}

In this section we interpret the ascending FG process defined in \Cref{sec:FG_measures}
as a nonintersecting path model and a dimer model,
and prove \Cref{thm:intro_dimers_dominoes}
from Introduction.

\subsection{Layered five vertex model}
\label{sub:five_vertex_models}

\begin{figure}[htpb]
	\centering
	\includegraphics[width=.8\textwidth]{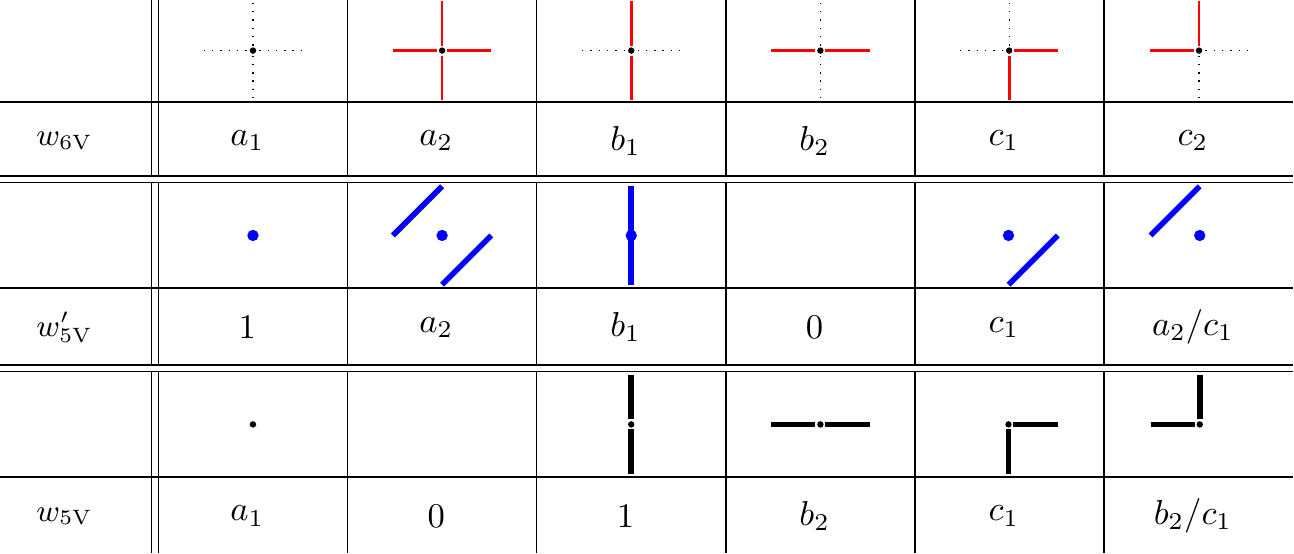}
	\caption{The six vertex weights and two families of five vertex weights.}
	\label{fig:6V_5V_5V}
\end{figure}

Let us take six vertex weights $w_{\mathrm{6V}}(i_1,j_1;i_2,j_2)$ as in 
\Cref{fig:6V_5V_5V}, top. Assume that they are free fermionic, that is, 
$a_1a_2+b_1b_2=c_1c_2$. Moreover, let $c_1\ne 0$.

Define two families of \emph{five vertex} weights, $w'_{\mathrm{5V}}$ and $w_{\mathrm{5V}}$, as
in \Cref{fig:6V_5V_5V}, middle and bottom, respectively. One readily sees that these vertex 
weights also satisfy the free fermion condition. 

The six vertex configuration can be replaced by a vertical
concatenation of two five vertex configurations
\cite[Section 4.7]{wheeler2018hall}.
Let us recall the construction.
Consider a stacked two-vertex configuration with the weight $w_{\mathrm{5V}}'$ at the top, and weight
$w_{\mathrm{5V}}$ at the bottom. Then we claim that 
these two five vertex weights produce the same partition function as $w_{\mathrm{6V}}$:
\begin{lemma}
	\label{lemma:from_6V_to_5V}
	For any fixed $I_1,I_2,J_1,j_2,j_2'\in  \left\{ 0,1 \right\}$, we have
	\begin{equation*}
		\sum_{j_1,j_1',k\in \left\{ 0,1 \right\}} 
		w_{\mathrm{5V}}(I_1,j_1;k,j_2)\,
		w_{\mathrm{5V}}'(k,j_1';I_2,j_2')\,\mathbf{1}_{j_1+j_1'=J_1}
		=
		w_{\mathrm{6V}}(I_1,J_1;I_2,j_2+j_2').
	\end{equation*}
	In particular, if $j_2+j_2'$ is greater than $1$, then the left-hand side vanishes.
\end{lemma}
\begin{proof}
	This is done by a straightforward verification.
	Let us illustrate just two cases. First, for $I_1=I_2=J_1=1$, we have
	two configurations to be considered separately (as they correspond to different 
	exit boundary conditions $(j_2,j_2')$):
	\begin{equation*}
		\includegraphics{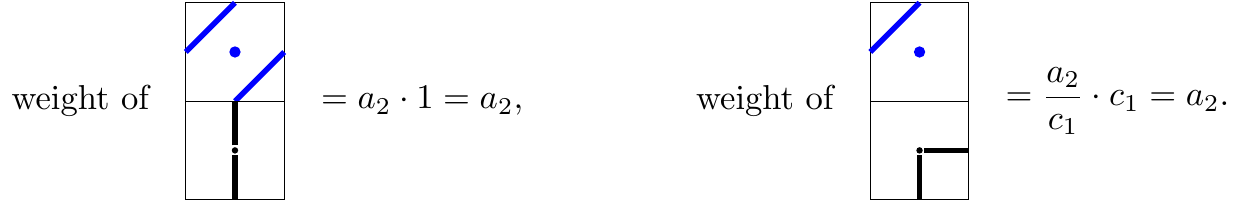}
	\end{equation*}
	Second, for $I_1=0$, $J_1=1$, $I_2=1$, we have two configurations 
	to be considered together (as they have the same $(j_2,j_2')$):
	\begin{equation*}
		\includegraphics{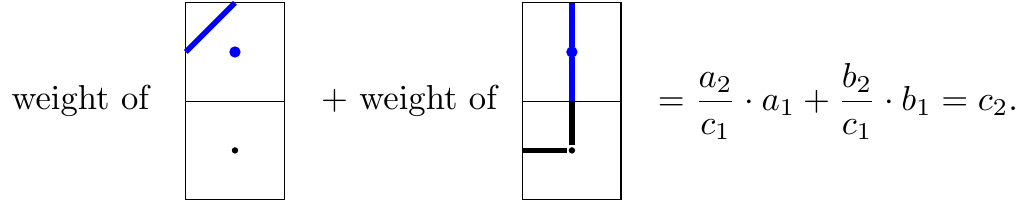}
	\end{equation*}
	All other cases are obtained similarly.
\end{proof}

\Cref{lemma:from_6V_to_5V} implies that the
ascending FG process $\mathscr{AP}(\lambda^{(1)},\ldots,\lambda^{(T)} )$
\eqref{eq:ascending_process} 
can be realized as a partition function of a 
path configuration in 
$\mathbb{Z}_{\ge1}\times \mathscr{I}_{T,N}$,
where 
\begin{equation}
	\label{eq:I_T_N_layered_stuff}
	\mathscr{I}_{T,N}=\{1,1',\ldots,T,T',T+1,(T+1)',\ldots,T+N,(T+N)'\}.
\end{equation}
Namely, take
the vertex weights at the odd horizontals 
(numbered $1\le i\le T+N$)
to be
$W_{\mathrm{5V}}$, $1\le i\le  T$ or $\widehat{W}_{\mathrm{5V}}$, $T+1\le i\le T+N$,
and the vertex weights at the even horizontals (numbered by $i'$, $1\le i\le T+N$)
to be 
$W'_{\mathrm{5V}}$ or $\widehat{W}_{\mathrm{5V}}$ in a similar way.
Here $W_{\mathrm{5V}}$, $\mathrm{W}_{5V}'$ are 
constructed from the six vertex weights
$W$ \eqref{eq:weights_W} as in \Cref{fig:6V_5V_5V}, 
and similarly $\widehat{W}_{\mathrm{5V}},\widehat{W}_{\mathrm{5V}}'$
are constructed from $\widehat{W}$ \eqref{eq:weights_W_hat}.
The boundary conditions in 
$\mathbb{Z}_{\ge1}\times \mathscr{I}_{T,N}$,
are the same as the boundary conditions for the ascending FG process: there are $N$ paths entering from below, and 
these $N$ paths exit far to the right through the topmost $N$ odd horizontals.
See \Cref{fig:FG_process_5vertex} for an illustration.

\begin{proposition}
	\label{prop:two_five_vertex_models}
	With the above notation, the joint distribution of the arrow configurations in the 
	layered five vertex model as in \Cref{fig:FG_process_5vertex}, bottom,
	joining 
	horizontals $i'$ and $i+1$, $1\le i\le T$,
	is the same as the joint distribution of $\mathcal{S}(\lambda^{(i)})$, $1\le i\le T$, 
	under the ascending FG process \eqref{eq:ascending_process}.
	(Here we are using notation $\mathcal{S}(\lambda)$ from \eqref{eq:S_lambda_notation}.)
\end{proposition}

Note that the paths in the layered five vertex model
are drawn to be nonintersecting (as in \Cref{fig:FG_process_5vertex}, bottom).
This allows to identify this model with a dimer model in \Cref{sub:dimer_model} below.

\begin{figure}[ht]
	\centering
	\includegraphics[width=.7\textwidth]{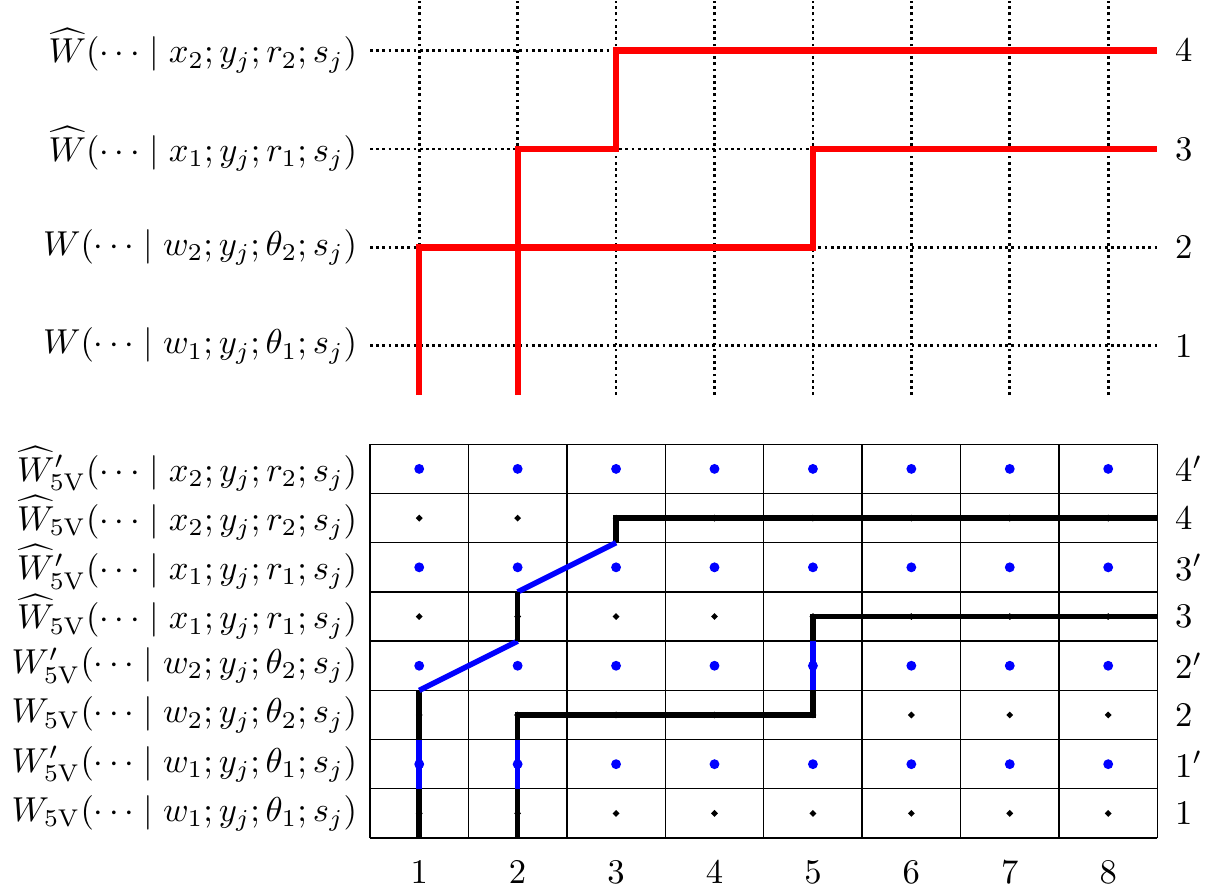}
	\caption{Top: A path configuration under an FG process with $N=T=2$,
		where $\lambda^{(1)}=(0,0)$ and $\lambda^{(2)}=(3,1)$.
		Bottom: 
	One of the possible path configurations under the five vertex realization of the 
	FG process, which corresponds to the particular six vertex configuration given at the top.
	The index $j$ in $y_j,s_j$ is the horizontal coordinate.}
	\label{fig:FG_process_5vertex}
\end{figure}
\begin{figure}[ht]
	\centering
	\includegraphics[width=.6\textwidth]{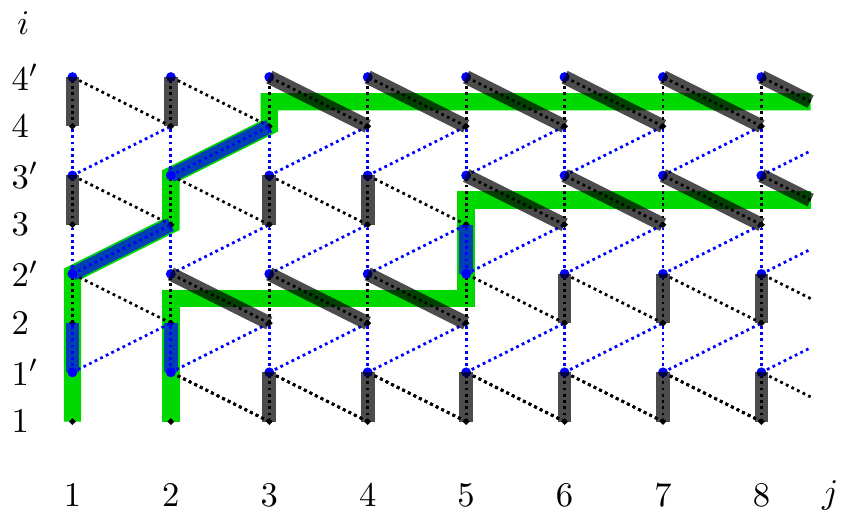}
	\caption{Graph in $\mathbb{Z}_{\ge1}\times \mathscr{I}_{T,N}$ 
		(with $T=N=2$)
		and a dimer covering corresponding to
		the layered five vertex 
		configuration in \Cref{fig:FG_process_5vertex}, bottom. The 
		paths of the layered five-vertex model are shown in green.}
	\label{fig:dimer_rail_yard}
\end{figure}

\begin{figure}[ht]
	\centering
	\includegraphics[width=.7\textwidth]{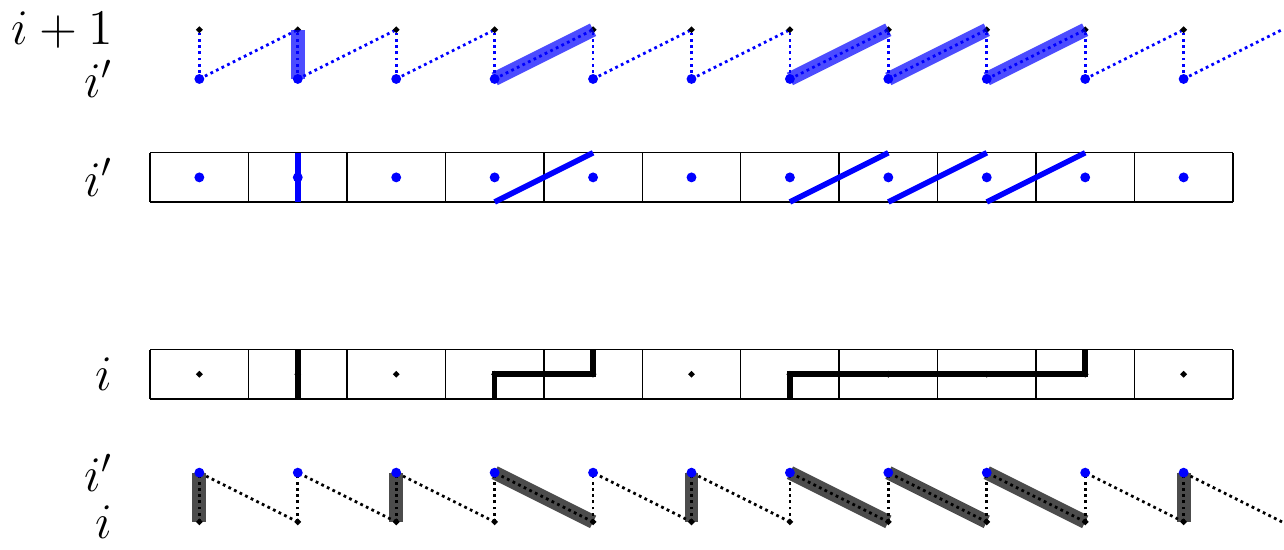}
	\caption{Illustration of the correspondence between five vertex path
	configurations and dimer configurations at two different types of
	layers.}
	\label{fig:dimer_rail_yard_layers}
\end{figure}

\subsection{Dimer model}
\label{sub:dimer_model}

Consider a layered bipartite graph $\mathscr{G}_{T,N}$ 
with vertices $\mathbb{Z}_{\ge1}\times \mathscr{I}_{T,N}$
(cf. \eqref{eq:I_T_N_layered_stuff}) in which edges connect 
the following vertices:
\begin{equation*}
	(j,i)-(j,i'),\qquad 
	(j,i)-(j-1,i'),\qquad 
	(j,i')-(j,i+1),\qquad 
	(j,i')-(j+1,i+1),
\end{equation*}
where $j\ge1$ and $1\le i\le T+N$.
In addition, remove vertices 
$(1,1),(2,1),\ldots,(N,1) $
from the graph together with all edges incident to these vertices.
See \Cref{fig:dimer_rail_yard} for an illustration.
This graph is equivalent to a particular case of a rail-yard graph \cite{boutillier2015dimers}.

Then we construct a one-to-one mapping from the layered five vertex configuration to 
a dimer covering
(i.e., a perfect matching)
of the graph $\mathscr{G}_{T,N}$. This is done layer by layer as in 
\Cref{fig:dimer_rail_yard_layers}.
Recall that the probability weight of a particular dimer covering
is proportional to the product of the weights of all edges that are covered.
We refer to 
\cite{KOS2006},
\cite{Kenyon2007Lecture},
\cite{gorin2021lectures}
for basics on dimer models on bipartite graphs.
Note that due to the behavior of the five vertex 
paths far to the right, in our dimer covering
far to the right 
we will 
almost surely see only dimers $(j,i)-(j,i')$ for $1\le i\le T$, and only dimers $(j,i)-(j-1,i')$ 
for
$T+1\le i\le T+N$. This also follows from 
the form of our edge weights:

\begin{figure}[ht]
	\centering
	\includegraphics[width=\textwidth]{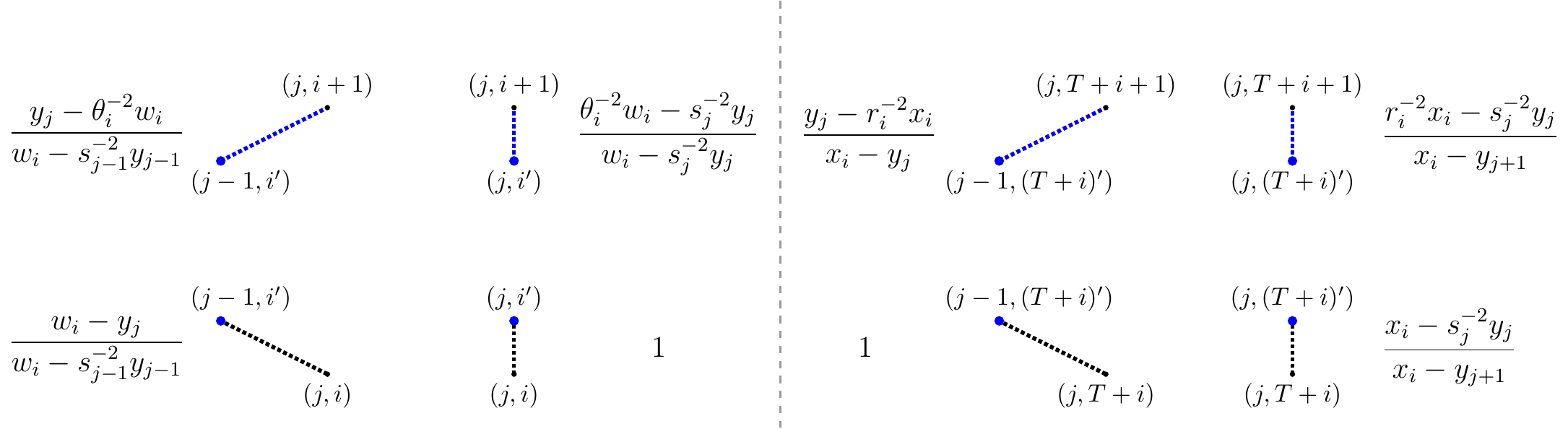}
	\caption{Edge weights in the dimer model representing the layered 
		five vertex model from \Cref{fig:FG_process_5vertex}, bottom. The left four weights correspond to the 
	lower $T$ rows, so $1\le i\le T$. The right four weights appear in the upper $N$ rows, 
	and there we have $1\le i\le N$.}
	\label{fig:dimer_edge_weights}
\end{figure}

\begin{proposition}
	\label{prop:five_vertex_to_dimers}
	Under the identification between the layered five vertex model
	and the dimer model
	as in 
	\Cref{fig:dimer_rail_yard,fig:dimer_rail_yard_layers}, 
	the probability measure (coming from the FG process)
	is equivalent to the dimer model
	with the edge weights given in \Cref{fig:dimer_edge_weights}.
\end{proposition}
\begin{proof}
	From the
	identification between the five vertex paths 
	and the dimer covering (see \Cref{fig:dimer_rail_yard_layers}),
	using the six to five vertex conversion table
	in \Cref{fig:6V_5V_5V}, 
	we obtain 
	dimer weights as in \Cref{fig:dimer_proof_new}.%
	\begin{figure}[ht]
		\centering
		\includegraphics[width=.9\textwidth]{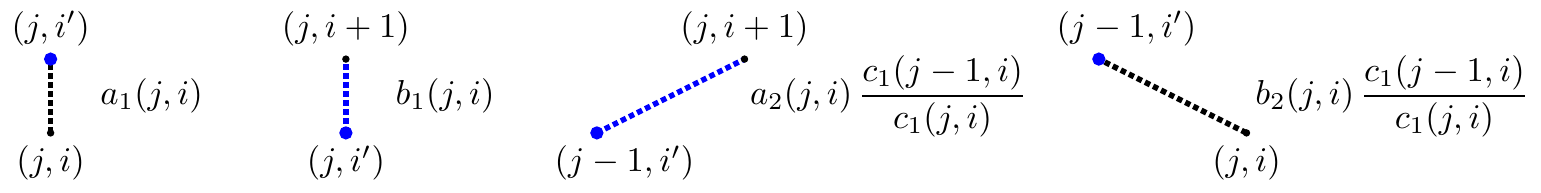}
		\caption{Dimer weights in the proof of \Cref{prop:five_vertex_to_dimers}.}
		\label{fig:dimer_proof_new}
	\end{figure}
	The concrete values of $a_1,a_2,b_1,b_2,c_1,c_2$
	depending on the coordinates $(j,i)$
	are equal to
	$W$
	\eqref{eq:weights_W}
	or $\widehat W$
	\eqref{eq:weights_W_hat},
	as depicted in \Cref{fig:FG_process_5vertex}.

	One readily sees that 
	in the $W$ part, 
	the weights in \Cref{fig:dimer_proof_new}
	produce the left four weights in \Cref{fig:dimer_edge_weights},
	as desired.

	In the $\widehat W$ part (containing the parameters $x_j,r_j$)
	let us in addition multiply the dimer weights around each vertex $(j,i')$ by 
	\begin{equation*}
		\frac{c_1(j+1,i)}{c_1(j,i)}=\frac{x_i-y_j}{x_i-y_{j+1}}.
	\end{equation*}
	This does not change the dimer model on finite subgraphs
	(cf. \cite[Section 3.10]{Kenyon2007Lecture}), but
	makes the weight of each edge $(j,i)-(j-1,i')$ to be $1$. This
	agrees with the fact that such dimers appear infinitely often 
	in the full dimer model on our infinite graph.
	This leads to the right four weights in \Cref{fig:dimer_edge_weights},
	and so we are done.
\end{proof}

\subsection{Random domino tilings}
\label{sub:tilings_from_dimers}

The representation of the FG process as a 
dimer model on a bipartite graph
described in \Cref{sub:dimer_model}
is useful for asymptotic analysis (performed below in this part),
mainly
due to 
the clear dependence of the edge weights on the Cartesian coordinates 
in the plane, cf.
\Cref{fig:dimer_edge_weights}.
Here we provide its equivalent interpretation as an
inhomogeneous domino tiling model.

In the lattice in \Cref{fig:dimer_rail_yard},
shift each row $i'$ to the right by $\frac{1}{2}$. This
transforms the lattice 
into a subset of the square lattice,
see
\Cref{fig:domino_from_dimer}, left.
Then rotate the whole picture
$45^\circ$ clockwise, and interpret dimers as $1\times 2$ dominoes. 
In this way, 
the ascending FG process \eqref{eq:ascending_process} is represented 
as a random domino tiling of the half-infinite strip
with the zigzag boundary
and
with $N$ additional
unit squares removed from the top of the
southwest part of the boundary.
The corresponding domino tiling is 
given in 
\Cref{fig:domino_from_dimer}, right. The domino weights
are inhomogeneous and, moreover,
depend on the parity of the 
coordinates. The weights are also given in 
\Cref{fig:domino_from_dimer}.
Note that this domino tiling model and the weights
are the same as in \Cref{fig:intro_dimers,fig:intro_dimer_weights}
from Introduction, up to a rotation by $45^\circ$.
Thus, we have completed the 
proof of
\Cref{thm:intro_dimers_dominoes}.

\medskip

\begin{figure}[htp]
	\centering
	\includegraphics[width=\textwidth]{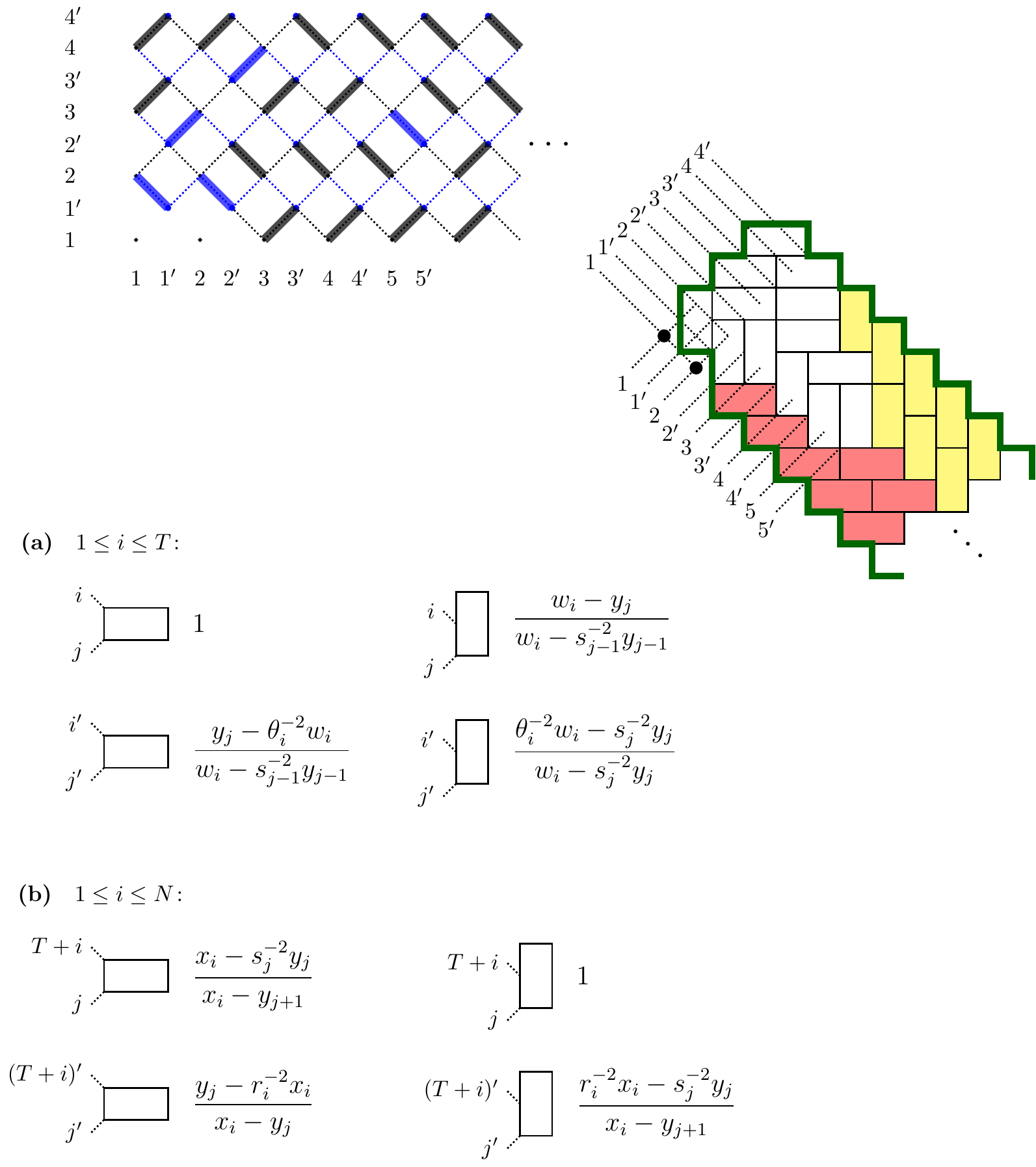}
	\caption{Mapping from the dimer model 
		in \Cref{fig:dimer_rail_yard}
		with weights given in \Cref{fig:dimer_edge_weights}
		to domino tilings of a half-strip.
		The dominoes which are repeated infinitely many times in the 
		down-right direction are 
		shaded.}
	\label{fig:domino_from_dimer}
\end{figure}

Random domino tilings (in particular, of the Aztec Diamond) is a classical subject
in combinatorics and probability
\cite{elkies1992alternating},
\cite{cohn-elki-prop-96},
\cite{CohnKenyonPropp2000},
\cite{KOS2006}.
Our model in 
\Cref{fig:domino_from_dimer}
is a particular case of a larger family of 
domino tilings, the 
\emph{steep tilings} of an infinite strip \cite{bouttier2017aztec}.
While our inhomogeneous domino weights are more general
than that in steep tilings,
the latter
allow for more general 
boundary conditions which are not yet available in 
our FG process setup. 
Moreover, in steep tilings
one can prescribe an arbitrary sequence 
$\in \left\{ +,- \right\}^{2\ell}$
of
asymptotic domino directions at infinity
(i.e., the directions of the shaded
dominoes in
\Cref{fig:domino_from_dimer} repeating infinitely many times).
We remark that different asymptotic directions of dominoes
may be modeled in our setup by 
passing to the fully general
FG processes (\Cref{sub:general_FG}),
but we will not consider this generality in the present work.

Below in this \Cref{partIII} we discuss bulk asymptotics of the
inhomogeneous domino tiling model displayed in \Cref{fig:domino_from_dimer}
coming from ascending FG processes.

\begin{remark}[Noncolliding lattice walks]
	\label{rmk:noncolliding_lattice_paths}
	When $w_i=y_j$ for all $i,j$, one of the dimer weights vanishes,
	see the left part of \Cref{fig:dimer_edge_weights}. 
	Thus, 
	in the bottom
	$T$ double rows in \Cref{fig:dimer_rail_yard}
	we can erase the edges carrying the zero weight. 
	In this way we obtain the hexagonal lattice. 
	Therefore, when $w_i \equiv y_j$,
	one can interpret the dimer model
	as a model of $N$ noncolliding lattice walks
	as in, e.g., \cite{konig2002non}
	(see also \cite{gorin2019universality} 
	for an equivalent lozenge tiling picture).
	The endpoints of these noncolliding lattice walks are 
	distributed according to the probability weights 
	coming from $F_{\lambda}(\mathbf{x};\mathbf{y};\mathbf{r};\mathbf{s})$.
	On the other hand, the bulk asymptotics of our dimer model 
	would lead to a certain class of random lozenge tilings of the whole plane.
	We briefly discuss these measures in 
	\Cref{sec:inhom_sine_kernel}.
\end{remark}

\section{Asymptotics in the bulk}
\label{sec:asymptotics}

\subsection{Scaling. Global parameters and local sequences}
\label{sub:scaling_and_parameter_assumptions}

In this section we perform bulk asymptotic analysis of the correlation
kernel $K_{\mathscr{AP}}(t,a;t',a')$ \eqref{eq:ascending_FG_process_kernel_text}.
By ``bulk asymptotics'' we mean the scaling around a global
position far from the boundary of the system, 
such that 
discrete lattice structure is preserved.
More precisely, take $N\to+\infty$ to be an independent parameter
going to infinity, and set for the variables in the kernel:
\begin{equation}
	\label{eq:scaling_parameters_1}
	T\gg N,\qquad 
	t=\lfloor \tau N \rfloor ,\quad t'=t+\Delta t ,\qquad
	a=\lfloor \alpha N \rfloor ,\quad a'=a+\Delta a,
\end{equation}
where $\tau,\alpha\in \mathbb{R}_{>0}$,
$\Delta t,\Delta a\in \mathbb{Z}$ are fixed.
Note that since the kernel 
$K_{\mathscr{AP}}$
does not depend on $T$, we just need to take $T$ large enough so that $t,t'\le T$
could grow linearly with $N$.

\begin{remark}
	\label{rmk:edge_limit}
	Along with the bulk limit behavior,
	dimer models typically possess other interesting scaling limits.
	In particular, the limit at the edge of the liquid region
	(cf. \Cref{fig:liquid_region} below)
	should bring the Airy kernel or its multiparameter
	deformations obtained in \cite{BorodinPeche2009}. 
	We do not anticipate
	new determinantal kernels at the edge in the case of generic inhomogeneous
	parameters, but it would be interesting to probe
	whether special choices of the parameters lead to interesting phase transitions in the
	edge behavior. We leave this question out of the scope of the present paper.
\end{remark}

As we aim to capture a nontrivial lattice limit in the bulk, we may set the parameters
$x_i,w_i,\theta_i,y_j,s_j$ of the ascending FG process
to be constant outside a finite neighborhood 
of the global position.
Let this neighborhood be of
size $L\in \mathbb{Z}_{\ge1}$, and later
(in \Cref{sec:inhom_sine_kernel}) we also take
the \emph{cutoff}
$L$ to infinity.
While this restricts the generality of the global limit shape
and global fluctuations
(such as the Gaussian Free Field fluctuations, cf. \cite{Kenyon2001GFF}, \cite{Petrov2012GFF}),
this specialization does not restrict the local lattice behavior.
More precisely, set
$x_i=x_*$ for all $i=1,\ldots,N$,
\begin{equation}
	\label{eq:w_local_global}
	w_i=\begin{cases}
		w_*,&|i-\lfloor \tau N \rfloor|> L;\\
		w^\circ_{i-\lfloor \tau N \rfloor },&
		|i - \lfloor \tau N \rfloor |\le L,
	\end{cases}
		\qquad
		\theta_i=\begin{cases}
			\theta_*,&|i-\lfloor \tau N \rfloor |> L;\\
		\theta^\circ_{i-\lfloor \tau N \rfloor },&
		|i - \lfloor \tau N \rfloor| \le L,
	\end{cases}
\end{equation}
and similarly 
\begin{equation}
	\label{eq:y_local_global}
	y_j=
	\begin{cases}
		y_*,& |j-\lfloor \alpha N \rfloor | > L;\\
		y^\circ_{j-\lfloor \alpha N \rfloor }, &
		|j-\lfloor \alpha N \rfloor | \le L,
	\end{cases}
	\qquad 
	s_j=
	\begin{cases}
		s_*,& |j-\lfloor \alpha N \rfloor | > L;\\
		s^\circ_{j-\lfloor \alpha N \rfloor }, &
		|j-\lfloor \alpha N \rfloor | \le L.
	\end{cases}
\end{equation}
In other words, we 
have passed to the global parameters $x_*,w_*,\theta_*,y_*,s_*$,
and the local
sequences $\{w_i^\circ\},\{\theta_i^\circ\},\{ y_j^\circ \},
\{ s_j^\circ \}$ with $|i|,|j|\le L$.
For convenience, let us also write $w_i^\circ=w_*$ for $|i| > L$, 
and similarly for
$y_j^\circ,\theta_i^\circ,s_j^\circ$.
Both the global parameters and the local sequences are fixed
and do not depend on
$N$.

\medskip

All the parameters
of the ascending FG process
must satisfy \eqref{eq:ys_FG_parameters},
\eqref{eq:x_r_spec_parameters},
and
\eqref{eq:compatible_specializations} in order for the
process to be well-defined as a probability distribution 
\eqref{eq:ascending_process}
on 
sequences of signatures. 
Under the scaling assumptions
\eqref{eq:w_local_global}--\eqref{eq:y_local_global},
these conditions read
\begin{equation}
	\label{eq:params_conditions_for_asymptotics}
	x_*<y_j,
	\qquad 
	w_i<y_j<\theta_{i}^{-2}w_i<s_j^{-2}y_j,
	\qquad 
	\left|
	\frac{x_*-s_*^{-2}y_*}{x_*-y_*}
	\frac{w_*-y_*}{w_*-s_*^{-2}y_*}
	\right|<1-\delta<1
\end{equation}
for all $i,j$.
Here we have dropped the assumptions $x_i,w_i,y_j>0$, 
cf. \Cref{rmk:nonnegative_condition_on_x_for_convenience}.
The ascending FG process does not depend on the parameters $r_i$,
so there are no conditions on the $r_i$'s.

The last inequality in
\eqref{eq:params_conditions_for_asymptotics}
is needed for the convergence of the 
series for the normalizing constant 
\eqref{eq:Z_ascending_process}
in the 
FG process weights 
(which in general is guaranteed by \eqref{eq:compatible_specializations}). 
Clearly, the convergence of this series is determined only by the global parameters.

By taking $x_*$ sufficiently small (close to $-\infty$), 
we see that the last inequality in \eqref{eq:params_conditions_for_asymptotics}
follows from $w_*<y_*<s_*^{-2}y_*$ and hence hold automatically.
Therefore, we may and will assume that the global
parameters
satisfy
\begin{equation}
	\label{eq:5_points}
	x_*<w_*<y_*<\theta_*^{-2}w_*<s_*^{-2}y_*.
\end{equation}

\subsection{Steepest descent}
\label{sub:steepest_descent}

Let us rewrite the correlation kernel
$K_{\mathscr{AP}}$
\eqref{eq:ascending_FG_process_kernel_text}
in the following form adapted to the scaling regime 
from \Cref{sub:scaling_and_parameter_assumptions}:
\begin{equation}
	\label{eq:K_AP_asymptotic_equivalence_1}
	\begin{split}
		&K_{\mathscr{AP}}(t,a;t',a')
		\\&\hspace{5pt}=
		\frac{1}{(2\pi\mathbf{i})^2}
		\oint\oint
		\frac{\mathscr{E}_N(u)/\mathscr{E}_N(v)\, du \, dv}{u-v}\,
		\frac{y_{a}(1-s_{a}^{-2})}{v-y_{a}}
		\frac{1}{u-s_{a'}^{-2}y_{a'}}
		\frac{
			\prod_{c=1}^{a'}
			\frac{u-s_c^{-2}y_c}{u-y_c}
		}{
			\prod_{c=1}^{a}
			\frac{u-s_c^{-2}y_c}{u-y_c}
		}
		\frac{
			\prod_{c=1}^{t'}\frac{u-w_c}{u-\theta_c^{-2}w_c}
		}{
			\prod_{c=1}^t \frac{u-w_c}{u-\theta_c^{-2}w_c}
		},
	\end{split}
\end{equation}
where 
\begin{equation*}
	\mathscr{E}_N(u):=
	(u-x_*)^{-N}
	\prod_{k=1}^{N}(u-y_k)
	\prod_{c=1}^{a}
	\frac{u-s_c^{-2}y_c}{u-y_c}
	\prod_{c=1}^{t}
	\frac{u-w_c}{u-\theta_c^{-2}w_c}.
\end{equation*}
In \eqref{eq:K_AP_asymptotic_equivalence_1}
both integration contours
are positively oriented simple curves.
The $u$ and $v$ contours encircle, respectively,
all the points
$y_*,y_i^\circ,
\theta_*^{-2}w_*,
(\theta_i^\circ)^{-2}w_i^\circ$
and 
all the points
$y_*,y_i^\circ,
w_{*},w_i^\circ$,
and no other poles of the integrand except $u=v$. For the latter pole,
the $u$ contour is outside the $v$ contour 
for $\Delta t=t'-t\ge0$, and inside for $\Delta t<0$.
The integration contours exist thanks to 
\eqref{eq:params_conditions_for_asymptotics}.
The ratios
$\prod_{c=1}^{a'}/\prod_{c=1}^a$ and
$\prod_{c=1}^{t'}/\prod_{c=1}^t$ in \eqref{eq:K_AP_asymptotic_equivalence_1}
outside of $\mathscr{E}_N$
are finite products
(which later will depend only on local sequences
of parameters).

Our aim is to perform the steepest descent analysis 
of $K_{\mathscr{AP}}$
based on critical points of 
\begin{equation*}
	\mathscr{S}_N(u) := \frac{1}{N}\log \mathscr{E}_N(u),
\end{equation*}
following \cite[Sections 3.1, 3.2]{Okounkov2002}.
Namely, if for all $N$ large enough on the integration contours we have 
\begin{equation}
	\label{eq:steep_condition}
	\Re \left(  
	\mathscr{S}_N(u)-\mathscr{S}_N(v)
	\right) 
	<0
\end{equation}
(here and below $\Re$ and $\Im$ stand for the real and imaginary parts, respectively),
then the integral containing
$e^{N(\mathscr{S}_N(u)-\mathscr{S}_N(v))}$
goes to zero exponentially with $N$.
We achieve \eqref{eq:steep_condition} by
deforming the original $u$ and $v$ integration contours
so that they pass through complex conjugate critical points
$z,\bar z$ of $\mathscr{S}_N(u)$ (first, we need to show that such critical points exist).
In the process of deforming the contours to 
those with
\eqref{eq:steep_condition},
certain residues will survive and contribute to the limit of
$K_{\mathscr{AP}}$.
Observe that this argument works for arbitrary branches of the logarithms
in $\log \mathscr{E}_N(u)$.

Since the integration 
contours in \eqref{eq:K_AP_asymptotic_equivalence_1}
can be chosen bounded, we have
\begin{equation}
	\label{eq:S_u_defn}
	\begin{split}
		&\mathscr{S}_N(u)=-\log(u-x_*)+(1-\alpha)
		\log(u-y_*)+\alpha \log(u-y_*/s_*^2)
		\\
		&\hspace{80pt}
		+\tau\log(u-w_*)-\tau\log (u-w_*/\theta_*^2)
		+\textnormal{remainder}=:
		\mathscr{S}(u)+\textnormal{remainder},
	\end{split}
\end{equation}
where $|\textnormal{remainder}|< C(L)/N$.
The constant $C(L)$ independent of $N$
comes from two sources.
First, we dropped the integer parts 
in $a\approx \alpha N$, $t\approx \tau N$
in $\mathscr{E}_N$. Second,
in $\mathscr{E}_N$
we have
replaced the
local parameters $y_{j-\lfloor \alpha N \rfloor }$, 
$|j|\le L$ (see \eqref{eq:w_local_global}--\eqref{eq:y_local_global}),
by $y_*$, and similarly for $w_i,\theta_i^{-2},s_j^{-2}y_j$. We see that the
critical points of $\mathscr{S}_N(u)$ are 
close (as $N\to+\infty$)
to the critical points of $\mathscr{S}(u)$.
One can check that the critical point
equation $\mathscr{S}'(u)=0$
reduces to a cubic polynomial equation in $u$. Indeed, the terms
containing $u^4$ cancel out because the sum of the coefficients 
by all the logarithms is zero.

\subsection{Moving the contours}
\label{sub:concrete_moving_contours}

Let us fix global parameters satisfying \eqref{eq:5_points},
and investigate the behavior of the function $\mathscr{S}(u)$
given by \eqref{eq:S_u_defn}.
The cubic polynomial equation for the critical points
of $\mathscr{S}(u)$
reads
\begin{equation}
	\label{eq:cubic}
		\left( \alpha y_*(s_*^{-2}-1)-\tau w_*(\theta_*^{-2}-1)+y_*-x_* \right)u^3
		+c_2 u^2+c_1 u+c_0
		=0,
\end{equation}
for certain polynomial functions $c_0,c_1,c_2$ of the global parameters
$x_*,w_*,\theta_*,y_*,s_*$ whose explicit expressions 
we omit for shorter notation.

\begin{definition}
	\label{def:liquid_region}
	Denote by $\mathscr{L}$ the region in the plane $(\alpha,\tau)\in \mathbb{R}_{\ge0}^2$ where the 
	discriminant of the cubic \eqref{eq:cubic}
	is negative (see
\Cref{fig:liquid_region} for an example).
	In $\mathscr{L}$ the cubic equation 
	has two nonreal complex conjugate roots. Denote by
	$z=z(\alpha,\tau)$ the root belonging to
	the upper half plane.
\end{definition}

\begin{figure}[ht]
	\centering
	\includegraphics[width=\textwidth]{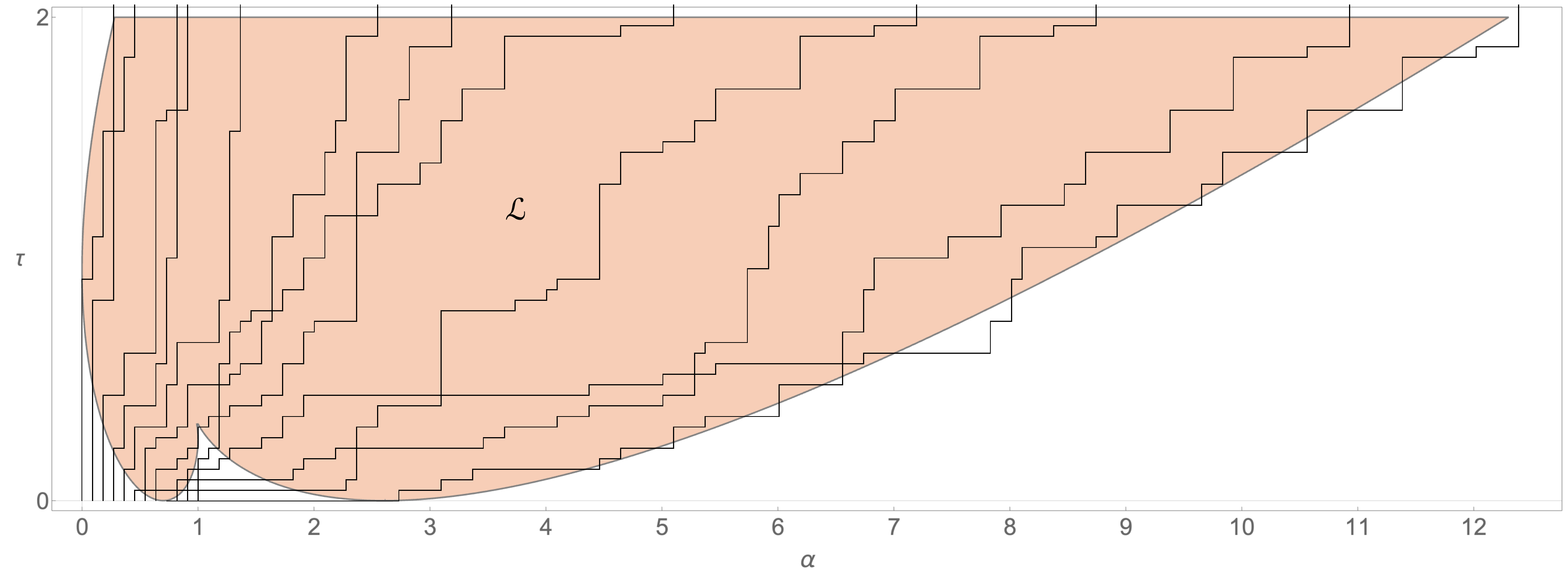}
	\caption{In the shaded region 
	the discriminant of the cubic equation
	\eqref{eq:cubic} is negative.
	We also sketch possible lattice path behavior
	under the FG process.
	The parameters in the figure are
	equal to 
	$x_*=\frac{1}{2}, w_*=\frac{2}{3}, 
	s_*=\frac{1}{2},
	y_*=\frac{9}{10}$, and
	$\theta_*=\frac{4}{5}$.}
	\label{fig:liquid_region}
\end{figure}

\begin{lemma}
	\label{lemma:z_map}
	The map $z\colon \mathscr{L}\to \mathbb{C}$
	is a diffeomorphism between $\mathscr{L}$
	and the open upper half plane.
	The region $\mathscr{L}$ is unbounded.
\end{lemma}
\begin{proof}
	For each $(\alpha,\tau)\in \mathscr{L}$, 
	there is a unique root $z(\alpha,\tau)$ in the upper half plane,
	so the map $z$ is injective.
	Substituting $u=X+i Y$ into the cubic equation 
	\eqref{eq:cubic}
	for $\mathscr{S}'(u)=0$,
	we may find $(\alpha,\tau)$ as rational functions of $(X,Y)$.
	These rational functions define the map 
	$z^{-1}\colon \mathbb{C}\to \mathscr{L}$.

	One can check that the image under $z^{-1}$
	of the real line (corresponding to $Y=0$)
	is precisely the curve where the discriminant of \eqref{eq:cubic}
	vanishes, i.e., the boundary of~$\mathscr{L}$.
	Moreover, an explicit computation shows that
	the discriminant of \eqref{eq:cubic} has the form
	$-(\textnormal{rational expression})^2Y^2$, so it is
	manifestly
	negative 
	for all $(\alpha,\tau)$ expressed through $(X,Y)$ with $Y\ne 0$.
	Since
	$Y$ enters $z^{-1}$
	only as $Y^2$, we see that $z^{-1}\colon \mathbb{C}\to \mathscr{L}$
	is two-to-one and in particular maps the upper half plane to~$\mathscr{L}$.
	As $z$ and $z^{-1}$ are clearly differentiable, 
	$z$ is indeed a diffeomorphism.

	To see that the region $\mathscr{L}$ is unbounded,
	one can check that the boundary point
	$(\alpha,\tau)$ of $\mathscr{L}$ corresponding to $(X,Y)=(x_*,0)$ is at 
	$\alpha=\tau=+\infty$.
\end{proof}

Fix $(\alpha,\tau)\in \mathscr{L}$.
Let us look at 
the \emph{steepest descent integration contours}
$\Im \mathscr{S}(u)=\Im \mathscr{S}(z(\alpha,\tau))$.
Since $\mathscr{S}(u)$ involves logarithms, let us now choose their branches
to have cuts in the lower half plane, so that $\mathscr{S}(u)$ is holomorphic
in the upper half plane and up to the real line except the 5 points
\eqref{eq:5_points}.

\begin{lemma}
	\label{lemma:Spp_nonzero}
	The second derivative $\mathscr{S}''(z(\alpha,\tau))$ at the critical point
	is nonzero.
\end{lemma}
\begin{proof}
	Fix a point $z=X+\mathbf{i} Y$ in the upper half plane, and 
	substitute $\alpha$ and $\tau$ as functions of $(X,Y)$ 
	under $z^{-1}$ (see the proof of the previous \Cref{lemma:z_map})
	into $\mathscr{S}''(X+\mathbf{i} Y)$.
	One can check that the resulting rational expression in $X,Y\in \mathbb{R}$ 
	(with complex coefficients)
	does not vanish unless $Y=0$, which is outside the upper half plane.
\end{proof}

By \Cref{lemma:Spp_nonzero},
the behavior of $\mathscr{S}(u)$ at the critical point $z$ is exactly quadratic. 
Therefore, there are four half-contours
with
$\Im \mathscr{S}(u)=\Im \mathscr{S}(z)$
leaving $z$. 
On two of them 
we have 
$\Re \mathscr{S}(u)<\Re \mathscr{S}(z)$, $u\ne z$,
and on the other two we have
$\Re \mathscr{S}(v)>\Re \mathscr{S}(z)$, $v\ne z$.
Going around $z$ these half-contours interlace. 
Let us denote these contours as read 
in the clockwise direction around $z$ by
$\Gamma_1^-,\Gamma_1^+,\Gamma_2^-,\Gamma_2^+$.
See \Cref{fig:contours} for an illustration.

\begin{figure}[ht]
	\centering
	\includegraphics[width=.69\textwidth]{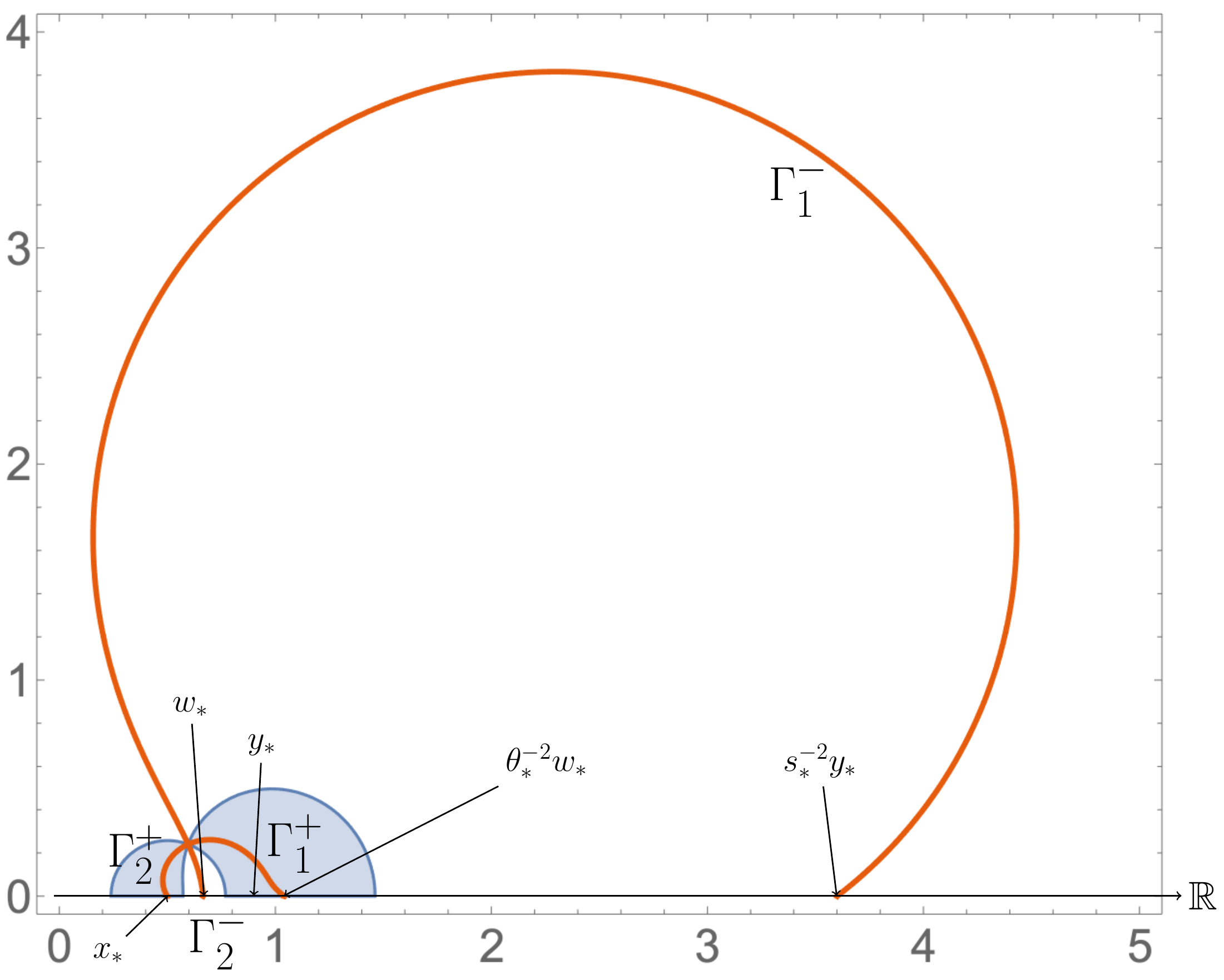}
	\caption{Steepest descent contours $\Gamma_{1,2}^{\pm}$ in the upper half plane. The four half-contours intersect at the critical point $z$.
	In the shaded regions we have $\Re(\mathscr{S}(u)-\mathscr{S}(z))>0$.
	The graph corresponds to $\alpha=\tau=\frac{3}{2}$ and other parameters
	as in \Cref{fig:liquid_region}.}
	\label{fig:contours}
\end{figure}

\begin{lemma}
	\label{lemma:at_infinity}
	We have $\lim_{R\to+\infty}\Re \mathscr{S}(R e^{\mathbf{i}p})=0$, uniformly in 
	$p\in[0,\pi]$.
\end{lemma}
\begin{proof}
	Taylor expanding each logarithm 
	in $\mathscr{S}$ \eqref{eq:S_u_defn}, we have
	\begin{equation*}
		\Re\log(Re^{\mathbf{i}p}-c)=
		\log R-\frac{c}{R}\cos p+O(R^{-2}),
	\end{equation*}
	where $O(R^{-2})$ is uniform in $p$.
	Since the coefficients by the logarithms in 
	\eqref{eq:S_u_defn} sum to $0$, the behavior
	of $\Re \mathscr{S}(R e^{\mathbf{i}p})$ is not $\log R$ but $O(R^{-1})$,
	uniformly in $p$.
\end{proof}

By \Cref{lemma:at_infinity}, the half-contours $\Gamma_{1,2}^{\pm}$
cannot escape to infinity because on them the real part of
$\mathscr{S}$ grows to $+\infty$ or decays to $-\infty$.
Since the function $\mathscr{S}$ is holomorphic in the upper half plane, 
these half-contours must end at the real line.
More precisely, each of these contours must end at one of the 
logarithmic singularities \eqref{eq:5_points} of $\mathscr{S}$. 
The signs of $\Re \mathscr{S}$ at these singularities are
\begin{equation}
	\label{eq:Re_behavior_new}
	\Re\mathscr{S}(x_*)=+\infty,\quad 
	\Re\mathscr{S}(w_*)=-\infty,\quad 
	\Re\mathscr{S}(y_*)=(\alpha-1)\infty,\quad 
	\Re\mathscr{S}(\theta_*^{-2}w_*)=+\infty,\quad 
	\Re\mathscr{S}(s_*^{-2}y_*)=-\infty.
\end{equation}
We see that 
$\Gamma_1^-$ must end at $s_*^{-2}y_*$;
$\Gamma_1^+$ must end at 
$y_*$ if $\alpha>1$ or at
$\theta_*^{-2}w_*$;
$\Gamma_2^-$ must end at $y_*$ if $\alpha<1$ or at $w_*$;
and $\Gamma_2^+$ must end at $x_*$.

\begin{lemma}
	\label{lemma:moving_contours}
	Assume that the parameters of the ascending FG process
	are 
	as described in \Cref{sub:scaling_and_parameter_assumptions}.
	Then
	\begin{enumerate}[$\bullet$]
		\item 
			The 
			$u$ contour in \eqref{eq:K_AP_asymptotic_equivalence_1}
			can be deformed,
			without picking residues other than at $u=v$,
			to a positively oriented contour which 
			crosses the real line at $s_*^{-2}y_*$,
			coincides 
			with 
			$\Gamma_1^-$ till $z$, 
			then with 
			$\Gamma_2^-$ till a small neighborhood of
			the real line, then crosses $\mathbb{R}$ again between 
			$x_*$ and all 
			$y_j$. Then the contour
			extends to the lower half plane symmetrically.
		\item 
			The 
			$v$ contour in \eqref{eq:K_AP_asymptotic_equivalence_1}
			can be deformed,
			without picking residues other than at $v=u$,
			to a positively oriented contour which crosses the real line between
			$\sup y_j$ and $\inf s_j^{-2}y_j$,
			in a small neighborhood of $\mathbb{R}$ joins 
			$\Gamma_1^+$ and coincides with it till $z$, 
			then coincides $\Gamma_2^+$ till the real line which
			it crosses again at $x_*$.
			Then the contour extends to the lower half plane symmetrically.
	\end{enumerate}
	Moreover, on the new contours we have 
	$\Re \left(  
	\mathscr{S}(u)-\mathscr{S}(v)
	\right) \le 0$, with equality only for $u=v=z$.
\end{lemma}

Let us denote the new contours afforded by 
\Cref{lemma:moving_contours}
by $\Gamma^{st}_u,\Gamma^{st}_v$, 
see 
\Cref{fig:contours_deform} for an example.
They are 
positively oriented simple closed curves in the full complex plane.

\begin{proof}[Proof of \Cref{lemma:moving_contours}]
	Let the new contours 
	$\Gamma^{st}_u,\Gamma^{st}_v$
	coincide 
	with the steepest descent ones except 
	in a small neighborhood of 
	$\mathbb{R}$. In the neighborhood of $\mathbb{R}$,
	let us change the steepest descent
	contours $\Gamma^{\pm}_{1,2}$ such that 
	$\Re \left(  
	\mathscr{S}(u)-\mathscr{S}(v)
	\right) \le 0$ still holds on the new contours
	$\Gamma^{st}_u,\Gamma^{st}_v$.
	Moreover, we would like the 
	contour deformation
	from the contours \eqref{eq:K_AP_asymptotic_equivalence_1} to
	$\Gamma^{st}_u,\Gamma^{st}_v$,
	to not pick any residues
	at poles
	$w_i^\circ,
	y_j^\circ,
	(\theta_i^\circ)^{-2}w_i^\circ,
	(s_j^\circ)^{-2}y_j^\circ$. 
	coming from the local parameters
	(call these the \emph{local residues}).
	Thanks to 
	\eqref{eq:Re_behavior_new} and the previous statements
	in this subsection,
	such a contour deformation does not cross 
	the poles
	\eqref{eq:5_points} coming from the global parameters.

	Let us now show the absence of local residues.
	On the right, the original $v$ contour crossed $\mathbb{R}$
	between $y_j$ and $s_j^{-2}y_j$.
	The steepest descent contour $\Gamma_1^+$
	crosses $\mathbb{R}$ at $y_*$ or $\theta_*^{-2}w_*$.
	In the latter case, we let $\Gamma^{st}_{v}$ coincide with 
	$\Gamma_1^+$. In the former case, 
	we may need to change the contour (in a small neighborhood of $\mathbb{R}$)
	so that it is still around all the $y_j$'s 
	(this case is illustrated in \Cref{fig:contours_deform}). 
	As the $y_j$'s are all to the left of
	$\theta_*^{-2}w_*$ and 
	$\Re \mathscr{S}(\theta_*^{-2}w_*)=+\infty$, the new contour
	still satisfies
	$\Re \mathscr{S}(v)>\Re \mathscr{S}(z)$.
	Since $v=(\theta_i^\circ)^{-2}w_i^\circ$ are not poles of the integrand,
	this deformation does not pick any local residues.
	Then the new $v$ contour joins 
	$\Gamma_1^+$ and follows it till $z$, then follows $\Gamma_2^+$
	till $x_*$ where it crosses the real line.
	We see that the new $v$ contour is still around
	all $y_j,w_i$, and not $s_j^{-2}y_j$, so 
	no local residues are picked.

	The argument for the $u$ contour is similar. 
	On the right, the original $u$ contour crossed the real line between 
	$\theta_i^{-2}w_i$ and $s_j^{-2}y_j$,
	and can be deformed to 
	coincide with $\Gamma_1^-$ which
	crosses $\mathbb{R}$ 
	at $s_*^{-2}y_*$
	without picking
	local residues at $(\theta_i^\circ)^{-2}w_i^\circ$
	(there are no poles at $u=s_j^{-2}y_j$).
	On the left, the original $u$ contour crossed $\mathbb{R}$
	between $x_*$ and $y_j$.
	The steepest descent contour $\Gamma_2^-$
	can cross the real line at $w_*$ or $y_*$. In the
	former case, we simply make the new $u$ contour coincide with $\Gamma_2^-$ and cross
	$\mathbb{R}$ at $w_*$. In the latter case, 
	we deform the $u$ contour in a neighborhood of $\mathbb{R}$
	so that it still encircles all $y_j$. As 
	the $y_j$'s are all to the right of $w_*$
	and $\Re \mathscr{S}(w_*)=-\infty$, on the new
	$u$ contour we have 
	$\Re \mathscr{S}(u)<\Re \mathscr{S}(z)$.
	
	We see that the desired contour deformation
	exists, with
	$\Re \left(  
	\mathscr{S}(u)-\mathscr{S}(v)
	\right) \le 0$
	on the new contours.
\end{proof}

\begin{figure}[htpb]
	\centering
	\includegraphics[width=\textwidth]{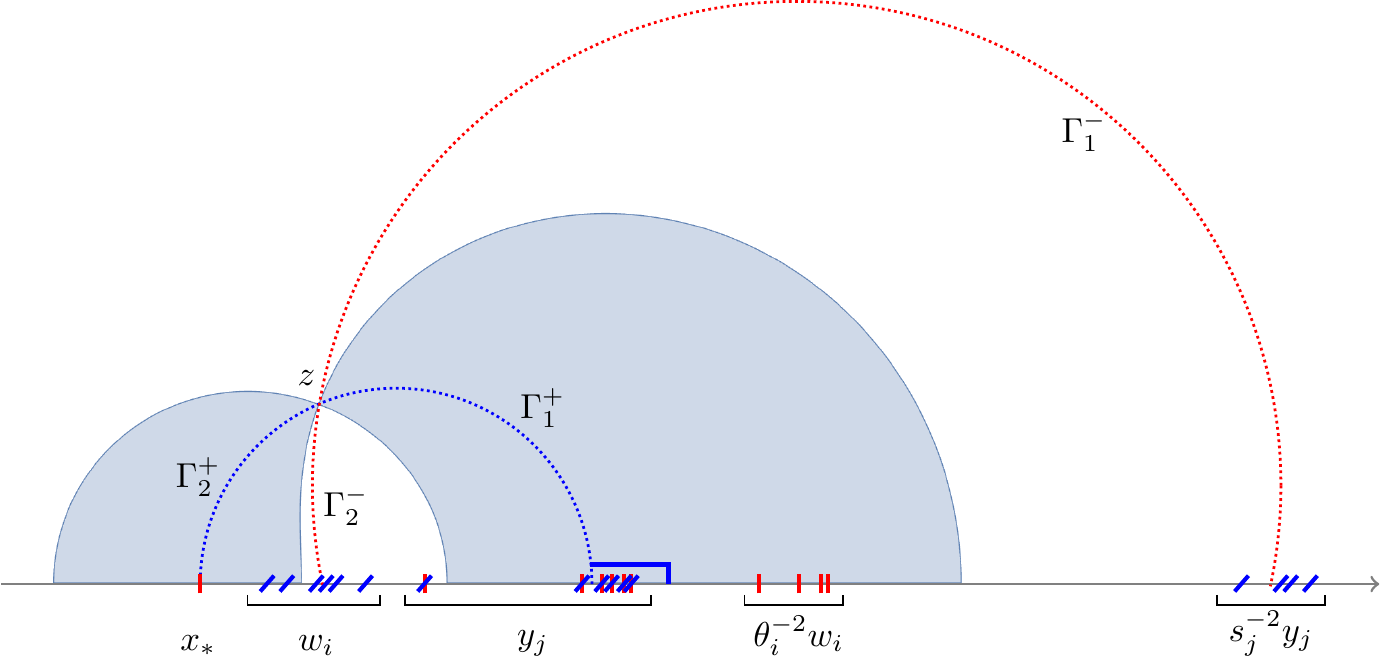}
		\caption{Deformation of the integration contours in the 
			proof of \Cref{lemma:moving_contours}.
			The shaded region is where $\Re \mathscr{S}(v)>\Re \mathscr{S}(z)$, 
			and the $v$ contour must be inside this region.
			The $u$ poles on $\mathbb{R}$ are $x_*,y_j$, and $\theta_i^{-2}w_i$, and
			the $v$ poles on $\mathbb{R}$ are $w_i,y_j,s_j^{-2}y_j$.
			We need
			to modify the new $v$ contour so that in a neighborhood of $\mathbb{R}$ it 
			diverges from 
			$\Gamma_1^+$ and encircles all $y_j$'s.}
	\label{fig:contours_deform}
\end{figure}

\subsection{Asymptotics of the kernel}
\label{sub:kernel_K_AP_asymptotics}

In the previous \Cref{sub:concrete_moving_contours}
we showed how to deform the integration contours
for the correlation kernel
$K_{\mathscr{AP}}$ \eqref{eq:K_AP_asymptotic_equivalence_1}
to the steepest descent contours $\Gamma^{st}_u,\Gamma^{st}_v$.
It remains to collect the residues at $u=v$ arising 
from this deformation. Denote by 
$\mathscr{I}_{t,a;t',a'}(u,v)$ the integrand in the 
double contour integral in \eqref{eq:K_AP_asymptotic_equivalence_1}.

\begin{lemma}
	\label{lemma:residue_u_v}
	The double contour integral 
	in \eqref{eq:K_AP_asymptotic_equivalence_1} is equal to 
	\begin{equation}
		\label{eq:kernel_single_plus_double}
		-\frac{1}{2\pi\mathbf{i}}\int_{\bar z}^{z}
		\mathop{\mathrm{Res}}_{u=v}\mathscr{I}_{t,a;t',a'}(u,v)\, dv
		+
		\frac{1}{(2\pi \mathbf{i})^2}\oint_{\Gamma^{st}_u}du
		\oint_{\Gamma^{st}_v}dv \, \mathscr{I}_{t,a;t',a'}(u,v),
	\end{equation}
	where in the 
	single integral the arc
	from $\bar z$ to $z$ is as follows:
	\begin{enumerate}[$\bullet$]
		\item If $\Delta t=t'-t \ge0$,
			the arc crosses the real line to the left of $w_*$ and all $w_i^\circ$;
		\item 
			If $\Delta t<0$,
			the arc crosses the real 
			line 
			between $\theta_*^{-2}w_*,(\theta_i^\circ)^{-2}w_i^\circ$ and $s_*^{-2}y_*,
			(s_j^\circ)^{-2}y_j^\circ$.
	\end{enumerate}
\end{lemma}
\begin{proof}
	This follows by considering 
	the contour deformation in two cases. For $\Delta t\ge0$, the 
	$u$ contour is around the $v$ one, so we pick the residue
	at $u=v$ and integrate it over
	the left portion of $\Gamma^{st}_v$,
	which is $\Gamma_2^+$ and its symmetric
	copy in the lower half plane.
	When $\Delta t<0$, we integrate minus the 
	residue at $u=v$ over the right portion of the 
	contour $\Gamma^{st}_{v}$.
\end{proof}

Recall that the number of local parameters
differing from the global ones
participating in \eqref{eq:kernel_single_plus_double}
is at most $2L$, see \Cref{sub:scaling_and_parameter_assumptions}.
Therefore, 
in \eqref{eq:kernel_single_plus_double}
the double contour integral 
decays to zero.
More precisely, this rate of decay is bounded by
$C(L)\cdot N^{-\frac12}$,
where $C(L)$ is independent of $N$. 
All of this contribution comes from a small neighborhood
of $z$, and outside of a finite neighborhood of $z$ the decay is exponential in $N$.

The surviving term in \eqref{eq:kernel_single_plus_double}
given by the single integral
is a new determinantal correlation kernel on $\mathbb{Z}^{2}$:
\begin{equation}
	\label{eq:inhom_sine_first_definition}
	\begin{split}
		&K_{\textnormal{2d}}^z(\mathsf{t},\mathsf{a};\mathsf{t}',\mathsf{a}')
		:=
		-\frac{1}{2\pi\mathbf{i}}\int_{\bar z}^{z}
		\mathop{\mathrm{Res}}_{u=v}
		\mathscr{I}_{\mathsf{t}+\lfloor \tau N \rfloor ,
		\mathsf{a}+\lfloor \alpha N \rfloor ;
		\mathsf{t}'+\lfloor \tau N \rfloor ,
		\mathsf{a}'+\lfloor \alpha N \rfloor }(u,v)\, dv
		\\&\hspace{20pt}=
		-
		\frac{1}{2\pi\mathbf{i}}
		\int_{\bar z}^{z}
		\frac{y_{\mathsf{a}}^\circ
		(1-(s_{\mathsf{a}}^\circ)^{-2})\, dv}
		{(v-y_{\mathsf{a}}^\circ)
		(v-(s^\circ_{\mathsf{a}'})^{-2}y_{\mathsf{a}'}^\circ)}
		\frac{
			\prod_{c=-\infty}^{\mathsf{a}'}
			\frac{v-(s^\circ_c)^{-2}y_c^\circ}{v-y_c^\circ}
		}{
			\prod_{c=-\infty}^{\mathsf{a}}
			\frac{v-(s_c^\circ)^{-2}y_c^\circ}{v-y_c^\circ}
		}
		\frac{
			\prod_{c=-\infty}^{\mathsf{t}'}\frac{v-w_c^\circ}
			{v-(\theta_c^\circ)^{-2}w_c^\circ}
		}{
			\prod_{c=-\infty}^{\mathsf{t}} 
			\frac{v-w_c^\circ}{v-(\theta_c^\circ)^{-2}w_c^\circ}
		},
	\end{split}
\end{equation}
where 
$\mathsf{t},\mathsf{a},\mathsf{t}',\mathsf{a}'\in \mathbb{Z}$
are fixed and the integration contours are 
as in \Cref{lemma:residue_u_v}.
Note that the ratios of the products from $-\infty$ in the 
second line in \eqref{eq:inhom_sine_first_definition}
are actually finite.

Recalling the scaling in 
\Cref{sub:scaling_and_parameter_assumptions},
we see that the kernel 
$K^z_{\textnormal{2d}}$ is independent of $N$
and depends on the following data:
\begin{enumerate}[$\bullet$]
	\item Cutoff $L\in \mathbb{Z}_{\ge1}$;
	\item Four local sequences
		$\{w_i^\circ\},\{\theta_i^\circ\},\{ y_j^\circ \},
		\{ s_j^\circ \}$, $|i|,|j|\le L$;
	\item Four global parameters $w_*,\theta_*,y_*,s_*$, where
		in \eqref{eq:inhom_sine_first_definition}
		we use the notation 
		$w_i^\circ=w_*$ for $|i| > L$, and similarly
		for $\theta_i^\circ,y_j^\circ,s_j^\circ$;
	\item Complex number 
		$z$ in the upper half plane.
\end{enumerate}
We call $K^z_{\textnormal{2d}}$
the (two-dimensional) \emph{inhomogeneous discrete sine kernel}
and discuss it in detail in \Cref{sec:inhom_sine_kernel} below. 
\begin{remark}
	Note that with the cutoff $L$, the kernel
	$K^z_{\textnormal{2d}}$
	defines a determinantal process on the whole plane $\mathbb{Z}^{2}$,
	but its parameters $w_i^\circ,\theta_i^\circ,y_j^\circ,s_j^\circ$ only vary in 
	a window of size $2L$.
	In
	\Cref{sub:def_inhom_sine} below
	we take the limit $L\to+\infty$, and arrive at a fully inhomogeneous 
	kernel with parameters varying in the full plane.
\end{remark}

Let us summarize the notation and 
establish the main result of this section.
Fix a cutoff parameter $L\in \mathbb{Z}_{\ge1}$
and
four local sequences
$\{w_i^\circ\},\{\theta_i^\circ\},\{ y_j^\circ \},
\{ s_j^\circ \}$, $|i|,|j|$,
satisfying
$w_i^\circ<y_j^\circ<(\theta_{i}^\circ)^{-2}w_i^\circ<(s_j^\circ)^{-2}y_j^\circ$
for all $i,j$.
Take global parameters $x_*,w_*,y_*,\theta_*,s_*$
satisfying \eqref{eq:5_points}, where
$x_*$ is sufficiently close to $-\infty$,
$w_*\in (\min w_i,\max w_i)$,
$y_*\in (\min y_j,\max y_j)$,
$\theta_*\in (\min \theta_i,\max \theta_i)$,
and
$s_*\in (\min s_j,\max s_j)$.
This ensures that 
the ascending FG process 
\eqref{eq:ascending_process}
with the parameters 
given in \eqref{eq:w_local_global}--\eqref{eq:y_local_global}
exists thanks to
\eqref{eq:params_conditions_for_asymptotics}.
Fix a complex number $z$ in the upper half plane.

Recall from \Cref{def:liquid_region}
the region $\mathscr{L}$
in the plane $(\alpha,\tau)\in \mathbb{R}_{\ge0}^2$
determined by 
$x_*,w_*,\theta_*,y_*,s_*$.
Let $(\alpha,\tau)=(\alpha(z),\tau(z))$ 
be the image of our point $z$ under the 
diffeomorphism from the upper half plane
to $\mathscr{L}$ (\Cref{lemma:z_map}).

We take the scaling 
\eqref{eq:scaling_parameters_1}
of $T,N,t,t',a,a'$
determined by $(\alpha,\tau)$.
Take the ascending FG process and let its parameters
$x_i,w_i,\theta_i,y_j,s_j$ behave as defined in 
\eqref{eq:w_local_global}--\eqref{eq:y_local_global}.
That is, $x_i=x_*$ are all the same, the parameters $w_i,\theta_i$ vary for $i$
in the $L$-neighborhood
of $\tau N$, and $y_j,s_j$ vary for $j$ in the $L$-neighborhood
of $\alpha N$. 
Outside these neighborhoods the parameters are constant.
Adopt the notation 
$w_i^\circ=w_*$ for $|i| > L$, and similarly
for the other three families $\theta_i^\circ,y_j^\circ,s_j^\circ$. 

\begin{theorem}
	[\Cref{thm:intro_limit_bulk} from Introduction]
	\label{thm:bulk_limit}
	Under the scaling and assumptions described before the theorem,
	the correlation kernel $K_{\mathscr{AP}}$ 
	(given by \eqref{eq:ascending_FG_process_kernel_text} or
	\eqref{eq:K_AP_asymptotic_equivalence_1})
	of the 
	ascending FG process converges to the 
	two-dimensional
	inhomogeneous discrete sine kernel 
	\eqref{eq:inhom_sine_first_definition}:
	\begin{equation*}
		\lim_{N\to+\infty}K_{\mathscr{AP}}
		(
		\mathsf{t}+\lfloor \tau N \rfloor ,
		\mathsf{a}+\lfloor \alpha N \rfloor ;
		\mathsf{t}'+\lfloor \tau N \rfloor ,
		\mathsf{a}'+\lfloor \alpha N \rfloor
		)=
		K_{\textnormal{2d}}^z(\mathsf{t},\mathsf{a};\mathsf{t}',\mathsf{a}'),
	\end{equation*}
	where
	$\mathsf{t},\mathsf{a},\mathsf{t}',\mathsf{a}'\in \mathbb{Z}$
	are fixed. 
\end{theorem}
\begin{proof}
	Fix $(\alpha,\tau)$ as the image of $z$ under the diffeomorphism
	from the upper half plane to $\mathscr{L}$.
	Let us look at the $N$-dependent kernel
	$K_{\mathscr{AP}}$ \eqref{eq:K_AP_asymptotic_equivalence_1}
	and consider the critical point $z_N(\alpha,\tau)$ 
	(in the upper
	half plane)
	of the $N$-dependent function $\mathscr{S}_N(u)$ \eqref{eq:S_u_defn}.
	Compared to $\mathscr{S}(u)$, 
	$\mathscr{S}_N(u)$ may depend on local parameters 
	$w_i^\circ,\theta_i^\circ,y_j^\circ,s_j^\circ$,
	but recall that 
	the difference is bounded by $C(L)/N$.
	Therefore, $z_N$
	is close 
	to the critical point $z(\alpha,\tau)$ 
	from \Cref{def:liquid_region}.

	Let us deform the integration 
	contours in \eqref{eq:K_AP_asymptotic_equivalence_1}
	to coincide,
	outside a neighborhood
	of $z_N$ (which at the same time is a neighborhood of $z$),
	with the contours
	$\Gamma^{st}_u,\Gamma^{st}_v$ described 
	in \Cref{lemma:moving_contours}.
	In this neighborhood, 
	let the contours pass through $z_N$ along the steepest descent
	directions. Thanks to \Cref{lemma:moving_contours},
	this deformation of contours
	does not cross any
	poles at
	$x_i,w_i,y_j,\theta_i^{-2}w_i,s_j^{-2}y_j$
	which could lead to residues.
	The only residues which this deformation of contours could produce
	are at 
	$u=v$, and these residues are accounted for in
	\Cref{lemma:residue_u_v}.

	After the contour deformation,
	$K_{\mathscr{AP}}$ becomes a sum of 
	a single integral from $\bar z_N(\alpha,\tau)$ to $z_N(\alpha,\tau)$
	and a double integral. In the $N\to+\infty$ limit,
	the double integral disappears, and the 
	single integral turns into an integral 
	from $\bar z$ to $z$,
	which is precisely the two-dimensional inhomogeneous sine kernel
	$K^z_{\textnormal{2d}}$. This completes the proof.
\end{proof}

\section{Inhomogeneous discrete sine kernel}
\label{sec:inhom_sine_kernel}

In this section we discuss the
two-dimensional inhomogeneous discrete sine
kernel $K^{z}_{\textnormal{2d}}$
defined by \eqref{eq:inhom_sine_first_definition},
and consider its many degenerations to known 
correlation kernels.
In particular, we prove
\Cref{thm:intro_K_z} from Introduction.
For simplicity, in this section 
we drop the ``${}^\circ$'' notation from
the parameters $w_i,\theta_i,y_j,s_j$ of the kernel.

\subsection{Definition of the kernel}
\label{sub:def_inhom_sine}

It is convenient to introduce 
the following inhomogeneous analogues of power
functions to write 
down the kernel $K^{z}_{\textnormal{2d}}$:
\begin{definition}[Inhomogeneous powers]
	\label{def:inhom_power}
	For any two sequences
	$\mathbf{b}=\{b_i \}_{i\in \mathbb{Z}}$ and
	$\mathbf{c}=\{c_i \}_{i\in \mathbb{Z}}$,
	define the following ``inhomogeneous powers'':
	\begin{equation}
		\label{eq:inhomogeneous_power}
		\mathcal{P}_{n,n'}(u\mid \mathbf{b};\mathbf{c}):=
		\begin{cases}
			\displaystyle\prod_{j=n+1}^{n'}\frac{u-b_j}{u-c_j},&n<n';\\
			1,&n=n';\\
			\displaystyle\prod_{j=n'+1}^{n}\frac{u-c_j}{u-b_j},&n>n',
		\end{cases}
		\qquad \qquad n,n'\in \mathbb{Z}.
	\end{equation}
\end{definition}

We will now define
the kernel $K_{\textnormal{2d}}^z$
depending on four sequences of parameters
\begin{equation}
	\label{eq:wtys_for_inhom_sine}
	\mathbf{w}=\{w_i \},\ 
	\boldsymbol\uptheta=\{\theta_i\},\ 
	\mathbf{y}=\{ y_j \},\ 
	\mathbf{s}=\{ s_j \}
	,\qquad 
	i,j\in \mathbb{Z},
\end{equation}
and on a point
$z$ in the upper
half complex plane.
Assume that the parameter sequences satisfy
\begin{equation}
	\label{eq:bulk_after_limit_parameters_condition}
	\begin{split}
		&
		\sup\nolimits_i w_i< 
		\inf\nolimits_j y_j\le
		\sup\nolimits_j y_j<
		\inf\nolimits_i \theta_i^{-2}w_i\le 
		\sup\nolimits_i \theta_i^{-2}w_i<
		\inf\nolimits_j s_j^{-2}y_j.
	\end{split}
\end{equation}

\begin{definition}
	\label{def:inhom_sine}
	The \emph{two-dimensional inhomogeneous extended sine kernel} is defined 
	as follows:
	\begin{equation}
		\label{eq:inhom_sine_kernel}
			K_{\textnormal{2d}}^{z}(t,a;t',a')=
			-
			\frac{1}{2\pi\mathbf{i}}\int_{\bar z}^{z}
			\frac{y_a(1-s_a^{-2})}{(u-y_a)(u-s_{a'}^{-2}y_{a'})}
			\,\mathcal{P}_{a,a'}(u\mid \mathbf{s}^{-2}\mathbf{y};\mathbf{y})
			\,
			\mathcal{P}_{t,t'}(u\mid \mathbf{w};\boldsymbol\uptheta^{-2}\mathbf{w})
			\,
			du,
	\end{equation}
	where
	$t,a,t',a'\in \mathbb{Z}$.
	The integration contour 
	is an arc
	from $\bar z$ to $z$ which 
	crosses the real line 
	\begin{enumerate}[$\bullet$]
		\item 
			To the left of all
			$w_i$ when $\Delta t=t'-t\ge 0$;
		\item 
			Between 
			$\theta_i^{-2}w_i$ and $s_j^{-2}y_j$
			when $\Delta t <0$.
	\end{enumerate}
	This integration arc
	exists thanks
	to \eqref{eq:bulk_after_limit_parameters_condition}.
\end{definition}

The kernel \eqref{eq:inhom_sine_kernel}
is the same as \eqref{eq:inhom_sine_first_definition},
up to changes in notation and the removal of the cutoff 
parameter $L\in \mathbb{Z}_{\ge1}$.

\begin{theorem}[\Cref{thm:intro_K_z} from Introduction]
	\label{thm:K_z_stochastic}
	Under the assumptions \eqref{eq:bulk_after_limit_parameters_condition}
	and for any $z$ in the upper half plane with $\Im z>0$, the 
	kernel $K^z_{\textnormal{2d}}$ 
	\eqref{eq:inhom_sine_kernel}
	defines a determinantal random point process 
	on $\mathbb{Z}^2$. 
\end{theorem}
\begin{proof}
	We show 
	using \Cref{thm:bulk_limit}
	that the determinantal point process
	defined by $K^z_{\textnormal{2d}}$
	arises as a limit of a determinantal random point process
	coming from an ascending FG process. 

	Given the data \eqref{eq:wtys_for_inhom_sine}
	satisfying
	\eqref{eq:bulk_after_limit_parameters_condition},
	pick global parameters
	$x_*\in (-\infty,\inf w_i)$, 
	$w_*\in (\inf w_i,\sup w_i)$,
	$y_*\in (\inf y_j,\sup y_j)$,
	$\theta_*\in (\inf \theta_i,\sup \theta_i)$,
	and
	$s_*\in (\inf s_j,\sup s_j)$.
	For any cutoff $L\in \mathbb{Z}_{\ge1}$,
	define truncated local parameter sequences
	\begin{equation}
		\label{eq:cutoff_param_sequences}
		w_i^{(L)}=
		\begin{cases}
			w_i,& |i|\le L;\\
			w_*,& |i| >L,
		\end{cases}
		\quad
		\theta_i^{(L)}=
		\begin{cases}
			\theta_i,& |i|\le L;\\
			\theta_*,& |i| >L,
		\end{cases}
		\quad 
		y_j^{(L)}=
		\begin{cases}
			y_j,& |j|\le L;\\
			y_*,& |j| >L,
		\end{cases}
		\quad
		s_j^{(L)}=
		\begin{cases}
			s_j,& |j|\le L;\\
			s_*,& |j| >L.
		\end{cases}
	\end{equation}
	Taking $x_*$ smaller if necessary, one can make 
	sure that \eqref{eq:params_conditions_for_asymptotics}
	holds. Thus,
	the ascending FG process \eqref{eq:ascending_process}
	with global parameters
	$x_*,w_*,\theta_*,y_*,s_*$ and
	local sequences 
	\eqref{eq:cutoff_param_sequences}
	is well-defined.
	Take the scaling location $(\alpha,\tau)$
	corresponding to $z$ as in \Cref{lemma:z_map}.
	Applying \Cref{thm:bulk_limit},
	we see that the kernel $K^{z,(L)}_{\textnormal{2d}}$
	with the $L$-truncated parameter sequences
	\eqref{eq:cutoff_param_sequences} is the bulk lattice limit of the
	kernel $K_{\mathscr{AP}}$ of the FG process.
	Therefore, in the $L$-truncated case the kernel
	$K^{z,(L)}_{\textnormal{2d}}$ indeed
	defines a stochastic process.

	Now observe that
	for any $t,a,t',a'\in \mathbb{Z}$
	and $L>\max\{|t|,|a|,|t'|,|a'|\}$,
	the matrix element 
	$K^z_{\textnormal{2d}}(t,a;t',a')$ \eqref{eq:inhom_sine_kernel}
	does not depend on $L$ or the global parameters
	involved in the truncation \eqref{eq:cutoff_param_sequences}.
	Therefore, as $L\to+\infty$,
	the matrix elements of
	$K^{z,(L)}_{\textnormal{2d}}$ 
	stabilize to those of 
	$K^z_{\textnormal{2d}}$.
	This implies that $K^z_{\textnormal{2d}}$
	also defines a stochastic process, as desired.
\end{proof}

In fact, our conditions
on the parameters $w_i,\theta_i,y_j,s_j$ 
of the inhomogeneous discrete sine kernel
are 
natural in the following sense:

\begin{lemma}
	\label{lemma:conditions_on_parameters_natural}
	Conditions 
	\eqref{eq:bulk_after_limit_parameters_condition}
	are equivalent to the fact that the 
	domino weights
	depending on $w_i,\theta_i,y_j,s_j$ 
	(given in \Cref{fig:domino_from_dimer}, \textnormal{(a)\/})
	are positive and 
	separated from zero and infinity.
\end{lemma}
\begin{proof}
	For fixed $i_1,j_1$ 
	the positivity of the domino weights
	depending only on $w_{i_1},\theta_{i_1},y_{j_1},s_{j_1}$
	is equivalent to either
	$w_{i_1}<y_{j_1}<\theta_{i_1}^{-2}w_{i_1}<s_{j_1}^{-2}y_{j_1}$
	or 
	$w_{i_1}>y_{j_1}>\theta_{i_1}^{-2}w_{i_1}>s_{j_1}^{-2}y_{j_1}$. 
	Let us pick $i_2\ne i_1,j_2\ne j_1$, and 
	for 
	$w_{i_2},\theta_{i_2},y_{j_2},s_{j_2}$ we have similarly
	one of the two strings of inequalities.
	If, say,
	$w_{i_1}<y_{j_1}<\theta_{i_1}^{-2}w_{i_1}<s_{j_1}^{-2}y_{j_1}$
	but
	$w_{i_2}>y_{j_2}>\theta_{i_2}^{-2}w_{i_2}>s_{j_2}^{-2}y_{j_2}$,
	then at $i=i_1$, $j=j_2$ one readily sees that both
	possibilities
	\begin{equation*}
		w_{i_1}<y_{j_2}<\theta_{i_1}^{-2}w_{i_1}<s_{j_2}^{-2}y_{j_2}
		\quad\textnormal{or}
		\quad
		w_{i_1}>y_{j_2}>\theta_{i_1}^{-2}w_{i_1}>s_{j_2}^{-2}y_{j_2}
	\end{equation*}
	lead to a contradiction. 

	Therefore, 
	it must be either
	$w_{i}<y_{j}<\theta_{i}^{-2}w_{i}<s_{j}^{-2}y_{j}$
	or 	
	$w_{i}>y_{j}>\theta_{i}^{-2}w_{i}>s_{j}^{-2}y_{j}$
	simultaneously for all $i,j$.
	If it's the latter, observe that the domino
	weights are invariant under the simultaneous
	sign flips $w_i\mapsto-w_i$,
	$y_j\mapsto-y_j$ for all $i,j$,
	which turns the conditions with
	``$>$'' into those with ``$<$''.
	Thus, we see that picking
	the ``$<$'' sign in all conditions does not restrict the generality.
\end{proof}

Thus, by \Cref{thm:K_z_stochastic} and \Cref{lemma:conditions_on_parameters_natural},
the kernel $K^z_{\textnormal{2d}}$
defines a \emph{bona fide} stochastic determinantal point process on $\mathbb{Z}^{2}$
for a maximally generic open family of parameters.
Moreover, setting some of the domino
weights 
in \Cref{fig:domino_from_dimer}, (a)
to zero also leads to a stochastic process via a straightforward limit transition.
In the rest of this section we compare the kernel
$K_{\textnormal{2d}}^{z}$ to similar known kernels, in one and then in two dimensions.

\subsection{Discrete sine kernel in one dimension}
\label{sub:inhom_sine_1d_to_usual_sine}

In one-dimensional slices (corresponding to fixing $t=t'\in \mathbb{Z}$),
the process is independent of $t$, and the kernel 
\eqref{eq:inhom_sine_kernel}
becomes an inhomogeneous analogue of the discrete sine kernel:
\begin{equation}
	\label{eq:inhom_sine_kernel_1d}
	K_{\textnormal{1d}}^{z}(a,a')=-
	\frac{1}{2\pi\mathbf{i}}\int_{\bar z}^{z}
	\frac{y_a(1-s_a^{-2})}{(u-y_a)(u-s_{a'}^{-2}y_{a'})}\,
	\mathcal{P}_{a,a'}(u\mid \mathbf{s}^{-2}\mathbf{y};\mathbf{y})
	\,
	du.
\end{equation}
The integration contour passes to the left of all $y_j$.

The kernel $K^{z}_{1d}$ is clearly not translation 
invariant, and the density 
function is given by
\begin{equation*}
	\rho_a=
	K_{\textnormal{1d}}^{z}(a,a)=
	-
	\frac{1}{2\pi\mathbf{i}}\int_{\bar z}^{z}
	\frac{y_a(1-s_a^{-2})\,du}{(u-y_a)(u-s_{a}^{-2}y_{a})}
	=
	\frac{1}{2\pi\mathbf{i}}
	\int_{V_a(\bar z)}^{V_a(z)} \frac{dv}{v}
	=
	\frac{\arg V_a(z)}{\pi}
	,
\end{equation*}
where we used the change of variables
\begin{equation}
	\label{eq:notation_change_of_variables}
	v=V_a(u):=
	\frac{u-y_a}{u-s_a^{-2}y_a}.
\end{equation}
This change of variables swaps the lower and upper half planes, 
hence the minus sign in front of the integral disappears.
The integration in $v$ is over an
arc crossing the real line to
the right of the origin.

\medskip

In the homogeneous case $y_a=y$, $s_a=s$ for all $a\in \mathbb{Z}$,
the change of variables $V_a=V$ does not depend on $a$, and
the density
$\rho_a\equiv \rho
=\frac{1}{\pi}\arg V(z)$
is constant.
We see that \eqref{eq:inhom_sine_kernel_1d} essentially becomes the
usual \emph{discrete sine kernel}:
\begin{equation}
	\label{eq:discrete_sine_usual}
	K_{\textnormal{1d,\,hom}}^{z}(a,a')=
	\frac{1}{2\pi\mathbf{i}}\int_{V(\bar z)}^{V(z)}\frac{dv}{v^{a'-a+1}}
	=|V(z)|^{a-a'}\,\frac{\sin\left( \pi\rho(a'-a) \right)}{\pi(a'-a)},\qquad
	a,a'\in \mathbb{Z}.
\end{equation}
The factor $|V(z)|^{a-a'}$ is a so-called ``gauge transformation'',
and can be removed from the kernel without changing the determinantal process.
The discrete sine kernel in one dimension
and the corresponding determinantal point process
were obtained in \cite{Borodin2000b} as a bulk limit 
of Plancherel random partitions.
This point process
arises from many other discrete determinantal point processes
as a lattice (bulk) scaling limit.

\subsection{Periodic discrete sine kernel in one dimension}
\label{sub:inhom_sine_1d_periodic}

The fully inhomogeneous kernel
$K_{\textnormal{1d}}^{z}(a,a')$
\eqref{eq:inhom_sine_kernel_1d}
on $\mathbb{Z}$
can be specialized to a $k$-periodic kernel on $\mathbb{Z}$, for any $k\ge 2$
(the case $k=1$ leads to the discrete sine kernel \eqref{eq:discrete_sine_usual}).
Here let us consider the case with $k=2$, 
and take a further degeneration. Namely, set
\begin{equation*}
	y_i=s_i^2 c_i,\qquad c_i
	=
	c_0\mathbf{1}_{i\equiv 0 \, \mathrm{mod}\,2}
	+
	c_1\mathbf{1}_{i\equiv 1 \, \mathrm{mod}\,2},
\end{equation*}
and after that send $s_i\to0$. This leads to the following
kernel:
\begin{equation}
	\label{eq:even_odd_kernel}
	\begin{split}
		&\begin{bmatrix}
			K_{\textnormal{1d}}(2n,2n')&
			K_{\textnormal{1d}}(2n,2n'+1)\\
			K_{\textnormal{1d}}(2n+1,2n')&
			K_{\textnormal{1d}}(2n+1,2n'+1)
		\end{bmatrix}
		\\&\hspace{70pt}=
		\frac{1}{2\pi\mathbf{i}}\int_{\bar z}^{z}
		\begin{bmatrix}
			\frac{c_0}{1-c_0/u}
			&
			c_0
			\\
			\frac{c_1}{ \left( 1-c_0/u \right)\left( 1-c_1/u \right)}
			&
			\frac{c_1}{ 1-c_1/u }
		\end{bmatrix}
		\left( \bigl(1-\frac{c_0}{u}\bigr)\bigl( 1-\frac{c_1}{u} \bigr) \right)^{\Delta n}
		\frac{du}{u^2},
	\end{split}
\end{equation}
where the integration contour crosses the real line 
to the left of $0$,
and $\Delta n=n'-n$. Here the matrix form is just a shorthand for four
different integral expressions for the matrix elements of the kernel,
depending on the parity.
The kernel \eqref{eq:even_odd_kernel} is invariant under translations by $2\mathbb{Z}$,
but not by $\mathbb{Z}$ if $c_0\ne c_1$.

This correlation kernel \eqref{eq:even_odd_kernel}
in the $2\mathbb{Z}$ periodic case has appeared (together with
its $\mathbb{Z}$-translation invariant extension into the second dimension) in
\cite[Theorem 3.1]{Mkrtchyan2014Periodic}
in the study of
plane partitions with weights periodic in one direction.
See also an extension to the arbitrary $k\mathbb{Z}\times \mathbb{Z}$ periodic case in
\cite[Theorem 4.2]{Mkrtchyan2019}.

Note also that $k\mathbb{Z}\times \mathbb{Z}$
periodic kernels from \cite{Mkrtchyan2019}
also arise as particular cases of the
determinantal kernels for
Gibbs ensembles of nonintersecting paths
from \cite{borodin2010gibbs}
or as a bulk limit from the periodic Schur process
\cite{borodin2007periodic}.
Indeed, the setup of 
\cite{borodin2010gibbs} allows for fully inhomogeneous parameters
in one direction (but requires full $\mathbb{Z}$
translation invariance in the second direction).
Therefore, 
in one-dimensional slices the point processes
from \cite{borodin2010gibbs}
produce all
our one-dimensional kernels $K_{\textnormal{1d}}^{z}$ \eqref{eq:inhom_sine_kernel_1d}
with the most general parameters.

\subsection{Two-dimensional homogeneous kernel}
\label{sub:sine_2d_discussion}

Let us proceed to discussing two-dimensional kernels.
First, 
let us identify the fully homogeneous case 
$K_{\textnormal{2d}}^{z}$ \eqref{eq:inhom_sine_kernel}.
That is, let $w_i=w$ for all $i\in \mathbb{Z}$,
and similarly for $\theta,y,s$.
Recall the change of variables $V_a(u)$ \eqref{eq:notation_change_of_variables}
which is independent of $a$ in the homogeneous case, and we denote it simply by $V(u)$.
We have
\begin{equation}
	\label{eq:2d_hom_kernel}
	K_{\textnormal{2d,\,hom}}^{z}(t,a;t',a')=
	\frac{(\mathrm{const})^{\Delta t}}{2\pi\mathbf{i}}\int_{V(\bar z)}^{V(z)}
	\left(
		\frac{v-V(w)}
		{v-V( \theta^{-2}w)}
	\right)^{\Delta t}
	\frac{dv}{v^{\Delta a+1}},
\end{equation}
where
$\Delta t=t'-t$,
$\Delta a=a'-a$, and
the integration contour
crosses the real line between
$V(w)$ and $1$ for $\Delta t\ge0$, and 
to the left of $V(\theta^{-2}w)$ for $\Delta t<0$.

The correlation kernel $K_{\textnormal{2d,\,hom}}^{z}$ can be identified with
the bulk limiting kernel in the liquid phase of the model of
random
domino
tilings of the Aztec diamond, when the dominoes are mapped to a determinantal process
on $\mathbb{Z}^2$. The one-dimensional sine kernel for the Aztec diamond model was
obtained in
\cite[Theorem 2.10]{johansson2002non}, and the
two-dimensional bulk kernel may be read off from the more general theory of
\cite{KOS2006}, or deduced as a bulk limit from \cite[(2.21)]{Johansson2005arctic}.

Moreover,
$K_{\textnormal{2d,\,hom}}^{z}$ also arises as a particular member of
the family of extensions of the one-dimensional discrete sine kernel
constructed in
\cite{borodin2010gibbs}. Namely, to get \eqref{eq:2d_hom_kernel}, 
one should alternate the
``alpha'' and the ``beta'' factors,
$(1-\alpha^+ v)^{-1}$ and $(1+\beta^+ v)$,
in \cite[(2)]{borodin2010gibbs}.

\medskip

Setting $w=y$, that is, $V(w)=0$,
turns the correlation kernel
$K_{\textnormal{2d,\,hom}}^{z}$ \eqref{eq:2d_hom_kernel}
(up to a gauge factor which does not change the determinantal process)
into the \emph{incomplete beta kernel} introduced in 
\cite{okounkov2003correlation}.
The incomplete beta kernel is
the determinantal kernel of the
ergodic translation invariant Gibbs measure on lozenge tilings of the plane (viewed as
a determinantal process on $\mathbb{Z}^2$, cf. \Cref{rmk:noncolliding_lattice_paths}),
which is unique up to specifying the \emph{slope}. The slope
is a two-dimensional real parameter which can be mapped (in our notation)
to the point $z$ in the upper half plane.
We refer to
\cite{Sheffield2008},
\cite{KOS2006}
for further details.
Universality of the 
incomplete beta kernel in 
the model of uniformly random lozenge tilings
in general domains
was established recently in 
\cite{aggarwal2019universality}.

\subsection{Other dimer models with periodic weights}
\label{sub:other_dimer_models}

We see that our two-dimensional 
inhomogeneous sine kernel
$K_{\textnormal{2d}}^{z}$ (\Cref{def:inhom_sine})
generalizes the bulk lattice distributions
arising in domino and lozenge tilings.
Our generalization allows for 
inhomogeneous parameters in both lattice directions.

By taking the parameters to be periodic (as in \Cref{sub:inhom_sine_1d_periodic}
but in both directions), one gets doubly periodic determinantal kernels with
periods $k\mathbb{Z}\times m\mathbb{Z}$ for arbitrary $k,m\ge1$.
When both $k,m>1$, the explicit form of the doubly periodic kernels is new.
For $m=1$, the $k\mathbb{Z}\times \mathbb{Z}$ periodic kernels appeared
in
\cite{borodin2007periodic},
\cite{borodin2010gibbs}, 
\cite{Mkrtchyan2014Periodic}, 
\cite{Mkrtchyan2019}.

While in principle our doubly periodic kernels fall into the general framework of \cite{KOS2006},
rewriting the general double integral formula for the kernel from
\cite[Theorem 4.3]{KOS2006} in an arc integral form as in 
$K_{\textnormal{2d}}^{z}$ \eqref{eq:inhom_sine_kernel}
is a nontrivial transformation.
Moreover, the fully inhomogeneous (non-periodic) kernels in both directions
do not immediately follow from the general theory of \cite{KOS2006}.
It might be possible to obtain non-periodic fully inhomogeneous
kernels as limits of the periodic
ones from \cite{KOS2006}, but the 
double integral
form of the latter kernels
does not
seem well-suited for such a limit transition.

\medskip

Observe that our two-dimensional inhomogeneous discrete sine kernel
$K_{\textnormal{2d}}^{z}$
corresponds only
to the liquid (also called rough) phase of our path ensembles / domino tilings
coming from the ascending FG processes.
In the rough phase one expects 
the variance of the height difference to grow logarithmically with the distance \cite{KOS2006}.
This behavior is proven in many cases, and the fluctuations are
identified with the Gaussian Free Field, 
cf. \cite{Kenyon2001GFF}, \cite{Petrov2012GFF}, \cite{gorin2021lectures}.

The liquid phase local behavior described by $K^z_{\textnormal{2d}}$
should be contrasted with that in the gaseous (also called smooth) phase
in which the
height differences have bounded variance \cite{KOS2006}.
The gaseous phase is present in doubly periodic 
(in particular, $2\mathbb{Z}\times 2\mathbb{Z}$ periodic)
domino tilings
\cite{chhita2016domino},
\cite{duits2017two},
\cite{berggren2021domino}, see also, e.g., \cite{charlier2021doubly} for a 
discussion of the case of lozenge tilings.
We see that
our dimer edge weights (see \Cref{fig:domino_from_dimer})
do not produce
gaseous phases.
This is because our weights
are
\emph{not} fully generic like in, e.g., 
\cite{chhita2016domino},
and instead depend on the parameters
in quite a special way.
In particular,
in the $2\mathbb{Z}\times 2\mathbb{Z}$ periodic case
we have verified that
the domino weights are gauge equivalent
(in the sense of \cite[Section 3.10]{Kenyon2007Lecture}),
in a nontrivial way, to 
weights periodic in only one direction.
This seems to be the reason for not seeing gaseous phases in 
the bulk of ascending FG processes.


\appendix

\newpage

\part{Appendix}
\label{part:appendix}

\section{Formulas for $F_\lambda$ and $G_\lambda$}
\label{appA:F_G_formula_proofs}

Here we employ the row operators (defined in \Cref{sub:row_operators})
to get explicit formulas for the partition functions $F_\lambda$ and $G_\lambda$
of the free fermion six vertex model, and thus prove 
\Cref{thm:F_formula,thm:G_formula}.
This Appendix accompanies \Cref{sec:F_G_symm_funct} and 
employs algebraic Bethe Ansatz type computations. They follow
\cite[Section 4.5]{BorodinPetrov2016inhom} (but are more involved in the case of $G_\lambda$), see also
Part~VII and in particular Appendix VII.2 of 
\cite{QISM_book}.

\subsection{Proof of \texorpdfstring{\Cref{thm:F_formula}}{formula for F}}
\label{appA:F}

\subsubsection{Recalling the notation}
\label{appA:F_recall}

Throughout this subsection we fix a
signature
$\lambda=(\lambda_1,\ldots,\lambda_N \ge 0)$ with $N$ parts, 
and sequences 
\begin{equation*}
	\mathbf{x}=(x_1,\ldots,x_N )
	,\qquad 
	\mathbf{y}=(y_1,y_2,\ldots )
	,\qquad 
	\mathbf{r}=(r_1,\ldots,r_N )
	,\qquad  
	\mathbf{s}=(s_1,s_2,\ldots )
	.
\end{equation*}
Recall (\Cref{def:F_function})
that the function $F_\lambda(\mathbf{x};\mathbf{y};\mathbf{r};\mathbf{s})$
is the partition function of the free fermion six vertex model with 
weights $\widehat{W}$ \eqref{eq:weights_W_hat}
and with boundary conditions determined by $\lambda$.

In this subsection we prove \Cref{thm:F_formula} stating that $F_\lambda$
is given by the determinantal expression \eqref{eq:F_det_formula_in_theorem}
involving the functions $\varphi_k(x)$ \eqref{eq:phi_def}.
For convenience, let us explicitly reproduce the desired formula here:
\begin{equation}
	\label{eq:F_det_formula_appendix}
	\begin{split}
		&
		F_\lambda(\mathbf{x};\mathbf{y};\mathbf{r};\mathbf{s})
		=
		\Biggl(
			\prod_{i=1}^{N}x_i(r^{-2}_i-1)
			\prod_{1\le i<j\le N}\frac{r_i^{-2}x_i-x_j}{x_i-x_j}
		\Biggr)
		\\&\hspace{140pt}
		\times\det
		\biggl[ 
				\frac{1}{y_{\lambda_j+N-j+1}-x_i}
				\prod_{m=1}^{\lambda_j+N-j}
				\frac{y_m-s_m^2x_i}{s_m^2(y_m-x_i)}
		\biggr]_{i,j=1}^{N}.
	\end{split}
\end{equation}

For the proof we will need the row operators $\widehat{A},\widehat{B},\widehat{C},\widehat{D}$
defined by \eqref{eq:abcdv_hat}--\eqref{eq:abcdv_n2_hat}. These operators 
are built from the weights $\widehat{W}$,
depend on two numbers
$x,r$ and the sequences $\mathbf{y}, \mathbf{s}$, and 
act (from the right) on tensor products of two-dimensional spaces
$V^{(k)}=\mathop{\mathrm{span}}\{ e_0^{(k)},e_1^{(k)} \} \simeq \mathbb{C}^2$.
To the signature $\lambda$ we associate the element 
$e_{\mathcal{S}(\lambda)}$ in the (formal) infinite tensor product
$V^{(1)}\otimes V^{(2)}\otimes\ldots$,
where we take $e^{(k)}_1$ in 
the $k$-th place if and only if $k\in \mathcal{S}(\lambda)$ and $e_0^{(k)}$ otherwise,
see \Cref{sub:signature_states}.
For example, the empty signature~$\varnothing$ (which has $0$ parts)
corresponds to $e_{\varnothing}=e_0^{(1)}\otimes e_0^{(2)}\otimes \ldots $.

By \Cref{prop:F_G_row_op}, $F_\lambda(\mathbf{x};\mathbf{y};\mathbf{r};\mathbf{s})$ 
is the coefficient of 
$e_{\mathcal{S}(\lambda)}$ in 
$e_{\varnothing}\widehat{B}(x_N,r_N)\ldots \widehat{B}(x_1,r_1) $,
and for the proof of \Cref{thm:F_formula} we proceed to evaluate this coefficient.
One of our main tools is the Yang--Baxter equation stated as 
a family of commutation relations between the operators 
(see \Cref{prop:ABCD_YBE_Hat}).

\subsubsection{Action on a tensor product of two spaces}
\label{appA:F_two_spaces}

The crucial part of the argument is to consider the action of 
$\widehat{B}(x_N,r_N)\ldots \widehat{B}(x_1,r_1)$
on a tensor product of two spaces, $V_1\otimes V_2$.
Using the second identity from
\eqref{eq:abcdv_n2_hat}, namely, 
$(v_1 \otimes v_2)  \widehat{B} 
=  
v_1 \widehat{D} \otimes v_2 \widehat{B} +  v_1 \widehat{B} \otimes v_2 \widehat{A}$,
we see that 
\begin{equation}
	\label{eq:F_proof_B_product_1}
	\widehat{B}(x_N,r_N)\ldots \widehat{B}(x_1,r_1)
	=
	\sum_{\mathcal{I}\subseteq \left\{ 1,\ldots,N  \right\}}
	X_{\mathcal{I}}(\mathbf{x};\mathbf{r})\otimes Y_{\mathcal{I}}(\mathbf{x};\mathbf{r}),
\end{equation}
where
\begin{equation*}
	\begin{split}
		X_{\mathcal{I}}(\mathbf{x};\mathbf{r})&=X_N(\mathcal{I};x_N,r_N)X_{N-1}(\mathcal{I};x_{N-1},r_{N-1})\ldots X_1(\mathcal{I};x_1,r_1), 
		\\
		Y_{\mathcal{I}}(\mathbf{x};\mathbf{r})&=Y_N(\mathcal{I};x_N,r_N)Y_{N-1}(\mathcal{I};x_{N-1},r_{N-1})\ldots Y_1(\mathcal{I};x_1,r_1),
		\\
		X_i(\mathcal{I};x_i,r_i)&=\begin{cases}
			\widehat{D}(x_i,r_i),&i \in \mathcal{I};\\
			\widehat{B}(x_i,r_i),&i \notin \mathcal{I},
		\end{cases}
		\qquad \qquad 
		Y_i(\mathcal{I};x_i,r_i)=\begin{cases}
			\widehat{B}(x_i,r_i),&i \in \mathcal{I};\\
			\widehat{A}(x_i,r_i),&i \notin \mathcal{I}.
		\end{cases}
		\qquad 
	\end{split}
\end{equation*}
Now, using the commutation relations 
\eqref{eq:BD_hat_relation}--\eqref{eq:BA_hat_relation}
from \Cref{prop:ABCD_YBE_Hat}, 
we move all the operators 
$\widehat B$ to the right in both $X_\mathcal{I}$ and $Y_\mathcal{I}$,
which allows to rewrite \eqref{eq:F_proof_B_product_1}
as 
\begin{equation}
	\label{eq:F_proof_B_product_2}
	\begin{split}
		&
		\sum_{\substack{I \cup J = \{1,\ldots,  N \} \\ I' \cup J' = \{1,\ldots,  N\} }}
		c_{I; I'} (\mathbf{x}; \mathbf{r})  \,
		\widehat{D} (x_{j_{N - k}}, r_{j_{N - k}})  \ldots \widehat{D} (x_{j_1}, r_{j_1})  \widehat{B} (x_{i_k}, r_{i_k}) \ldots \widehat{B} (x_{i_1}, r_{i_1}) 
		\\
		&\hspace{120pt} \otimes \widehat{A} (x_{j_{N - m}'}, r_{j_{N - m}'}) \ldots \widehat{A} (x_{j_1'}, r_{j_1'}) \widehat{B} (x_{i_m'}, r_{i_m'}) \ldots \widehat{B} (x_{i_1'}, r_{i_1'}),
	\end{split}
\end{equation}
for some rational functions $c_{I; I'} (\mathbf{x}; \mathbf{r})$,
where we have denoted $|I| = k$ and $|I'| = m$,
defined $J = \left\{ 1,\ldots,N  \right\} \setminus I$ and
$J' = \left\{ 1,\ldots,N  \right\} \setminus I'$, and ordered the
indices such that $i_\alpha<i_\beta,i_\alpha'<i_\beta',j_\alpha<j_\beta$, and 
$j_\alpha'<j_\beta'$ for all $\alpha<\beta$.
Here we also employed the commutativity of 
$\widehat{A}$ \eqref{eq:A2hatA1hat} and $\widehat{D}$ \eqref{eq:D2hatD1hat}.
In fact, here one can already see from
\eqref{eq:BD_hat_relation}--\eqref{eq:BA_hat_relation}
that $m=N-k$, but we will get this relation (and a stronger relation between 
the sets $I,I',J,J'$) in the next \Cref{lemma:F_proof_lemma1}.

\begin{remark}
	\label{rmk:uniqueness_of_coeffs}
	Let us make an important observation about the coefficients 
	$c_{I;I'}(\mathbf{x};\mathbf{r})$. 
	Namely, these coefficients
	are computed using only the commutation relations 
	for the operators $\widehat{A},\widehat{B},\widehat{C},\widehat{D}$, and we 
	argue that the $c_{I;I'}(\mathbf{x};\mathbf{r})$'s do not depend on the order of 
	applying the commutation relations.
	This property is based on the fact that for generic parameters $(x,r)$, 
	there exists a representation
	of 
	$\begin{bmatrix}
		\widehat{A}(x,r)&\widehat{B}(x,r)\\\widehat{C}(x,r)&\widehat{D}(x,r)
	\end{bmatrix}$
	subject to the same commutation relations, and a highest weight vector 
	(annihilated by $\widehat{C}$ and an eigenfunctions of $\widehat{A},\widehat{D}$)
	$\mathsf{v}_0$
	in that representation, such that 
	vectors $\Big(\prod_{k\in \mathcal{K}}\widehat{B}(x_k,r_k)\Big)\mathsf{v}_0$, 
	with $\mathcal{K}$ ranging over all subsets of $\{1,2,\ldots,N\}$,
	are linearly independent.
	This fact is a corollary of \cite[Lemma\;14]{FelderVarchenko1996}:
	our operators are based on the free fermion 
	six vertex weights, and the cited paper deals with 
	more general eight vertex case.

	Therefore, if we apply 
	the commutation relations in two ways
	and get different coefficients 
	$c_{I;I'}(\mathbf{x};\mathbf{r})$ 
	in \eqref{eq:F_proof_B_product_2},
	then
	we can apply these commutation relations in the above highest weight representation,
	which contradicts the linear independence.
\end{remark}

\begin{lemma}
 	\label{lemma:F_proof_lemma1}
	We have $c_{I;I'}(\mathbf{x};\mathbf{r})=0$
	if $I\cap I'\ne \varnothing$ or $J\cap J'\ne \varnothing$.
\end{lemma}
\begin{proof}
	The two claims with 
	$I\cap I'\ne \varnothing$ 
	and 
	$J\cap J'\ne \varnothing$
	are analogous,
	so we only prove the first one.

	Suppose $I\cap I'\ne \varnothing$. Since
	the operators $\widehat B(x,r)$ commute up to a 
	scalar factor (see
	\eqref{eq:B2hat_B1hat_commute}),
	we may assume that $I\cap I'\ni N$ by permuting terms in
	the left-hand side of \eqref{eq:F_proof_B_product_1}.
	
	Observe that no
	summand in \eqref{eq:F_proof_B_product_1} 
	with $X_N (\mathcal{I}; x_N, r_N) =
	\widehat{D} (x_N, r_N)$ (i.e., 
	$N\in \mathcal{I}$) contributes to a nonzero value of
	$c_{I; I'} (\mathbf{x}, \mathbf{r}) $. 
	Indeed, in this case
	the operator
	$\widehat{D}(x_N,r_N)$ is the leftmost term in 
	$X_{\mathcal{I}}$, and thus it does not get involved in the commutation relations
	of the form \eqref{eq:BD_hat_relation}, which means that one cannot 
	obtain $\widehat{B}(x_N,r_N)$ from this term.
	Similarly, no summand in \eqref{eq:F_proof_B_product_1} 
	with $Y_N
	(\mathcal{I}; x_N, r_N) = \widehat{A} (x_N, r_N)$ (i.e., $N\notin \mathcal{I}$)
	contributes to a
	nonzero value of $c_{I; I'} (\mathbf{x}; \mathbf{r})$.

	However, for any $\mathcal{I}\subset \left\{ 0,1 \right\}^{N}$
	we either have $X_N (\mathcal{I}; x_N, r_N) = \widehat{D} (x_N, r_N)$ or 
	$Y_N(\mathcal{I}; x_N, r_N) = \widehat{A} (x_N, r_N)$, and so we cannot obtain 
	$\widehat{B}(x_N,r_N)$ in both tensor factors. Therefore, terms with $I\cap I'\ne \varnothing$
	are zero.
\end{proof}

We see that in \eqref{eq:F_proof_B_product_2} it must be $I=J'$ and $I'=J$, 
and we may abbreviate $c_{I;I'}=c_I$. We thus rewrite 
\eqref{eq:F_proof_B_product_1}--\eqref{eq:F_proof_B_product_2} as
\begin{equation}
	\label{eq:F_proof_B_product_3}
	\begin{split}
		&
		\widehat{B}(x_N,r_N)\ldots \widehat{B}(x_1,r_1)
		\\
		&\hspace{20pt}=
		\sum_{\substack{I \cup J = \{1,\ldots,  N \}}}
		c_{I} (\mathbf{x}; \mathbf{r})  \,
		\widehat{D} (x_{j_{N - k}}, r_{j_{N - k}})  \ldots \widehat{D} (x_{j_1}, r_{j_1})  
		\widehat{B} (x_{i_k}, r_{i_k}) \ldots \widehat{B} (x_{i_1}, r_{i_1}) 
		\\
		&\hspace{150pt} \otimes 
		\widehat{A} (x_{i_k},r_{i_k}) \ldots \widehat{A} (x_{i_1}, r_{i_1}) 
		\widehat{B} (x_{j_{N-k}}, r_{j_{N-k}}) \ldots \widehat{B} (x_{j_1}, r_{j_1}).
	\end{split}
\end{equation}
We will now evaluate the coefficients $c_I(\mathbf{x};\mathbf{r})$.
First, set $I=\left\{ N-k+1,N-k+2,\ldots,N  \right\}$. Then the operator
\begin{equation}
	\label{eq:DB_AB_proof_B_product}
	\begin{split}
		&\widehat{D}(x_{N-k},r_{N-k})\ldots \widehat{D}(x_1,r_1)
		\widehat{B}(x_{N},r_N)\ldots \widehat{B}(x_{N-k+1},r_{N-k+1})
		\\&\hspace{90pt}\otimes
		\widehat{A}(x_{N},r_N)\ldots \widehat{A}(x_{N-k+1},r_{N-k+1})
		\widehat{B}(x_{N-k},r_{N-k})\ldots \widehat{B}(x_1,r_1)
	\end{split}
\end{equation}
might come from \eqref{eq:F_proof_B_product_1} only
for $\mathcal{I}=\{1,2,\ldots,N-k \}$, in which case
\begin{multline*}
	X_{\mathcal{I}}(\mathbf{x};\mathbf{r})\otimes
	Y_{\mathcal{I}}(\mathbf{x};\mathbf{r})=
	\widehat{B}(x_{N},r_N)\ldots \widehat{B}(x_{N-k+1},r_{N-k+1})
	\widehat{D}(x_{N-k},r_{N-k})\ldots \widehat{D}(x_1,r_1)
	\\\otimes
	\widehat{A}(x_{N},r_N)\ldots \widehat{A}(x_{N-k+1},r_{N-k+1})
	\widehat{B}(x_{N-k},r_{N-k})\ldots \widehat{B}(x_1,r_1).
\end{multline*}
In the first term, we use the commutation relation 
\eqref{eq:BD_hat_relation} to place the $\widehat{D}$ operators on the left
and 
extract the coefficient $c_{I}(\mathbf{x};\mathbf{r})$
of \eqref{eq:DB_AB_proof_B_product}.
We have thus established:

\begin{lemma}
	For $I=I_k:=\left\{ N-k+1,N-k+2,\ldots,N  \right\}$, the rational function $c_I$ is equal to
	\begin{equation*}
		c_{I_k}
		(\mathbf{x};\mathbf{r})=
		\prod_{i=1}^{N-k}
		\prod_{j=N-k+1}^{N}
		\frac{r_i^{-2}x_i-x_j}{x_i-x_j}.
	\end{equation*}
\end{lemma}

We are now in a position to compute $c_I(\mathbf{x};\mathbf{r})$ 
for arbitrary $I\subset \left\{ 1,\ldots,N  \right\}$ of size $k$ (where $k$ is also arbitrary)
by permuting the $\widehat{B}$ operators
in the left-hand side of \eqref{eq:F_proof_B_product_1} 
thanks to the commutation relation \eqref{eq:B2hat_B1hat_commute}.
For each such $I$,
let $\sigma$ be a permutation of $\left\{ 1,\ldots,N  \right\}$
which is increasing on the intervals $\left\{ 1,\ldots,N-k  \right\}$ and
$\left\{ N-k+1,\ldots,N  \right\}$, and 
sends $\left\{ N-k+1,N-k+2,\ldots,N  \right\}$ to $I$.

\begin{lemma}
	With the above notation, we have
	\begin{equation}
		\label{eq:F_proof_B_product_4}
		c_I(\mathbf{x};\mathbf{r})=\mathop{\mathrm{sgn}}(\sigma)
		\prod_{1\le i<j\le N}\frac{r_i^{-2}x_i-x_j}{x_i-x_j}
		\prod_{\substack{i,j\in I\\i<j}}
		\left( \frac{r_i^{-2}x_i-x_j}{x_i-x_j} \right)^{-1}
		\prod_{\substack{i,j\notin I\\ i<j}}
		\left( \frac{r_i^{-2}x_i-x_j}{x_i-x_j} \right)^{-1}.
	\end{equation}
\end{lemma}
\begin{proof}
	The claim follows
	from the fact that
	\begin{equation*}
		c_I(\mathbf{x};\mathbf{r})=
		c_{I_k}(\sigma(\mathbf{x});\sigma(\mathbf{r}))
		\prod_{\substack{j\le N-k,\, i\ge N-k+1\\ \sigma(i)<\sigma(j)}}
		\frac
		{r_{\sigma(i)}^{-2}x_{\sigma(i)}-x_{\sigma(j)}}
		{r_{\sigma(j)}^{-2}x_{\sigma(j)}-x_{\sigma(i)}},
	\end{equation*}
	which in turn holds thanks to \eqref{eq:B2hat_B1hat_commute} 
	via induction on the length of the permutation
	$\sigma$ (which is the minimal number of elementary
	transpositions required to represent $\sigma$ as their product).
	Then we can further simplify:
	\begin{equation*}
		\begin{split}
			&
			c_{I_k}(\sigma(\mathbf{x});\sigma(\mathbf{r}))
			\prod_{\substack{j\le N-k,\, i\ge N-k+1\\ \sigma(i)<\sigma(j)}}
			\frac
			{r_{\sigma(i)}^{-2}x_{\sigma(i)}-x_{\sigma(j)}}
			{r_{\sigma(j)}^{-2}x_{\sigma(j)}-x_{\sigma(i)}}
			=
			\prod_{i\in I,\,j\notin I}\frac{x_i-r_j^{-2}x_j}{x_i-x_j}
			\prod_{i\in I,\,j\notin I \, i<j}\frac{x_j-r_i^{-2}x_i}{x_i-r_j^{-2}x_j}
			\\&\hspace{30pt}=
			\prod_{i\in I, \, j\notin I}(x_i-x_j)^{-1}
			\prod_{i,j\in I, \, i<j}(x_j-r_i^{-2}x_i)^{-1}
			\prod_{i,j\notin I, \, i<j}(x_j-r_i^{-2}x_i)^{-1}
			\prod_{1\le i<j\le N}(x_j-r_i^{-2}x_i),
		\end{split}
	\end{equation*}
	which leads to the desired right-hand side of \eqref{eq:F_proof_B_product_4}.
	The signature of the permutation $\sigma$ arises by 
	turning $x_i-x_j$ into $x_j-x_i$ for each pair $i\notin I,\,j\in I$ with $i>j$.
\end{proof}

\subsubsection{Completing the proof}
\label{appA:F_complete_proof}

We are now in a position to prove the determinantal formula
\eqref{eq:F_det_formula_appendix}, which finalizes the proof of \Cref{thm:F_formula}.
The goal is to express the 
coefficient of
$e_{\mathcal{S}(\lambda)}$ in 
$e_{\varnothing}\widehat{B}(x_N,r_N)\ldots \widehat{B}(x_1,r_1) $.
We are going to repeatedly apply identity 
\eqref{eq:F_proof_B_product_3} (with the coefficients $c_I(\mathbf{x};\mathbf{r})$
given by \eqref{eq:F_proof_B_product_4})
to vectors of the form $e_0^{(m)} \otimes v$, $m=1,2,\ldots $.

Observe that $e_0^{(m)}\widehat{B}(x,r)\widehat{B}(x',r')=0$. Therefore,
any nonzero summand in \eqref{eq:F_proof_B_product_3}
must have $|I| \le 1$.
Moreover, 
when $I=\{i\}$ has one element,
any such nonzero contribution to the coefficient of $e_{\mathcal{S}(\lambda)}$ 
should have $i\in \mathcal{S}(\lambda)=\{\lambda_N+1, \lambda_{N-1}+2, \ldots, \lambda_1+N \}$.
Therefore, each
step of the repeated application of \eqref{eq:F_proof_B_product_3}
for which we choose $|I| = 1$ corresponds to a number from $1$ to $N$
(indicating which element of $\mathcal{S}(\lambda)$ is selected), and these numbers 
must be distinct. We encode this information by a permutation $\tau\in \mathfrak{S}_N$.
Using the facts that
\begin{equation*}
	\begin{split}
		c_{ \varnothing}(\mathbf{x};\mathbf{r})&= 1,
		\qquad \quad 
		c_{ \{k \} }(\mathbf{x};\mathbf{r})=(-1)^{N-k}
		\prod_{i=1}^{k-1}
		\frac{r_i^{-2}x_i-x_k}{x_i-x_k}
		\prod_{j=k+1}^{N}\frac{r_k^{-2}x_k-x_j}{x_k-x_j},
		\\
		e_0^{(m)} \widehat{A}(x,r)&= e_0^{(m)},\qquad 
		e_0^{(m)} \widehat{B}(x,r)=\frac{x(1-r^2)}{r^2(y_m-x)}\,e_1^{(m)},\qquad 
		e_0^{(m)} \widehat{D}(x,r)=\frac{y_m-s_m^2x}{s_m^2(y_m-x)}\,e_0^{(m)},
	\end{split}
\end{equation*}
we see that the coefficient of $e_{\mathcal{S}(\lambda)}$ in 
$e_{\varnothing}\widehat{B}(x_N,r_N)\ldots \widehat{B}(x_1,r_1) $
is equal to
\begin{equation*}
	\prod_{1\le i<j\le N}\frac{r_i^{-2}x_i-x_j}{x_i-x_j}
	\sum_{\tau\in \mathfrak{S}_N}
	\mathop{\mathrm{sgn}}(\tau)
	\prod_{k=1}^{N}
	\left( 
		\frac{x_{\tau(k)}(r_{\tau(k)}^{-2}-1)}{y_{\lambda_{k}+N-k+1}-x_{\tau(k)}}
		\prod_{m=1}^{\lambda_k+N-k}
		\frac{y_m-s_m^2x_{\tau(k)}}{s_m^2(y_m-x_{\tau(k)})}
	\right).
\end{equation*}
Note that the prefactor 
$\prod_{1\le i<j\le N}\frac{r_i^{-2}x_i-x_j}{x_i-x_j}$
arises by taking the product of the $c_{ \{k \} }$'s over all
$k=1,\ldots,N$, 
but in this product for each next term the number $N$ of variables decreases by one.
Therefore, we end up with a product over $i<j$ instead
of over all pairs $i\ne j$.
This completes the proof of \Cref{thm:F_formula}.

\subsection{Proof of \texorpdfstring{\Cref{thm:G_formula}}{formula for G}}
\label{appA:G}

\subsubsection{Recalling the notation}
\label{appA:G_recall}

Throughout this subsection we fix 
$M,N\ge1$,
a
signature $\lambda=(\lambda_1\ge\ldots\ge \lambda_N \ge0)$ with $N$ parts, 
and sequences of complex parameters
\begin{equation*}
	\mathbf{x}=(x_1,\ldots,x_M ),\qquad 
	\mathbf{y}=(y_1,y_2,\ldots ),\qquad 
	\mathbf{r}=(r_1,\ldots,r_M ),\qquad 
	\mathbf{s}=(s_1,s_2,\ldots ).
\end{equation*}
Recall from \Cref{def:G_function}
the function 
$G_\lambda(\mathbf{x};\mathbf{y};\mathbf{r};\mathbf{s})
=G_{\lambda/0^N}(\mathbf{x};\mathbf{y};\mathbf{r};\mathbf{s})$
which is the partition function of the free fermion six vertex model
with weights $W$ \eqref{eq:weights_W}
and with boundary conditions determined by $\lambda$.

Our aim is to prove \Cref{thm:G_formula} which gives an explicit formula for $G_\lambda$ 
\eqref{eq:G_SP_formula_in_theorem}
in terms
of a sum over a pair of permutations. 
The argument is longer than in the case of $F_\lambda$ from \Cref{appA:F}
but also involves manipulations with row operators.
Namely, we utilize the operators $A,B,C,D$ given by 
\eqref{eq:abcdv}--\eqref{eq:abcdv_n2}.
They are built from the vertex weights $W$
and depend on $x,r$ and the sequences $\mathbf{y}, \mathbf{s}$.
These operators act (from the left)
on 
tensor products of two-dimensional spaces
$V^{(k)}=\mathop{\mathrm{span}}\{ e_0^{(k)},e_1^{(k)} \} \simeq \mathbb{C}^2$,
where $k\ge1$.
Recall (\Cref{sub:signature_states})
that to $\lambda$ we associate the vector $e_{\mathcal{S}(\lambda)}$ 
in the finitary subspace $\mathscr{V}$ of the infinite tensor product
$V^{(1)}\otimes V^{(2)}\otimes \ldots $,
where we take $e^{(k)}_1$ in 
the $k$-th place if and only if $k\in \mathcal{S}(\lambda)$ and $e_0^{(k)}$ otherwise,
see \Cref{sub:signature_states}.
Let us also set 
\begin{equation}
	\label{eq:notation_state_e_1_N}
	e_{[1,N]}
	=
	e^{(1)}_1\otimes \ldots \otimes e^{(N)}_1\otimes e^{(N+1)}_0 \otimes e^{(N+2)}_0\otimes \ldots.
\end{equation}
Equip all tensor products
of the spaces $V^{(k)}$
with the inner product defined by
$\langle e_{\mathcal{T}},e_{\mathcal{T}'} \rangle = \mathbf{1}_{\mathcal{T}=\mathcal{T}'}$
(here we use the notation $e_{\mathcal{T}}$ as in \eqref{eq:e_calT_notation}). 
Then by \Cref{prop:F_G_row_op} we have
\begin{equation*}
	G_{\lambda}(\mathbf{x};\mathbf{y};\mathbf{r};\mathbf{s})=
	\left\langle e_{\mathcal{S}(\lambda)}, D(x_M,r_M)\ldots D(x_2,r_2)D(x_1,r_1)e_{[1,N]} \right\rangle .
\end{equation*}
We will compute the above coefficient of $e_{\mathcal{S}(\lambda)}$ in the action of the 
product of the $D$ operators 
using the Yang--Baxter equation stated in \Cref{prop:ABCD_YBE}
as a series of commutation relations between the operators
$A,B,C$, and $D$.

\begin{remark}
	\label{rmk:Shorthand_notation_ABCD_appA2}
	Sometimes, to shorten some formulas in the proofs, we will use notation
	$A_i,B_i,C_i$, or $D_i$ for 
	$A(x_i,r_i),B(x_i,r_i),C(x_i,r_i)$,
	and $D(x_i,r_i)$, respectively.
\end{remark}

\subsubsection{Action of $D$ operators on a two-fold tensor product}
\label{appA:G_action_two}

The next two statements, \Cref{lemma:C_commutation,lemma:D_action_twofold},
are parallel to the computations with the row operators
performed in \Cref{appA:F_two_spaces} in the proof of the formula for $F_\lambda$.

\begin{lemma}
	\label{lemma:C_commutation}
	Let $\sigma\in \mathfrak{S}_M$ be a permutation. Then
	\begin{equation*}
		C(x_{\sigma(M)},r_{\sigma(M)})\ldots 
		C(x_{\sigma(1)},r_{\sigma(1)}) =
		C(x_M,r_M)\ldots C(x_1,r_1)\prod_{\substack{1\le i<j\le M\\ \sigma(j)<\sigma(i)}} 
		\frac{r_{\sigma(j)}^{-2}x_{\sigma(j)}-x_{\sigma(i)}}{r_{\sigma(i)}^{-2}x_{\sigma(i)}-x_{\sigma(j)}}.
	\end{equation*}
\end{lemma}
\begin{proof}
	This is proven by induction on the length of the permutation $\sigma$
	using the commutation relation 
	\eqref{eq:C2C1_commute}
	between the $C$ operators.
\end{proof}

\begin{lemma}
	\label{lemma:D_action_twofold}
	As operators on a tensor product of two spaces $V_1\otimes V_2$, we have
	\begin{equation}
		\label{eq:D_action_twofold}
		\begin{split}
			&
			D(x_M,r_M)D(x_{M-1},r_{M-1})\ldots D(x_1,r_1)
			\\&\hspace{10pt}
			= 
			\sum_{\mathcal{I}\subseteq \left\{ 1,\ldots,M  \right\}}
			\Biggl(
				\prod_{i\in \mathcal{I},\, j\notin \mathcal{I}}
				\frac{r_i^{-2}x_i-x_j}{r_i^{-2}x_i-r_j^{-2}x_j}
			\Biggr)
			B(x_{i_k},r_{i_k})\ldots B(x_{i_1},r_{i_1}) 
			D(x_{j_{M-k}},r_{j_{M-k}})\ldots D(x_{j_1},r_{j_1}) 
			\\&\hspace{180pt}\otimes
			D(x_{j_{M-k}},r_{j_{M-k}})\ldots D(x_{j_1},r_{j_1}) 
			C(x_{i_k},r_{i_k})\ldots C(x_{i_1},r_{i_1}).
		\end{split}
	\end{equation}
	Here
	$\mathcal{I}=(i_1<\ldots<i_k )$ and $\mathcal{J}=\left\{ 1,\ldots,M  \right\}\setminus \mathcal{I}=
	(j_1<\ldots<j_{M-k} )$.
\end{lemma}
\begin{proof}
	In the proof we use the shorthand notation for the operators from \Cref{rmk:Shorthand_notation_ABCD_appA2}.
	By the last identity in \eqref{eq:abcdv_n2}, the action of $D_M\ldots D_1$
	on $V_1\otimes V_2$ is given by 
	\begin{equation}
		\label{eq:D_twofold_proof_1}
		\sum_{\mathcal{K}\subseteq \{1,\ldots,M \}}
		X_{\mathcal{K}}\otimes Y_{\mathcal{K}},
	\end{equation}
	where 
	$X_{\mathcal{K}}=X_M(\mathcal{K})\ldots X_1(\mathcal{K}) $,
	$Y_{\mathcal{K}}=Y_M(\mathcal{K})\ldots Y_1(\mathcal{K}) $,
	with 
	\begin{equation*}
		X_i(\mathcal{K})=\begin{cases}
			B_i,&i\in \mathcal{K};\\
			D_i,&i\notin \mathcal{K},
		\end{cases}
		\qquad 
		Y_i(\mathcal{K})=\begin{cases}
			C_i,&i\in \mathcal{K};\\
			D_i,&i\notin \mathcal{K}.
		\end{cases}
	\end{equation*}
	Next, by repeated use of relations
	\eqref{eq:D2_B1} and \eqref{eq:C2_D1},
	the sum \eqref{eq:D_twofold_proof_1}
	can be expressed in the form
	\begin{equation}
		\label{eq:D_twofold_proof_2}
		\sum_{I,I'\subseteq \left\{ 1,\ldots,M  \right\}}
		h_{I;I'}(\mathbf{x};\mathbf{r})\,
		B_{i_k}\ldots B_{i_1}D_{j_{M-k}}\ldots D_{j_1} \otimes
		D_{i_m'}\ldots D_{i_1'}
		C_{j'_{M-m}}\ldots C_{j_1'} 
		,
	\end{equation}
	where $h_{I;I'}(\mathbf{x};\mathbf{r})$ are 
	rational functions in $\mathbf{x}=(x_1,\ldots,x_M )$ and 
	$\mathbf{r}=(r_1,\ldots,r_M )$, and the indices are
	\begin{equation*}
		\begin{split}
			I=(i_1<\ldots<i_k ),
			\qquad 
			I'&=(i_1'<\ldots<i_m' ),
			\\ 
			J=I^c=(j_1<\ldots<j_{M-k} ),
			\qquad 
			J'&=(I')^c=(j_1'<\ldots<j_{M-m}' ).
		\end{split}
	\end{equation*}
	
	By looking at relations
	\eqref{eq:D2_B1}, \eqref{eq:C2_D1}
	closer,
	one can already see that 
	$m=M-k$ in \eqref{eq:D_twofold_proof_2}.
	By \Cref{rmk:uniqueness_of_coeffs},
	the coefficients 
	$h_{I;I'}(\mathbf{x};\mathbf{r})$ are independent of the 
	order in which we apply the commutation relations between 
	the operators $A,B,C,D$ to get from 
	\eqref{eq:D_twofold_proof_1} to 
	\eqref{eq:D_twofold_proof_2}.
	
	By the same argument as in \Cref{lemma:F_proof_lemma1},
	one can show that $h_{I;I'}(\mathbf{x};\mathbf{r})=0$
	if $I\cap I'\ne \varnothing$ or $J\cap J'\ne \varnothing$.
	Thus, it must be that $I=J'$ and $J=I'$, and
	we may rewrite $h_I(\mathbf{x};\mathbf{r})=h_{I;I'}(\mathbf{x};\mathbf{r})$.
	This implies that we may write \eqref{eq:D_twofold_proof_2}
	as
	\begin{equation}
		\label{eq:D_twofold_proof_3}
		\sum_{I\cup J= \{1,\ldots,M \}}
		h_{I}(\mathbf{x};\mathbf{r})\,
		B_{i_k}\ldots B_{i_1}D_{j_{M-k}}\ldots D_{j_1} \otimes
		D_{j_{M-k}}\ldots D_{j_1}
		C_{i_k}\ldots C_{i_1}. 
	\end{equation}

	It remains to evaluate the coefficients $h_I(\mathbf{x};\mathbf{r})$
	in
	\eqref{eq:D_twofold_proof_3}. This is simpler than for 
	the case of $F_\lambda$ considered in \Cref{appA:F_two_spaces}.
	First, assume that 
	$I=I_k:=\left\{ 1,2,\ldots,k   \right\}$.
	In this case, applying 
	\eqref{eq:D2_B1} and \eqref{eq:C2_D1}
	to a term $X_{\mathcal{K}} \otimes
	Y_{\mathcal{K}}$ in \eqref{eq:D_twofold_proof_1}
	only gives rise to a
	nonzero multiple of $B_{k} \cdots B_{1} 
	D_{M} \cdots D_{k+1} \otimes D_{M} \cdots D_{k+1} C_{k} \cdots C_{1}$ 
	as a summand only if $\mathcal{K} = I_k$.
	Indeed, otherwise let $k_0=\min \mathcal{K}^c\le k$.
	In any expression of $X_{\mathcal{K}} \otimes
	Y_{\mathcal{K}}$ as a linear combination of $B_{i_k}  \cdots B_{i_1}
	D_{j_{M - k}} \cdots D_{j_1} \otimes D_{j_{M - k}} \cdots D_{j_1}
	C_{i_k} \cdots C_{i_1}$,
	one needs to commute $D_{k_0}$ to the right through $X_{k_0-1},\ldots,X_1 $,
	which implies that $j_1\le k_0\le k$. Therefore, it must be $\mathcal{K}=I_k$.

	For $\mathcal{K}=I_k$,
	the only way of 
	obtaining 
	$B_{k} \cdots B_{1} 
	D_{M} \cdots D_{k+1} \otimes D_{M} \cdots D_{k+1} C_{k} \cdots C_{1}$ 
	from $X_{\mathcal{K}} \otimes Y_{\mathcal{K}}$ 
	is through using 
	\eqref{eq:D2_B1}
	to commute each $D_j$ to the
	right of each $B_i$. This produces a factor of 
	$(r_i^{-2}x_i-x_j)/(r_i^{-2}x_i-r_j^{-2}x_j)$
	for each such commutation, and so
	\begin{equation*}
		h_{I_k}(\mathbf{x};\mathbf{r})=
		\prod_{i\in I_k,\,j\notin I_k}
		\frac{r_i^{-2}x_i-x_j}{r_i^{-2}x_i-r_j^{-2}x_j}.
	\end{equation*}

	Finally, to get $h_I$ for general $I$,
	observe that the operators $D_i$ commute by \eqref{eq:D2D1_commute},
	and therefore 
	$h_{\sigma(I_k)}(\mathbf{x};\mathbf{r})=h_{I_k}(\sigma(\mathbf{x});\sigma(\mathbf{r}))$,
	which are precisely the coefficients in the claimed identity in the present lemma, where $\sigma$
	takes $I_k$ to an arbitrary $I$.
	This completes the proof.
\end{proof}

For the next proposition, recall the notation
$d = d(\lambda) \ge 0$ which is
the integer such that $\lambda_d \ge d$ and $\lambda_{d + 1} < d + 1$,
and
$\mu = (\mu_1< \mu_2< \ldots < \mu_d) =
\{1,\ldots,N \}  \setminus \big( \mathcal{S}(\lambda) \cap \{1,\ldots,N \} \big)$.
Also consider the $N$-fold tensor product
$V^{(1)}\otimes \ldots \otimes V^{(N)}$,
and take the following vectors in this space
\begin{equation}
	\label{eq:e_S_N_lambda_vector}
	e_{\mathcal{S}_N(\lambda)}:=e^{(1)}_{m_1}\otimes e^{(2)}_{m_2}\otimes \ldots\otimes e^{(N)}_{m_N},
	\qquad 
	e_{[1,N]}=e^{(1)}_{1}\otimes e^{(2)}_{1}\otimes \ldots\otimes e^{(N)}_{1},
\end{equation}
where $m_i=\mathbf{1}_{i\in \mathcal{S}(\lambda)}$,
and with $e_{[1,N]}$ we are slightly abusing the notation, cf. \eqref{eq:notation_state_e_1_N}.

\begin{proposition}
	\label{prop:G_proof_main_proposition_1}
	With the above notation, 
	for any vectors $v_1,v_2\in V^{(N+1)}\otimes V^{(N+2)}\otimes \ldots $
	we have
	\begin{equation}
		\label{eq:G_proof_main_proposition_1}
		\begin{split}
			&
			\left\langle e_{\mathcal{S}_N(\lambda)}\otimes v_2,
			D(x_M,r_M)\ldots D(x_2,r_2)D(x_1,r_1) (e_{[1,N]}\otimes v_1)\right\rangle 
			\\&
			\hspace{2pt}=
			\prod_{j=1}^{M}\prod_{k=1}^{N}
			\frac{y_k-s_k^2r_j^{-2}x_j}{y_k-s_k^2 x_j}
			\sum_{\substack{\mathcal{I}\subseteq \left\{ 1,\ldots,M  \right\}\\ |\mathcal{I}|=d}}
			\left\langle v_2,
			\biggl( 
				\prod_{j\notin \mathcal{I}}D(x_j,r_j)
			\biggr)
			C(x_{i_d},r_{i_d})
			\ldots 
			C(x_{i_1},r_{i_1})
			v_1
			\right\rangle 
			\\&\hspace{80pt}
			\times
			\prod_{i \in \mathcal{I},\, j\notin \mathcal{I}}
			\frac{r_i^{-2}x_i-x_j}{r_i^{-2}x_i-r_j^{-2}x_j}
			\prod_{i,j \in \mathcal{I},\,i<j}
			\frac{r_i^{-2}x_i-x_j}{r_i^{-2}x_i-r_j^{-2}x_j}
			\\&\hspace{80pt}
			\times
			\sum_{\sigma\in \mathfrak{S}_d}
			\mathop{\mathrm{sgn}}(\sigma)
			\prod_{j=1}^d
			\biggl(
				\frac{s_{\mu_j}^2 x_{i_{\sigma (j)}} \big(
				r^{-2}_{i_{\sigma (j)}} - 1\big)}{y_{\mu_j} - s_{\mu_j}^2
				r^{-2}_{i_{\sigma (j)}} x_{i_{\sigma (j)}}} 
				\prod_{k =
				\mu_j + 1}^N 
				\frac{s_k^2 \big(r^{-2}_{i_{\sigma (j)}}
				x_{i_{\sigma (j)}} - y_k \big)}{y_k - s_k^2 r^{-2}_{i_{\sigma (j)}}
				x_{i_{\sigma (j)}}}
			\biggr),
		\end{split}
	\end{equation}
	where 
	$\mathcal{I}=(i_1<\ldots<i_d )$.
\end{proposition}
The right-hand side of 
\eqref{eq:G_proof_main_proposition_1}
vanishes if $d(\lambda)>M$. Observe that the 
same is true for the left-hand side. Indeed,
a single $D$ operator moves at most one vertical arrow somewhere to the right,
and $d$ is the number of gaps (sites with no vertical arrows)
among $\left\{ 1,\ldots,N  \right\}$ in the configuration encoded by $e_{\mathcal{S}_N(\lambda)}$, 
so $d$ should not be larger than $M$.

\begin{proof}[Proof of \Cref{prop:G_proof_main_proposition_1}]
	In this proof we use the shorthand notation for the operators, see \Cref{rmk:Shorthand_notation_ABCD_appA2}.
	As a first step, we consider how the action of the 
	product of the $D$ and $C$ operators like in the right-hand side of \eqref{eq:G_proof_main_proposition_1} acts on 
	tensor products. Fix an integer $n>0$, a
	subset $\mathcal{H}=\left\{ h_1,h_2,\ldots,h_k  \right\}\subseteq \left\{ 1,\ldots,M  \right\}$,
	and $u,w\in V^{(n+1)}\otimes V^{(n+2)}\otimes \ldots $.
	Then we have
	\begin{equation}
		\label{eq:G_proof_main_proposition_1_proof1}
		\begin{split}
			\Bigg\langle e_1^{(n)} \otimes w, & \bigg( \prod_{j
			\notin \mathcal{H}} D_j \bigg) C_{h_k} \cdots C_{h_1} \big(
			e_1^{(n)} \otimes u \big) \Bigg\rangle \\ & = \Bigg\langle
			e_1^{(n)}, \bigg( \prod_{j \notin \mathcal{H}} D_j
			\bigg) A_{h_k} \cdots A_{h_1} e_1^{(n)} \Bigg\rangle
			\Bigg\langle w, \bigg( \prod_{j \notin \mathcal{H}}
			D_j \bigg) C_{h_k} \cdots C_{h_1} u \Bigg\rangle. 
		\end{split}
	\end{equation}
	and
	\begin{equation}
		\label{eq:G_proof_main_proposition_1_proof2}
		\begin{split}
			&
			\Bigg\langle e_0^{(n)} \otimes w, \bigg( \prod_{j
			\notin \mathcal{H}} D_j \bigg) C_{h_k} \cdots C_{h_1} \big(
			e_1^{(n)} \otimes u \big) \Bigg\rangle \\ &\hspace{20pt} =
			\sum_{i \notin \mathcal{H}}  \Bigg\langle
			e_0^{(n)}, B_i \bigg( \prod_{j \notin \mathcal{H}
			\cup \{ i \}} D_j \bigg) A_{h_k} \cdots A_{h_1} e_1^{(n)}
			\Bigg\rangle \Bigg\langle w, \bigg( \prod_{j \notin
			\mathcal{H} \cup \{ i \}} D_j \bigg) C_i C_{h_k} \cdots C_{h_1}
			u \Bigg\rangle \\ &\hspace{40pt} \times \prod_{j
			\notin \mathcal{H} \cup \{ i \} } \frac{r^{-2}_i x_i -
			x_j}{r^{-2}_i x_i - r^{-2}_j x_j}.
		\end{split}
	\end{equation}
	Indeed, observe that 
	$C(x,r)$ maps $e_1^{(n)}$ to $0$, so 
	by
	the third statement in \eqref{eq:abcdv_n2} we have
	\begin{equation*}
		C_{h_k}\ldots C_{h_1}(e_1^{(n)}\otimes u)= 
		A_{h_k}\ldots A_{h_1} e_1^{(n)}\otimes  
		C_{h_k}\ldots C_{h_1}u.
	\end{equation*}
	When applying a product of the $D_j$'s to this vector, 
	a nonzero term with $e_1^{(n)}$ in the first tensor factor
	may appear only if we act each time by the operators $D$ on both
	tensor factors, see the fourth statement in \eqref{eq:abcdv_n2}.
	This (together with the fact 
	that $\langle \cdot,\cdot \rangle $
	is multiplicative with respect to the tensor product)
	leads to \eqref{eq:G_proof_main_proposition_1_proof1}.
	For 
	\eqref{eq:G_proof_main_proposition_1_proof2}, we use
	\Cref{lemma:D_action_twofold} expressing the action of a product of the $D_j$'s on a
	tensor product, and observe that a nonzero term with
	$e_0^{(n)}$ in the first
	tensor factor
	may appear only if $|\mathcal{I}| =1$ in the right-hand side of \eqref{eq:D_action_twofold}.

	The action of all the operators on $e_1^{(n)}$ in the right-hand sides 
	of \eqref{eq:G_proof_main_proposition_1_proof1}--\eqref{eq:G_proof_main_proposition_1_proof2}
	is explicit by 
	\eqref{eq:abcdv}
	and
	\eqref{eq:weights_W}:
	\begin{equation*}
		\begin{split}
			&
			\Bigg\langle
			e_1^{(n)}, \bigg( \displaystyle\prod_{j \notin \mathcal{H}} D_j
			\bigg) A_{h_k} \cdots A_{h_1} e_1^{(n)} \Bigg\rangle
			=
			\prod_{k\in \mathcal{H}}
			\frac{s_n^2(x_k-r_k^2 y_n)}{r_k^2(y_n-s_n^2x_k)}
			\prod_{j\notin \mathcal{H}}
			\frac{y_n-s_n^2r_j^{-2} x_j}{y_n-s_n^2x_j};
			\\
			&
			\Bigg\langle
			e_0^{(n)}, B_i \bigg( \displaystyle\prod_{j \notin \mathcal{H}
			\cup \{ i \}} D_j \bigg) A_{h_k} \cdots A_{h_1} e_1^{(n)}
			\Bigg\rangle
			\\
			&
			\hspace{120pt}=
			\frac{s_n^2x_i(r_i^{-2}-1)}{y_n-s_n^2x_i}
			\prod_{k\in \mathcal{H}}
			\frac{s_n^2(x_k-r_k^2 y_n)}{r_k^2(y_n-s_n^2x_k)}
			\prod_{j\notin \mathcal{H} \cup \left\{ i \right\}}
			\frac{y_n-s_n^2r_j^{-2} x_j}{y_n-s_n^2x_j}.
		\end{split}
	\end{equation*}
	This means that we can continue our identities as
	\begin{equation}
		\label{eq:G_proof_main_proposition_1_proof1_next}
		\begin{split}
			\eqref{eq:G_proof_main_proposition_1_proof1}&=
			\Bigg\langle w, \bigg( \displaystyle\prod_{j \notin \mathcal{H}}
			D_j \bigg) C_{h_k} \cdots C_{h_1} u \Bigg\rangle
			\prod_{k\in \mathcal{H}}
			\frac{s_n^2(x_k-r_k^2 y_n)}{r_k^2(y_n-s_n^2x_k)}
			\prod_{j\notin \mathcal{H}}
			\frac{y_n-s_n^2r_j^{-2} x_j}{y_n-s_n^2x_j};
			\\
			\eqref{eq:G_proof_main_proposition_1_proof2}&=
			\sum_{i \notin \mathcal{H}} 
			\Bigg\langle w, \bigg( \prod_{j \notin
			\mathcal{H} \cup \{ i \}} D_j \bigg) C_i C_{h_k} \cdots C_{h_1}
			u \Bigg\rangle 
			\frac{s_n^2x_i(r_i^{-2}-1)}{y_n-s_n^2x_i}
			\prod_{k\in \mathcal{H}}
			\frac{s_n^2(x_k-r_k^2 y_n)}{r_k^2(y_n-s_n^2x_k)}
			\\ &\hspace{80pt} \times 
			\prod_{j\notin \mathcal{H} \cup \left\{ i \right\}}
			\frac{y_n-s_n^2r_j^{-2} x_j}{y_n-s_n^2x_j}
			\prod_{j
			\notin \mathcal{H} \cup \{ i \} } \frac{r^{-2}_i x_i -
			x_j}{r^{-2}_i x_i - r^{-2}_j x_j}.
		\end{split}
	\end{equation}

	Now
	we can evaluate
	\begin{equation*}
		\left\langle e_{\mathcal{S}_N(\lambda)}\otimes v_2,
		D(x_M,r_M)\ldots D(x_2,r_2)D(x_1,r_1) (e_{[1,N]}\otimes v_1)\right\rangle
	\end{equation*}
	by repeatedly using 
	\eqref{eq:G_proof_main_proposition_1_proof1_next}. 
	Start with $\mathcal{H}=\varnothing$, and apply the 
	first identity in \eqref{eq:G_proof_main_proposition_1_proof1_next}
	for each $n\notin \mu=\left\{ 1,\ldots,N  \right\}\setminus \mathcal{S}(\lambda)$,
	and the second identity in \eqref{eq:G_proof_main_proposition_1_proof1_next} for each $n\in \mu$.
	Each application of the latter
	involves choosing an index $i\notin \mathcal{H}$. 
	This freedom is encoded by the data
	$(\mathcal{I},\sigma)$, where $\mathcal{I}=\{i_1<i_2< \ldots < i_d \}\subseteq\left\{ 1,\ldots,M  \right\}$ 
	and $\sigma\in \mathfrak{S}_d$,
	such that at each step when $n=\mu_k\in \mu$ we remove the index 
	$i_{\sigma(k)}$.
	For each fixed $(\mathcal{I},\sigma)$
	we have the following factors in the resulting expansion:
	\begin{enumerate}[$\bullet$]
		\item 
			The inner product term 
			$\displaystyle \Bigl\langle v_2, \Bigl( \prod\limits_{j\notin \mathcal{I}} D_j\Bigr)
			C_{i_d}\ldots C_{i_1} 
			v_1\Bigr\rangle 
			\prod\limits_{1\le \upalpha<\upbeta\le d\colon \sigma(\upbeta)<\sigma (\upalpha)}
			\frac{r_{i_{\sigma(\upbeta)}}^{-2}
			x_{i_{\sigma(\upbeta)}}-
			x_{i_{\sigma(\upalpha)}}}
			{r_{i_{\sigma(\upalpha)}}^{-2}
			x_{i_{\sigma(\upalpha)}}-x_{i_{\sigma(\upbeta)}}}
			$,
			where the last factor comes from reordering the $C$ operators thanks to \Cref{lemma:C_commutation}.
		\item The factor 
			$\displaystyle \prod\limits_{i\in \mathcal{I},\, j\notin \mathcal{I}}
			\frac{r^{-2}_i x_i-x_j}{r^{-2}_i x_i - r^{-2}_j x_j}
			\prod\limits_{1\le \upalpha<\upbeta\le d}
			\frac{r^{-2}_{i_{\sigma(\upalpha)}}
			x_{i_{\sigma(\upalpha)}}-x_{i_{\sigma(\upbeta)}}}{r^{-2}_{i_{\sigma(\upalpha)}}
			x_{i_{\sigma(\upalpha)}} - r^{-2}_{i_{\sigma(\upbeta)}}
			x_{i_{\sigma(\upbeta)}}}$
			arises by applying the second identity in \eqref{eq:G_proof_main_proposition_1_proof1_next}
			for each $n\in \mu$.
			Reordering the denominator in the second factor gives
			\begin{equation*}
			\prod\limits_{1\le \upalpha<\upbeta\le d}
			\frac{1}
			{r^{-2}_{i_{\sigma(\upalpha)}}
			x_{i_{\sigma(\upalpha)}} - r^{-2}_{i_{\sigma(\upbeta)}}
			x_{i_{\sigma(\upbeta)}}}
			=\mathop{\mathrm{sgn}}(\sigma)
			\prod_{i,j\in \mathcal{I},\, i<j}
			\frac{1}{r_i^{-2}x_i-r_j^{-2}x_j}.
			\end{equation*}
		\item The product $\displaystyle\prod\limits_{j=1}^d
			\frac{s_{\mu_j}^2 x_{i_{\sigma(j)}}(r_{i_{\sigma(j)}}^{-2}-1)}{y_{\mu_j}-s_{\mu_j}^2x_{i_{\sigma(j)}}}$
			is composed of one factor per each 
			application of the second identity in \eqref{eq:G_proof_main_proposition_1_proof1_next}
			corresponding to $n=\mu_j\in \mu$.
		\item The product 
			$\displaystyle\prod_{j=1}^{d}\prod_{n=\mu_j+1}^N
			\frac{s_n^2(x_{i_{\sigma(j)}}-r_{i_{\sigma(j)}}^2 y_n)}{r_{i_{\sigma(j)}}^2(y_n-s_n^2x_{i_{\sigma(j)}})}
			$
			arises from both identities in
			\eqref{eq:G_proof_main_proposition_1_proof1_next}
			which contain the same products over $k\in \mathcal{H}$.
		\item Finally, the product 
			$\displaystyle
			\biggl(
				\prod_{n=1}^N\prod_{j=1}^{M}
				\frac{y_n-s_n^2r_j^{-2} x_j}{y_n-s_n^2x_j}
			\biggr)
			\biggl(
				\prod_{j=1}^d
				\prod_{n=\mu_j}^{N}
				\frac{y_n-s_n^2x_{i_{\sigma(j)}}}{y_n-s_n^2r_{i_{\sigma(j)}}^{-2} x_{i_{\sigma(j)}}}
			\biggr)
			$
			arises from the products over $j\notin \mathcal{H}$ or $j\notin\mathcal{H}\cup \left\{ i \right\}$
			in \eqref{eq:G_proof_main_proposition_1_proof1_next}.
	\end{enumerate}
	Combining all the terms
	yields the desired identity.
\end{proof}

\subsubsection{Commutation of the operators $C$ and $D$}
\label{appA:G_CD_commutation}

In this subsection we establish one of the key formulas concerning the commutation of the 
operators $C$ and $D$. We fix $M,N\ge1$ and sequences of complex numbers
\begin{equation*}
	\mathbf{x}=(x_1,\ldots,x_N ),\qquad 
	\mathbf{r}=(r_1,\ldots,r_N ),\qquad 
	\mathbf{w}=(w_1,\ldots,w_M ),\qquad 
	\boldsymbol\uptheta=(\theta_1,\ldots,\theta_M ).
\end{equation*}

\begin{proposition}
	\label{prop:C_D_commutation}
	We have 
	\begin{equation}
		\label{eq:C_D_commutation}
		\begin{split}
			&D(x_N,r_N)\ldots D(x_1,r_1) 
			C(w_M,\theta_M)\ldots C(w_1,\theta_1) 
			\\&\hspace{5pt}
			=
			\sum_{\substack{\mathcal{I}\subseteq\left\{ 1,\ldots,N  \right\}\\
			\mathcal{H}\subseteq \left\{ 1,\ldots,M  \right\}\\
			|\mathcal{I}|+|\mathcal{H}|=M}}
			C(x_{i_k},r_{i_k})\ldots C(x_{i_1},r_{i_1})\,
			C(w_{h_{M-k}},\theta_{h_{M-k}})\ldots C(w_{h_1},\theta_{h_1}) 
			\prod_{j\notin \mathcal{H}}D(w_j,\theta_j)
			\prod_{j\notin \mathcal{I}}D(x_j,r_j)
			\\&\hspace{40pt}
			\times
			\prod_{i\in \mathcal{I}}(1-r_i^{-2})x_i
			\prod_{i\in \mathcal{I},\, j\notin \mathcal{I}}\frac{r_j^{-2}x_j-x_i}{x_j-x_i}
			\prod_{h\in \mathcal{H},\, j\notin \mathcal{H}}\frac{1}{w_j-w_h}
			\prod_{h\in \mathcal{H},\, j\notin \mathcal{I}}\frac{r_j^{-2}x_j-w_h}{x_j-w_h}
			\prod_{i\in \mathcal{I},\, j\notin \mathcal{H}}\frac{1}{x_i-w_j}
			\\&\hspace{40pt}
			\times
			\prod_{i,j\in \mathcal{I},\, i<j}(r_i^{-2}x_i-x_j)
			\prod_{i,h\in \mathcal{H},\, h<i}\frac{1}{\theta_i^{-2}w_i-w_h}
			\prod_{1\le i<j\le M}(\theta_j^{-2}w_j-w_i).
		\end{split}
	\end{equation}
	Here $\mathcal{I}=(i_1<\ldots<i_k )$ and 
	$\mathcal{H}=(h_1<\ldots<h_{M-k} )$.
\end{proposition}
Recall that the operators $D(x_j,r_j)$ commute by \eqref{eq:D2D1_commute},
so we can write their products in any order. This is not the case 
for the operators $C(w_j,\theta_j)$, which is why their order in \eqref{eq:C_D_commutation}
must be specified explicitly.

The rest of this subsection is devoted to the proof of \Cref{prop:C_D_commutation}.
As a first step, let us establish the claim for $M=1$:
\begin{lemma}[\Cref{prop:C_D_commutation} for $M=1$]
	\label{lemma:C_D_commutation_M1_case}
	We have
	\begin{equation}
		\label{eq:C_D_commutation_M1_case}
		\begin{split}
			&
			D(x_N,r_N)\ldots D(x_1,r_1)
			C(w,\theta)
			=
			C(w,\theta)
			D(x_1,r_1)\ldots D(x_N,r_N)
			\prod_{j=1}^{N}\frac{r_j^{-2}x_j-w}{x_j-w}
			\\&\hspace{100pt}
			+
			\sum_{i=1}^{N}
			\biggl(
				C(x_i,r_i)D(w,\theta)\prod_{j\ne i}D(x_j,r_j)
			\biggr)
			\frac{(1-r_i^{-2})x_i}{x_i-w}\prod_{j\ne i}\frac{r_j^{-2}x_j-x_i}{x_j-x_i}.
		\end{split}
	\end{equation}
\end{lemma}
\begin{proof}
	The first term containing $C(w,\theta)
	D(x_1,r_1)\ldots D(x_N,r_N)$ may only arise 
	if we are picking the first summand in 
	\eqref{eq:D2_C1} for each commutation. This produces the desired product
	$\prod_{j=1}^{N}\frac{r_j^{-2}x_j-w}{x_j-w}$ as a prefactor.

	Now let us explain how to get the summand in the second sum corresponding to $i=1$.
	Thanks to the commutativity of the $D(x_j,r_j)$'s, the form of the other summands then would follow.
	To get the term containing $C(x_1,r_1)D(w,\theta)D(x_2,r_2)\ldots D(x_N,r_N)$,
	we must pick the second summand in \eqref{eq:D2_C1} once,
	when moving $C(w,\theta)$ to the left of $D(x_1,r_1)$. This produces 
	$C(x_1,r_1)D(w,\theta)\frac{(1-r_1^{-2})x_1}{x_1-w}$.
	After that, we move $C(x_1,r_1)$ to the left of all the other $D(x_j,r_j)$'s,
	always picking the first summand in \eqref{eq:D2_C1}.
	This produces the desired identity.
\end{proof}

We now consider the general case $M,N\ge1$ of \eqref{eq:C_D_commutation}.
First, repeatedly using 
relations \eqref{eq:C2C1_commute}, \eqref{eq:D2D1_commute}, and
\eqref{eq:D2_C1}, we have
\begin{equation}
	\label{eq:C_D_commutation_proof_1}
	\begin{split}
			&D(x_N,r_N)\ldots D(x_1,r_1) 
			C(w_M,\theta_M)\ldots C(w_1,\theta_1) 
			\\&\hspace{25pt}
			=
			\sum_{\mathcal{I},\mathcal{H}}
			C(x_{i_k},r_{i_k})\ldots C(x_{i_1},r_{i_1})\,
			C(w_{h_{M-k}},\theta_{h_{M-k}})\ldots C(w_{h_1},\theta_{h_1}) 
			\\&\hspace{140pt}
			\times
			\prod_{j\notin \mathcal{H}}D(w_j,\theta_j)
			\prod_{j\notin \mathcal{I}}D(x_j,r_j)
			R_{\mathcal{I};\mathcal{H}}(\mathbf{w};\mathbf{x};\boldsymbol\uptheta;\mathbf{r}),
	\end{split}
\end{equation}
where the sum is taken over
$\mathcal{I}\subseteq\left\{ 1,\ldots,N  \right\}$ and 
$\mathcal{H}\subseteq \left\{ 1,\ldots,M  \right\}$,
such that 
$|\mathcal{I}| =k$, $|\mathcal{H}| =M-k$, and $k$ is arbitrary (see \eqref{eq:C_D_commutation}).
Here $R_{\mathcal{I};\mathcal{H}}$ are some rational functions
which we will now evaluate.
\begin{lemma}[Evaluation of $R_{\mathcal{I};\mathcal{H}}$ in a special case]
	\label{lemma:C_D_commutation_H_special_case}
	Let $\mathcal{H}=\{1,2,\ldots,M-k \}$,
	and $\mathcal{I}=(i_1<\ldots<i_k )\subseteq\left\{ 1,\ldots,N  \right\}$
	with $|\mathcal{I}| =k$ be arbitrary.
	Then
	\begin{equation}
		\label{eq:C_D_commutation_H_special_case}
		\begin{split}
			&
			R_{\mathcal{I};\mathcal{H}}(\mathbf{w};\mathbf{x};\boldsymbol\uptheta;\mathbf{r})
			=
			\prod_{i\in \mathcal{I}}(1-r_i^{-2})x_i
			\prod_{i\in \mathcal{I},\, j\notin \mathcal{I}}\frac{r_j^{-2}x_j-x_i}{x_j-x_i}
			\prod_{h\in \mathcal{H},\, j\notin \mathcal{H}}\frac{1}{w_j-w_h}
			\prod_{h\in \mathcal{H},\, j\notin \mathcal{I}}\frac{r_j^{-2}x_j-w_h}{x_j-w_h}
			\\&\hspace{40pt}
			\times
			\prod_{i\in \mathcal{I},\, j\notin \mathcal{H}}\frac{1}{x_i-w_j}
			\prod_{i,j\in \mathcal{I},\, i<j}(r_i^{-2}x_i-x_j)
			\prod_{i,h\in \mathcal{H},\, h<i}\frac{1}{\theta_i^{-2}w_i-w_h}
			\prod_{1\le i<j\le M}(\theta_j^{-2}w_j-w_i).
		\end{split}
	\end{equation}
\end{lemma}
\begin{proof}
	From the left-hand side of \eqref{eq:C_D_commutation_proof_1}, we apply \eqref{eq:D2_C1}
	(together with permutation relations \eqref{eq:C2C1_commute}, \eqref{eq:D2D1_commute} for
	the operators $C,D$)
	to move all the operators $C$ to the left of all the operators $D$. 
	The operator 
	\begin{equation*}
		C(x_{i_k},r_{i_k})\ldots C(x_{i_1},r_{i_1})
		C(w_{M-k},\theta_{M-k})\ldots C(w_1,\theta_1) 
		\prod_{j=M-k+1}^{M}D(w_j,\theta_j)
		\prod_{j\notin \mathcal{I}}D(x_j,r_j)
	\end{equation*}
	may arise, after a sequence of
	applications of \Cref{lemma:C_D_commutation_M1_case},
	only if there exists a permutation $\sigma\in \mathfrak{S}_k$
	such that the following two conditions are met:
	\begin{enumerate}[$\bullet$]
		\item When moving each $C(w_{M-k+j},\theta_{M-k+j})$, $1\le j\le k$, to the left,
			turn $(w_{M-k+j},\theta_{M-k+j})$ into $(x_{i_{\sigma(j)}},r_{i_{\sigma(j)}})$.
			This corresponds to picking the second summand in \eqref{eq:D2_C1}, and this 
			swapping of parameters may happen only once per each $C$ operator.
		\item 
			When moving each $C(w_j,\theta_j)$, $1\le j\le M-k$, to the left, 
			we always pick the first summand in \eqref{eq:D2_C1}, and 
			the parameters $(w_j,\theta_j)$ stay the same throughout the exchanges. 
	\end{enumerate}
	
	To be able to put all the coefficients together,
	denote 
	$\sigma_t (\mathcal{I}) = \big( i_{\sigma (t)}, i_{\sigma (t + 1)},
	\ldots , i_{\sigma (k)} \big)$ for each $1\le t\le k$.
	Then, for each integer $1\le j\le k$, when attempting to commute
	$C(w_{M - k + j}, \theta_{M - k + j})$ to the left of 
	\begin{equation*}
	\prod_{h \notin \sigma_{j + 1} (\mathcal{I})}  D (x_h, r_h) \prod_{h = j + 1}^k D (w_{M - k + h}, \theta_{M - k + h}),
	\end{equation*}
	we obtain 
	\begin{equation*}
		C \big( x_{i_{\sigma (j)}}, r_{i_{\sigma (j)}} \big) \prod_{h \notin \sigma_j (\mathcal{I})} D (x_h, r_h) \prod_{h = j}^k D (w_{M - k + h}, \theta_{M - k + h}).
	\end{equation*}
	By \Cref{lemma:C_D_commutation_M1_case}, this contributes a factor of
	\begin{equation}
		\label{eq:C_D_commutation_proof_2}
		\frac{\big( 1 - r^{-2}_{i_{\sigma (j)}} \big) x_{i_{\sigma
		(j)}}}{x_{i_{\sigma (j)}} - w_{M - k + j}} \prod_{h = M - k + j +
		1}^M \frac{\theta_h^{-2} w_h - x_{i_{\sigma (j)}}}{w_h - x_{i_{\sigma
		(j)}}} \prod_{h \notin \sigma_j (\mathcal{I})} \frac{r^{-2}_h x_h -
		x_{i_{\sigma (j)}}}{x_h - x_{i_{\sigma (j)}}}.
	\end{equation}
	This deals with the first case above when we swap the parameters between $C$ and $D$ operators.

	In the second case when we do not swap the parameters, 
	each $C (w_j, \theta_j)$ for $1\le j\le M-k$ must be commuted to the left of 
	$\prod_{h \notin \mathcal{I}}  D (x_h, r_h) \prod_{h = M - k + 1}^M D (w_h, \theta_h)$,
	which contributes
	\begin{equation}
	\label{eq:C_D_commutation_proof_3}
	\prod_{h \notin \mathcal{I}} \frac{r^{-2}_h
		x_h - w_j}{x_h - w_j} \prod_{h = M - k + 1}^M
		\frac{\theta_h^{-2} w_h - w_j}{w_h - w_j}.  
	\end{equation}	
		
	Observe that
	\begin{equation}
		\label{eq:C_D_commutation_proof_4}
		\begin{split}
			\prod_{j = 1}^k \big( 1 -
			r^{-2}_{i_{\sigma (j)}} \big) x_{i_{\sigma (j)}} &=
			\prod_{i \in \mathcal{I}} (1 - r^{-2}_i) x_i; 
			\\ 
			\prod_{j = 1}^k \prod_{h \notin \sigma_j
			(\mathcal{I})} \frac{r^{-2}_h x_h - x_{i_{\sigma
			(j)}}}{x_h - x_{i_{\sigma (j)}}} &= \prod_{
				i\in \mathcal{I},\, h \notin  \mathcal{I}} \frac{r^{-2}_h
			x_h - x_i}{x_h - x_i} \prod_{1 \le h < j \le k}
			\frac{r^{-2}_{i_{\sigma (h)}} x_{i_{\sigma (h)}} -
			x_{i_{\sigma (j)}}}{x_{i_{\sigma (h)}} - x_{i_{\sigma (j)}}},
		\end{split}
	\end{equation}

	Now, combining the product of 
	\eqref{eq:C_D_commutation_proof_2} over $1\le j\le k$ and 
	\eqref{eq:C_D_commutation_proof_3} over $1\le j\le M-k$, and using 
	\eqref{eq:C_D_commutation_proof_4}, we see that 
	the desired coefficient depending on $\sigma\in \mathfrak{S}_k$ is equal to
	\begin{equation}
		\label{eq:C_D_commutation_proof_5}
		\begin{split}
			&
			\prod_{i \in \mathcal{I}}  (1 - r^{-2}_i) x_i
			\prod_{i \in \mathcal{I},\, h \notin
			\mathcal{I}} \frac{r^{-2}_h x_h - x_i}{x_h - x_i}
			\prod_{j = 1}^{M - k} \Bigg( \prod_{h
				\notin \mathcal{I}} \frac{r^{-2}_h x_h - w_j}{x_h - w_j}
				\prod_{h = M - k + 1}^M \frac{\theta^{-2}_h w_h -
			w_j}{w_h - w_j} \Bigg) 
			\\ & \hspace{30pt}
			\times 
			\prod_{1 \le h < j
			\le k} \frac{r^{-2}_{i_{\sigma (h)}} x_{i_{\sigma (h)}} -
			x_{i_{\sigma (j)}}}{x_{i_{\sigma (h)}} - x_{i_{\sigma (j)}}}
			\prod_{j = 1}^k \Bigg(
			\frac{1}{x_{i_{\sigma (j)}} - w_{M - k + j}}
			\prod_{h = M - k + j + 1}^M \frac{\theta^{-2}_h w_h -
			x_{i_{\sigma (j)}}}{w_h - x_{i_{\sigma (j)}}} \Bigg).
		\end{split}
	\end{equation}
	Note that this is the coefficient of the operator
	\begin{equation*}
		C(x_{i_{\sigma(k)}},r_{i_{\sigma(k)}})
		\ldots 
		C(x_{i_{\sigma(1)}},r_{i_{\sigma(1)}})
		C(w_{M-k},\theta_{M-k})\ldots C(w_1,\theta_1) 
		\prod_{j=M-k+1}^{M}D(w_j,\theta_j)
		\prod_{j\notin \mathcal{I}}D(x_j,r_j),
	\end{equation*}
	and permuting the first $k$ of the $C$ operators to the desired order
	$C(x_{i_k},r_{i_k})\ldots C(x_{i_1},r_{i_1}) $
	results in an additional factor
	\begin{equation}
		\label{eq:C_D_commutation_proof_6}
		\prod_{\substack{1\le \upalpha<\upbeta\le k\\ \sigma(\upbeta)<\sigma(\upalpha)}}
		\frac{r^{-2}_{i_{\sigma(\upbeta)}}x_{i_{\sigma(\upbeta)}}-x_{i_{\sigma(\upalpha)}}}
		{r^{-2}_{i_{\sigma(\upalpha)}}x_{i_{\sigma(\upalpha)}}-x_{i_{\sigma(\upbeta)}}},
	\end{equation}
	by \Cref{lemma:C_commutation}.

	This implies that the full coefficient 
	$R_{\mathcal{I};\mathcal{H}}(\mathbf{w};\mathbf{x};\boldsymbol\uptheta;\mathbf{r})$
	equals to the sum of 
	\eqref{eq:C_D_commutation_proof_5} times \eqref{eq:C_D_commutation_proof_6} over all $\sigma\in \mathfrak{S}_k$.
	We have
	\begin{equation*}
		\prod_{1 \le h < j
		\le k} \frac{r^{-2}_{i_{\sigma (h)}} x_{i_{\sigma (h)}} -
		x_{i_{\sigma (j)}}}{x_{i_{\sigma (h)}} - x_{i_{\sigma (j)}}}
		\prod_{\substack{1\le \upalpha<\upbeta\le k\\ \sigma(\upbeta)<\sigma(\upalpha)}}
		\frac{r^{-2}_{i_{\sigma(\upbeta)}}x_{i_{\sigma(\upbeta)}}-x_{i_{\sigma(\upalpha)}}}
		{r^{-2}_{i_{\sigma(\upalpha)}}x_{i_{\sigma(\upalpha)}}-x_{i_{\sigma(\upbeta)}}}
		=
		\mathrm{\mathop{sgn}}(\sigma)
		\prod_{i,j\in \mathcal{I},\, i<j}\frac{r_i^{-2}x_i-x_j}{x_i-x_j}.
	\end{equation*}
	Therefore, the summation over $\sigma$ amounts to computing 
	the determinant:
	\begin{equation}
		\label{eq:sum_over_sigma_G_proof}
		\begin{split}
			&
			\sum_{\sigma\in \mathcal{I}}
			\mathop{\mathrm{sgn}}(\sigma)
			\prod_{j = 1}^k \Bigg(
			\frac{1}{x_{i_{\sigma (j)}} - w_{M - k + j}}
			\prod_{h = M - k + j + 1}^M \frac{\theta^{-2}_h w_h -
			x_{i_{\sigma (j)}}}{w_h - x_{i_{\sigma (j)}}} \Bigg)
			\\&\hspace{120pt}=
			\det
			\left[
				\frac{1}{x_{i_{\upbeta}} - w_{M - k + \upalpha}}
				\prod_{h = M - k + \upalpha + 1}^M
				\frac{\theta^{-2}_h w_h -
				x_{i_{\upbeta}}}{w_h - x_{i_{\upbeta}}}
			\right]_{\upalpha,\upbeta=1}^{k}
			.
		\end{split}
	\end{equation}
	We have already computed this determinant (up to renaming the variables)
	in 
	\eqref{eq:fully_deformed_Cauchy_determinant},
	and so
	\begin{equation*}
		\eqref{eq:sum_over_sigma_G_proof}=
		\prod_{i\in \mathcal{I},\, j\notin \mathcal{H}}
		\frac1{x_i-w_j}
		\prod_{i,j\notin \mathcal{H},\,i<j}(\theta_j^{-2}w_j-w_i)
		\prod_{i,j\in \mathcal{I},\,i<j}(x_i-x_j)
		,
		%
		%
		%
	\end{equation*}
	where we recalled that $\mathcal{H}=\left\{ 1,2,\ldots,M-k  \right\}$.
	Combining this with the remainder of \eqref{eq:C_D_commutation_proof_5},
	we arrive at the desired expression
	\eqref{eq:C_D_commutation_H_special_case}, thus concluding the proof of \Cref{lemma:C_D_commutation_H_special_case}.
\end{proof}

Finally, to get $R_{\mathcal{I};\mathcal{H}}$ for general $\mathcal{H}$, we can permute the $C$ operators
in the left-hand side of \eqref{eq:C_D_commutation}
thanks to \eqref{eq:C2C1_commute}. More precisely, the two expressions 
\begin{equation*}
	\begin{split}
		&
		C(w_M,\theta_M)\ldots C(w_1,\theta_1)\prod_{1\le i<j\le M}\frac{1}{\theta_j^{-2}w_j-w_i},
		\\
		&
		C(w_{h_{M-k}},\theta_{h_{M-k}})\ldots C(w_{h_1},\theta_{h_1})
		\prod_{i,j\in \mathcal{H},\,i<j}\frac{1}{\theta_j^{-2}w_j-w_i}
	\end{split}
\end{equation*}
are symmetric in $(w_i,\theta_i)$, $1\le i\le M$, and 
$(w_h,\theta_h)$, $h\in \mathcal{H}$, respectively.
Defining
\begin{equation}
	\label{eq:C_D_commutation_proof_7}
	\widehat{R}_{\mathcal{I};\mathcal{H}}(\mathbf{w};\mathbf{x};\boldsymbol\uptheta;\mathbf{r})
	=
	R_{\mathcal{I};\mathcal{H}}(\mathbf{w};\mathbf{x};\boldsymbol\uptheta;\mathbf{r})\,
	\frac{\prod_{i,j\in \mathcal{H},\,i<j}(\theta_j^{-2}w_j-w_i)}{\prod_{1\le i<j\le M}(\theta_j^{-2}w_j-w_i)},
\end{equation}
we see that for any permutation $\tau\in \mathfrak{S}_M$ we have
$
\widehat{R}_{\mathcal{I};\tau(\mathcal{H})}(\tau(\mathbf{w});\mathbf{x};\tau(\boldsymbol\uptheta);\mathbf{r})
=
\widehat{R}_{\mathcal{I};\mathcal{H}}(\mathbf{w};\mathbf{x};\boldsymbol\uptheta;\mathbf{r})
$.
The renormalization in \eqref{eq:C_D_commutation_proof_7} cancels out with the two last factors in 
$R_{\mathcal{I};\left\{ 1,\ldots,M-k  \right\}}$ in
\eqref{eq:C_D_commutation_H_special_case}.
This together with the symmetry of \eqref{eq:C_D_commutation_proof_7}
implies that $R_{\mathcal{I};\mathcal{H}}$
for general $\mathcal{H}$ is given by the same formula.
We have thus completed 
the proof of \Cref{prop:C_D_commutation}.

\subsubsection{Action of $C$ operators on a two-fold tensor product}
\label{appA:G_C_action}

In this subsection we perform computations with 
row operators acting on tensor products
which are
parallel to those in \Cref{appA:F_two_spaces,appA:G_action_two},
but now involve the $C$ operators.

\begin{lemma}
	\label{lemma:C_action_two}
	Let $\mathbf{x}=(x_1,\ldots,x_M )$, $\mathbf{r}=(r_1,\ldots,r_M )$.
	On any tensor product $V_1\otimes V_2$ we have:
	\begin{equation}
		\label{eq:action_of_C_twofold}
		\begin{split}
			&
			C(x_M,r_M)\ldots C(x_1,r_1)
			\\
			&\hspace{3pt}
			=
			\sum_{\mathcal{I}\subseteq \{1,\ldots,M \}}
			C(x_{i_k},r_{i_k})\ldots C(x_{i_1},r_{i_1}) 
			A(x_{j_{M-k}},r_{j_{M-k}})\ldots A(x_{j_1},r_{j_1}) 
			\\&\hspace{30pt}\otimes
			C(x_{j_{M-k}},r_{j_{M-k}})\ldots C(x_{j_1},r_{j_1}) 
			D(x_{i_k},r_{i_k})\ldots D(x_{i_1},r_{i_1})
			\\&\hspace{40pt}\times
			\prod_{i\in \mathcal{I},\, j\in \mathcal{J}}\frac{1}{x_i-x_j}
			\prod_{1\le i<j\le M}(r_j^{-2}x_j-x_i)
			\prod_{i,j\in \mathcal{I},\,i<j}
			\frac{1}{r_j^{-2}x_j-x_i}
			\prod_{i,j\in \mathcal{J},\,i<j}
			\frac{1}{r_j^{-2}x_j-x_i},
		\end{split}
	\end{equation}
	where
	$\mathcal{I}=(i_1<\ldots<i_k )$ and $\mathcal{J}=\left\{ 1,\ldots,M  \right\}\setminus \mathcal{I}=
	(j_1<\ldots<j_{M-k} )$.
\end{lemma}
\begin{proof}
	In the proof we use the 
	shorthand notation for the operators from \Cref{rmk:Shorthand_notation_ABCD_appA2}.
	Due to \eqref{eq:abcdv_n2}, relations in \Cref{prop:ABCD_YBE},
	and an argument identical to the beginning 
	of the proof of \Cref{lemma:D_action_twofold},
	we see that 
	the left-hand side of \eqref{eq:action_of_C_twofold}
	can be written in the form
	\begin{equation*}
		\sum_{\mathcal{I}\subseteq\left\{ 1,\ldots,M  \right\}}
		h_{\mathcal{I}}(\mathbf{x};\mathbf{r})
		C_{i_k}\ldots C_{i_1}A_{j_{M-k}}\ldots A_{j_1} 
		\otimes
		C_{j_{M-k}}\ldots C_{j_1} D_{i_k}\ldots D_{i_{1}},
	\end{equation*}
	where the notation $\mathcal{I},\mathcal{J}$ is as in \eqref{eq:action_of_C_twofold}.

	We first evaluate $h_{\mathcal{I}}$ in the special case 
	$\mathcal{I}=\mathcal{I}_k=\left\{ M-k+1,\ldots,M-1,M  \right\}$.
	The contribution containing the operator 
	$C_M\ldots C_{M-k+1}A_{M-k}\ldots A_1\otimes
	C_{M-k}\ldots C_1 D_{M}\ldots D_{M-k+1}$
	may arise only if we use \eqref{eq:D2_C1} in the second tensor factor
	to commute all $C_j$, $j\notin \mathcal{I}_k$, 
	to the left of all $D_i$, $i\in \mathcal{I}$,
	without swapping their arguments.
	Each such commutation gives rise to the factor
	$\frac{r_i^{-2}x_i-x_j}{x_i-x_j}$. Therefore,
	\begin{equation}
		\label{eq:action_of_C_twofold_proof}
		\begin{split}
			&h_{\mathcal{I}_k}(\mathbf{x};\mathbf{r})
			=
			\prod_{i=M-k+1}^{M}\prod_{j=1}^{M-k}\frac{r_i^{-2}x_i-x_j}{x_i-x_j}
			\\
			&\hspace{25pt}=
			\prod_{i\in \mathcal{I}_k,\, j\notin \mathcal{I}_k}\frac{1}{x_i-x_j}
			\prod_{1\le i<j\le M}(r_j^{-2}x_j-x_i)
			\prod_{i,j\in \mathcal{I}_k,\,i<j}
			\frac{1}{r_j^{-2}x_j-x_i}
			\prod_{i,j\notin \mathcal{I}_k,\,i<j}
			\frac{1}{r_j^{-2}x_j-x_i}.
		\end{split}
	\end{equation}

	Next, 
	thanks to \eqref{eq:C2C1_commute} the three expressions 
	\begin{equation*}
		\frac{C_M \ldots C_1}{\prod_{1\le i<j\le M}(r_j^{-2}x_j-x_i)},
		\qquad 
		\frac{C_{i_k} \ldots C_{i_1}}{\prod_{i,j\in \mathcal{I},\,i<j}(r_j^{-2}x_j-x_i)},
		\qquad 
		\frac{C_{j_{M-k}} \ldots C_{j_1}}{\prod_{i,j\notin \mathcal{I},\,i<j}(r_j^{-2}x_j-x_i)}
	\end{equation*}
	are symmetric in the pairs $(x_i,r_i)$ of variables they depend on
	(where $1\le i\le M$, $i\in \mathcal{I}$, and $i\notin \mathcal{I}$, respectively).
	Therefore, the function
	\begin{equation*}
		\widehat 
		h_{\mathcal{I}}(\mathbf{x};\mathbf{r})=
		h_{\mathcal{I}}(\mathbf{x};\mathbf{r})\,
		\frac{
			\prod_{i,j\in \mathcal{I},\, i<j}(r_j^{-2}x_j-x_i)
			\prod_{i,j\notin \mathcal{I},\, i<j}(r_j^{-2}x_j-x_i)
		}{
			\prod_{1\le i<j\le M}(r_j^{-2}x_j-x_i)
		}
	\end{equation*}
	satisfies $\widehat h_{\tau(\mathcal{I})}(\mathbf{x};\mathbf{r})
	=\widehat h_{\mathcal{I}}(\tau^{-1}(\mathbf{x});\tau^{-1}(\mathbf{r}))$
	for any permutation $\tau\in \mathfrak{S}_M$.
	Together with \eqref{eq:action_of_C_twofold_proof}
	this shows that for any $\mathcal{I}$ we have
	$\widehat h_{\mathcal{I}}(\mathbf{x};\mathbf{r})=\prod_{i\in \mathcal{I},\,j\notin \mathcal{I}}(x_i-x_j)^{-1}$,
	which implies
	the claim.
\end{proof}

In the next proposition, 
let $e_0 = e_0^{(i_1)} \otimes e_0^{(i_2)} \otimes \cdots \otimes
e_0^{(i_n)} \in V^{(i_1)} \otimes V^{(i_2)} \otimes \cdots \otimes
V^{(i_n)}$ for any integers $i_1 < i_2 < \cdots < i_n$. 
Moreover, fix $M \ge 1, N \ge 0$, and $\mathcal{T} = (t_1<
t_2< \ldots < t_M) \subset \mathbb{Z}_{\ge1}$.
Define the vector $e_{\mathcal{T}; N} =
e_{m_1}^{(N + 1)} \otimes e_{m_2}^{(N + 2)} \otimes \cdots \in V^{(N +
1)} \otimes V^{(N + 2)} \otimes \cdots$, where $m_i=1$ 
if $i\in \mathcal{T}$, and $0$ otherwise.

\begin{proposition}
	\label{prop:C_action_two}
	With the above notation we have
	\begin{equation*}
		\begin{split}
		&
		\bigl\langle e_{\mathcal{T}; N}, C(x_M,r_M) \cdots C(x_1,r_1)\, e_0 \bigr\rangle 
		=
		\prod_{1 \le i< j \le M} \frac{r_j^{-2} x_j - x_i}{x_i - x_j}
		\\ & \hspace{120pt}
		\times 
		\sum_{\sigma \in \mathfrak{S}_M}
		\mathop{\mathrm{sgn}}(\sigma)
		\prod_{j = 1}^M \Bigg( \frac{y_{t_j + N}
		\big( 1 - s_{t_j + N}^2 \big)}{y_{t_j + N} - s_{t_j + N}^2
		x_{\sigma (j)}} \prod_{i = N + 1}^{t_j + N - 1}
		\frac{s_i^2 \big(y_i - x_{\sigma (j)} \big)}{y_i - s_i^2
		x_{\sigma (j)}} \Bigg),
		\end{split}
	\end{equation*}
	where the inner product is taken in the space $V^{(N+1)}\otimes V^{(N+2)}\otimes\ldots$.
\end{proposition}
Observe that this formula is determinantal, and is in fact equivalent to the 
determinantal formula for $F_\lambda$ from \Cref{thm:F_formula}
proven in \Cref{appA:F}, up to swapping horizontal 
arrows with empty horizontal edges, and renormalizing. 
Here, however, we present an independent proof which is more convenient
given our previous statements.
\begin{proof}[Proof of \Cref{prop:C_action_two}]
	In the proof we use the shorthand notation for the operators from \Cref{rmk:Shorthand_notation_ABCD_appA2}.
	Fix $n>N$ and vectors $v_1,v_2\in V^{(n+1)}\otimes V^{(n+2)}\otimes\ldots $.
	By \Cref{lemma:C_action_two}, we have
	\begin{equation}
		\label{eq:C_action_two_in_proposition_proof}
		\begin{split}
			\big\langle e_0^{(n)} \otimes v_2, C_M C_{M - 1} \cdots C_1 e_0
			\big\rangle 
			&= 
			\big\langle e_0^{(n)}, A_M A_{M - 1} \cdots A_1
			e_0^{(n)} \big\rangle \langle v_2, C_M C_{M - 1} \cdots C_1 e_0
			\rangle; 
			\\ 
			\big\langle e_1^{(n)} \otimes v_2, C_M C_{M - 1}
			\cdots C_1 e_0 \big\rangle 
			&= 
			\sum_{i = 1}^M
			\big\langle e_1^{(n)}, C_i A_M \cdots A_{i + 1} A_{i - 1} \cdots
			A_1 e_0^{(n)} \big\rangle 
			\\ & 
			\qquad \qquad \times \langle v_2,
			C_M \cdots C_{i + 1} C_{i - 1} \cdots C_1 D_i e_0 \rangle 
			\\ &
			\qquad \qquad \times \prod_{j \ne i}
			\frac{1}{x_i - x_j} \prod_{j = 1}^{i -
			1} (r^{-2}_i x_i - x_j) \prod_{j = i + 1}^M 
			(r^{-2}_j x_j - x_i).
		\end{split}
	\end{equation}
	These quantities can be computed as follows:
	\begin{equation*}
		\begin{split}
			D_ie_0& =e_0;
			\\
			\big\langle e_0^{(n)},  A_M A_{M - 1} \cdots A_1 e_0^{(n)}
			\big\rangle 
			&= \prod_{j = 1}^M
			\frac{s_n^2 (y_n - x_j)}{y_n - s_n^2 x_j}; 
			\\
			\big\langle e_1^{(n)}, C_i A_M \cdots A_{i + 1} A_{i - 1} \cdots
			A_1 e_0^{(n)} \big\rangle 
			&= \frac{y_n (1 -
			s_n^2)}{y_n - s_n^2 x_i} \prod_{j \ne i}
			\frac{s_n^2 (y_n - x_j)}{y_n - s_n^2 x_j},
		\end{split}
	\end{equation*}
	using the definition of the operators \eqref{eq:abcdv}
	and formulas for the vertex weights $W$ \eqref{eq:weights_W}.
	Therefore, \eqref{eq:C_action_two_in_proposition_proof} is continued as
	\begin{equation}
		\label{eq:C_action_two_in_proposition_proof_2}
		\begin{split}
			&
			\big\langle e_0^{(n)} \otimes v_2, C_M C_{M - 1} \cdots C_1 e_0
			\big\rangle  = \langle v_2, C_M C_{M - 1} \cdots C_1 e_0
			\rangle
			\prod_{j = 1}^M
			\frac{s_n^2 (y_n - x_j)}{y_n - s_n^2 x_j};
			\\
			&\big\langle e_1^{(n)} \otimes v_2, C_M C_{M - 1}
			\cdots C_1 e_0 \big\rangle  = \sum_{i = 1}^M
			\langle v_2,
			C_M \cdots C_{i + 1} C_{i - 1} \cdots C_1 e_0 \rangle
			\\& \hspace{50pt}\times 
			\frac{y_n (1 -
			s_n^2)}{y_n - s_n^2 x_i} \prod_{j \ne i}
			\frac{s_n^2 (y_n - x_j)}{y_n - s_n^2 x_j}
			\prod_{j \ne i}
			\frac{1}{x_i - x_j} \prod_{j = 1}^{i -
			1} (r^{-2}_i x_i - x_j) \prod_{j = i + 1}^M 
			(r^{-2}_j x_j - x_i).
		\end{split}
	\end{equation}
	
	Now we can evaluate $\langle  e_{\mathcal{T};N},C_M \ldots C_1 e_0  \rangle $
	by repeatedly applying \eqref{eq:C_action_two_in_proposition_proof_2}. Throughout these applications,
	we use first or second identity in 
	\eqref{eq:C_action_two_in_proposition_proof_2}, respectively,
	for each $n$ belonging or not belonging to the set $\left\{ t_1+N,t_2+N,\ldots,t_M+N  \right\}$.
	In the latter case, for $n=N+t_j$, we choose which index $i=i_j\in  \left\{ 1,\ldots,M  \right\}$
	to remove. These choices are encoded by a permutation $\sigma\in \mathfrak{S}_M$ as
	$i_j=\sigma(j)$. This leads to the desired claim, where, in particular, 
	$\mathop{\mathrm{sgn}}(\sigma)$ 
	arises from reordering the denominators $x_{\sigma(i)}-x_{\sigma(j)}$ to $x_i-x_j$ over all $1\le i<j\le M$.
\end{proof}
\subsubsection{Completing the proof}
\label{appA:G_completing_proof}

To finalize the proof of \Cref{thm:G_formula}, 
let us recall the formula to be established.
Fix 
an arbitrary signature 
$\lambda=(\lambda_1\ge \ldots\ge \lambda_N \ge0)$.
Let $d = d(\lambda) \ge 0$ denote
the integer such that $\lambda_d \ge d$ and $\lambda_{d + 1} < d + 1$.
Denote by $\ell_j=\lambda_j+N-j+1$, $j=1,\ldots,N $, the elements of the set $\mathcal{S}(\lambda)$.
Moreover, we define
$\mu = (\mu_1< \mu_2< \ldots < \mu_d) =
\{1,\ldots,N \}  \setminus \big( \mathcal{S}(\lambda) \cap \{1,\ldots,N \} \big)$.
Our goal is to show that
\begin{align}
	\nonumber
	&G_{\lambda} (\mathbf{x}; \mathbf{y}; \mathbf{r}; \mathbf{s}) 
	=
	\prod_{j=1}^{M}\prod_{k=1}^{N}
	\frac{y_k-s_k^2r_j^{-2}x_j}{y_k-s_k^2 x_j}
	\sum_{\substack{\mathcal{I},\mathcal{J}\subseteq \left\{ 1,\ldots,M  \right\}\\ |\mathcal{I}|=|\mathcal{J}|=d}}
	\prod_{\substack{i\in \mathcal{I}\\ 1\le j\le M}}(r_i^{-2}x_i-x_j)
	\prod_{\substack{i\in \mathcal{I}\\j\in \mathcal{I}^c}}
	\frac{1}{r_i^{-2}x_i-r_j^{-2}x_j}
	\\
	\label{eq:G_lambda_formula_in_appendix_desired}
	&\hspace{70pt}
	\times
	\prod_{\substack{i,j\in \mathcal{I}\\i<j}}\frac{1}{r_i^{-2}x_i-r_j^{-2}x_j}
	\prod_{\substack{i\in \mathcal{I}^c\\j\in \mathcal{J}}}(r_i^{-2}x_i-x_j)
	\prod_{\substack{i\in \mathcal{J}^c\\j\in \mathcal{J}}}
	\frac{1}{x_i-x_j}
	\prod_{\substack{i,j\in \mathcal{J}\\i<j}}\frac1{x_j-x_i}
	\\
	\nonumber
	&\hspace{70pt}
	\times
	\sum_{\sigma,\rho\in \mathfrak{S}_d}
	\mathop{\mathrm{sgn}}(\sigma\rho)
	\prod_{h = 1}^d \biggl( \frac{y_{\ell_h}
	\big( 1 - s_{\ell_h}^2 \big)}{y_{\ell_h} - s_{\ell_h}^2
	x_{j_{\rho (h)}}} \prod_{i = N + 1}^{\ell_h - 1}
	\frac{s_i^2 \big(y_i - x_{j_{\rho (h)}} \big)}{y_i - s_i^2
	x_{j_{\rho (h)}}} \biggr)
	\\&\hspace{100pt}\times
	\prod_{m=1}^d
	\biggl(
		\frac{s_{\mu_m}^2 }{y_{\mu_m} - s_{\mu_m}^2
		r^{-2}_{i_{\sigma (m)}} x_{i_{\sigma (m)}}} 
		\prod_{k =
		\mu_m + 1}^N 
		\frac{s_k^2 \big(r^{-2}_{i_{\sigma (m)}}
		x_{i_{\sigma (m)}} - y_k \big)}{y_k - s_k^2 r^{-2}_{i_{\sigma (m)}}
		x_{i_{\sigma (m)}}}
	\biggr).
	\nonumber
\end{align}	
where
$\mathcal{I}= (i_1< i_2< \ldots < i_d)$
and $\mathcal{J}= (j_1< j_2< \ldots < j_d)$.

\medskip

Recall that
\begin{equation*}
	G_{\lambda}(\mathbf{x};\mathbf{y};\mathbf{r};\mathbf{s})= \left\langle
	e_{\mathcal{S}_N(\lambda)}\otimes e_{\mathcal{S}_{>N}}(\lambda), D(x_M,r_M)\ldots
	D(x_2,r_2)D(x_1,r_1)(e_{[1,N]}\otimes e_0) \right\rangle,
\end{equation*}
where we have split the vectors into $e_{\mathcal{S}_N(\lambda)},e_{[1,N]}\in
V^{(1)}\otimes\ldots\otimes  V^{(N)}$ (cf. \eqref{eq:e_S_N_lambda_vector}), 
and the remaining two vectors belong to 
$V^{(N+1)}\otimes V^{(N+2)}\otimes \ldots $. 
Note that the vector $e_{\mathcal{S}_{>N}(\lambda)}$ has exactly $d$ tensor
components of the form $e_1^{(k)}$,
and the other components are of the form $e_0^{(k)}$.
We can use
\Cref{prop:G_proof_main_proposition_1} to write:
\begin{equation}
	\label{eq:G_final_proof_1}
	\begin{split}	
		&
		\left\langle e_{\mathcal{S}_N(\lambda)}\otimes e_{\mathcal{S}_{>N}(\lambda)},
		D(x_M,r_M)\ldots D(x_2,r_2)D(x_1,r_1) (e_{[1,N]}\otimes e_0)\right\rangle 
		\\&
		\hspace{2pt}=
		\prod_{j=1}^{M}\prod_{k=1}^{N}
		\frac{y_k-s_k^2r_j^{-2}x_j}{y_k-s_k^2 x_j}
		\sum_{\substack{\mathcal{I}\subseteq \left\{ 1,\ldots,M  \right\}\\ |\mathcal{I}|=d}}
		\left\langle e_{\mathcal{S}_{>N}(\lambda)},
		\biggl( 
			\prod_{j\notin \mathcal{I}}D(x_j,r_j)
		\biggr)
		C(x_{i_d},r_{i_d})
		\ldots 
		C(x_{i_1},r_{i_1})
		e_0
		\right\rangle 
		\\&\hspace{80pt}
		\times
		\prod_{i \in \mathcal{I},\, j\notin \mathcal{I}}
		\frac{r_i^{-2}x_i-x_j}{r_i^{-2}x_i-r_j^{-2}x_j}
		\prod_{i,j \in \mathcal{I},\,i<j}
		\frac{r_i^{-2}x_i-x_j}{r_i^{-2}x_i-r_j^{-2}x_j}
		\\&\hspace{80pt}
		\times
		\sum_{\sigma\in \mathfrak{S}_d}
		\mathop{\mathrm{sgn}}(\sigma)
		\prod_{j=1}^d
		\biggl(
			\frac{s_{\mu_j}^2 x_{i_{\sigma (j)}} \big(
			r^{-2}_{i_{\sigma (j)}} - 1\big)}{y_{\mu_j} - s_{\mu_j}^2
			r^{-2}_{i_{\sigma (j)}} x_{i_{\sigma (j)}}} 
			\prod_{k =
			\mu_j + 1}^N 
			\frac{s_k^2 \big(r^{-2}_{i_{\sigma (j)}}
			x_{i_{\sigma (j)}} - y_k \big)}{y_k - s_k^2 r^{-2}_{i_{\sigma (j)}}
			x_{i_{\sigma (j)}}}
		\biggr).
	\end{split}
\end{equation}
Let us denote 
\begin{equation*}
	D_{\mathcal{I}^c}:=\prod_{i\notin \mathcal{I}}D(x_j,r_j),
	\qquad 
	C_{\mathcal{I}}:=
	C(x_{i_d},r_{i_d})
	\ldots 
	C(x_{i_1},r_{i_1}),
\end{equation*}
and use similar notation in what follows.
In particular, in all such products of the $C$ operators the indices are decreasing from 
left to right.
Employ \Cref{prop:C_D_commutation} to write
\begin{align*}
	&D_{\mathcal{I}^c}C_{\mathcal{I}}=
	\sum_{\substack{\mathcal{K}\subseteq \mathcal{I}^c,\, \mathcal{H}\subseteq \mathcal{I}\\
	|\mathcal{K}|+|\mathcal{H}|=d}}
	C_{\mathcal{K}}C_{\mathcal{H}}D_{\mathcal{I}\setminus \mathcal{H}}D_{\mathcal{I}^c \setminus \mathcal{K}}
	\\&\hspace{60pt}
	\times
	\prod_{k\in \mathcal{K}}(1-r_k^{-2})x_k
	\prod_{i\in \mathcal{K}\cup \mathcal{H},\, j\in \mathcal{I}^c\setminus \mathcal{K}}\frac{r_j^{-2}x_j-x_i}{x_j-x_i}
	\prod_{h\in \mathcal{H},\, j\in \mathcal{I}\setminus\mathcal{H}}\frac{1}{x_j-x_h}
	\prod_{i\in \mathcal{K},\, j\in \mathcal{I}\setminus \mathcal{H}}\frac{1}{x_i-x_j}
	\\&\hspace{60pt}
	\times
	\prod_{i,j\in \mathcal{K},\, i<j}(r_i^{-2}x_i-x_j)
	\prod_{i,h\in \mathcal{H},\, h<i}\frac{1}{r_i^{-2}x_i-x_h}
	\prod_{i,j\in \mathcal{I},\, i<j}(r_j^{-2}x_j-x_i).
\end{align*}
Let us insert this into \eqref{eq:G_final_proof_1}.
Observe that all operators $D$ preserve the vector $e_0$.
Thus, we can continue the computation as
\begin{align*}
	&
	\left\langle e_{\mathcal{S}_N(\lambda)}\otimes e_{\mathcal{S}_{>N}(\lambda)},
	D(x_M,r_M)\ldots D(x_2,r_2)D(x_1,r_1) (e_{[1,N]}\otimes e_0)\right\rangle 
	\\&
	=
	\prod_{j=1}^{M}\prod_{k=1}^{N}
	\frac{y_k-s_k^2r_j^{-2}x_j}{y_k-s_k^2 x_j}
	\sum_{\substack{\mathcal{I}\subseteq \left\{ 1,\ldots,M  \right\}\\ |\mathcal{I}|=d}}
	\prod_{i,j\in \mathcal{I}}(r_i^{-2}x_i-x_j)
	\prod_{\substack{i \in \mathcal{I}\\ j\in \mathcal{I}^c}}
	\frac{r_i^{-2}x_i-x_j}{r_i^{-2}x_i-r_j^{-2}x_j}
	\prod_{\substack{i,j \in \mathcal{I}\\i<j}}
	\frac{1}{r_i^{-2}x_i-r_j^{-2}x_j}
	\\&\hspace{60pt}
	\times
	\sum_{\substack{\mathcal{K}\subseteq \mathcal{I}^c,\, \mathcal{H}\subseteq \mathcal{I}\\
	|\mathcal{K}|+|\mathcal{H}|=d}}
	\left\langle e_{\mathcal{S}_{>N}(\lambda)},
	C_\mathcal{K}C_\mathcal{H}
	e_0
	\right\rangle 
	\prod_{k\in \mathcal{K}}(1-r_k^{-2})x_k
	\prod_{i\in \mathcal{K}\cup \mathcal{H},\, j\in \mathcal{I}^c\setminus \mathcal{K}}\frac{r_j^{-2}x_j-x_i}{x_j-x_i}
	\\&\hspace{60pt}\times
	\prod_{h\in \mathcal{H},\, j\in \mathcal{I}\setminus\mathcal{H}}\frac{1}{x_j-x_h}
	\prod_{i\in \mathcal{K},\, j\in \mathcal{I}\setminus \mathcal{H}}\frac{1}{x_i-x_j}
	\prod_{i,j\in \mathcal{K},\, i<j}(r_i^{-2}x_i-x_j)
	\prod_{i,h\in \mathcal{H},\, h<i}\frac{1}{r_i^{-2}x_i-x_h}
	\\&\hspace{60pt}
	\times
	\sum_{\sigma\in \mathfrak{S}_d}
	\mathop{\mathrm{sgn}}(\sigma)
	\prod_{j=1}^d
	\biggl(
		\frac{s_{\mu_j}^2 }{y_{\mu_j} - s_{\mu_j}^2
		r^{-2}_{i_{\sigma (j)}} x_{i_{\sigma (j)}}} 
		\prod_{k =
		\mu_j + 1}^N 
		\frac{s_k^2 \big(r^{-2}_{i_{\sigma (j)}}
		x_{i_{\sigma (j)}} - y_k \big)}{y_k - s_k^2 r^{-2}_{i_{\sigma (j)}}
		x_{i_{\sigma (j)}}}
	\biggr).
\end{align*}

Now we are going to apply \Cref{prop:C_action_two} to compute the remaining inner product. 
Recall that $e_{\mathcal{S}_{>N}(\lambda)}$ has exactly $d$ tensor components equal to $e_1^{(m)}$,
for $m \in\left\{ \ell_1,\ldots,\ell_d  \right\}$.
Denote $(x_1',\ldots,x_d' )=(x_{h_1},\ldots,x_{h_{|\mathcal{H}|}}, x_{k_1},\ldots,x_{k_{|\mathcal{K}|}} )$,
where $h_1<\ldots<h_{|\mathcal{H}|} $, $k_1<\ldots<k_{|\mathcal{K}|} $.
Then we have
\begin{equation}
	\label{eq:G_final_proof_99}
	\begin{split}
		&\left\langle e_{\mathcal{S}_{>N}(\lambda)},
		C_\mathcal{K}C_\mathcal{H}
		e_0
		\right\rangle 
		=
		(-1)^{\frac{d(d-1)}{2}}
		\prod_{i,j\in \mathcal{H},\, i<j} \frac{r_j^{-2} x_j - x_i}{x_i - x_j}
		\prod_{i,j\in \mathcal{K},\, i<j} \frac{r_j^{-2} x_j - x_i}{x_i - x_j}
		\\
		&
		\hspace{50pt}
		\times
		\prod_{i\in \mathcal{H},\,j\in \mathcal{K}} \frac{r_j^{-2} x_j - x_i}{x_i - x_j}
		\sum_{\rho \in \mathfrak{S}_d}
		\mathop{\mathrm{sgn}}(\rho)
		\prod_{j = 1}^d \biggl( \frac{y_{\ell_j}
		\big( 1 - s_{\ell_j}^2 \big)}{y_{\ell_j} - s_{\ell_j}^2
		x'_{\rho (j)}} \prod_{i = N + 1}^{\ell_j - 1}
		\frac{s_i^2 \big(y_i - x'_{\rho (j)} \big)}{y_i - s_i^2
		x'_{\rho (j)}} \biggr)
		.
	\end{split}
\end{equation}
The sign $(-1)^{\frac{d(d-1)}{2}}$ arises from the fact that the $t_j$'s in \Cref{prop:C_action_two}
are increasing, while the $\ell_j$'s in \eqref{eq:G_final_proof_99} are decreasing, so the 
sign of $\rho$ has to be multiplied by $(-1)^{\frac{d(d-1)}{2}}$.
This allows to continue our computation as follows:
\begin{align*}
	&
	\left\langle e_{\mathcal{S}_N(\lambda)}\otimes e_{\mathcal{S}_{>N}(\lambda)},
	D(x_M,r_M)\ldots D(x_2,r_2)D(x_1,r_1) (e_{[1,N]}\otimes e_0)\right\rangle 
	\\
	&\hspace{20pt}
	=
	(-1)^{\frac{d(d-1)}{2}}
	\prod_{j=1}^{M}\prod_{k=1}^{N}
	\frac{y_k-s_k^2r_j^{-2}x_j}{y_k-s_k^2 x_j}
	\\&\hspace{40pt}
	\times
	\sum_{\substack{\mathcal{I}\subseteq \left\{ 1,\ldots,M  \right\}\\ |\mathcal{I}|=d}}
	\prod_{\substack{i\in \mathcal{I}\\ 1\le j\le M}}(r_i^{-2}x_i-x_j)
	\prod_{\substack{i\in \mathcal{I}\\j\in \mathcal{I}^c}}
	\frac{1}{r_i^{-2}x_i-r_j^{-2}x_j}
	\prod_{\substack{i,j\in \mathcal{I}\\i<j}}\frac{1}{r_i^{-2}x_i-r_j^{-2}x_j}
	\\
	&\hspace{40pt}
	\times
	\sum_{\substack{\mathcal{K}\subseteq \mathcal{I}^c,\, \mathcal{H}\subseteq \mathcal{I}\\
	|\mathcal{K}|+|\mathcal{H}|=d}}
	\prod_{\substack{i\in \mathcal{I}^c\\j\in \mathcal{K}\cup \mathcal{H}}}(r_i^{-2}x_i-x_j)
	\prod_{\substack{i\notin \mathcal{K}\cup \mathcal{H}\\j\in \mathcal{K}\cup \mathcal{H}}}
	\frac{1}{x_i-x_j}
	\prod_{\substack{i,j\in \mathcal{H}\\i<j}}\frac1{x_i-x_j}
	\prod_{\substack{i,j\in \mathcal{K}\\i<j}}\frac1{x_i-x_j}
	\prod_{\substack{i\in \mathcal{H}\\j\in \mathcal{K}}}\frac1{x_i-x_j}
	\\
	&\hspace{70pt}
	\times
	\sum_{\sigma,\rho\in \mathfrak{S}_d}
	\mathop{\mathrm{sgn}}(\sigma\rho)
	\prod_{j = 1}^d \biggl( \frac{y_{\ell_j}
	\big( 1 - s_{\ell_j}^2 \big)}{y_{\ell_j} - s_{\ell_j}^2
	x'_{\rho (j)}} \prod_{i = N + 1}^{\ell_j - 1}
	\frac{s_i^2 \big(y_i - x'_{\rho (j)} \big)}{y_i - s_i^2
	x'_{\rho (j)}} \biggr)
	\\&\hspace{100pt}\times
	\prod_{j=1}^d
	\biggl(
		\frac{s_{\mu_j}^2 }{y_{\mu_j} - s_{\mu_j}^2
		r^{-2}_{i_{\sigma (j)}} x_{i_{\sigma (j)}}} 
		\prod_{k =
		\mu_j + 1}^N 
		\frac{s_k^2 \big(r^{-2}_{i_{\sigma (j)}}
		x_{i_{\sigma (j)}} - y_k \big)}{y_k - s_k^2 r^{-2}_{i_{\sigma (j)}}
		x_{i_{\sigma (j)}}}
	\biggr).
\end{align*}
Upon denoting $\mathcal{J}=\mathcal{K}\cup\mathcal{H}=(j_1<\ldots<j_d )$, we arrive at the 
desired statement \eqref{eq:G_lambda_formula_in_appendix_desired}. 
Note that reordering the indices in $(x_1',\ldots,x_d' )$ in the increasing order
leads to an extra $\pm$ sign coming from $\mathop{\mathrm{sgn}}(\rho)$, 
but this sign is compensated by writing 
\begin{equation}
	\label{eq:G_final_proof_99_1}
	\prod_{i,j\in \mathcal{H},\,i<j}\frac1{x_i-x_j}
	\prod_{i,j\in \mathcal{K},\,i<j}\frac1{x_i-x_j}
	\prod_{i\in \mathcal{H},\,j\in \mathcal{K}}\frac1{x_i-x_j}
	=
	\pm
	\prod_{i,j\in \mathcal{J},\,i<j}\frac{1}{x_i-x_j}
\end{equation}
(equivalently, one may refer to the symmetry as in the proof of \Cref{lemma:C_action_two}).
Finally,
replacing $x_i-x_j$ in \eqref{eq:G_final_proof_99_1} with $x_j-x_i$ absorbs the sign $(-1)^{\frac{d(d-1)}{2}}$.
This completes the proof of 
\Cref{thm:G_formula}.

\section{Correlation kernel via Eynard--Mehta approach}
\label{appB:Eynard_Mehta}

Here we prove 
\Cref{thm:ascending_FG_process_kernel} on 
the determinantal structure of the FG measures and 
processes. We employ an Eynard--Mehta type approach
based on \cite{borodin2005eynard},
see also \cite{eynard1998matrices}.

\subsection{Representation of the ascending FG process in a determinantal form}
\label{sup:appB_det_form_representation}

Recall the notation of the ascending FG process \eqref{eq:ascending_process} from \Cref{sub:probability_from_Cauchy}.
Throughout
\Cref{appB:Eynard_Mehta}
we omit the notation $\mathbf{y},\mathbf{s}$
in the functions $G_{\mu/\varkappa}(w_i;\mathbf{y};\theta_i;\mathbf{s})$
and other similar quantities.

Here we use the 
determinantal formulas for the functions $F_\lambda$
(\Cref{thm:F_formula})
and 
$G_\mu$, $G_{\nu/\lambda}$ to rewrite the probabilities 
\eqref{eq:ascending_process} in a determinantal form.
The formulas for $G_\mu$ and $G_{\nu/\lambda}$
are of Jacobi--Trudy type and follow from 
Cauchy identities and biorthogonality as in 
\Cref{sub:G_det_formula}.

Recall the notation \eqref{eq:phi_def}:
\begin{equation*}
	\varphi_k(x)=
	\frac{1}{y_{k+1}-x}
	\prod_{j=1}^{k}
	\frac{y_j-s_j^2x}{s_j^2(y_j-x)}
	,\qquad k\ge0.
\end{equation*}
By \Cref{thm:F_formula}, we have
\begin{equation}
	\label{eq:F_det_formula_for_Eynard}
	F_\lambda(\rho)=
	\mathrm{const}\cdot
	\det\left[ \varphi_{\lambda_j+N-j}(x_i) \right]_{i,j=1}^{N},
\end{equation}
where the constant is independent of $\lambda$ (we adopt this convention for all 
such constants throughout \Cref{appB:Eynard_Mehta}, and will denote all of them 
by $\mathrm{const}$).

Next, recall the functions
$\psi_k$
\eqref{eq:psi_def}:
\begin{equation*}
	\psi_k(x)
	=
	\frac{y_{k+1}(s_{k+1}^2-1)}{y_{k+1}-s^2_{k+1}x}\,
	\prod_{j=1}^{k}
	\frac{s_j^2(y_j-x)}{y_j-s_j^2x}
	,\qquad k\ge1.
\end{equation*}
For
$(\mathbf{w};\boldsymbol\uptheta)=(w_a,\ldots,w_b;\theta_a,\ldots,\theta_b)$,
$a\le b$, 
let us define a slight generalization of
\eqref{eq:nontilde_h_function_for_JT_G_nonskew}:
\begin{equation*}
	\mathsf{h}_{k,p}(\mathbf{w};\boldsymbol\uptheta)
	:=
	\frac{1}{2\pi\mathbf{i}}\oint_{\Gamma_{y,w}}dz\,
	\frac{\psi_k(z)}{y_p-z}
	\prod_{j=a}^{b}\frac{z-\theta_j^{-2}w_j}{z-w_j},
	\qquad k\ge0,\quad p\ge1,
\end{equation*}
where the integration contour $\Gamma_{y,w}$ is positively oriented, 
surrounds all $y_i, w_j$, and leaves out all $s_i^{-2}y_i$.
The function $G_{\lambda^{(1)}}$ in \eqref{eq:ascending_process}
has the following determinantal form (with $a=b=1$ in $\mathsf{h}_{k,l}$):
\begin{equation}
	\label{eq:G_pure_det_formula_for_Eynard}
		G_{\lambda^{(1)}}
		(w_1;\theta_1)
		=
		\mathrm{const}\cdot
		\det\bigl[ \mathsf{h}_{\lambda^{(1)}_i+N-i,\,j}(w_1;\theta_1) \bigr]_{i,j=1}^{N}.
\end{equation}

Finally, recall the functions 
$\widetilde{\mathsf{h}}_l$ \eqref{eq:tilde_h_function_for_JT_G}
and 
$\mathsf{g}_{l/k}$
\eqref{eq:g_skew_better_notation}:
\begin{equation*}
	\mathsf{g}_{l/k}(\mathbf{w};\boldsymbol\uptheta)
	=
	\widetilde{\mathsf{h}}_{l-k}(\mathbf{w};\tau_k \mathbf{y};\boldsymbol\uptheta;\tau_k \mathbf{s})
	=
	\frac{\mathbf{1}_{l\ge k}}{2\pi \mathbf{i}}\oint_{\Gamma_{y,w}}
	dz\,
	\varphi_k(z)\psi_l(z)
	\prod_{j=a}^{b}\frac{z-\theta_j^{-2}w_j}{z-w_j},
\end{equation*}
where the integration contour is around $y_j,w_i$ and not $s_j^{-2}y_j$,
and 
$(\mathbf{w};\boldsymbol\uptheta)=(w_a,\ldots,w_b;\theta_a,\ldots,\theta_b)$ with $a\le b$.
The skew functions in \eqref{eq:ascending_process} take the following determinantal form:
\begin{equation}
	\label{eq:G_skew_det_formula_for_Eynard}
	G_{\lambda^{(t)}/\lambda^{(t-1)}}(w_t;\theta_t)
	=
	\det
	\bigl[ 
		\mathsf{g}_{(\lambda^{(t)}_i+N-i)/(\lambda^{(t-1)}_j+N-j)}
		(w_t;\theta_t) 
	\bigr]_{i,j=1}^{N}.
\end{equation}

We observe that when evaluated at a single pair of variables $(w;\theta)$, both
$\mathsf{h}_{k,j}$ (for $k\ge j$) and $\mathsf{g}_{l/k}$ become explicit:
\begin{lemma}
	\label{lemma:explicit_Eynard_h_g}
	We have
	\begin{equation*}
			\mathsf{h}_{k,j}(w;\theta)=
			\begin{cases}
			\dfrac{w(1-\theta^{-2})\psi_k(w)}{y_j-w},&k \ge j;\\
			\textnormal{\small{}a non-product expression},& k < j,
			\end{cases}
			\quad 
			\mathsf{g}_{l/k}(w;\theta)=
			\begin{cases}
				w(1-\theta^{-2})\varphi_k(w)\psi_l(w),& l>k;\\[5pt]
				\dfrac{\theta^{-2}w-s_{k+1}^{-2}y_{k+1}}{w-s_{k+1}^{-2}y_{k+1}},& l=k;\\
				0,& l<k.
			\end{cases}
	\end{equation*}
\end{lemma}
\begin{proof}
	For $\mathsf{h}_{k,j}$ with $k\ge j$, the only pole inside the contour is $z=w$, which leads
	to the desired formula.
	The exact form of the functions
	$\mathsf{h}_{k,j}$ with $k < j$ is not very explicit (apart from
	the original contour integral expression), but they are 
	not involved in 
	our computations.

	For $\mathsf{g}_{l/k}$, in the case $l=k$, the only singularity outside 
	the contour is $z=s_{k+1}^{-2}y_k$, 
	and for $l>k$ the only singularity inside the contour is $z=w$. 
	The respective residues in these two cases lead to the desired formulas.
	For $l<k$, there are no singularities outside the integration contours, and 
	the integral vanishes.
\end{proof}

Putting together \eqref{eq:F_det_formula_for_Eynard}, 
\eqref{eq:G_pure_det_formula_for_Eynard}, and \eqref{eq:G_skew_det_formula_for_Eynard}, we get:
\begin{proposition}
	\label{prop:det_form_for_Eynard}
	The probability weights under the
	ascending FG process \eqref{eq:ascending_process}
	have the following product-of-determinants form.
	For $\ell^{(t)}_j:=\lambda^{(t)}_{j}+N+1-j$, we have
	\begin{equation*}
		\begin{split}&
		\mathscr{AP}(\lambda^{(1)},\lambda^{(2)},\ldots,\lambda^{(T)})
		\\&\hspace{30pt}=
		\mathrm{const}\cdot
		\det\bigl[ \mathsf{h}_{\ell^{(1)}_i-1,\,j}(w_1;\theta_1) \bigr]
		\prod_{t=2}^{T}
		\det\bigl[ 
			\mathsf{g}_{(\ell^{(t)}_i-1)/(\ell^{(t-1)}_j-1)}
			(w_t;\theta_t) 
		\bigr]
		\det\bigl[ \varphi_{\ell^{(T)}_j-1}(x_i) \bigr],
		\end{split}
	\end{equation*}
	where all determinants are 
	taken with respect to $1\le i,j\le N$, and 
	$\mathrm{const}$ is a normalizing constant which does not depend on 
	the $\ell^{(j)}$'s.
\end{proposition}

\subsection{Application of the Eynard--Mehta theorem}
\label{sub:appB_Eynard_application}

The form of the probability weights 
as in \Cref{prop:det_form_for_Eynard}
puts the ascending FG process into the domain of applicability of the Eynard--Mehta theorem
(see, for example, \cite{eynard1998matrices}, \cite[Theorem 1.4]{borodin2005eynard}).
To express the determinantal correlation kernel of the point process
\begin{equation}
	\label{eq:determ_pp_Eynard}
	\{(t,\ell^{(t)}_j)\colon t=1,\ldots,T,\,j=1,\ldots,N \}\subset\left\{ 1,\ldots,T  \right\}
	\times\mathbb{Z}_{\ge1},\qquad \ell^{(t)}_j=\lambda^{(t)}_j+N+1-j,
\end{equation}
one first needs to invert the $N\times N$ ``Gram matrix'' given by
\begin{equation}
	\label{eq:Eynard_Mij_summation}
	M_{ij}=\sum_{a_1,\ldots,a_m\ge 0}\mathsf{h}_{a_1,i}(w_1;\theta_1)
	\,
	\mathsf{g}_{a_2/a_1}(w_2;\theta_2)
	\ldots 
	\mathsf{g}_{a_T/a_{T-1}}(w_T;\theta_T)
	\,
	\varphi_{a_T}(x_j).
\end{equation}
Note that by \Cref{lemma:explicit_Eynard_h_g},
this series converges absolutely under the condition
\eqref{eq:Eynard_asc_FG_condition}.

\begin{proposition}
	\label{prop:Eynard_Gram_matrix}
	We have
	\begin{equation}
		\label{eq:Eynard_Mij_computation}
		M_{ij}=
		\frac{1}{y_i-x_j}\prod_{t=1}^{T}\frac{x_j-\theta_t^{-2}w_t}{x_j-w_t}.
	\end{equation}
\end{proposition}
The proof is based on the following lemma:
\begin{lemma}
	\label{lemma:phi_psi_summation}
	Let $\bigl|
		\frac{u-s_j^{-2}y_j}{u-y_j}
		\frac{v-y_j}{v-s_j^{-2}y_j}
	\bigr|<1-\delta<1$ for all sufficiently large $j\ge1$.
	Then we have
	\begin{equation*}
		\sum_{k=0}^{\infty}\varphi_k(u)\psi_k(v)=\frac{1}{u-v}.
	\end{equation*}
\end{lemma}
\begin{proof}
	We have
	\begin{equation*}
		\begin{split}
			\sum_{k=0}^{\infty}\varphi_k(u)\psi_k(v)&=
			\sum_{k=0}^{\infty}
			\frac{1}{u-y_{k+1}}
			\frac{y_{k+1}(1-s_{k+1}^{-2})}{v-s_{k+1}^{-2}y_{k+1}}
			\prod_{j=1}^{k}
			\frac{u-s_j^{-2}y_j}{u-y_j}
			\frac{v-y_j}{v-s_j^{-2}y_j}
			\\&=
			\sum_{k=0}^{\infty}
			\frac{1}{u-v}\left( 1-
			\frac{u-s_{k+1}^{-2}y_{k+1}}{u-y_{k+1}}
			\frac{v-y_{k+1}}{v-s_{k+1}^{-2}y_{k+1}}
			\right)
			\prod_{j=1}^{k}
			\frac{u-s_j^{-2}y_j}{u-y_j}
			\frac{v-y_j}{v-s_j^{-2}y_j},
		\end{split}
	\end{equation*}
	and the sum telescopes to $1 / (u-v)$ if it converges 
	(which holds under the condition in the hypothesis).
\end{proof}

\begin{proof}[Proof of \Cref{prop:Eynard_Gram_matrix}]
	We represent $\mathsf{h}_{a_1,i}$ as an integral over $z_1$,
	and each $\mathsf{g}_{a_t/a_{t-1}}$ as an integral 
	over $z_t$, $2\le t\le T$. Initially all the integration variables 
	belong to the same contour $\Gamma_{y,w}$.
	However, in order to apply \Cref{lemma:phi_psi_summation}
	under the integrals, we need to have the following conditions on the contours
	for all sufficiently large $k\ge1$:
	\begin{equation*}
		\Bigl|
			\frac{z_{t+1}-s_k^{-2}y_k}{z_{t+1}-y_k}
			\frac{z_t-y_k}{z_{t}-s_k^{-2}y_k}
		\Bigr|
		<1-\delta<1
		,\qquad 
		\Bigl|
			\frac{x_j-s_k^{-2}y_k}{x_j-y_k}
			\frac{z_T-y_k}{z_{T}-s_k^{-2}y_k}
		\Bigr|
		<1-\delta<1,
	\end{equation*}
	where $t=1,\ldots,T-1$, $j=1,\ldots,N $.
	Clearly, under certain restrictions on the parameters,
	such contours exist. 
	Moreover, we may also choose them to be nested:
	$z_1$ around all $y_k$ and $w_t$, $z_{B}$ around $z_A$ if $B>A$,
	and all contours must leave outside all the points $s_k^{-2}y_k$.
	On these contours, we have
	by \Cref{lemma:phi_psi_summation}:
	\begin{equation*}
		M_{ij}=
		\frac{1}{(2\pi\mathbf{i})^{T}}
		\oint\ldots\oint 
		\frac{1}{y_i-z_1}
		\frac{dz_1\ldots dz_T }{(x_j-z_T)(z_T-z_{T-1})\ldots(z_2-z_1)}
		\prod_{t=1}^{T}\frac{z_t-\theta_t^{-2}w_t}{z_t-w_t}.
	\end{equation*}
	This integral is computed as follows. 
	First, for $z_T$ there is a single pole $z_T=x_j$
	outside the contour (and the integrand 
	has the zero residue at infinity).
	Taking the residue clears the denominator
	$x_j-z_T$ and substitutes $z_T=x_j$. After that, 
	we repeat the procedure for $z_{T-1},\ldots,z_1$,
	which leads to the desired formula.

	Finally, the restrictions on the parameters under which the contours exist 
	are lifted by an analytic continuation, since 
	\Cref{lemma:explicit_Eynard_h_g,lemma:phi_psi_summation} imply that 
	the summation
	in \eqref{eq:Eynard_Mij_summation} produces an a priori rational function.
\end{proof}

The matrix $M=[M_{ij}]_{i,j=1}^{N}$ is readily inverted:
\begin{lemma}
	\label{lemma:Eynard_M_inversion}
	We have, for $i,j=1,\ldots,N$,
	\begin{equation}
		\label{eq:Eynard_M_inversion}
		\begin{split}
			M_{ij}^{-1}&=
			\frac{1}{x_i-y_j}
			\frac{\prod_{k=1}^{N}(x_i-y_k)(y_j-x_k)}{\prod_{k\ne i}(x_i-x_k)\prod_{k\ne j}(y_j-y_k)}
			\prod_{t=1}^{T}\frac{x_i-w_t}{x_i-\theta_t^{-2}w_t}
			\\&=
			\frac{1}{(2\pi\mathbf{i})^2}
			\oint_{\Gamma_{x_i}}d\xi\oint_{\Gamma_{y_j}}d\eta\,
			\frac{1}{\xi-\eta}
			\prod_{k=1}^{N}\frac{(\xi-y_k)(\eta-x_k)}{(\xi-x_k)(\eta-y_k)}
			\prod_{t=1}^{T}\frac{\xi-w_t}{\xi-\theta_t^{-2}w_t},
		\end{split}
	\end{equation}
	where the contours for $\xi$ and $\eta$ are
	small nonintersecting
	positively oriented circles around $x_i$ and 
	$y_j$, respectively, which do not include any other poles of the integrand.
\end{lemma}
\begin{proof}
	The first expression for $M_{ij}^{-1}$ is obtained
	using the Cauchy determinant,
	since all minors (and hence all cofactors) of $M$ are determinants of similar form.
	The contour integral expression corresponds to taking 
	residues at the simple poles $\xi=x_i$
	and $\eta=y_j$.
\end{proof}

By the Eynard--Mehta theorem as in \cite[Theorem 1.4]{borodin2005eynard}, the correlation kernel of the 
determinantal point process
\eqref{eq:determ_pp_Eynard}
on $\left\{ 1,\ldots,T  \right\}\times \mathbb{Z}_{\ge1}$
takes the form (the shifts $a+1,a'+1$ correspond to the shifts in the 
determinantal representation in \Cref{prop:det_form_for_Eynard}):
\begin{align}
	\nonumber
		&
		K_{\mathscr{AP}}(t,a+1;t',a'+1)
		\\&\hspace{10pt}=
		-\mathbf{1}_{t>t'}\sum_{\alpha_{t'+1},\ldots,\alpha_{t-1}\ge0 }
		\mathsf{g}_{\alpha_{t'+1}/a'}(w_{t'+1};\theta_{t'+1})
		\ldots
		\mathsf{g}_{\alpha_{t-1}/\alpha_{t-2}}(w_{t-1};\theta_{t-1})
		\mathsf{g}_{a/\alpha_{t-1}}(w_t;\theta_t)
		\nonumber
		\\ &\hspace{20pt}+
		\sum_{i,j=1}^{N}
		M_{ji}^{-1}
		\sum_{\alpha_1,\ldots,\alpha_{t-1}\ge0 }
		\mathsf{h}_{\alpha_1,i}(w_1;\theta_1)
		\mathsf{g}_{\alpha_2/\alpha_1}(w_2;\theta_2)
		\ldots
		\mathsf{g}_{a/\alpha_{t-1}}(w_t;\theta_t)
		\label{eq:K_AP_Eynard_first}
		\\
		\nonumber&\hspace{60pt}
		\times
		\sum_{\beta_{t'+1},\ldots,\beta_T\ge0 }
		\mathsf{g}_{\beta_{t'+1}/a'}(w_{t'+1};\theta_{t'+1})
		\ldots \mathsf{g}_{\beta_T/\beta_{T-1}}(w_T;\theta_T)
		\varphi_{\beta_T}(x_j).
\end{align}
The iterated sums over the $\alpha_j$'s in the first and the second terms are finite and thus converge,
and the sum over the $\beta_j$'s is infinite but converges 
under \eqref{eq:Eynard_asc_FG_condition}, see \Cref{lemma:explicit_Eynard_h_g}.

\subsection{Computation of the kernel}
\label{eq:appB_computation_kernel}

Let us now compute all the sums in \eqref{eq:K_AP_Eynard_first}, and 
arrive at the resulting formula for the correlation kernel.

For the first summand arising when $t>t'$, we 
pass to the nested contours
($z_{B}$ around $z_A$ if $B>A$) as in the proof of \Cref{prop:Eynard_Gram_matrix}.
We obtain
\begin{align}
		&
		\sum_{\alpha_{t'+1},\ldots,\alpha_{t-1}\ge0 }
		\mathsf{g}_{\alpha_{t'+1}/a'}(w_{t'+1};\theta_{t'+1})
		\ldots
		\mathsf{g}_{\alpha_{t-1}/\alpha_{t-2}}(w_{t-1};\theta_{t-1})
		\mathsf{g}_{a/\alpha_{t-1}}(w_t;\theta_t)
		\label{eq:K_Eynard_additional_summand_1}
		\\&\hspace{5pt}
		\nonumber=
		\frac{1}{(2\pi\mathbf{i})^{t-t'}}
		\oint\ldots\oint 
		\varphi_{a'}(z_{t'+1})\psi_a(z_t)\,
		\frac{dz_{t'+1}\ldots dz_t}{(z_{t'+2}-z_{t'+1})\ldots(z_{t-1}-z_{t-2})(z_t-z_{t-1}) }
		\prod_{i=t'+1}^t \frac{z_i-\theta_i^{-2}w_i}{z_i-w_i},
\end{align}
where we extended the sum over the
$\alpha_j$'s to all $\alpha_j\ge 0$
under the integral, and the infinite sums under the integral are
computed using \Cref{lemma:phi_psi_summation}.
Next, 
deforming the contours $z_{t-1},z_{t-2},\ldots,z_{t'+1} $
(in this order) to infinity, each integration in $z_i$ picks up a residue at a single pole outside the 
integration contour at $z_i=z_t$. This leaves a single integral:
\begin{equation}
		\label{eq:K_Eynard_additional_summand}
		\eqref{eq:K_Eynard_additional_summand_1}=
		\frac{1}{2\pi\mathbf{i}}
		\oint_{\Gamma_{y,w}} dz \,
		\varphi_{a'}(z)\psi_a(z)\,
		\prod_{i=t'+1}^t \frac{z-\theta_i^{-2}w_i}{z-w_i}.
\end{equation}

Arguing in a similar manner, we can compute
\begin{align*}
		&\sum_{\alpha_1,\ldots,\alpha_{t-1}\ge0 }
		\mathsf{h}_{\alpha_1,i}(w_1;\theta_1)
		\mathsf{g}_{\alpha_2/\alpha_1}(w_2;\theta_2)
		\ldots
		\mathsf{g}_{a/\alpha_{t-1}}(w_t;\theta_t)
		\\&\hspace{20pt}
		=
		\frac{1}{(2\pi\mathbf{i})^{t}}
		\oint\ldots\oint
		\frac{\psi_a(z_t)}{y_i-z_1}
		\frac{dz_1\ldots dz_t }{(z_2-z_1)(z_3-z_2)\ldots (z_t-z_{t-1}) }
		\prod_{d=1}^t \frac{z_d-\theta_d^{-2}w_d}{z_d-w_d}
		\\&\hspace{20pt}
		=
		\frac{1}{2\pi\mathbf{i}}
		\oint_{\Gamma_{y,w}}
		\frac{\psi_a(z)\,dz}{y_i-z}
		\prod_{d=1}^t \frac{z-\theta_d^{-2}w_d}{z-w_d},
\end{align*}
and
\begin{align*}
		&
		\sum_{\beta_{t'+1},\ldots,\beta_T\ge0 }
		\mathsf{g}_{\beta_{t'+1}/a'}(w_{t'+1};\theta_{t'+1})
		\ldots \mathsf{g}_{\beta_T/\beta_{T-1}}(w_T;\theta_T)
		\varphi_{\beta_T}(x_j)
		\\&\hspace{10pt}
		=
		\frac{1}{(2\pi\mathbf{i})^{T-t'}}
		\oint\ldots\oint 
		\frac{\varphi_{a'}(z_{t'+1})}{x_j-z_T}
		\frac{dz_{t'+1}\ldots dz_T}{(z_{t'+2}-z_{t'+1})\ldots(z_{T-1}-z_{T-2})(z_T-z_{T-1}) }
		\prod_{c=t'+1}^T \frac{z_c-\theta_c^{-2}w_c}{z_c-w_c}
		\\&\hspace{10pt}
		=
		\varphi_{a'}(x_j)
		\prod_{c=t'+1}^T \frac{x_j-\theta_c^{-2}w_c}{x_j-w_c}.
\end{align*}
In the latter computation we pick the residues at $z_T=x_j,\ldots,z_{t'+1}=x_j$
(in this order), which is the only pole outside the corresponding integration contour.
Finally, we take the last two 
quantities, multiply by $M_{ji}^{-1}$, and sum as in \eqref{eq:K_AP_Eynard_first}.
Using \eqref{eq:Eynard_M_inversion}, we have
\begin{align*}
		&
		\sum_{i,j=1}^N
		\frac{1}{(2\pi\mathbf{i})^2}
		\oint_{\Gamma_{x_j}}d\xi\oint_{\Gamma_{y_i}} d \eta\,
		\frac{1}{\xi-\eta}
		\prod_{k=1}^{N}\frac{(\xi-y_k)(\eta-x_k)}{(\xi-x_k)(\eta-y_k)}
		\prod_{t=1}^{T}\frac{\xi-w_t}{\xi-\theta_t^{-2}w_t}
		\\
		&\hspace{40pt}
		\times
		\frac{1}{2\pi\mathbf{i}}
		\oint_{\Gamma_{y,w}}
		\frac{\psi_a(z)\,dz}{y_i-z}
		\prod_{d=1}^t \frac{z-\theta_d^{-2}w_d}{z-w_d}
		\,\varphi_{a'}(x_j)
		\prod_{c=t'+1}^T \frac{x_j-\theta_c^{-2}w_c}{x_j-w_c}
		\\&\hspace{20pt}
		=
		\frac{1}{(2\pi\mathbf{i})^3}
		\oint_{\Gamma_{x}}d\xi\oint_{\Gamma_{y}} d \eta
		\oint_{\Gamma_{y,w}}
		dz
		\,
		\frac{1}{\eta-\xi}
		\frac{1}{\eta-z}
		\prod_{k=1}^{N}\frac{(\xi-y_k)(\eta-x_k)}{(\xi-x_k)(\eta-y_k)}
		\\
		&\hspace{40pt}
		\times
		\frac{y_{a+1}(1-s_{a+1}^{-2})}{z-s_{a+1}^{-2}y_{a+1}}
		\frac{1}{y_{a'+1}-\xi}
		\prod_{j=1}^{a}
		\frac{z-y_j}{z-s_j^{-2}y_j}
		\prod_{j=1}^{a'}
		\frac{\xi-s_j^{-2}y_j}{\xi-y_j}
		\prod_{d=1}^t \frac{z-\theta_d^{-2}w_d}{z-w_d}
		\prod_{c=1}^{t'}\frac{\xi-w_c}{\xi-\theta_c^{-2}w_c}.
\end{align*}
To obtain the latter expression
we substituted $x_j=\xi$, $y_i=\eta$, and
changed the contours $\Gamma_x,\Gamma_y$ for these variables to encircle 
all $x_k$'s or all $y_k$'s, respectively, while leaving all other poles outside.
Observe now that the only pole in $\eta$ outside the integration contour
which produces a nonzero residue is at $\eta=z$.
Indeed, the residue at $\eta=\xi$ eliminates all poles inside the $\xi$ contour, and thus vanishes. 
Therefore, we may continue the above computation as follows:
\begin{align*}
	&=
		\frac{1}{(2\pi\mathbf{i})^2}
		\oint_{\Gamma_{x}}d\xi
		\oint_{\Gamma_{y,w}}
		dz
		\,
		\frac{1}{z-\xi}
		\prod_{k=1}^{N}\frac{(\xi-y_k)(z-x_k)}{(\xi-x_k)(z-y_k)}
		\\
		&\hspace{40pt}
		\times
		\frac{y_{a+1}(1-s_{a+1}^{-2})}{z-s_{a+1}^{-2}y_{a+1}}
		\frac{1}{\xi-y_{a'+1}}
		\prod_{j=1}^{a}
		\frac{z-y_j}{z-s_j^{-2}y_j}
		\prod_{j=1}^{a'}
		\frac{\xi-s_j^{-2}y_j}{\xi-y_j}
		\prod_{d=1}^t \frac{z-\theta_d^{-2}w_d}{z-w_d}
		\prod_{c=1}^{t'}\frac{\xi-w_c}{\xi-\theta_c^{-2}w_c}.
\end{align*}
Let us drag the $\xi$ contour through infinity, 
so that now it encircles the $z$ contour
$\Gamma_{y,w}$, and also all the points $\theta_i^{-2}w_i$.
This leads to an extra minus sign.

Finally, we need to add the additional summand
\eqref{eq:K_Eynard_additional_summand} if $t>t'$.
In this case, observe that dragging the $z$ contour so that it is outside of the $\xi$ contour
produces the same expression as \eqref{eq:K_Eynard_additional_summand},
but with the opposite sign. 
Moreover,
we need to undo the
shifts $a+1,a'+1$ corresponding to the
determinantal representation in \Cref{prop:det_form_for_Eynard}. Renaming the integration variables
as $\xi=u$, $z=v$ 
leads to the final expression
for the correlation kernel of the ascending FG process:
\begin{equation*}
	\begin{split}
		&K_{\mathscr{AP}}(t,a;t',a')
		=
		\frac{1}{(2\pi\mathbf{i})^2}
		\oint_{\Gamma_{y,w,\theta^{-2}w}}du
		\oint_{\Gamma_{y,w}}
		dv
		\,
		\frac{1}{u-v}
		\prod_{k=1}^{N}\frac{(u-y_k)(v-x_k)}{(u-x_k)(v-y_k)}
		\\
		&\hspace{20pt}
		\times
		\frac{y_{a}(1-s_{a}^{-2})}{v-s_{a}^{-2}y_{a}}
		\frac{1}{u-y_{a'}}
		\prod_{j=1}^{a-1}
		\frac{v-y_j}{v-s_j^{-2}y_j}
		\prod_{j=1}^{a'-1}
		\frac{u-s_j^{-2}y_j}{u-y_j}
		\prod_{d=1}^t \frac{v-\theta_d^{-2}w_d}{v-w_d}
		\prod_{c=1}^{t'}\frac{u-w_c}{u-\theta_c^{-2}w_c}.
	\end{split}
\end{equation*}
where the $u$ contour is outside for $t\le t'$, and
the $v$ contour is outside for $t>t'$.
This completes the proof of \Cref{thm:ascending_FG_process_kernel} in the ascending 
FG process case.

\bibliography{bib}

\begin{bibdiv}
\begin{biblist}

\bib{agg-bor-wh2020-sl1n}{article}{
      author={Aggarwal, A.},
      author={Borodin, A.},
      author={Wheeler, M.},
       title={{Colored Fermionic Vertex Models and Symmetric Functions}},
        date={2021},
     journal={arXiv preprint},
        note={arXiv:2101.01605 [math.CO]},
}

\bib{aggarwal2019universality}{article}{
      author={Aggarwal, A.},
       title={{Universality for Lozenge Tiling Local Statistics}},
        date={2019},
     journal={arXiv preprint},
        note={arXiv:1907.09991 [math.PR]},
}

\bib{AndersonGuionnetZeitouniBook}{book}{
      author={Anderson, G.W.},
      author={Guionnet, A.},
      author={Zeitouni, O.},
       title={{An introduction to random matrices}},
   publisher={Cambridge University Press},
        date={2010},
}

\bib{theodoros2019_determ}{article}{
      author={Assiotis, T.},
       title={{Determinantal Structures in Space Inhomogeneous Dynamics on
  Interlacing Arrays}},
        date={2020},
     journal={Ann. Inst. H. Poincar\'e},
      volume={21},
       pages={909\ndash 940},
        note={arXiv:1910.09500 [math.PR]},
}

\bib{baxter2007exactly}{book}{
      author={Baxter, R.},
       title={{Exactly solved models in statistical mechanics}},
   publisher={Courier Dover Publications},
        date={2007},
}

\bib{betea2019periodic}{article}{
      author={Betea, D.},
      author={Bouttier, J.},
       title={{The periodic Schur process and free fermions at finite
  temperature}},
        date={2019},
     journal={Mathematical Physics, Analysis and Geometry},
      volume={22},
      number={1},
       pages={3},
        note={arXiv:1807.09022 [math-ph]},
}

\bib{boutillier2015dimers}{article}{
      author={Boutillier, C.},
      author={Bouttier, J.},
      author={Chapuy, G.},
      author={Corteel, S.},
      author={Ramassamy, S.},
       title={Dimers on rail yard graphs},
        date={2017},
     journal={Annales de l'Institut Henri Poincar{\'e} D},
      volume={4},
      number={4},
       pages={479\ndash 539},
        note={arXiv:1504.05176 [math-ph]},
}

\bib{brubaker2011schur}{article}{
      author={Brubaker, B.},
      author={Bump, D.},
      author={Friedberg, S.},
       title={{Schur polynomials and the Yang-Baxter equation}},
        date={2011},
     journal={Comm. Math. Phys.},
      volume={308},
      number={2},
       pages={281},
}

\bib{BorodinCorwin2011Macdonald}{article}{
      author={Borodin, A.},
      author={Corwin, I.},
       title={Macdonald processes},
        date={2014},
     journal={Probab. Theory Relat. Fields},
      volume={158},
       pages={225\ndash 400},
        note={arXiv:1111.4408 [math.PR]},
}

\bib{bouttier2017aztec}{article}{
      author={Bouttier, J.},
      author={Chapuy, G.},
      author={Corteel, S.},
       title={{From Aztec diamonds to pyramids: steep tilings}},
        date={2017},
     journal={Trans. AMS},
      volume={369},
      number={8},
       pages={5921\ndash 5959},
        note={arXiv:1407.0665 [math.CO]},
}

\bib{berggren2021domino}{article}{
      author={Berggren, T.},
       title={{Domino tilings of the Aztec diamond with doubly periodic
  weightings}},
        date={2021},
     journal={Ann. Probab.},
      volume={49},
      number={4},
       pages={1965\ndash 2011},
        note={arXiv:1911.01250 [math.PR]},
}

\bib{BorFerr2008DF}{article}{
      author={Borodin, A.},
      author={Ferrari, P.},
       title={{Anisotropic growth of random surfaces in 2+1 dimensions}},
        date={2014},
     journal={Commun. Math. Phys.},
      volume={325},
       pages={603\ndash 684},
        note={arXiv:0804.3035 [math-ph]},
}

\bib{bump2011factorial}{article}{
      author={Bump, D.},
      author={McNamara, P.},
      author={Nakasuji, M.},
       title={{Factorial Schur functions and the Yang-Baxter equation}},
        date={2014},
     journal={Rikkyo-daigaku-sugaku-zasshi},
      volume={63},
      number={1-2},
       pages={23\ndash 45},
        note={arXiv:1108.3087 [math.CO]},
}

\bib{borodin2000distributions}{article}{
      author={Borodin, A.},
      author={Olshanski, G.},
       title={Distributions on partitions, point processes, and the
  hypergeometric kernel},
        date={2000},
     journal={Communications in Mathematical Physics},
      volume={211},
      number={2},
       pages={335\ndash 358},
        note={arXiv:math/9904010 [math.RT]},
}

\bib{Borodin2000b}{article}{
      author={Borodin, A.},
      author={Okounkov, A.},
      author={Olshanski, G.},
       title={{Asymptotics of Plancherel measures for symmetric groups}},
        date={2000},
     journal={Jour. AMS},
      volume={13},
      number={3},
       pages={481\ndash 515},
        note={arXiv:math/9905032 [math.CO]},
}

\bib{borodin2007periodic}{article}{
      author={Borodin, A.},
       title={{Periodic Schur process and cylindric partitions}},
        date={2007},
     journal={Duke J. Math.},
      volume={140},
      number={3},
       pages={391\ndash 468},
        note={arXiv:math/0601019 [math.CO]},
}

\bib{Borodin2009}{incollection}{
      author={Borodin, A.},
       title={Determinantal point processes},
        date={2011},
   booktitle={Oxford handbook of random matrix theory},
      editor={Akemann, G.},
      editor={Baik, J.},
      editor={Di~Francesco, P.},
   publisher={Oxford University Press},
        note={arXiv:0911.1153 [math.PR]},
}

\bib{BorodinPeche2009}{article}{
      author={Borodin, A.},
      author={Peche, S.},
       title={Airy kernel with two sets of parameters in directed percolation
  and random matrix theory},
        date={2008},
     journal={Jour. Stat. Phys.},
      volume={132},
      number={2},
       pages={275\ndash 290},
        note={arXiv:0712.1086v3 [math-ph]},
}

\bib{BorodinPetrov2016inhom}{article}{
      author={Borodin, A.},
      author={Petrov, L.},
       title={Higher spin six vertex model and symmetric rational functions},
        date={2018},
     journal={Selecta Math.},
      volume={24},
      number={2},
       pages={751\ndash 874},
        note={arXiv:1601.05770 [math.PR]},
}

\bib{borodin2005eynard}{article}{
      author={Borodin, A.},
      author={Rains, E.M.},
       title={{Eynard--Mehta theorem, Schur process, and their Pfaffian
  analogs}},
        date={2005},
     journal={J. Stat. Phys},
      volume={121},
      number={3},
       pages={291\ndash 317},
        note={arXiv:math-ph/0409059},
}

\bib{BereleRegev}{article}{
      author={Berele, A.},
      author={Regev, A.},
       title={{Hook Young diagrams with applications to combinatorics and
  representations of Lie superalgebras}},
        date={1987},
     journal={Adv. Math.},
      volume={64},
      number={2},
       pages={118\ndash 175},
}

\bib{borodin2010gibbs}{article}{
      author={Borodin, A.},
      author={Shlosman, S.},
       title={{Gibbs ensembles of nonintersecting paths}},
        date={2010},
     journal={Commun. Math. Phys.},
      volume={293},
      number={1},
       pages={145\ndash 170},
        note={arXiv:0804.0564 [math-ph]},
}

\bib{cohn-elki-prop-96}{article}{
      author={Cohn, Henry},
      author={Elkies, Noam},
      author={Propp, James},
       title={Local statistics for random domino tilings of the {A}ztec
  diamond},
        date={1996},
     journal={Duke Math. J.},
      volume={85},
      number={1},
       pages={117\ndash 166},
  url={http://dx.doi.org.ezproxy.library.wisc.edu/10.1215/S0012-7094-96-08506-3},
        note={arXiv:math/0008243 [math.CO]},
}

\bib{charlier2021doubly}{article}{
      author={Charlier, C.},
       title={Doubly periodic lozenge tilings of a hexagon and matrix valued
  orthogonal polynomials},
        date={2021},
     journal={Studies in Applied Mathematics},
      volume={146},
      number={1},
       pages={3\ndash 80},
        note={arXiv:2001.11095 [math-ph]},
}

\bib{chhita2016domino}{article}{
      author={Chhita, S.},
      author={Johansson, K.},
       title={{Domino statistics of the two-periodic Aztec diamond}},
        date={2016},
     journal={Adv. Math.},
      volume={294},
       pages={37\ndash 149},
        note={arXiv:1410.2385 [math.PR]},
}

\bib{CohnKenyonPropp2000}{article}{
      author={Cohn, H.},
      author={Kenyon, R.},
      author={Propp, J.},
       title={A variational principle for domino tilings},
        date={2001},
     journal={Jour. AMS},
      volume={14},
      number={2},
       pages={297\ndash 346},
        note={arXiv:math/0008220 [math.CO]},
}

\bib{decreusefond2016determinantal}{incollection}{
      author={Decreusefond, L.},
      author={Flint, I.},
      author={Privault, Nicolas},
      author={T., Giovanni~L.},
       title={Determinantal point processes},
        date={2016},
   booktitle={Stochastic analysis for poisson point processes},
   publisher={Springer},
       pages={311\ndash 342},
}

\bib{duits2017two}{article}{
      author={Duits, M.},
      author={Kuijlaars, A.},
       title={{The two periodic Aztec diamond and matrix valued orthogonal
  polynomials}},
        date={2020},
     journal={Journal of the European Mathematical Society},
      volume={23},
      number={4},
       pages={1075\ndash 1131},
        note={arXiv:1712.05636 [math.PR]},
}

\bib{dyson1962brownian}{article}{
      author={Dyson, F.J.},
       title={{A Brownian motion model for the eigenvalues of a random
  matrix}},
        date={1962},
     journal={Journal of Mathematical Physics},
      volume={3},
      number={6},
       pages={1191\ndash 1198},
}

\bib{elkies1992alternating}{article}{
      author={Elkies, N.},
      author={Kuperberg, G.},
      author={Larsen, M.},
      author={Propp, J.},
       title={Alternating-sign matrices and domino tilings},
        date={1992},
     journal={Jour. Alg. Comb.},
      volume={1},
      number={2-3},
       pages={111\ndash 132 and 219\ndash 234},
}

\bib{eynard1998matrices}{article}{
      author={Eynard, B.},
      author={Mehta, M.L.},
       title={{Matrices coupled in a chain: I. Eigenvalue correlations}},
        date={1998},
     journal={J. Phys. A},
      volume={31},
       pages={4449\ndash 4456},
}

\bib{Faddeev_Lectures}{article}{
      author={Faddeev, L.D.},
       title={{How algebraic Bethe ansatz works for integrable model}},
        date={1996},
        note={Les-Houches lectures (1996), arXiv:hep-th/9605187},
}

\bib{felderhof1973diagonalization2}{article}{
      author={Felderhof, B.U.},
       title={Diagonalization of the transfer matrix of the free-fermion model.
  ii},
        date={1973},
     journal={Physica},
      volume={66},
      number={2},
       pages={279\ndash 297},
}

\bib{felderhof1973diagonalization3}{article}{
      author={Felderhof, B.U.},
       title={Diagonalization of the transfer matrix of the free-fermion model.
  iii},
        date={1973},
     journal={Physica},
      volume={66},
      number={3},
       pages={509\ndash 526},
}

\bib{felderhof1973direct}{article}{
      author={Felderhof, B.U.},
       title={Direct diagonalization of the transfer matrix of the zero-field
  free-fermion model},
        date={1973},
     journal={Physica},
      volume={65},
      number={3},
       pages={421\ndash 451},
}

\bib{fomin1994grothendieck}{inproceedings}{
      author={Fomin, S.},
      author={Kirillov, A.N.},
       title={{Grothendieck polynomials and the Yang-Baxter equation}},
        date={1994},
   booktitle={Proc. formal power series and alg. comb},
       pages={183\ndash 190},
}

\bib{fomin1996yang}{article}{
      author={Fomin, S.},
      author={Kirillov, A.N.},
       title={{The Yang-Baxter equation, symmetric functions, and Schubert
  polynomials}},
        date={1996},
     journal={Discrete Mathematics},
      volume={153},
      number={1-3},
       pages={123\ndash 143},
}

\bib{feher2012equivariant}{article}{
      author={Feh{\'e}r, L.},
      author={N{\'e}methi, A.},
      author={Rim{\'a}nyi, R.},
       title={Equivariant classes of matrix matroid varieties},
        date={2012},
     journal={Commentarii Mathematici Helvetici},
      volume={87},
      number={4},
       pages={861\ndash 889},
        note={arXiv:0812.4871 [math.AG]},
}

\bib{forrester2019meet}{article}{
      author={Forrester, P.J.},
       title={{Meet Andr{\'e}ief, Bordeaux 1886, and Andreev, Kharkov
  1882--1883}},
        date={2019},
     journal={Random Matrices: Theory and Applications},
      volume={8},
      number={02},
       pages={1930001},
        note={arXiv:1806.10411 [math-ph]},
}

\bib{ferrari2006domino}{article}{
      author={Ferrari, P.L.},
      author={Spohn, H.},
       title={{Domino tilings and the six-vertex model at its free-fermion
  point}},
        date={2006},
     journal={Jour. Phys. A},
      volume={39},
      number={33},
       pages={10297},
        note={arXiv:cond-mat/0605406 [cond-mat.stat-mech]},
}

\bib{FelderVarchenko1996}{article}{
      author={Felder, G.},
      author={Varchenko, A.},
       title={Algebraic {B}ethe ansatz for the elliptic quantum group
  {$E_{\tau,\eta}({\rm sl}_2)$}},
        date={1996},
        ISSN={0550-3213},
     journal={Nuclear Phys. B},
      volume={480},
      number={1-2},
       pages={485\ndash 503},
         url={http://dx.doi.org/10.1016/S0550-3213(96)00461-0},
        note={arXiv:q-alg/9605024},
}

\bib{gaudin1960demonstration}{article}{
      author={Gaudin, M.},
       title={Une d{\'e}monstration simplifi{\'e}e du th{\'e}oreme de wick en
  m{\'e}canique statistique},
        date={1960},
     journal={Nuclear Physics},
      volume={15},
       pages={89\ndash 91},
}

\bib{gorin2021lectures}{article}{
      author={Gorin, Vadim},
       title={Lectures on random lozenge tilings},
        date={2021},
     journal={Cambridge Studies in Advanced Mathematics. Cambridge University
  Press},
        note={Available at
  \url{https://people.math.wisc.edu/~vadicgor/Random_tilings.pdf}},
}

\bib{gleizer2000littlewood}{article}{
      author={Gleizer, O.},
      author={Postnikov, A.},
       title={{Littlewood-Richardson coefficients via Yang-Baxter equation}},
        date={2000},
     journal={International Mathematics Research Notices},
      volume={2000},
      number={14},
       pages={741\ndash 774},
}

\bib{gorin2019universality}{article}{
      author={Gorin, V.},
      author={Petrov, L.},
       title={Universality of local statistics for noncolliding random walks},
        date={2019},
     journal={Annals of Probability},
      volume={47},
      number={5},
       pages={2686\ndash 2753},
        note={arXiv:1608.03243 [math.PR]},
}

\bib{guo2019identities}{article}{
      author={Guo, P.},
      author={Sun, S.},
       title={Identities on factorial grothendieck polynomials},
        date={2019},
     journal={Adv. Appl. Math.},
      volume={111},
       pages={101933},
        note={arXiv:1812.04390 [math.CO]},
}

\bib{gunna2022integrable}{article}{
      author={Gunna, A.},
      author={Scrimshaw, T.},
       title={Integrable systems and crystals for edge labeled tableaux},
        date={2022},
     journal={arXiv preprint},
        note={arXiv:2202.06004 [math.CO]},
}

\bib{hardt2021lattice}{article}{
      author={Hardt, A.},
       title={{Lattice Models, Hamiltonian Operators, and Symmetric
  Functions}},
        date={2021},
     journal={arXiv preprint},
        note={arXiv:2109.14597 [math.RT]},
}

\bib{hamel1995lattice}{article}{
      author={Hamel, A.M.},
      author={Goulden, I.P.},
       title={{Lattice paths and a Sergeev-Pragacz formula for skew
  supersymmetric functions}},
        date={1995},
     journal={Canad. J. Math.},
      volume={47},
      number={2},
       pages={364\ndash 382},
}

\bib{peres2006determinantal}{article}{
      author={Hough, J.B.},
      author={Krishnapur, M.},
      author={Peres, Y.},
      author={Vir{\'a}g, B.},
       title={{Determinantal processes and independence}},
        date={2006},
     journal={Probability Surveys},
      volume={3},
       pages={206\ndash 229},
        note={arXiv:math/0503110 [math.PR]},
}

\bib{ikeda2009excited}{article}{
      author={Ikeda, T.},
      author={Naruse, H.},
       title={{Excited Young diagrams and equivariant Schubert calculus}},
        date={2009},
     journal={Trans. AMS},
      volume={361},
      number={10},
       pages={5193\ndash 5221},
        note={arXiv:math/0703637 [math.AG]},
}

\bib{johansson2002non}{article}{
      author={Johansson, K.},
       title={{Non-intersecting paths, random tilings and random matrices}},
        date={2002},
     journal={Probab. Theory Relat. Fields},
      volume={123},
      number={2},
       pages={225\ndash 280},
        note={arXiv:math/0011250 [math.PR]},
}

\bib{Johansson2005arctic}{article}{
      author={Johansson, K.},
       title={The arctic circle boundary and the {A}iry process},
        date={2005},
     journal={Ann. Probab.},
      volume={33},
      number={1},
       pages={1\ndash 30},
        note={arXiv:math/0306216 [math.PR]},
}

\bib{Johansson2005lectures}{article}{
      author={Johansson, K.},
       title={Random matrices and determinantal processes},
        date={2005},
        note={arXiv:math-ph/0510038},
}

\bib{Kac1990InfiniteDim}{book}{
      author={Kac, Victor~G.},
       title={Infinite-dimensional {L}ie algebras},
     edition={Third edition},
   publisher={Cambridge University Press},
     address={Cambridge},
        date={1990},
        ISBN={0-521-37215-1; 0-521-46693-8},
         url={http://dx.doi.org/10.1017/CBO9780511626234},
      review={\MR{1104219 (92k:17038)}},
}

\bib{QISM_book}{book}{
      author={Korepin, V.},
      author={Bogoliubov, N.},
      author={Izergin, A.},
       title={Quantum inverse scattering method and correlation functions},
   publisher={Cambridge University Press, Cambridge},
        date={1993},
        ISBN={0-521-37320-4; 0-521-58646-1},
}

\bib{Kenyon2001GFF}{article}{
      author={Kenyon, R.},
       title={{Dominos and the Gaussian Free Field}},
        date={2001},
     journal={Ann. Probab.},
      volume={29},
      number={3},
       pages={1128\ndash 1137},
        note={arXiv:math-ph/0002027},
}

\bib{Kenyon2007Lecture}{article}{
      author={Kenyon, R.},
       title={Lectures on dimers},
        date={2009},
        note={arXiv:0910.3129 [math.PR]},
}

\bib{kitanine2002spin}{article}{
      author={Kitanine, N.},
      author={Maillet, J.M.},
      author={Slavnov, N.A.},
      author={Terras, V.},
       title={Spin--spin correlation functions of the xxz-12 heisenberg chain
  in a magnetic field},
        date={2002},
     journal={Nucl. Phys. B},
      volume={641},
      number={3},
       pages={487\ndash 518},
        note={arXiv:hep-th/0201045},
}

\bib{OkounkovKenyon2007Limit}{article}{
      author={Kenyon, R.},
      author={Okounkov, A.},
       title={Limit shapes and the complex {B}urgers equation},
        date={2007},
     journal={Acta Math.},
      volume={199},
      number={2},
       pages={263\ndash 302},
        note={arXiv:math-ph/0507007},
}

\bib{Konig2005}{article}{
      author={K{\"o}nig, W.},
       title={{Orthogonal polynomial ensembles in probability theory}},
        date={2005},
     journal={Probab. Surv.},
      volume={2},
       pages={385\ndash 447},
        note={arXiv:math/0403090 [math.PR]},
}

\bib{konig2002non}{article}{
      author={K{\"o}nig, W.},
      author={O'Connell, N.},
      author={Roch, S.},
       title={{Non-colliding random walks, tandem queues, and discrete
  orthogonal polynomial ensembles}},
        date={2002},
     journal={Electron. J. Probab.},
      volume={7},
      number={5},
       pages={1\ndash 24},
}

\bib{korff2013cylindric}{article}{
      author={Korff, C.},
       title={{Cylindric versions of specialised Macdonald functions and a
  deformed Verlinde algebra}},
        date={2013},
     journal={Commun. Math. Phys.},
      volume={318},
      number={1},
       pages={173\ndash 246},
        note={arXiv:1110.6356 [math-ph]},
}

\bib{korff2021cylindric}{article}{
      author={Korff, C.},
       title={{Cylindric Hecke Characters and Gromov--Witten Invariants via the
  Asymmetric Six-Vertex Model}},
        date={2021},
     journal={Commun. Math. Phys.},
      volume={381},
      number={2},
       pages={591\ndash 640},
        note={arXiv:1906.02565 [math-ph]},
}

\bib{KOS2006}{article}{
      author={Kenyon, R.},
      author={Okounkov, A.},
      author={Sheffield, S.},
       title={Dimers and amoebae},
        date={2006},
     journal={Ann. Math.},
      volume={163},
       pages={1019\ndash 1056},
        note={arXiv:math-ph/0311005},
}

\bib{kirillov1988bethe}{article}{
      author={Kirillov, A.N.},
      author={Reshetikhin, N.Yu.},
       title={{The Bethe ansatz and the combinatorics of Young tableaux}},
        date={1988},
     journal={Journal of Soviet Mathematics},
      volume={41},
      number={2},
       pages={925\ndash 955},
}

\bib{kulesza2012determinantal}{article}{
      author={Kulesza, A.},
      author={Taskar, B.},
       title={{Determinantal Point Processes for Machine Learning}},
        date={2012},
     journal={Foundations and Trends in Machine Learning},
      volume={5},
      number={2--3},
       pages={123\ndash 286},
        note={arXiv:1207.6083 [stat.ML]},
}

\bib{lascoux20076}{article}{
      author={Lascoux, A.},
       title={{The 6 vertex model and Schubert polynomials}},
        date={2007},
     journal={SIGMA},
      volume={3},
       pages={029},
        note={arXiv:math/0610719 [math.CO]},
}

\bib{lascoux1997flag}{article}{
      author={Lascoux, A.},
      author={Leclerc, B.},
      author={Thibon, J.-Y.},
       title={{Flag varieties and the Yang-Baxter equation}},
        date={1997},
     journal={Letters in Mathematical Physics},
      volume={40},
      number={1},
       pages={75\ndash 90},
}

\bib{lyons2003determinantal}{article}{
      author={Lyons, R.},
       title={Determinantal probability measures},
        date={2003},
     journal={Publ. IHES},
      volume={98},
       pages={167\ndash 212},
        note={arXiv:math/0204325 [math.PR]},
}

\bib{macchi1975coincidence}{article}{
      author={Macchi, O.},
       title={The coincidence approach to stochastic point processes},
        date={1975},
     journal={Adv. Appl. Probab.},
      volume={7},
      number={1},
       pages={83\ndash 122},
}

\bib{macdonald1992schur_Theme}{article}{
      author={Macdonald, I.G.},
       title={Schur functions: Theme and variations},
        date={1992},
     journal={S{\'e}m. Lothar. Combin},
      volume={28},
       pages={5\ndash 39},
}

\bib{Macdonald1995}{book}{
      author={Macdonald, I.G.},
       title={Symmetric functions and {H}all polynomials},
     edition={2},
   publisher={Oxford University Press},
        date={1995},
}

\bib{mcnamara2009factorial}{article}{
      author={McNamara, P.},
       title={{Factorial Schur functions via the six vertex model}},
        date={2009},
     journal={arXiv preprint},
        note={arXiv:0910.5288 [math.CO]},
}

\bib{mehta1960density}{article}{
      author={Mehta, M.L.},
      author={Gaudin, M.},
       title={On the density of eigenvalues of a random matrix},
        date={1960},
     journal={Nuclear Physics},
      volume={18},
       pages={420\ndash 427},
}

\bib{Mkrtchyan2014Periodic}{article}{
      author={Mkrtchyan, S.},
       title={Plane partitions with 2-periodic weights},
        date={2014},
     journal={Letters in Mathematical Physics},
      volume={104},
      number={9},
       pages={1053\ndash 1078},
        note={arXiv:1309.4825 [math.PR]},
}

\bib{Mkrtchyan2019}{article}{
      author={Mkrtchyan, S.},
       title={Turning point processes in plane partitions with periodic weights
  of arbitrary period},
        date={2019},
     journal={arXiv preprint},
        note={arXiv:1908.01246 [math.PR]},
}

\bib{molev2009comultiplication}{article}{
      author={Molev, A.},
       title={{Comultiplication rules for the double Schur functions and Cauchy
  identities}},
        date={2009},
     journal={Electron. J. Comb.},
      volume={R13},
        note={arXiv:0807.2127 [math.CO]},
}

\bib{motegi2017izergin}{article}{
      author={Motegi, K.},
       title={{Izergin-Korepin analysis on the projected wavefunctions of the
  generalized free-fermion model}},
        date={2017},
     journal={Advances in Mathematical Physics},
      volume={2017},
       pages={7563781},
        note={arXiv:1704.03575 [math-ph]},
}

\bib{motegi2020integrability}{article}{
      author={Motegi, K.},
       title={{Integrability approach to Feh\'er-N\'emethi-Rim\'anyi-Guo-Sun
  type identities for factorial Grothendieck polynomials}},
        date={2020},
     journal={Nuclear Physics B},
      volume={954},
       pages={114998},
        note={arXiv:1909.02278 [math.CO]},
}

\bib{MPP2}{article}{
      author={Morales, A.H.},
      author={Pak, I.},
      author={Panova, G.},
       title={{Hook formulas for skew shapes II. Combinatorial proofs and
  enumerative applications}},
        date={2017},
     journal={SIAM J. Discr. Math.},
      volume={31},
      number={3},
       pages={1953\ndash 1989},
        note={arXiv:1610.04744 [math.CO]},
}

\bib{MPP3}{article}{
      author={Morales, A.},
      author={Pak, I.},
      author={Panova, G.},
       title={{Hook formulas for skew shapes III. Multivariate and product
  formulas}},
        date={2019},
     journal={Algebraic Combinatorics},
      volume={2},
      number={5},
       pages={815\ndash 861},
        note={arXiv:1707.00931 [math.CO]},
}

\bib{moens2003determinantal}{article}{
      author={Moens, E.M.},
      author={Van~der Jeugt, J.},
       title={{A determinantal formula for supersymmetric Schur polynomials}},
        date={2003},
     journal={Journal of Algebraic Combinatorics},
      volume={17},
      number={3},
       pages={283\ndash 307},
}

\bib{nagao1998multilevel}{article}{
      author={Nagao, T.},
      author={Forrester, P.J.},
       title={{Multilevel dynamical correlation functions for Dyson's Brownian
  motion model of random matrices}},
        date={1998},
     journal={Physics Letters A},
      volume={247},
      number={1-2},
       pages={42\ndash 46},
}

\bib{nakagawa2001tableau}{incollection}{
      author={Nakagawa, J.},
      author={Noumi, M.},
      author={Shirakawa, M.},
      author={Yamada, Y.},
       title={{Tableau representation for Macdonald's ninth variation of Schur
  functions}},
        date={2001},
   booktitle={Physics and combinatorics},
   publisher={World Scientific},
       pages={180\ndash 195},
}

\bib{okounkov2001infinite}{article}{
      author={Okounkov, A.},
       title={{Infinite wedge and random partitions}},
        date={2001},
     journal={Selecta Math.},
      volume={7},
      number={1},
       pages={57\ndash 81},
        note={arXiv:math/9907127 [math.RT]},
}

\bib{Okounkov2002}{incollection}{
      author={Okounkov, A.},
       title={Symmetric functions and random partitions},
        date={2002},
   booktitle={Symmetric functions 2001: Surveys of developments and
  perspectives},
      editor={Fomin, S.},
   publisher={Kluwer Academic Publishers},
        note={arXiv:math/0309074 [math.CO]},
}

\bib{olshanski2019interpolation}{article}{
      author={Olshanski, G.},
       title={{Interpolation Macdonald polynomials and Cauchy-type
  identities}},
        date={2019},
     journal={Jour. Comb. Th. A},
      volume={162},
       pages={65\ndash 117},
        note={arXiv:1712.08018 [math.CO]},
}

\bib{oota2003quantum}{article}{
      author={Oota, T.},
       title={Quantum projectors and local operators in lattice integrable
  models},
        date={2003},
     journal={Jour. Phys. A},
      volume={37},
      number={2},
       pages={441},
        note={arXiv:hep-th/0304205},
}

\bib{okounkov2003correlation}{article}{
      author={Okounkov, A.},
      author={Reshetikhin, N.},
       title={{Correlation function of Schur process with application to local
  geometry of a random 3-dimensional Young diagram}},
        date={2003},
     journal={Jour. AMS},
      volume={16},
      number={3},
       pages={581\ndash 603},
        note={arXiv:math/0107056 [math.CO]},
}

\bib{pauling1935structure}{article}{
      author={Pauling, L.},
       title={The structure and entropy of ice and of other crystals with some
  randomness of atomic arrangement},
        date={1935},
     journal={Journal of the American Chemical Society},
      volume={57},
      number={12},
       pages={2680\ndash 2684},
}

\bib{Petrov2012GFF}{article}{
      author={Petrov, L.},
       title={{Asymptotics of Uniformly Random Lozenge Tilings of Polygons.
  Gaussian Free Field}},
        date={2015},
     journal={{Ann. Probab.}},
      volume={43},
      number={1},
       pages={1\ndash 43},
      eprint={1206.5123},
        note={arXiv:1206.5123 [math.PR].},
}

\bib{Pak_F_Petrov2020}{article}{
      author={Pak, I.},
      author={Petrov, F.},
       title={Hidden symmetries of weighted lozenge tilings},
        date={2020},
     journal={Electron. J. Combin.},
      volume={27},
      number={3},
       pages={P3.44},
        note={arXiv:2003.14236 [math.CO]},
}

\bib{reshetikhin2010lectures}{incollection}{
      author={Reshetikhin, N.},
       title={Lectures on the integrability of the 6-vertex model},
        date={2010},
   booktitle={{Exact Methods in Low-dimensional Statistical Physics and Quantum
  Computing}},
   publisher={Oxford Univ. Press},
       pages={197\ndash 266},
        note={arXiv:1010.5031 [math-ph]},
}

\bib{Sheffield2008}{article}{
      author={Sheffield, S.},
       title={Random surfaces},
        date={2005},
     journal={Ast\'erisque},
      volume={304},
        note={arXiv:math/0304049 [math.PR]},
}

\bib{Soshnikov2000}{article}{
      author={Soshnikov, A.},
       title={Determinantal random point fields},
        date={2000},
     journal={Russian Mathematical Surveys},
      volume={55},
      number={5},
       pages={923\ndash 975},
        note={arXiv:math/0002099 [math.PR]},
}

\bib{tsilevich2006quantum}{article}{
      author={Tsilevich, N.},
       title={{Quantum inverse scattering method for the q-boson model and
  symmetric functions}},
        date={2006},
     journal={Functional Analysis and Its Applications},
      volume={40},
      number={3},
       pages={207\ndash 217},
        note={arXiv:math-ph/0510073},
}

\bib{wheeler2018hall}{article}{
      author={Wheeler, M.},
      author={Zinn-Justin, P.},
       title={{Hall polynomials, inverse Kostka polynomials and puzzles}},
        date={2018},
     journal={Jour. Comb. Th. A},
      volume={159},
       pages={107\ndash 163},
        note={arXiv:1603.01815 [math-ph]},
}

\bib{yau2013wigner}{inproceedings}{
      author={Yau, Horng-Tzer},
       title={{The Wigner-Dyson-Gaudin-Mehta Conjecture}},
organization={International Press of Boston},
        date={2013},
   booktitle={Notices of the international congress of chinese mathematicians},
      volume={1},
       pages={10\ndash 13},
}

\bib{zinn2000six}{article}{
      author={Zinn-Justin, Paul},
       title={{Six-vertex model with domain wall boundary conditions and
  one-matrix model}},
        date={2000},
     journal={Phys. Rev. E},
      volume={62},
      number={3},
       pages={3411},
        note={arXiv:math-ph/0005008},
}

\bib{zinn2009littlewood}{article}{
      author={Zinn-Justin, P.},
       title={{Littlewood--Richardson coefficients and integrable tilings}},
        date={2009},
     journal={Electron. J. Comb.},
      volume={16},
      number={R12},
       pages={1},
        note={arXiv:0809.2392 [math-ph]},
}

\bib{ZinnJustin20096Vertex}{article}{
      author={Zinn-Justin, P.},
       title={{Six-Vertex, Loop and Tiling models: Integrability and
  Combinatorics}},
        date={2009},
        note={arXiv:0901.0665 [math-ph]},
}

\end{biblist}
\end{bibdiv}

\end{document}